\documentclass[twoside,11pt]{article}

\usepackage{blindtext}

\usepackage{jmlr2e}
\hypersetup{ hidelinks }

\usepackage{mathrsfs}
\usepackage{graphicx}
\usepackage{color}
\usepackage{bm}
\usepackage{algorithm}
\usepackage{algpseudocode}
\usepackage{algpascal}
\usepackage{url}
\usepackage{afterpage}
\usepackage{enumitem} 
\usepackage{multirow}
\usepackage{booktabs}
\usepackage{subfigure}

\usepackage{mathtools}
\usepackage{dsfont}
\usepackage{alphabeta}

\usepackage{xr}

\usepackage{import}

\usepackage{mathtools}
\usepackage{dsfont}
\usepackage[utf8]{inputenc}
\usepackage[greek,english]{babel}
\usepackage{alphabeta}
\usepackage{caption}
\usepackage{comment}

\usepackage{aligned-overset}

\def\cA{\mathcal A}

\def\cE{\mathcal E}

\def\cI{\mathcal I}

\def\cN{\mathcal N}
\def\cP{\mathcal P}
\def\cQ{\mathcal Q}
\def\cR{\mathcal R}
\def\cS{\mathcal S}
\def\cT{\mathcal T}
\def\cU{\mathcal U}

\newcommand{\bA}{{\bf A}}

\newcommand{\bD}{{\bf D}}

\newcommand{\bF}{{\bf F}}

\newcommand{\bI}{{\bf I}}

\newcommand{\bT}{{\bf T}}

\newcommand{\bV}{{\bf V}}

\newcommand{\bX}{{\bf X}}

\newcommand{\bY}{{\bf Y}}

\newcommand{\bbE}{{\mathbb E}}

\newcommand{\bbN}{{\mathbb N}}
\newcommand{\bbP}{{\mathbb P}}
\newcommand{\bbR}{{\mathbb R}}

\newcommand{\Var}{\mbox{{\rm Var}}}

\newcommand{\bSigma}{\bm{\Sigma}}

\newcommand{\bOmega}{\bm{\Omega}}

\newcommand{\iid}{{i.i.d.}}

\newcommand{\bc}{\begin{center}}
\newcommand{\ec}{\end{center}}
\newcommand{\be}{\begin{equation}}
\newcommand{\ee}{\end{equation}}
\newcommand{\ba}{\begin{array}}
\newcommand{\ea}{\end{array}}
\newcommand{\bean}{\setlength\arraycolsep{1pt}\begin{eqnarray*}}
\newcommand{\eean}{\end{eqnarray*}}
\newcommand{\bea}{\setlength\arraycolsep{1pt}\begin{eqnarray}}
\newcommand{\eea}{\end{eqnarray}}
\newcommand{\ben}{\begin{enumerate}}
\newcommand{\een}{\end{enumerate}}
\newcommand{\bed}{\begin{itemize}}
\newcommand{\eed}{\end{itemize}}

\DeclareMathOperator*{\argmax}{argmax}
\DeclareMathOperator*{\argmin}{argmin}

\newcommand{\rmd}{\mathrm{d}}

\newcommand{\thetaMAP}[1][t]{\widehat{\theta}_{#1}}
\newcommand{\thetaBest}[1][t]{\theta_{#1}^{\ast}}
\newcommand{\fullthetaBest}[1][t]{\theta_{1:#1}^{\ast}}
\newcommand{\fullthetaMLE}[1][t]{\widehat{\theta}_{1:#1}^{{\textnormal {\rm \texttt{ML}}}}}
\newcommand{\fullthetapMLE}[1][t]{\widehat{\theta}_{1:#1}}

\newcommand{\FisherMAP}[1][t]{\widetilde{\bF}_{#1, \thetaMAP[#1]}}
\newcommand{\FisherBest}[1][t]{\widetilde{\bF}_{#1, \thetaBest[#1]}}
\newcommand{\FisherTilde}[2][t]{\widetilde{\bF}_{#1, #2}}
\newcommand{\Fisher}[2][t]{\bF_{#1, #2}}
\newcommand{\FullFisher}[2][t]{\bF_{1:#1, #2}}
\newcommand{\FullFisherTilde}[2][t]{\widetilde{\bF}_{1:#1, #2}}

\newcommand{\LA}{\textnormal {\rm \texttt{LA}}}

\newcommand{\TV}{\textnormal {\rm \texttt{TV}}}
\newcommand{\KL}{\textnormal {\rm \texttt{KL}}}
\newcommand{\local}{{\rm local}}
\newcommand{\tail}{{\rm tail}}
\newcommand{\symmPD}{\cS_{\texttt{++}}^{p}}

\newcommand{\eff}{{\rm eff}}
\newcommand{\est}{{\rm est}}
\newcommand{\bias}{{\rm bias}}

\newcommand{\ML}{\textnormal {\rm \texttt{ML}}}

\newcommand{\scrE}{\bm{\mathscr{E}}}

\usepackage{lastpage}
\jmlrheading{23}{2025}{1-\pageref{LastPage}}{5/25; Revised 12/25}{2/26}{25-0989}{Jeyong Lee, Junhyeok Choi and Minwoo Chae}

\ShortHeadings{Online Bernstein--von Mises theorem}{Lee, Choi and Chae}
\firstpageno{1}

\begin{document}

\title{Online Bernstein--von Mises theorem}

\author{\name Jeyong Lee \email jylee1024@postech.ac.kr \\
       \addr Department of Industrial and Management Engineering\\
       Pohang University of Science and Technology\\
       Pohang, Gyeongbuk 37673, South Korea
       \AND
       \name Junhyeok Choi \email cjunh4810@postech.ac.kr \\
       \addr Department of Industrial and Management Engineering\\
       Pohang University of Science and Technology\\
       Pohang, Gyeongbuk 37673, South Korea
       \AND
       \name Minwoo Chae\thanks{Corresponding author.} \email mchae@postech.ac.kr \\
       \addr Department of Industrial and Management Engineering\\
       Pohang University of Science and Technology\\
       Pohang, Gyeongbuk 37673, South Korea       
       }

\editor{Daniel Roy}

\maketitle

\begin{abstract}
Online learning is an inferential paradigm in which parameters are updated incrementally from sequentially available data, in contrast to batch learning, where the entire dataset is processed at once. In this paper, we assume that mini-batches from the full dataset become available sequentially. The Bayesian framework, which updates beliefs about unknown parameters after observing each mini-batch, is naturally suited for online learning. At each step, we update the posterior distribution using the current prior and new observations, with the updated posterior serving as the prior for the next step. However, this recursive Bayesian updating is rarely computationally tractable unless the model and prior are conjugate. When the model is regular, the updated posterior can be approximated by a normal distribution, as justified by the Bernstein--von Mises theorem. We adopt a variational approximation at each step and investigate the frequentist properties of the final posterior obtained through this sequential procedure. Under mild assumptions, we show that the accumulated approximation error becomes negligible once the mini-batch size exceeds a threshold depending on the parameter dimension. As a result, the sequentially updated posterior is asymptotically indistinguishable from the full posterior.
\end{abstract}


\section{Introduction} \label{sec:intro}

Online learning is an inferential paradigm in which parameters are updated sequentially as new data arrive. Unlike batch learning, which processes a fixed dataset all at once, online learning incrementally adjusts parameters with each new observation. This approach is particularly well-suited for analyzing streaming data, where data become available sequentially. Moreover, even for fixed datasets, online learning can be an advantageous approach, as its algorithms are often significantly more computationally efficient than batch learning methods. Over the past few decades, substantial progress has been made in the development of online learning techniques for various statistical models, such as 
topic models \citep{hoffman2010online, wang2011online, kim2016online}, matrix factorization \citep{mairal2010online}, survival analysis \citep{xue2020online, wu2021online, choi2025online} and quantile regression \citep{chen2019quantile, lee2024fast}.
We provide a comprehensive review of online learning methodologies in Section \ref{sec:online_learning_review}.

The Bayesian philosophy, which updates beliefs about an unknown parameter after observing data, aligns well with the online learning paradigm. As new data arrive, one can update the posterior distribution using the current prior and the new data, with the updated posterior serving as the new prior. However, unless the model-prior pair is conjugate, this straightforward Bayesian approach is rarely computationally tractable, particularly when dealing with complex hierarchical models. Consequently, Bayesian approaches to online learning typically include an additional step to approximate the updated posterior with a simpler one, ensuring computational feasibility for subsequent updates \citep{opper1999bayesian, solla1999online}. A common choice for this additional step is variational approximation \citep{broderick2013streaming, lin2013online, nguyen2018variational}. 

In the Bayesian online learning framework described above, the posterior distribution must be repeatedly approximated as new data arrive. While the approximation error at each step may be negligible, the cumulative error over multiple updates may not be. Thus, a central concern is whether the accumulated approximation error remains small. In this paper, we provide a rigorous theoretical analysis of this issue by investigating the frequentist properties of the sequentially updated and approximated posterior distributions.

To set the scene, let $\cP = \{ \bbP_{\theta} : \theta \in \Theta \}$ be a statistical model indexed by $\theta \in \Theta$, where $\Theta \subset \bbR^p$ is the parameter space. Let $\bD = (Y_1, \ldots, Y_N)$ be the complete set of observations, and let $\Pi \in \cQ$ denote the (initial) prior on $\Theta$, where $\cQ$ is a class of priors. The corresponding posterior distribution is $\Pi(\cdot \mid \bD)$, which we refer to as the full posterior. We assume that the data become available sequentially in the order $\bD_1, \bD_2, \ldots, \bD_T$, where $\bD_t = (Y_{N_{t-1}+1}, Y_{N_{t-1}+2}, \ldots, Y_{N_t})$ represents the $t$-th mini-batch, with $N_0 = 0$. For convenience, we assume that each mini-batch has the same size, denoted by $n$, so that $N_t = nt$ for every $t$. 

With $\Pi_0 = \Pi$, the Bayesian online learning framework considered in this paper consists of the following inductive steps. For $t \geq 1$, let $\widetilde\Pi_t(\cdot \mid \bD_t)$ be the posterior distribution obtained by updating the prior $\Pi_{t-1}$ with the data $\bD_t$ using Bayes' formula:
\begin{align} \label{eq:online-Bayes}
    \widetilde{\Pi}_{t}( \cA \mid \bD_t) = \dfrac{
    \int_{\cA} \exp\left\{ L_{t}(\theta) \right\} \rmd \Pi_{t-1} (\theta)
    }{
    \int_{\Theta} \exp\left\{ L_{t}(\theta) \right\} \rmd \Pi_{t-1} (\theta)
    }, \quad \text{for any measurable } \ \cA \subset \Theta,
\end{align}
where $L_t(\cdot)$ is the log-likelihood corresponding to the $t$-th mini-batch $\bD_t$. Next, we approximate the updated posterior $\widetilde\Pi_t(\cdot \mid \bD_t)$ by projecting it onto $\cQ$, the space of prior distributions. Since we use a variational approximation, the approximated posterior is given by
\begin{align} \label{eq:KL-projection}
    \Pi_{t} = \argmin_{Q \in \cQ} K\left( Q; \: \widetilde{\Pi}_{t}(\cdot \mid \bD_t) \right),
\end{align}
where $K(P; Q)$ denotes the Kullback--Leibler (KL) divergence, defined as
\begin{align} \label{def:KL_divergence}
    K(P; Q) = \int \log\left( \dfrac{\rmd P}{\rmd Q} \right) \rmd P.
\end{align}
Thus, we obtain the sequence $(\Pi_t)_{t\leq T}$ of approximated posterior distributions, where $\Pi_t \in \cQ$ for every $t$.

The main goal of this paper is to provide sufficient conditions under which the accumulated approximation error in the online learning process remains negligible, ensuring that the sequentially updated posterior distribution $\Pi_T(\cdot)$ is nearly identical to the full posterior distribution $\Pi( \cdot \mid \bD)$. Although this is intuitively obvious when $\cQ$ is sufficiently large, a mathematically rigorous analysis is challenging even for $T=2$.

We assume that $Y_1, \ldots, Y_N$ are independent but not necessarily identically distributed, and there is the true parameter $\theta_0 \in \Theta$ generating data. This is a frequentist assumption that is commonly adopted in the Bayesian asymptotics literature. We also assume that the parametric model $\cP$ is regular in the sense that the log-likelihood function is locally approximately quadratic around $\theta_0$. Under this regularity condition, the full posterior distribution $\Pi(\cdot \mid \bD)$ is approximately normal, centered around the linear efficient estimator, with variance given by the inverse Fisher information matrix, as stated by the celebrated Bernstein--von Mises (BvM) theorem. Accordingly, the class of normal distributions serves as a natural choice for $\cQ$.

Our main theorem guarantees, under suitable assumptions, that the total variation distance between the full posterior $\Pi(\cdot \mid \bD)$ and the sequentially updated posterior $\Pi_T(\cdot)$ is sufficiently small with high probability. When the dimension $p$ of $\theta_0$ is fixed, the required condition boils down to $n \gg (\log N)^4$. The main results are formulated in a non-asymptotic framework and also cover the case where $p$ diverges at a polynomial rate with respect to the sample size $N$. In particular, when $\log p \asymp \log N$, $n \gg p^3$ is sufficeint; see Theorem \ref{thm:onlineBvM} for detailed statements. We refer to this result as the \textit{online BvM theorem}. It ensures the frequentist validity and asymptotic efficiency of statistical inferences, such as point estimation and uncertainty quantification, based on $\Pi_T$.

To prove the main results, sharp non-asymptotic results on the quadratic approximation of the log-likelihood function are crucial, for which substantial progress has been made in recent studies; see \cite{spokoiny2012parametric}, \cite{spokoiny2017penalized} and \cite{katsevich2024approximation}. In addition to these techniques, we develop several novel techniques, including bounding the KL divergence between the updated posterior and the approximated variational posterior, as well as handling accumulated errors. It is also important to note that, while the prior considered in the standard BvM theorem is typically a flat prior, the priors used in the intermediate steps of the online learning process are highly informative.

To the best of our knowledge, no theoretical work has studied the theory of Bayesian online learning with the level of rigor presented in this paper. In contrast, rigorous theoretical frameworks have recently been developed for frequentist online learning based on empirical risk minimization using one-pass stochastic gradient descent (SGD) algorithms. Specifically, \citet{toulis2017SGD} demonstrated that the \textit{implicit SGD} estimator is asymptotically efficient, admitting a limiting distribution with optimal variance. In addition, \citet{chen2020SGD} proposed the \textit{batch-means SGD} estimator, which yields asymptotically valid confidence intervals.

Since the one-pass SGD approaches described above process a single datum at each step, it is natural to ask whether similar results can be achieved with $T = N$. We conjecture that the online BvM theorem does not hold in this case. Although we do not have a formal proof, our numerical experiments in Section \ref{sec:numerical_experiments} provide supporting evidence. Specifically, we observe that the relative efficiency of the point estimator---the posterior mean obtained from the sequentially updated $\Pi_t$--- when $T = N$, compared to a batch estimator, is significantly greater than 1, indicating that it is not asymptotically efficient. We believe the online BvM theorem fails to hold because our results rely on the quadratic approximation of the log-likelihood, which is valid only when the mini-batch sample size $n$ is much larger than the parameter dimension $p$. Given the positive results of one-pass algorithms in \citet{toulis2017SGD} and \citet{chen2020SGD}, it would be interesting to develop an online Bayesian procedure with appropriate algorithmic modifications that ensure the validity of the online BvM theorem.

The remainder of this paper is organized as follows. In the next subsections, we provide a comprehensive review of online learning methods and introduce elementary notations. Section \ref{sec:setup} describes the basic setup and key definitions related to Bayesian online learning. Section \ref{sec:overview} provides a high-level overview of the online BvM theorem and presents a motivating example illustrating where our theory applies.
Sections \ref{sec:variational_posterior} and \ref{sec:pMLE} present the variational approximation of the sequentially updated posterior and the penalized M-estimation, respectively. Section \ref{sec:eigenvalue_analysis} provides a non-asymptotic analysis of several regularity quantities. Section \ref{sec:full_posterior} establishes the BvM theorem for the full posterior. Our main results concerning the online BvM theorem are presented in Section \ref{sec:online_variational_posterior}. Numerical results supporting our theory are provided in Section \ref{sec:numerical_experiments}. Concluding remarks follow in Section \ref{sec:discussion}, and all proofs and additional technical details are deferred to the Appendix.

\subsection{Related works} \label{sec:online_learning_review}
\subsubsection{Frequentist methods}
In this sub-section, we provide a brief introduction to the recent advancements in online statistical inference by surveying theoretical investigations of SGD-type estimators. Due to its computational advantages, the SGD estimator \citep{robbins1951stochastic} and its variants have been extensively studied in the frequentist online learning literature. As pioneering works, \citet{ruppert1988efficient} and \citet{polyak1992acceleration} independently proposed averaging of SGD iterates, with \citet{polyak1992acceleration} establishing the asymptotic normality of the averaged SGD estimator. More recently, \citet{toulis2017SGD} introduced the implicit SGD estimator, which exhibits stable performance in finite samples and is asymptotically normal under a suitably specified learning rate. 

However, in the online learning setting, the asymptotic normality alone does not guarantee that confidence intervals can be constructed. This is because the asymptotic covariance matrix is typically computed using the entire dataset, which is not available when data arrives sequentially. For example, if our interest lies in the Fisher information matrix evaluated at an estimator $\widehat{\theta}_{T}$, we would need to evaluate $\sum_{t=1}^{T} \partial^{2}L_{t}(\theta)/ \partial \theta \partial \theta^{\top}$ at $\widehat{\theta}_{T}$, where $L_{t}(\theta)$ denotes the log-likelihood function for the $t$-th mini-batch $\bD_{t}$. 
In a batch learning setup, the computation of this matrix is straightforward. However, in an online learning setting, the early mini-batches $\bD_{1}, ..., \bD_{T-1}$ are discarded once $\bD_{T}$ arrives, while $\widehat{\theta}_{T}$ is only available at time $T$. This limitation hinders the applicability of the conventional batch learning approach.

As a result, constructing tractable and asymptotically valid confidence intervals has become an important topic in online learning literature. To address this challenge, \citet{chen2020SGD} proposed a \textit{batch-means} method that aims to estimate the limiting covariance matrix, and \citet{zhu2023online} improved upon this method by eliminating the need for a priori knowledge of the total sample size $N$. 
In another line of work, \cite{lee2022fast} developed computationally efficient confidence intervals by applying a functional central limit theorem to SGD iterates, which further extended to quantile regression \citep{lee2024fast}. Alternatively, instead of directly estimating the limiting covariance matrix, \citet{fang2018online} constructed tractable confidence intervals by employing a bootstrap resampling procedure based on randomly perturbed SGD updates.

\subsubsection{Bayesian methods}

In recent years, significant methodological advancements have been made in Bayesian online learning. Several studies have proposed various approximation methods for updating the posterior distribution sequentially. For example, \citet{broderick2013streaming} introduced a general framework for large-scale and streaming data, and \citet{nguyen2018variational} adapted the online VB algorithm for neural networks. In the realm of nonparametric models, \citet{lin2013online} developed a VB algorithm for Dirichlet process mixture models, while \citet{jeong2023online} employed an assumed density filtering (ADF) approach for similar tasks. More recently, \citet{lambert2022recursive}\footnote{The procedure proposed in \citet{lambert2022recursive} coincides with the setup that will be introduced in Section \ref{sec:setup_detail}.} and \citet{lambert2023limited} proposed computationally efficient online VB approximations using the Gaussian variational family.

These works demonstrate that projecting onto a tractable class $\cQ$ of distributions is a popular strategy due to its computational feasibility. Among the proposed methods, both ADF and VB are prominent. However, ADF requires a moment-matching step for each mini-batch, which can be computationally costly for certain hierarchical models (e.g., topic models), whereas VB can avoid this step \citep{broderick2013streaming}. Consequently, VB and its variants have emerged as the preferred approaches in recent Bayesian online learning literature \citep{lin2013online, bui2017streaming, nguyen2018variational, lambert2022recursive, choi2025online}.


\subsection{Notations}

\begin{table}[hbt!]
\centering
\caption{Important notations}
\label{table:notations_main}
\renewcommand{\arraystretch}{1.3} 
\begin{tabular}{c c || c c }
\hline
\textbf{Notation} & \textbf{Location} & \textbf{Notation} & \textbf{Location} \\
\hline\hline
$\Theta (\theta, \bF, r), \ \Theta (\bF, r)$ 
  & \eqref{def:local_set}, \eqref{def:local_set_around_zero}
  & $\Fisher[t]{\theta}, \ \FisherTilde[t]{\theta}$ 
  & \eqref{def:Fisher_info_quantities} \\
$r_{\eff, t}, \ \widetilde{r}_{\eff, 1:t}, \ r_{\eff, 1:t}$
  & \eqref{def:estimation_quantities}, \eqref{def:batch_quantity}
  & $\scrE_{\est, 1}, \ \scrE_{\est, 2}$
  & \eqref{def:Gamma_n_event}, \eqref{def:Gamma_n2_event} \\
$p_{\eff, t}, \ \lambda_{t}$
  & \eqref{def:estimation_quantities}
  & $r_{\LA}, \ \widehat{\tau}_{3, t}, \ \widehat{\tau}_{4, t}$
  & \eqref{def:hat_tau_t_radius} \\
$M_n$
  & \eqref{assume:A1_2}, \eqref{assume:A1ast_2}
  & $\widehat{\tau}_{3, t, r}$
  & \eqref{def:tau_3tr} \\
$p_{\ast}$
  & (\textbf{EX})
  & $\tau_{3, t}^{\ast}$
  & \eqref{def:tau_3ast} \\
$d_{V}(\cdot,\cdot), \ K(\cdot \ ;\cdot)$
  & \eqref{def:TV_distance}, \eqref{def:KL_divergence}
  & $\epsilon_{n, t, \TV}, \ \epsilon_{n, t, \KL}$
  & Theorem \ref{thm:LA_TV}, \ref{thm:LA_KL} \\
$\thetaMAP, \ \thetaBest$
  & \eqref{def:maximizer_quantities_minibatch}
  & $K_{\min}, \ K_{\max}$
  & \eqref{assume:A2_1}, \eqref{assume:A2_2}, \eqref{assume:A2_3} \\
$\fullthetapMLE, \ \fullthetaMLE, \ \fullthetaBest$
  & \eqref{def:batch_maximizer_quantities}
  & $K_{\rm low}, \ K_{\rm up}$
  & Proposition \ref{prop:eigenvalue_order_main} \\
\hline
\end{tabular}
\end{table}

For two real numbers $a$ and $b$, $a \vee b$ and $a \wedge b$ denote the maximum and minimum of $a$ and $b$, respectively.
For two positive sequences $(a_n)$ and $(b_n)$, $a_n \lesssim b_n$ (or $a_n = O(b_n)$) means that $a_n \leq C b_n$ for some constant $C \in (0, \infty)$.
Also, $a_n \asymp b_n$ indicates that $a_n \lesssim b_n$ and $b_n \lesssim a_n$.
The notation $a_n \ll b_n$ (or $a_n = o(b_n)$) implies that $a_n / b_n \rightarrow 0$ as $n \rightarrow \infty$. We use the standard $O_P$ and $o_p$ notation for stochastic order symbols; see \citet{van2000asymptotic}.
With a slight abuse of notation, for $\nu_n \in \bbR^p$ satisfying $\| \nu_n \|_2 = o(1)$, we write $\nu_n = o(1)$.
For $1 \leq q \leq \infty$, $\|\cdot\|_{q}$ indicates the $\ell_q$-norm of a vector. 
For $m \in \bbN$, let $[m] = \{1, 2, ..., m\}$.

For $z = (z_j)_{j \in [p]} \in \bbR^{p}$ and $k \geq 2$, let
\begin{align*}
     z^{\otimes k} = \left( z_{i_1} \times ... \times z_{i_k} \right)_{i_1, ..., i_k \in [p]} \in \bbR^{p^{k}}.
\end{align*}
For two $k$-order tensors $\bA = (A_{i_1, ..., i_k})_{i_1, ..., i_k \in [p]} \in \bbR^{p^{k}}$ and $\mathbf{B} = (B_{i_1, ..., i_k})_{i_1, ..., i_k \in [p]} \in \bbR^{p^{k}}$, let
\begin{align*}
    \langle \bA, \mathbf{B} \rangle = \sum_{i_1, ..., i_k \in [p]} A_{i_1, ..., i_k} B_{i_1, ..., i_k}.
\end{align*}

Let $\cS_{\texttt{++}}^{p}$ denote the set of all $p \times p$-dimensional symmetric positive definite matrices. Let $\bI_{p} \in \bbR^{p \times p}$ denote the identity matrix.
For a matrix $\mathbf{A} = (a_{ij}) \in \bbR^{n \times p}$, let $\lambda_{\operatorname{min}}(\bA)$ and $\lambda_{\operatorname{max}}(\bA)$ denote the smallest and largest singular values of $\bA$, respectively.
For simplicity, $\| \bA \|_{2}$ will often be used interchangeably with $\lambda_{\operatorname{max}}(\bA)$.
Let $\| \bA \|_{\rm F} = (\sum_{ij} a_{ij}^{2})^{1/2}$ be the Frobenius norm.
For two distinct matrices $\bA, \mathbf{B} \in \bbR^{n \times n}$, $\bA \succeq \mathbf{B}$ means $\bA - \mathbf{B}$ is positive semi-definite matrix.
For a $k$-th order tensor $\bA = (A_{i_1, ..., i_k})_{i_1, ..., i_k \in [p]} \in \bbR^{p^{k}}$, define the operator norm of $\bA$ by
\begin{align*}
    \left\| \bA \right\|_{\rm op} = \sup_{u_{1}, ..., u_{k} \in \cU} \left| \langle \bA, u_{1} \otimes ... \otimes u_{k} \rangle \right|,
\end{align*}
where $\cU = \{ u \in \bbR^{p} : \| u \|_{2} = 1 \}$.

For $\sigma^{2} \geq 0$, a random vector $X \in \bbR^{p}$ is said to be $\text{SubG}(\sigma^2)$ if
\begin{align*}
    \log \bbE \exp \big( \alpha^{\top} (X - \bbE X) \big) \leq \sigma^{2} \left\| \alpha \right\|_{2}^{2}/2, \quad \forall \alpha \in \bbR^{p}
\end{align*}
With the convention, we consider that $\inf \emptyset = \infty$ in this paper.
Throughout our paper, the constants $c_1, c_2, ...$ may vary depending on the context.

\section{Preliminaries} \label{sec:setup}

\subsection{Setup for online learning} \label{sec:setup_detail}

In this subsection, we precisely formulate the Bayesian online learning procedure briefly introduced in the introduction. Suppose the entire dataset $\bD = (Y_i)_{i \in [N]}$ consists of independent (not necessarily identically distributed) observations $Y_1, \ldots, Y_N$. The dataset is partitioned into $T$ mini-batches of equal size $n$, so that the total sample size up to the $t$-th mini-batch is given by $N_t = nt$ for all $t \in \{0, 1, 2, \ldots, T\}$. The $t$-th mini-batch and the collection of all samples up to the $t$-th mini-batch are denoted as
\begin{align*}
    \bD_t = (Y_{N_{t-1} + 1}, Y_{N_{t-1} + 2}, ..., Y_{N_{t}}), 
    \quad \text{and} \quad
    \bD_{1:t} = (\bD_1, \bD_2, ..., \bD_{t}) = (Y_i)_{i \in [N_{t}]}.
\end{align*}

For $i \in [N]$, let $p_{\theta, i}(\cdot)$ be the probability density function for $Y_i$ parametrized by $\theta \in \Theta \subset \bbR^p$, and let $\ell_{\theta, i}(y) = \log p_{\theta, i} (y)$ be the log density. 
For $t \in [T]$, let 
\begin{align*}
    L_{t}(\theta) = L_{t}(\theta; \bD_t) = \sum_{i = N_{t-1} + 1}^{N_t} \ell_{\theta, i}(Y_i).
\end{align*}
Let $\bbP_{\theta}^{(N)}$ denote the joint probability measure corresponding to the product density function $(y_1, y_2, ..., y_N)$ $\mapsto \prod_{i = 1}^{N} p_{\theta, i}(y_i)$. 
We assume that the model is well-specified; that is, $\bD$ is generated from $\bbP_{\theta_0}^{(N)}$ for some \textit{true parameter} $\theta_0 \in \Theta$. 
Let $\bbP_{0, t}(\cdot)$ denote the joint probability measure corresponding to the product density 
\begin{align*}
(y_{N_{t-1} + 1}, y_{N_{t-1} + 2}, ..., y_{N_{t}}) \mapsto \prod_{i = N_{t-1} + 1}^{N_t} p_{\theta_0, i}(y_i),
\end{align*}
and let $\bbE_{t}$ denote the expectation under $\bbP_{0, t}$. Since the model is assumed to be well-specified, we have $\theta_0 = \argmax_{\theta \in \Theta} \bbE_{t} L_{t}(\theta)$ for all $t \in [T]$.

Now, we introduce our online learning procedure.
Let $\cQ$ be the collection of all Gaussian measures with nonsingular covariance matrices.
Given an initial prior $\Pi_0$ and the log-likelihood $L_{t}(\theta)$ for the $t$-th mini-batch data $\bD_t$, we iteratively define the posterior distribution $\widetilde{\Pi}_{t}(\cdot \mid \bD_t)$ and the corresponding variational approximation $\Pi_t(\cdot)$ as described in \eqref{eq:online-Bayes} and \eqref{eq:KL-projection}.
We denote by $\widetilde{\pi}_{t}(\cdot \mid \bD_t)$ and $\pi_{t}(\cdot)$ the density functions of $\widetilde{\Pi}_{t}(\cdot \mid \bD_t)$ and $\Pi_{t}(\cdot)$, respectively.

In particular, we consider a normal distribution $\cN(\mu_0, \bOmega_0^{-1})$ as the initial prior $\Pi_0$, where $\mu_0 \in \bbR^{p}$ and $\bOmega_0 \in \symmPD$. For $t \in \left\{ 0, 1, ..., T \right\}$, we denote 
\begin{align} \label{def:VB_param}
    \Pi_{t} = \cN(\mu_t, \bOmega_t^{-1})
\end{align}
for $\mu_t \in \bbR^{p}$ and $\bOmega_t \in \symmPD$.

\subsection{Definitions}

For a four times differentiable function $f : \bbR^{p} \rightarrow \bbR$ with $\theta \mapsto f(\theta)$, let
\begin{align*}
    &\nabla f(\theta) = \left( \dfrac{\partial}{\partial \theta_{i_1}} f(\theta) \right)_{i_1 \in [p]} \in \bbR^{p},  \quad 
    \nabla^2 f(\theta) = \left( \dfrac{\partial^2}{\partial \theta_{i_1} \partial \theta_{i_2}} f(\theta) \right)_{i_1, i_2 \in [p]} \in \bbR^{p \times p}, \\
    &\nabla^3 f(\theta) = \left( \dfrac{\partial^3}{\partial \theta_{i_1} \partial \theta_{i_2} \partial \theta_{i_3} } f(\theta) \right)_{i_1, i_2, i_3 \in [p]} \in \bbR^{p \times p \times p},  \\
    &\nabla^4 f(\theta) = \left( \dfrac{\partial^4}{\partial \theta_{i_1} \partial \theta_{i_2} \partial \theta_{i_3} \partial \theta_{i_4} } f(\theta) \right)_{i_1, i_2, i_3, i_4 \in [p]} \in \bbR^{p \times p \times p \times p}.
\end{align*}

To ensure an accurate approximation of the posterior, we will impose several smoothness conditions on the log-likelihood $L_{t}$, requiring that it is at least four times continuously differentiable with probability 1. 
For such a differentiable $L_t(\cdot)$, we introduce some notations used for posterior analysis. 
For $t \in [T]$ and $\theta \in \Theta$, define
\begin{align} \label{def:Fisher_info_quantities}
    \bF_{t, \theta} =  -\nabla^2 L_{t}(\theta) \in \bbR^{p \times p}, \quad 
    \widetilde{\bF}_{t, \theta} =  \bOmega_{t-1} +  \bF_{t, \theta},
\end{align}
where $\bOmega_{t-1}$ is defined in \eqref{def:VB_param}. 
For $\theta_{\rm c} \in \bbR^{p}$, $r \geq 0$ and $\bF \in \symmPD$, let 
\begin{align} \label{def:local_set}
    \Theta \left( \theta_{\rm c}, \bF, r \right) = \left\{ \theta \in \Theta : \big\| \bF^{1/2} \left( \theta - \theta_{\rm c} \right) \big\|_{2} \leq r  \right\}
\end{align}
denote the local elliptical vicinity of $\theta_{\rm c}$. For notational simplicity, let
\begin{align} \label{def:local_set_around_zero}
    \Theta \left(\bF, r \right) = \Theta \left( 0, \bF, r \right). 
\end{align}

Given the prior distribution $\Pi_{t-1}$, which reflects the information in $\bD_{1:t-1}$, we define the penalized log-likelihood function as
\begin{align} \label{def:penalized_likelihood}
    \widetilde{L}_{t}(\theta) = L_{t}(\theta) - \dfrac{1}{2} \left\| \bOmega_{t-1}^{1/2} \left( \theta - \mu_{t-1} \right) \right\|_{2}^{2}.
\end{align}
For each $t \in [T]$, the penalized maximum likelihood estimator (pMLE) and its population version are defined as
\begin{align} \label{def:maximizer_quantities_minibatch}
    \thetaMAP = \argmax_{\theta \in \Theta} \widetilde{L}_{t}(\theta), \quad  \thetaBest = \argmax_{\theta \in \Theta} \bbE_{t} \widetilde{L}_{t}(\theta).
\end{align}
From the standard M-estimation theory, it is expected that $\thetaMAP$ is close to $\thetaBest$; this will be addressed in Section \ref{sec:pMLE}. It is noteworthy that $\thetaMAP$ need not converge to $\theta_0$ because, in general, $\thetaBest \neq \theta_0$ due to the penalization term. Therefore, for $\thetaMAP$ to be a reliable estimator of $\theta_0$, the \textit{bias} $\| \thetaBest - \theta_0 \|_2$ should be sufficiently small. Roughly speaking, if $\| \bOmega_{t-1}^{1/2} (\theta_0  - \mu_{t-1}) \|_2$ is not too large, one can expect that $\| \thetaBest - \theta_0 \|_2$ will also be small. The magnitude of this bias will be addressed in Proposition \ref{prop:eigenvalue_order_main}, followed by a refined analysis in \eqref{eqn:refined_bias}.

\subsection{Smoothness condition} \label{sec:smoothness}
In this subsection, we introduce an essential tool for characterizing the smoothness structure of (concave) functions---namely, \textit{self-concordance}-type condition. For a simple illustration, consider a $3$-times differentiable function $f : \bbR \rightarrow \bbR$ satisfying
\begin{align*}
    | f'''(\theta) | \leq 2 f''(\theta)^{3/2}, \quad  \forall \theta \in \bbR,
\end{align*}
where $f''$ and $f'''$ denote the second and third derivative of $f(\cdot)$, respectively. Then, one can prove that
\begin{align*}
    \big( 1 - \delta(\theta, \eta) \big) f''(\theta) 
    \leq 
    f''(\eta) 
    \leq 
    1/\big( 1 - \delta(\theta, \eta) \big)^2 f''(\theta), \quad 
    \forall \theta, \eta \in \bbR \ \text{ with } \ \delta(\theta, \eta) \leq 1
\end{align*}
where $\delta(\theta, \eta) = |f''(\theta)^{1/2} (\eta - \theta)|$; see Theorem 5.1.7 in \citet{nesterov2018lectures} for a general statement.
Intuitively, this condition ensures that the second derivative of $f$ does not change too abruptly, which in turn facilitates the analysis of local quadratic approximations.

Recently, the self-concordance condition has been invoked in the statistical literature. Originally introduced in the context of convex optimization problems \citep{nesterov1994interior}, it has since been adapted for statistical applications \citep{bach2010self, ostrovskii2021finite, spokoiny2025inexact}. These conditions facilitate non-asymptotic quadratic approximation theory, leading to sharp theoretical analyses. For this purpose, we rely heavily on the smoothness structures induced by the self-concordance condition.

Now, we formalize the notion of smoothness.
Let $f : \Theta \rightarrow \bbR$ be a four-times differentiable function.
For some $\bF \in \symmPD$ and $\tau_3, r \geq 0$, we say that $f$ satisfies the \textit{third order smoothness} at $\theta \in \Theta$ with parameters $(\tau_3, \bF, r)$ if 
\begin{align} \label{def:3_smooth} 
    \sup_{u \in \Theta (\bF, r)} \sup_{z \in \bbR^{p}} 
    \dfrac{
    \left| \langle \nabla^3 f(\theta + u), z^{\otimes 3} \rangle \right|
    }{
    \left\| \bF^{1/2} z \right\|_{2}^{3}
    }
    \leq \tau_3.
\end{align}
Similarly, we say that $f$ satisfies the \textit{fourth order smoothness} at $\theta \in \Theta$ with parameters $(\tau_4, \bF, r)$ if
\begin{align} \label{def:4_smooth} 
    \sup_{u \in \Theta (\bF, r)} \sup_{z \in \bbR^{p}} 
    \dfrac{
    \left| \langle \nabla^4 f(\theta + u), z^{\otimes 4} \rangle \right|
    }{
    \left\| \bF^{1/2} z \right\|_{2}^{4}
    }
    \leq \tau_4.
\end{align}
By Theorem 2.1 in \cite{zhang2012best}, each left-hand side in \eqref{def:3_smooth} and \eqref{def:4_smooth} is equal to the following expression:
\begin{align} \label{def:3_4_smooth_equiv}  
    \sup_{u \in \Theta (\bF, r)} 
    \sup_{z_{1}, ..., z_{k} \in \bbR^{p}} 
    \dfrac{
    \left| \langle \nabla^{k} f(\theta + u), z_{1} \otimes \cdots \otimes z_{k} \rangle \right|
    }{
    \left\| \bF^{1/2} z_{1} \right\|_{2} 
    \times \cdots \times
    \left\| \bF^{1/2} z_{k} \right\|_{2}
    }, 
    \quad k \in \{3, 4\}.
\end{align}

It is noteworthy that $\tau_3$ and $\tau_4$ can typically be chosen to be sufficiently small. More specifically, for $\theta \in \Theta$, $\bF \in \symmPD$ and $r \geq 0$, we prove in Lemma \ref{lemma:tech_smooth_tau_bound} that \eqref{def:3_smooth} and \eqref{def:4_smooth} hold with 
\begin{align} \label{eqn:tau_bound_theory}
    \tau_{3} = \lambda_{\min}^{-3/2} ( \bF )  \sup_{u \in \Theta ( \bF, r )} \left\| \nabla^{3} f(\theta + u) \right\|_{\rm op} \text{ and } \
    \tau_{4} = \lambda_{\min}^{-2} ( \bF )  \sup_{u \in \Theta ( \bF, r )} \left\| \nabla^{4} f(\theta + u) \right\|_{\rm op},
\end{align}
respectively. To aid understanding, suppose that $f = L_{t}$, $\bF = -\nabla^2 L_{t}(\theta)$ with 
\begin{align*} 
    \lambda_{\min}(\bF) \asymp n, \quad 
    \sup_{u \in \Theta ( \bF, r )} \big( \left\| \nabla^{3} L_{t}(\theta + u) \right\|_{\rm op} \vee \left\| \nabla^{4} L_{t}(\theta + u) \right\|_{\rm op} \big) \asymp n.
\end{align*}
Then it follows that $\tau_3 \asymp n^{-1/2}$ and $\tau_4 \asymp n^{-1}$. Many of the subsequent analyses in this paper rely on the assumption that $\tau_3$ and $\tau_4$ are small enough. This can be achieved when $\lambda_{\min}(\bF)$ is large enough relative to the local third and fourth operator norms of $L_t$ as described in \eqref{eqn:tau_bound_theory}. In this sense, $\lambda_{\min}(\bF)$ behaves as an \textit{effective sample size} in our analysis. A detailed discussion of this point is deferred to Section \ref{sec:eigenvalue_analysis}.

\section{High-level overview of the online BvM theorem} \label{sec:overview}

In this section, we provide a high-level overview of the online BvM theorem and provide a motivating example illustrating the applicability of our theory.

\subsection{Proof framework and our contributions}
For $t \in [T]$, we define the $t$-th \textit{full posterior} by 
\begin{align} \label{def:full_posterior}
    \Pi( \cA \mid \bD_{1:t}) = \dfrac{
    \int_{\cA} \exp\left\{ L_{1:t}(\theta) \right\} \rmd \Pi_{0} (\theta)
    }{
    \int_{\Theta} \exp\left\{ L_{1:t}(\theta) \right\} \rmd \Pi_{0} (\theta)
    } \quad  \text{ for any measurable } \cA \subset \Theta,
\end{align}
where $L_{1:t}(\theta) = \sum_{s=1}^{t} L_{s}(\theta)$. 
As mentioned in Section \ref{sec:intro}, we aim to show that
\begin{align} \label{eqn:intro_claim}
    d_{V} \Big(  
    \Pi_{t},
    \Pi(\cdot \mid \bD_{1:t}) \big)
    \Big) 
    = o_P(1), \quad \forall t \in [T],
\end{align}
which we call the \textit{online BvM theorem}, stated in Theorem \ref{thm:onlineBvM}. 
To this end, we introduce a normal distribution $\widehat Q_{t}$ satisfying
\begin{align} \label{eqn:overview_assume}
    d_{V} \Big(  \widehat Q_{t}, \Pi(\cdot \mid \bD_{1:t})  \Big) = o_P(1)
\end{align}
for all $t \in [T]$. In Section \ref{sec:full_posterior}, $\widehat Q_{t}$ will be specified in Theorem \ref{thm:batch_posterior_LA}.

For regular parametric models, it is well known that the posterior distribution can be approximated by a normal distribution centered at an efficient estimator with variance equal to the inverse Fisher information matrix \citep{van2000asymptotic}. However, this normal distribution involves the Fisher information matrix, which depends on the unknown true parameter $\theta_{0}$. This motivates approximating the posterior by fully data-dependent distributions $\widehat Q_{t}$, a topic of significant interest in the Bayesian literature.

Substantial theoretical progress has been made in this direction, including Laplace approximation (LA) \citep{spokoiny2023dimension, spokoiny2025inexact, katsevich2023laplace} and VB methods \citep{katsevich2024approximation}. Accordingly, one may take such data-dependent distributions as $\widehat Q_{t}$. In Section \ref{sec:full_posterior}, we will provide rigorous justifications for these arguments. In the remainder of this section, we focus on conveying the high-level ideas.

Now we present our high-level proof strategy.
For $t \in [T]$, note that
\begin{align*}
    d_{V} \big( \Pi_{t}, \Pi(\cdot \mid \bD_{1:t}) \big)
    &\leq 
    d_{V} \big( \Pi_{t}, \widehat Q_{t} \big)
    +
    d_{V} \big( \widehat Q_{t}, \Pi(\cdot \mid \bD_{1:t}) \big) \\
    &=
    d_{V} \big( \Pi_{t}, \widehat Q_{t} \big)
    +
    o_P(1).
\end{align*}
Therefore, it suffices to bound the first term $d_{V} \big( \Pi_{t}, \widehat Q_{t} \big)$. 

To this end, we introduce an important formula. For two normal distributions $Q_1 = \cN(m_1, \bV_1^{-1})$ and $Q_2 = \cN(m_2, \bV_2^{-1})$ with sufficiently small $\| \bV_2^{-1/2} \bV_1 \bV_2^{-1/2} - \bI_p \|_{2}$, it holds that
\begin{align} \label{eqn:TV_error_normals_intro}
    d_{V}\left( Q_1, Q_2 \right) 
    \lesssim 
    \left\| \bV_1^{1/2} \big( m_1 - m_2 \big) \right\|_2 \vee \left\| \bV_2^{-1/2} \bV_1 \bV_2^{-1/2} - \bI_p \right\|_{\rm F}.
\end{align}
Suppose that $\widehat Q_{t} = \cN( \theta_t, \bF_t^{-1})$ for some $\theta_t \in \Theta$ and $\bF_t \in \symmPD$.\footnote{Here, $\theta_t$ and $\bF_t$ are written in simplified form. More precise notations and a rigorous construction of these quantities will be introduced in Section \ref{sec:full_posterior}.} 
By \eqref{eqn:TV_error_normals_intro}, the proof of \eqref{eqn:intro_claim} boils down to the problem of obtaining sharp upper bounds for the following quantities:
\begin{align} \label{eqn:intro_claim_2}
    \Big\| \bOmega_{t}^{1/2} \big( \theta_t - \mu_{t}  \big) \Big\|_{2} \ \text{ and } \
    \Big\| \bF_t^{-1/2} \bOmega_{t} \bF_t^{-1/2} - \bI_{p} \Big\|_{\rm F}.
\end{align}
Our main technical contribution lies in deriving explicit non-asymptotic bounds for these two terms in \eqref{eqn:intro_claim_2}.
To the best of our knowledge, no prior theoretical work has provided a rigorous non-asymptotic analysis of \eqref{eqn:intro_claim_2}. Specifically, Propositions \ref{prop:similar_MAP} and \ref{prop:similar_variance} in Section \ref{sec:online_variational_posterior} address the first and second terms, respectively. We will revisit this issue in Section \ref{sec:online_variational_posterior}.

\subsection{Example: logistic regression under random design}

This subsection summarizes our main results in the context of the logistic regression model.
Let $\bY = (Y_i)_{i \in [N]} \in \bbR^{N}$ be the response vector and $\bX = (X_{ij})_{i \in [N], j \in [p]} \in \bbR^{N \times p}$ be the design matrix.

To illustrate that our theory applies to logistic regression with a ``well-posed'' random design, we consider the random matrix setup where each entry of the design matrix $\bX$ is an i.i.d. standard normal, i.e., $X_{ij} \overset{\iid}{\sim} \cN(0, 1)$. For simplicity, we take the covariance matrix to be the identity matrix $\bI_p$; the analysis can be easily extended to a general covariance $\bSigma$ satisfying
\begin{align*}
    C^{-1} \leq \lambda_{\min} \big( \bSigma \big) \leq \lambda_{\max} \big( \bSigma \big) \leq C
\end{align*}
for some constant $C > 0$. 
With slight abuse of notation, hereafter, let $\bbP$ and $\bbE$ denote the joint probability measure and expectation corresponding to $(\bX, \bY)$, respectively. For a detailed description, see Appendix \ref{sec:proof_logit_example}.

To prove the online BvM theorem for the logistic regression model, we impose the following conditions.

\bed
\item[(\textbf{EX})] 
The true parameter $\theta_0$ and the initial prior parameters $\mu_0$ and $\bOmega_0$ satisfy
\begin{align*} 
    \left\| \theta_0 \right\|_{2} \leq K_1, \quad 
    \left\| \bOmega_{0}^{1/2} \big( \theta_0 - \mu_0 \big) \right\|_{2} 
    \leq K_2 p_{\ast}^{1/2}, \quad 
    \left\| \bOmega_{0} \right\|_{2} \leq K_3 p_{\ast},
\end{align*}
where $p_{\ast} = p \vee \log n \vee \log T$ and $K_1, K_2, K_3 > 0$ are some universal constants.
Furthermore, for a large enough constant $C = C(K_1, K_2, K_3) > 0$,
\begin{align*} 
        n \geq C \bigg[ \Big( p \log^6 (T \vee n) \log^{6}\left( 2n/p \right)  \Big) \vee \Big( p^{2} \log^{4}T \Big) \bigg].
\end{align*}
\eed
We now state that the online BvM theorem holds for the logistic regression model.
\begin{proposition} \label{prop:logit_intro_statement}
    Suppose that (\textbf{EX}) holds.
    Then, with $\bbP$-probability at least $1 - 5n^{-1} - 10e^{-n/72} -4(Np)^{-1}$, the following inequality holds uniformly for all $t \in [T]$:
    \begin{align*}
        d_{V} \Big( \Pi_{t}, \Pi(\cdot \mid \bD_{1:t}) \Big) 
        \leq  
        C  \left( \dfrac{p_{\ast}^{3}}{n} \right)^{1/2},
    \end{align*}
    where $C = C(K_1)$.
\end{proposition} 
Technical statements and proofs are deferred to Appendix \ref{sec:logit_example_app}; see Proposition \ref{prop:logit_fin_statement} for the precise statements.

\section{Variational approximation} \label{sec:variational_posterior}

In this section, we demonstrate that the posterior distribution $\widetilde{\Pi}_{t}\left(\cdot \mid \bD_{t} \right)$ is well approximated by its variational approximation $\Pi_t$. Before establishing the precise relation between $\Pi_t$ and $\widetilde{\Pi}_{t}\left(\cdot \mid \bD_{t} \right)$, we first show that the posterior can be accurately approximated by its Laplace approximation, defined as
\begin{align*}
    \Pi_{t}^{\LA} = \cN \left( \thetaMAP[t], \FisherMAP[t]^{-1} \right), \quad \forall t \in [T], 
\end{align*}
which admits a density function $\pi_{t}^{\LA}(\cdot)$.
Although our primary focus is on $\Pi_t$ rather than $\Pi_{t}^{\LA}$, understanding the concentration behavior of the pair $(\thetaMAP[t], \FisherMAP[t])$ allows us to facilitate the theoretical analysis needed to establish $\Pi_T \approx \Pi(\cdot \mid \bD)$. We will revisit this issue in Section \ref{sec:online_variational_posterior}.

Our main results and the corresponding proofs in this section are largely inspired by those in \citet{spokoiny2023dimension}. 
Notably, for the total variation (TV) metric, we refine the dimension dependency of Theorem 2.4 in \citet{spokoiny2023dimension} by leveraging advanced analysis in \citet{spokoiny2024estimation}. 
Moreover, for the KL divergence, we significantly improve the convergence rate presented in Theorem 2.6 of \cite{spokoiny2023dimension}; see the discussion following Theorem \ref{thm:LA_KL}.
Since the variational posterior $\Pi_t$ minimizes the KL divergence, obtaining a sharp upper bound for the KL divergence between $\Pi_{t}^{\LA}$ and $\widetilde{\Pi}_{t} (\cdot \mid \bD_{t} )$ is a crucial step in bounding $K(\Pi_t; \widetilde{\Pi}_{t} (\cdot \mid \bD_{t} ) )$. 
A detailed proof of this section is provided in Appendix \ref{sec:appendix_LA_VB_app}, where we present self-contained non-asymptotic proofs adapted to the online learning framework.

As discussed in Section \ref{sec:smoothness}, to accurately approximate $\widetilde{\Pi}_{t} (\cdot \mid \bD_{t} )$, we adopt self-concordance-type smoothness conditions. 
To this end, we impose the following assumption:
\bed
\item[(\textbf{A0})] 
For every $t \in [T]$, assume that the map $\theta \mapsto L_{t}(\theta)$ is concave on $\Theta$ and at least four times continuously differentiable with probability 1.
\eed
For the remainder of this paper, every event is treated as the intersection with the event under which assumption (\textbf{A0}) holds, without further explicit restatement. 

Let 
\begin{align}
\begin{aligned} \label{def:hat_tau_t_radius}
    r_{\LA} &= 2\sqrt{p} + \sqrt{2\log N} \\
    \widehat{\tau}_{3, t} 
    &= \inf \left\{ 
        \tau \in \bbR_{+} : 
        \sup_{u \in \Theta (\FisherMAP, 4r_{\LA})} \sup_{z \in \bbR^{p}} 
        \dfrac{
        \left| \langle \nabla^3 \widetilde{L}_{t}(\thetaMAP + u), z^{\otimes 3} \rangle \right|
        }{
        \left\| \FisherMAP^{1/2} z \right\|_{2}^{3}
        }
        \leq 
        \tau
    \right\}, \\
    \widehat{\tau}_{4, t} 
    &= \inf \left\{ 
        \tau \in \bbR_{+} : 
        \sup_{u \in \Theta (\FisherMAP, 4r_{\LA})} \sup_{z \in \bbR^{p}} 
        \dfrac{
        \left| \langle \nabla^4 \widetilde{L}_{t}(\thetaMAP + u), z^{\otimes 4} \rangle \right|
        }{
        \left\| \FisherMAP^{1/2} z \right\|_{2}^{4}
        }
        \leq 
        \tau
    \right\}.
\end{aligned}
\end{align}
Then, with probability 1, $\widetilde{L}_{t}(\cdot)$ satisfies the third and fourth order smoothness at $\thetaMAP$ with parameters $(\widehat{\tau}_{3, t}, \FisherMAP, 4r_{\LA})$ and $(\widehat{\tau}_{4, t}, \FisherMAP, 4r_{\LA})$, respectively. 
Note that $\widehat{\tau}_{3, t}$ and $\widehat{\tau}_{4, t}$ are finite almost surely under (\textbf{A0}) because $\widetilde{L}_{t}$ is four times continuously differentiable and $\FisherMAP \in \symmPD$.

For two probability distributions $P$ and $Q$, the total variation (TV) distance is defined as
\begin{align} \label{def:TV_distance}
    d_{V} \left( P,  Q \right) = \sup_{\cA} \left| P(\cA) - Q(\cA) \right|,
\end{align}
where the supremum is taken over all measurable sets $\cA$. The following theorem provides an upper bound on the total variation between $\Pi_{t}^{\LA}$ and $\widetilde{\Pi}_{t} (\cdot \mid \bD_{t} )$ in terms of $\widehat{\tau}_{3, t}$ and $\widehat{\tau}_{4, t}$.
We also note that Theorem \ref{thm:LA_TV} can be obtained from Theorem 2.2 in \citet{katsevich2025improved}. Although the form of (2.1) in that paper differs slightly from \eqref{def:hat_tau_t_radius}, the leading asymptotic orders are the same, and the proof techniques are similar.

\begin{theorem}[Laplace approximation: TV distance] \label{thm:LA_TV}
Suppose that (\textbf{A0}) holds. Then, with probability 1, we have
    \begin{align} \label{claim:LA_TV}
        d_{V} \left( \Pi_{t}^{\LA}(\cdot),  \widetilde{\Pi}_{t}\left(\cdot \mid \bD_{t} \right) \right) \leq K \epsilon_{n, t, \TV}, \quad  \forall t \in [T],
    \end{align}
    where
    \begin{align*}
        \epsilon_{n, t, \TV} 
        = 
            \left( \widehat{\tau}_{4, t} + \widehat{\tau}_{3, t}^2 \right) p^2 
            +
            \widehat{\tau}_{3, t} p
            +
            \widehat{\tau}_{3, t}^{3} \log^{3} N
            +
            e^{-8\log N - 8p}
    \end{align*}
    and $K > 0$ is a universal constant.
\end{theorem}

Theorem \ref{thm:LA_TV} implies that an accurate approximation in total variation can be achieved provided that $( \widehat{\tau}_{3, t} r_{\LA}^{2} ) \vee ( \widehat{\tau}_{4, t} p^{2} )$ is sufficiently small.
In particular, this condition guarantees that the quadratic approximation is accurate on the local region $\Theta (\thetaMAP, \FisherMAP, 4r_{\LA})$. In Proposition \ref{prop:eigenvalue_order_main}, we provide sufficient conditions under which the following asymptotic bounds hold on a given event:
\begin{align} \label{eqn:tau_bound}
    \widehat{\tau}_{3, t} = O(t^{-3/2} n^{-1/2}), \quad 
    \widehat{\tau}_{4, t} = O(t^{-2} n^{-1}). 
\end{align}
Based on these bounds, $n t^2 \gg (p + \log N)^2$ is sufficient to ensure $( \widehat{\tau}_{3, t} r_{\LA}^{2} ) \vee ( \widehat{\tau}_{4, t} p^{2}) = o(1)$. Furthermore, combining \eqref{eqn:tau_bound} with conditions $n t^3 \gg p^{-1} \log^3 N$ and $nt \gg p^2$, one can simplify the leading order of $\epsilon_{n, t, \TV}$ as
\begin{align} \label{eqn:TV_error}
    \epsilon_{n, t, \TV} = O(\widehat{\tau}_{3, t} p) = O \left( t^{-3/2} (p^2/n)^{1/2} \right)
\end{align}
For the batch learning setup ($t = 1$), this bound coincides with the sharp bounds established in recent studies \citep{katsevich2024laplace, katsevich2024approximation, spokoiny2024estimation}. In particular, \citet[][Section 2.4]{katsevich2023laplace} showed that the total variation error is lower bounded by a constant multiple of $(p^2/n)^{1/2}$. Therefore, the rate derived in Theorem \ref{thm:LA_TV} is optimal, at least in terms of its dependence on the sample size and parameter dimension. In this sense, for a fixed sample size $n$, the order $p = o(n^{1/2})$ is regarded as the \textit{critical dimension} for the validity of the Laplace approximation.

Note that the critical dimension of order $p = o(n^{1/2})$ established in the above literature is derived for a certain class of models, including logistic regression. However, it is important to emphasize that the critical dimension for the Laplace approximation---or more broadly, the asymptotic normality of the posterior distribution (i.e., the BvM assertion)---depends on the specific statistical model under consideration. For instance, \citet[][Section 4.1]{panov2015finite} demonstrated that the condition $p = o(n^{1/3})$ cannot be relaxed for the BvM assertion in a particular model. Similarly, \citet{chae2023adaptive} established the asymptotic normality of the posterior distribution in the current status model under the assumption $p = o(n^{1/3})$. Although there is no formal proof that $p = o(n^{1/3})$ is the critical dimension in this setting, the bound appears to be tight based on the corresponding frequentist theory \citep{tang2012likelihood} for the same model. On the other hand, \citet{yano2020frequentist} showed that asymptotic normality of the posterior holds under the condition $p^2 \log^3 n = o(n)$ in linear regression with an unknown error distribution. These examples illustrate that the critical dimension for the Laplace approximation generally depends on the underlying statistical model.


To approximate the posterior in terms of the KL divergence, an additional step is required to carefully control 
the behavior of the posterior density in the tail region. In particular, the log-likelihood ratio $\theta \mapsto \log (\pi_{t}^{\LA}(\theta)/ \widetilde{\pi}_{t}(\theta \mid \bD_t))$ on $\Theta^{\rm c} (\thetaMAP, \FisherMAP, 4r_{\LA})$ should not grow too rapidly. To quantify this growth rate, we introduce the following quantity:
\begin{align}
\begin{aligned} \label{def:tau_3tr}
    \widehat{\tau}_{3, t, r} 
    = \inf \left\{ 
        \tau \in \bbR_{+} : 
        \sup_{u \in \Theta (\FisherMAP, r)} \sup_{z \in \bbR^{p}} 
        \dfrac{
        \left| \langle \nabla^3 \widetilde{L}_{t}(\thetaMAP + u), z^{\otimes 3} \rangle \right|
        }{
        \left\| \FisherMAP^{1/2} z \right\|_{2}^{3}
        }
        \leq 
        \tau
    \right\}, \quad r > 4r_{\LA}.
\end{aligned}
\end{align}  
This quantity plays a key role in bounding the KL divergence $K ( \Pi_{t}^{\LA}(\cdot); \ \widetilde{\Pi}_{t}\left(\cdot \mid \bD_{t} \right) )$.
To ensure that this divergence remains sufficiently small, we impose the following regularity condition.

\bed
\item[(\textbf{KL})] 
Suppose that $N \geq 2$. Assume also that, on an event $\scrE_{1}$ the following inequalities hold uniformly for all $t \in [T]$: 
\begin{align} 
    \begin{aligned} \label{assume:KL}
        &\left( \widehat{\tau}_{3, t} r_{\LA}^{2} \right) \vee \left( \widehat{\tau}_{4, t} p^{2} \right) \leq \frac{1}{8}, \\         
        &\widehat{\tau}_{3 ,t, r} 
        \leq N e^{8p} \exp\bigg( \left[ \sqrt{p} + \sqrt{2 \log N} - 3 \right] r \bigg), \quad \forall r > 4r_{\LA}.
    \end{aligned}
\end{align}  
\eed

Assumption (\textbf{KL}) is very mild. In the proof of Proposition \ref{prop:eigenvalue_order_main}, we show that \eqref{assume:KL} holds under mild regularity conditions. Here, we provide a simple sufficient condition for \eqref{assume:KL}. According to the discussion following \eqref{eqn:tau_bound}, the condition $nt^2 \gg (p + \log N)^2$ is sufficient to ensure the first part of \eqref{assume:KL}.
Furthermore, by Lemma \ref{lemma:tech_smooth_tau_bound}, the second condition in \eqref{assume:KL} holds when
\begin{align*}
    \sup_{\theta \in \Theta (\thetaMAP, \FisherMAP, r)}
    \left\| \nabla^{3} L_{t}(\theta) \right\|_{\rm op} 
    \leq 
    \lambda_{\min}^{3/2}\left( \FisherMAP \right) N e^{8p} \exp\bigg( \left[ \sqrt{p} + \sqrt{2 \log N} - 3 \right] r \bigg).
\end{align*}
Note that the last display is rather mild and holds in many examples.
For instance, under the logistic regression with a simple random design setup, one can prove that
\begin{align*}
    \lambda_{\min} \big( \FisherMAP \big) \asymp nt, \quad 
    \max_{t \in [T]} \sup_{\theta \in \Theta} \left\| \nabla^{k} L_{t}(\theta) \right\|_{\rm op} \lesssim \ n, \quad  \forall k \in \left\{ 3, 4 \right\}   
\end{align*}
provided that $p^{2} \log^{12} (T \vee n) = o(n)$ (see Propositions \ref{prop:eigenvalues_logit} and \ref{prop:operator_norm_logit} for the precise statements).
Combining the last two displays, one easily checks that $n \gg p^2 \log^{12} (T \vee n)$ is sufficient for \eqref{assume:KL}.


\begin{theorem}[Laplace approximation: KL divergence] \label{thm:LA_KL}
    Suppose that (\textbf{A0}) and (\textbf{KL}) hold on some event $\scrE_{1}$.
    Then, on $\scrE_{1}$, the following inequality holds uniformly for all $t \in [T]$:
    \begin{align*}
        K \left( \Pi_{t}^{\LA}(\cdot); \ \widetilde{\Pi}_{t}\left(\cdot \mid \bD_{t} \right) \right) \leq K \epsilon_{n, t, \KL}^{2},
    \end{align*}
    where $K > 0$ is a universal constant and 
    \begin{align*}
        \epsilon_{n, t, \KL} = 
        \left( 
            \left[ \widehat{\tau}_{4, t} + \widehat{\tau}_{3, t}^2 \right] p^2 
            +
            \widehat{\tau}_{3, t}^{3} \log^{3} N
            +
            e^{- 7\log N}
        \right)^{1/2}.
    \end{align*}    
\end{theorem}

According to the bounds established in \eqref{eqn:tau_bound}, the leading order term of $\epsilon_{n, t, \KL}$ satisfies
\begin{align*}
    \epsilon_{n, t, \KL}^2 
    \lesssim ( \widehat{\tau}_{4, t} p^2) \vee (\widehat{\tau}_{3, t}^{3} \log^{3} N)
    \lesssim \left( \dfrac{p^2}{n t^2} \right) \vee \left( \dfrac{\log^2 N}{ n t^{3}} \right)^{3/2},
\end{align*}
which simplifies to $p^2/n$ when $n \gg p^{-4} \log^6 N$ and $t = 1$. 
As a comparable result for $t = 1$, \citet{spokoiny2023dimension} demonstrated, under certain conditions, that $K ( \Pi_{t}^{\LA}(\cdot); \ \widetilde{\Pi}_{t}\left(\cdot \mid \bD_{t} \right) )$ is bounded by $(p_{\eff}^3 / n)^{1/2}$ up to a constant factor. Here, $p_{\eff}$ denotes an effective dimension satisfying $p_{\eff} \leq p$, and we have $p_{\eff} \asymp p$ unless $\lambda_{\min}(\bOmega_0) \gg 1$. In this sense, our result represents a substantial improvement over existing ones because $p^2/n \ll (p_{\eff}^3 / n)^{1/2}$ provided that $n \gg p$ and $p_{\eff} \asymp p$.

Extending the results in Theorem \ref{thm:LA_KL} to the VB case is straightforward.
By the definition of $\Pi_t$ specified in \eqref{eq:KL-projection}, we have
\begin{align*}
    K \left( \Pi_{t}(\cdot); \ \widetilde{\Pi}_{t}\left(\cdot \mid \bD_{t} \right) \right)
    \leq 
    K \left( \Pi_{t}^{\LA}(\cdot); \ \widetilde{\Pi}_{t}\left(\cdot \mid \bD_{t} \right) \right)
    \lesssim
    \epsilon_{n, t, \KL}^2,
\end{align*}
which, by Pinsker's inequality, further implies that
\begin{align*}
    d_{V}\left( \Pi_{t}(\cdot),  \widetilde{\Pi}_{t}\left(\cdot \mid \bD_{t} \right) \right)
    \leq 
    \sqrt{ \dfrac{1}{2} K \left( \Pi_{t}^{\LA}(\cdot); \ \widetilde{\Pi}_{t}\left(\cdot \mid \bD_{t} \right) \right) }
    \lesssim   
    \epsilon_{n, t, \KL}.
\end{align*}
For a simplified comparison, if we take $\epsilon_{n, 1, \KL} = O(\sqrt{p^2/n})$, our result for the TV metric aligns with the rate demonstrated in \citet[][Corollary 2.1]{katsevich2024approximation}.

\begin{theorem}[Variational approximation] \label{thm:VB_KL}
    Suppose that (\textbf{A0}) and (\textbf{KL}) hold on some event $\scrE_{1}$. Then, on $\scrE_{1}$, the following inequalities hold uniformly for all $t \in [T]$:
    \begin{align*}
        &K \left( \Pi_{t}(\cdot); \ \widetilde{\Pi}_{t}\left(\cdot \mid \bD_{t} \right) \right) \leq K \epsilon_{n, t, \KL}^{2}, \quad 
        d_{V}\left( \Pi_{t}(\cdot),  \widetilde{\Pi}_{t}\left(\cdot \mid \bD_{t} \right) \right) \leq K \epsilon_{n, t, \KL},
    \end{align*}
    where $K > 0$ is a universal constant.
\end{theorem}

In the online setting, the sequential updates quickly yield a highly \textit{informative prior}. Specifically, we typically have $\lambda_{\min}(\bOmega_t) \gg 1$ and $\| \mu_t - \theta_0 \|_2 \ll 1$ for $t > 1$, and both quantities naturally scale with the sample size $N_t = nt$. To quantify the strength of such priors, we adopt the notion of \textit{effective sample size}, originally introduced by \citet{spokoiny2017penalized}, defined as $n_{\eff, t} = \lambda_{\min}(-\nabla^2 \widetilde{L}_{t}(\thetaBest))$. A strong prior tends to inflate the effective sample size. We will show in Section \ref{sec:eigenvalue_analysis} that $n_{\eff, t} \asymp N_t$ holds under suitable conditions.

A subtle but important phenomenon emerges in the rate of $\epsilon_{n, t, \TV}$. 
One might initially expect the TV error based on $N_t$ observations to scale as $(p^2 / N_t)^{1/2}$, matching the order $(p^2/n)^{1/2}$ at $t = 1$. In contrast, $\epsilon_{n, t, \TV} = O( t^{-1} [p^2/N_t]^{1/2} )$ in \eqref{eqn:TV_error} is strictly sharper in terms of its polynomial dependence on $t$. Notably, this improvement arises from the use of informative normal priors.

To provide more intuition behind this improvement, note that, by \eqref{eqn:tau_bound_theory},
\begin{align*}
     \widehat{\tau}_{k, t} & \leq \lambda_{\min}^{-k/2} ( \FisherMAP ) \sup_{\theta \in \Theta ( \thetaMAP, \FisherMAP, 4r_{\LA})} \big\| \nabla^{k} \widetilde{L}_{t}(\theta) \big\|_{\rm op}, \quad k \in \{3, 4\}.
\end{align*}
For a simple illustration, we assume that
\begin{align*}
    \lambda_{\min} ( \FisherMAP ) \asymp nt
    \quad \text{and} \quad 
    \sup_{\theta \in \Theta ( \thetaMAP, \FisherMAP, 4r_{\LA})} \left\| \nabla^{k} L_{t}(\theta) \right\|_{\rm op} \asymp n.
\end{align*}
Then, the asymptotic bound \eqref{eqn:tau_bound} holds, which leads to the improved result \eqref{eqn:TV_error}. Note that the normal prior affects $\widetilde{L}_{t}(\cdot)$ only through (up to) second-order terms, whereas the key regularity of our model is governed by the ratio between the third (or fourth) and second derivatives. By using informative normal priors, we can effectively increase the effective sample size (as $\lambda_{\min} ( \FisherMAP ) \asymp nt$) without altering the higher-order properties of the log-likelihood, since $\nabla^{k} \widetilde{L}_{t} = \nabla^{k} L_{t}$ for $k > 2$. In this sense, the normal prior effectively leverages the available prior information.



\section{Penalized M-estimation} \label{sec:pMLE}
This section presents theoretical results for the penalized M-estimator $\thetaMAP$, which has been well-studied in the literature.
Recently, \cite{spokoiny2012parametric, spokoiny2017penalized, spokoiny2024estimation}  developed advanced techniques for analyzing statistical behaviors of M-estimators with finite-sample guarantees.
Accordingly, we adopt these techniques with modifications to adapt them to our framework.

We say the model is \textit{stochastically linear} if the map $\theta \mapsto \zeta_{t}(\theta) = L_{t}(\theta) - \bbE_{t} L_{t}(\theta)$ is linear for every $t$; that is, $\zeta_{t}(\theta) = a^\top \theta + b$ for some (random) quantities $a \in \bbR^p$ and $b \in \bbR$. Note that stochastic linearity implies that $\nabla^{k} L_{t}(\theta)$ is non-random for all $t \in [T]$ and $k \in \{ 2, 3, 4 \}$. Also, since $\nabla \zeta_t(\theta)$ does not depend on $\theta$, we hereafter denote this random vector by $\nabla \zeta_t$. The stochastically linear framework encompasses many important statistical models, such as the logistic regression, Poisson regression, nonparametric regression \citep{spokoiny2025accuracy}, nonlinear inverse problem \citep{spokoiny2019bayesian}, and covariance estimation \citep{puchkin2025sharper}. See Section 1.3 in \citet{spokoiny2024estimation} for a detailed discussion.

We now introduce the assumptions and notations used in the estimation procedure.
\bed
\item[(\textbf{A1})] 
The model is stochastically linear. Also, $\bF_{t, \theta}$ is nonsingular for every $t \in [T]$ and $\theta \in \Theta(\theta_{0}, \bI_{p}, 1/2)$. Furthermore, there exist $\{ \bV_t : t \in [T] \} \subset \symmPD$ and $M_n > 1$ such that
\begin{align} \label{assume:A1_1}
    \bbP_{0, t} \bigg( \left\| \FisherBest^{-1/2} \nabla \zeta_{t} \right\|_{2} \geq r_{\eff, t} \bigg) &\leq e^{-(\log n + \log T)}, \\
    \label{assume:A1_2}
    \max_{t \in [T]} \sup_{\theta \in \Theta \left(\theta_0, \bI_{p}, 1/2 \right)} \left\| \Fisher[t]{\theta}^{-1} \bV_{t} \right\|_{2} &\leq \dfrac{M_n^2}{9},
\end{align}
where
\begin{align}
\begin{aligned} \label{def:estimation_quantities}
    & r_{\eff, t} = p_{\eff, t}^{1/2} + \sqrt{2\lambda_{t} (\log n + \log T) }, \\ 
    & p_{\eff, t} = \operatorname{tr}\left( \FisherBest^{-1} \bV_{t} \right), \quad
    \lambda_{t} = \left\| \FisherBest^{-1} \bV_{t} \right\|_{2}.    
\end{aligned}
\end{align}
\eed

As discussed in Section \ref{sec:smoothness}, we employ the self-concordance condition on $\bbE_t \widetilde{L}_t(\cdot)$ over the local set $\Theta(\thetaBest, \FisherBest, 4r_{\eff, t})$. 
For this purpose, we define the following quantity:
\begin{align} \label{def:tau_3ast}
    \tau_{3, t}^{\ast}
    = \inf \left\{ 
        \tau_3 \in \bbR_{+} : 
        \sup_{u \in \Theta (\FisherBest, 4r_{\eff, t})} \sup_{z \in \bbR^{p}} 
        \dfrac{
        \left| \langle \nabla^3 \bbE_{t} \widetilde{L}_{t}(\thetaBest + u), z^{\otimes 3} \rangle \right|
        }{
        \left\| \FisherBest^{1/2} z \right\|_{2}^{3}
        }
        \leq 
        \tau_3
    \right\}.
\end{align} 
Note that $\tau_{3, t}^{\ast}$ is random because both $\thetaBest$ and $\bOmega_{t-1}$ depend on $(\bD_1, ..., \bD_{t-1})$.
To prove the convergence of the pMLE $\thetaMAP$, we further impose the following assumption.

\bed
\item[(\textbf{Est})] On an event $\scrE_{2}$, the following inequality holds uniformly for all $t \in [T]$:
\begin{align} \label{assume:est}
    \tau_{3, t}^{\ast} r_{\eff, t} \leq 1/16.
\end{align}
\eed

Assumptions (\textbf{A1}) and (\textbf{Est}) ensure that the pMLE $\thetaMAP$ converges to $\thetaBest$ at an appropriate rate. More specifically, Theorem \ref{thm:penalized_estimation} shows that
\begin{align*}
    \thetaMAP \in \Theta \big( \thetaBest, \FisherBest, 4r_{\eff, t} \big), \quad \forall t \in [T]
\end{align*}
with high probability. Before stating the theorem, we provide a detailed discussion of the above assumptions.

In many examples, one can choose $\bV_t = c \Var(\nabla \zeta_t)$ for a sufficiently large constant $c > 0$. Other choices are also possible. For instance, if $\nabla \zeta_t \sim \text{subG}(\sigma_n^2)$ for some $\sigma_n^2 \geq 0$, then \eqref{assume:A1_1} holds with the choice $\bV_t = \sigma_n^2 \bI_p$. In this case, since $\nabla \zeta_t$ is the sum of $n$ independent random variables, we roughly have $\sigma_n^2 \asymp n$. Note that $\nabla \zeta_t$ can be sub-exponential rather than sub-Gaussian in certain important examples, such as the Poisson regression model. Even in sub-exponential cases, however, one can still verify \eqref{assume:A1_1} under suitable conditions; see Appendix A in \citet{spokoiny2017penalized} and Lemma B.2 in \citet{lee2025advances}.


The condition \eqref{assume:A1_2} is used to derive an appropriate bound for the radius $r_{\eff, t}$. In the batch learning setting, one may choose $r_{\eff, 1} \asymp \sqrt{p} + \sqrt{\log n}$ so that $\thetaMAP[1] \in \Theta ( \thetaBest[1], \FisherBest[1], 4r_{\eff, 1} )$ with high probability. In contrast, in the online learning setting, the order of $r_{\eff, t}$ exhibits an interesting behavior. In particular, note that the effective dimension $p_{\eff, t}$ decreases with $t$ because $\FisherBest$ and $\bV_t$ are of orders proportional to the sample sizes $nt$ and $n$, respectively. By combining \eqref{assume:A1_2} with additional conditions, one can derive the following bounds:
\begin{align} \label{eqn:effective_dim}
    p_{\eff, t} = O \big( M_n^2 t^{-1} p \big), \quad 
    r_{\eff, t} = O \big( M_n \sqrt{(p \vee \log n \vee \log T)/t} \big).
\end{align}
The results in \eqref{eqn:effective_dim} follow directly from Lemma~\ref{lemma:radius_upper_bound} and Proposition~\ref{prop:eigenvalue_order_main}.

The assumptions in (\textbf{A1}) are mild and hold in many examples (e.g., when the score function is sub-Gaussian). Statistical models with the sub-Gaussian score function encompass a wide range of examples; a logistic regression model is one of the popular cases. In Appendix \ref{sec:logit_example_app}, we present a theoretical verification of (\textbf{A1}) under a simple random design setting. For the logistic regression model, one can show \eqref{assume:A1_1} and \eqref{assume:A1_2} are satisfied with $\bV_t = \bX_{t}^{\top} \bX_{t}/4$ and $M_n = O_P(1)$, where $\bX_{t} \in \bbR^{n \times p}$ is the design matrix for the $t$-th mini-batch. 

The nonsingularity of $\Fisher[t]{\theta}$ on $\Theta(\theta_0, \bI, 1/2)$ can be verified when $L_t(\cdot)$ is strictly concave; this includes, for example, the logistic regression model. The specific choice of radius $1/2$ in \eqref{assume:A1_2} does not carry inherent meaning and can be replaced by a decaying sequence $r_n$, with some additional technical effort.

Assumption \eqref{assume:est} ensures that the effective sample size $n_{\eff, t}$ is sufficiently large relative to the effective dimension $p_{\eff, t}$, thereby guaranteeing the convergence of $\thetaMAP$ to $\thetaBest$. One can show that $\tau_{3, t}^{\ast}$ is of the order given in \eqref{eqn:tau_bound}. Combining this with \eqref{eqn:effective_dim}, we find that for a fixed $t \in [T]$, a sufficient condition for \eqref{assume:est} is 
\begin{align*}
    M_n^2 \big( p \vee \log T \big) = o( nt^{4} ).
\end{align*}
In the batch learning setup where $T = t = 1$, this condition reduces to the well-known requirement $p = o(n)$, provided that $M_n = O(1)$.

Note that the additional term $\log T$ in the definition of $r_{\eff, t}$ can be interpreted as the cost of requiring uniformity over $\{ 1, 2, ..., T \}$. 
From \eqref{assume:A1_1}, we have
\begin{align*}
    \bbP_{0} \bigg( \left\| \FisherBest^{-1/2} \nabla \zeta_{t} \right\|_{2} \geq r_{\eff, t} \: \text{ for some } t \in [T] \bigg) 
    &\leq T \cdot \max_{t \in [T]} \bbP_{0, t} \bigg( \left\| \FisherBest^{-1/2} \nabla \zeta_{t} \right\|_{2} \geq r_{\eff, t} \bigg) \\
    &\leq e^{-\log n} = n^{-1}.
\end{align*}

To state the main theorem in this section, let $\scrE_{\est, 1}$ be the event on which the following inequality holds:
\begin{align} \label{def:Gamma_n_event}
    \left\| \FisherBest^{-1/2} \nabla \zeta_{t} \right\|_{2} \leq r_{\eff, t} \quad \text{for all } t \in [T].
\end{align}
Then, $\bbP_0^{(N)}(\scrE_{\est, 1}) \geq 1 - n^{-1}$ under (\textbf{A1}). The event $\scrE_{\est, 1}$ plays a crucial role throughout this paper, as it ensures the statistical validity of the online estimation procedure.

\begin{theorem} \label{thm:penalized_estimation}
    Suppose that (\textbf{A0}), (\textbf{A1}) and (\textbf{Est}) hold.
    Then, on $\scrE_{\est, 1} \cap \scrE_{2}$, the following inequality holds uniformly for all $t \in [T]$:
    \begin{align*}
         \left\| \FisherBest^{1/2} \big( \thetaMAP - \thetaBest \big) \right\|_{2} \leq 4r_{\eff, t}, \quad \forall t \in [T].
    \end{align*}
\end{theorem}

Theorem \ref{thm:penalized_estimation} provides a sharp upper bound for $\| \thetaMAP - \thetaBest \|_{2}$, which corresponds to the variance term in the well-known bias–variance decomposition: 
\begin{align} \label{eqn:decomposition_pMLE}
    \big\| \thetaMAP - \theta_0 \big\|_2 
    \leq 
    \big\| \thetaMAP - \thetaBest \big\|_2 + \big\| \thetaBest - \theta_0 \big\|_2.
\end{align}
By \eqref{eqn:effective_dim} and Theorem \ref{thm:penalized_estimation}, we obtain
\begin{align*}
    \big\| \thetaMAP - \thetaBest \big\|_{2} = O \left( M_n \left[ p \vee \log n \vee \log T \right]^{1/2} t^{-1} n^{-1/2} \right)
\end{align*}
provided that $\lambda_{\min}(\FisherBest) \gtrsim nt$. Ignoring the logarithmic factor, the quantity $\| \thetaMAP - \thetaBest \|_{2}$ decreases at the rate $t^{-1} n^{-1/2}$, rather than the standard rate $(nt)^{-1/2}$. As discussed at the end of Section~\ref{sec:variational_posterior}, this improvement arises from the use of informative quadratic penalization.

Note that Theorem~\ref{thm:penalized_estimation} does not address $\| \thetaBest - \theta_0 \|_{2}$, which corresponds to the bias term in the decomposition \eqref{eqn:decomposition_pMLE}. In Sections~\ref{sec:eigenvalue_analysis} and \ref{sec:online_variational_posterior}, we show that $\| \thetaBest - \theta_0 \|_{2}$ decreases at the rate $(nt)^{-1/2}$. Given that this convergence rate cannot be improved in general, our result may be seen as a natural consequence.

\section{Analysis of eigenvalues and remainder terms} \label{sec:eigenvalue_analysis}

This section presents informative bounds on important quantities such as $\widehat{\tau}_{3, t}$, $\widehat{\tau}_{4, t}$, and $\tau_{3, t}^{\ast}$. A crucial step in analyzing these quantities is to ensure that $\thetaMAP$ and $\thetaBest$ lie within a local neighborhood of $\theta_0$, say $\Theta(\theta_0, \bI_p, 1/2)$. This localization guarantees an accurate quadratic approximation. Specifically, the quality of the quadratic approximation relies heavily on the magnitude of the smallest eigenvalues, such as $\lambda_{\min}(\FisherBest)$ and $\lambda_{\min}(\FisherMAP)$. Intuitively, these values serve as the \textit{effective sample size}. Therefore, we need to show that the effective sample size is proportional to the accumulated (actual) sample size $nt$. Once this step is established, it is not too difficult to bound the other quantities (e.g., $\widehat{\tau}_{3, t}$, $\widehat{\tau}_{4, t}$ and $\tau_{3, t}^{\ast}$).

However, unless the log-likelihood $L_t(\cdot)$ is strongly concave in the sense that $\lambda_{\min}(\bF_{t, \theta})$ is uniformly bounded below by a multiple of $n$ for all $\theta \in \Theta$, it is not straightforward to ensure that $\lambda_{\min}(\FisherBest) \wedge \lambda_{\min}(\FisherMAP) \gtrsim nt$.
If we assume the strong concavity of $L_t(\cdot)$, the analysis can be significantly simplified because the \textit{localization} step (e.g., $\thetaBest \in \Theta(\theta_0, \bI_p, 1/2)$) can be omitted, and it is possible to obtain $\lambda_{\min}(\FisherBest) \gtrsim nt$ directly.  For ease of analysis, this strong concavity assumption has been often adopted in the existing online learning literature \citep{chen2020SGD, zhu2023online}.  However, many important statistical models (e.g., GLMs) may not exhibit strong concavity when $\Theta = \bbR^p$. For instance, the logistic regression model---which is our main statistical application--- is not strongly concave but strictly concave.
To accommodate these models, we impose some regularity conditions in a local vicinity of $\theta_0$. The precise statements of the local regularity conditions are as follows:

\bed
\item[(\textbf{A2})] There exist constants $K_{\min} > 0$ and $K_{\max} \geq 1$ such that
\begin{align} 
\begin{aligned}\label{assume:A2_1}
    \min_{t \in [T]} \inf_{\theta \in \Theta (\theta_0, \bI_p, 1/2)} \lambda_{\min} \left( \Fisher[t]{\theta} \right) &\geq K_{\min}n, \\
    \max_{t \in [T]} \sup_{\theta \in \Theta (\theta_0, \bI_p, 1/2)} \lambda_{\max} \left( \Fisher[t]{\theta} \right) &\leq K_{\max}n, 
\end{aligned} 
\end{align}
\begin{align} 
\begin{aligned} \label{assume:A2_2}
    \max_{t \in [T]} \sup_{\theta \in \Theta (\theta_0, \bI_p, 1/2)} \left\| \nabla^{3} L_{t}(\theta) \right\|_{\rm op} &\leq K_{\max} n, \\
    \max_{t \in [T]} \sup_{\theta \in \Theta (\theta_0, \bI_p, 1/2)} \left\| \nabla^{4} L_{t}(\theta) \right\|_{\rm op} &\leq K_{\max} n,
\end{aligned} 
\end{align}
\begin{align} 
\begin{aligned} \label{assume:A2_3} 
    \max_{t \in [T]} \sup_{\theta \in \Theta (\theta_0, \bI_p, 1/2 + r)} \left\| \nabla^{3} L_{t}(\theta) \right\|_{\rm op} 
    &\leq K_{\max} N e^{8p} 
    e^{ ( \sqrt{p} + \sqrt{2 \log N} - 3 ) r }, \quad \forall r > 0.
\end{aligned} 
\end{align}
\eed

In Appendix \ref{sec:logit_example_app}, we show that assumption (\textbf{A2}) holds with high probability under the logistic regression model with random design. For precise statements, see Propositions \ref{prop:eigenvalues_logit} and \ref{prop:operator_norm_logit}.

For high-dimensional or nonparametric models, the effect of the prior may remain non-negligible even as the sample size increases \citep{cox1993analysis, freedman1999wald, johnstone2010high}. Since we allow the dimension to diverge, i.e., $p = p_n \to \infty$ as $n \to \infty$, we impose non-asymptotic conditions on the initial prior parameters $\mu_0$ and $\bOmega_0$.
Before we state specific conditions, we recall the following notation:
\begin{align*}
    p_{\ast} = p \vee \log n \vee \log T.
\end{align*}
\bed
\item[(\textbf{P})] The initial prior parameters $\mu_0$ and $\bOmega_0$ satisfy 
\begin{align*} 
    \left\| \bOmega_{0}^{1/2} \big( \theta_0 - \mu_0 \big) \right\|_{2} \leq \delta n^{1/2}, \quad
    \left\| \bOmega_{0} \right\|_{2} \leq K_{\max} p_{\ast}
\end{align*}
for a small enough constant $\delta = \delta(K_{\min}, K_{\max}) > 0$.
\eed

\bed
\item[(\textbf{S})] For a large enough constant $C = C(K_{\min}, K_{\max}) > 0$,
\begin{align*}
\begin{aligned} 
    n \geq C \big( \log^{2} T \vee M_n^2 \big) p_{\ast}^2.
\end{aligned}
\end{align*}
\eed

We now present the main results of this section.

\begin{proposition} \label{prop:eigenvalue_order_main}
    Suppose that (\textbf{A0}), (\textbf{A1}), (\textbf{A2}), (\textbf{S}), and (\textbf{P}) hold.
    Then, on $\scrE_{\est, 1}$, the following inequalities hold uniformly for all $t \in [T]$:
    \begin{align*} 
    \begin{aligned}
        \left\| \FisherBest[t]^{1/2} \big( \theta_0 - \thetaBest \big) \right\|_{2} &\leq K_{\rm up} M_n \sqrt{t p_{\ast}},  &
        \lambda_{\min} \big( \FisherBest \big) \wedge \lambda_{\min} \big( \FisherMAP \big) &\geq K_{\rm low} nt, \\  
        \left\| \FisherBest[t]^{1/2} \big( \thetaMAP - \thetaBest \big) \right\|_{2} &\leq K_{\rm up} M_n \sqrt{t^{-1} p_{\ast}},  &
        \lambda_{\max} \big( \FisherBest \big)  \vee  \lambda_{\max} \big( \FisherMAP \big) &\leq K_{\rm up} nt, 
    \end{aligned}
    \end{align*}
    and
    \begin{align*}
    \widehat{\tau}_{3, t} \vee \tau_{3, t}^{\ast}
    \leq
    K_{\rm up} t^{-3/2} n^{-1/2}, \quad 
    \widehat{\tau}_{4, t}
    \leq
    K_{\rm up} t^{-2} n^{-1}, \quad 
    \epsilon_{n, t, \KL} 
    \leq K_{\rm up} t^{-1} n^{-1/2} p_{\ast},
    \end{align*}        
    where $K_{\rm up} = K_{\rm up} (K_{\min}, K_{\max})$ and $K_{\rm low} = K_{\rm low} (K_{\min})$.
\end{proposition}
\begin{proof}
    See the proof of Proposition \ref{prop:eigenvalue_order}; Proposition \ref{prop:eigenvalue_order_main} is a special case of Proposition \ref{prop:eigenvalue_order}.
\end{proof}

From the results in Proposition \ref{prop:eigenvalue_order_main}, we can quantify the relation between pMLE $(\thetaMAP, \FisherMAP)$ and variational parameters $(\mu_t, \bOmega_t)$. Recall that
\begin{align*}
    d_{V}\Big( \Pi_{t}(\cdot),  \Pi_{t}^{\LA}(\cdot) \Big)
    \leq 
    d_{V}\Big( \Pi_{t}(\cdot),  \widetilde{\Pi}_{t}\left(\cdot \mid \bD_{t} \right) \Big) + 
    d_{V}\Big( \widetilde{\Pi}_{t}\left(\cdot \mid \bD_{t} \right),  \Pi_{t}^{\LA}(\cdot) \Big) 
    \lesssim 
    \epsilon_{n, t, \KL},
\end{align*}
where the inequality holds by Theorem \ref{thm:VB_KL} and Pinsker's inequality. If two normal distributions $Q_1 = \cN(m_1, \bV_1^{-1})$ and $Q_2 = \cN(m_2, \bV_2^{-1})$ with $d_{V}(Q_1, Q_2) \leq \epsilon$ for sufficiently small $\epsilon > 0$, \citet{arbas2023polynomial} showed in Theorem 1.8 that 
\begin{align*}
    \left\| \bV_1^{1/2} \big( m_1 - m_2 \big) \right\|_2 \vee \left\| \bV_2^{-1/2} \bV_1 \bV_2^{-1/2} - \bI_p \right\|_{\rm F} = O(\epsilon).
\end{align*}
Based on this fact, it follows from Proposition \ref{prop:eigenvalue_order_main} that on $\scrE_{\est, 1}$
\begin{align} \label{eqn:VB_param_bounds}
    \left\| \FisherMAP^{1/2} \big( \mu_t - \thetaMAP \big) \right\|_{2} \vee
    \left\| \bOmega_{t}^{-1/2} \FisherMAP \bOmega_{t}^{-1/2} - \bI_{p} \right\|_{\rm F} 
    \lesssim
    \epsilon_{n, t, \KL} = O(t^{-1} n^{-1/2} p_{\ast}).
\end{align}
Note that, due to the symmetry of the total variation distance, the above bound remains invariant under interchange of $(\thetaMAP, \FisherMAP)$ and $(\mu_t, \bOmega_t)$.
Importantly, \eqref{eqn:VB_param_bounds} plays a crucial role in establishing the online BvM assertion, namely that $d_{V}(\Pi_T, \Pi(\cdot \mid \bD)) \to 0$ in probability.

We conclude this section by discussing the bias term $\| \thetaBest - \theta_0 \|_2$. By Proposition \ref{prop:eigenvalue_order_main}, we have 
\begin{align} \label{eqn:loose_bound_bias}
    \left\| \thetaBest - \theta_0 \right\|_2 \lesssim \lambda_{\min}^{-1/2}\big( \FisherBest \big) M_n \sqrt{t p_{\ast}} \lesssim \left( \dfrac{p_{\ast}}{n} \right)^{1/2}.
\end{align}
Although this bound is sufficient to ensure that $\thetaBest \in \Theta(\theta_0, \bI_p, 1/2)$, the corresponding rate is not sharp enough to guarantee that $\| \thetaMAP - \theta_0 \|_2$ decreases at the optimal rate $(nt)^{-1/2}$ under the standard decomposition analysis in \eqref{eqn:decomposition_pMLE}.
Instead, we will obtain an appropriate rate of $\| \thetaMAP - \theta_0 \|_2$ via a batch learning estimator. We revisit this issue in Proposition \ref{prop:similar_MAP}.


\section{Bernstein--von Mises theorem for full posterior} \label{sec:full_posterior}
This section presents the BvM theorem of full posterior under the batch learning setting.
Recall the definition of \textit{full posterior} in \eqref{def:full_posterior}.
Also, we introduce some notations for analyzing $t$-th full posterior $\Pi( \cdot \mid \bD_{1:t})$.
For $t \in [T]$ and $\theta \in \Theta$, let 
\begin{align}
\begin{aligned} \label{def:batch_quantities}
    \widetilde{L}_{1:t}(\theta) &= -\dfrac{1}{2} \left\| \bOmega_{0}^{1/2} \big( \theta - \mu_0 \big) \right\|_{2}^{2} + L_{1:t}(\theta), \quad &
    \bF_{1:t, \theta} &= -\nabla^2 L_{1:t}(\theta) \\
    \zeta_{1:t}(\theta) &= L_{1:t}(\theta) - \bbE L_{1:t}(\theta), \quad &
    \FullFisherTilde[t]{\theta} &= \bOmega_0 + \FullFisher[t]{\theta}.
\end{aligned}
\end{align}
and
\begin{align} \label{def:batch_maximizer_quantities}
    \fullthetapMLE[t] = \argmax_{\theta \in \Theta} \widetilde{L}_{1:t}(\theta), \quad 
    \fullthetaMLE[t] = \argmax_{\theta \in \Theta} L_{1:t}(\theta), \quad 
    \fullthetaBest[t] = \argmax_{\theta \in \Theta} \bbE \widetilde{L}_{1:t}(\theta).
\end{align}
By applying similar approximation techniques in Section \ref{sec:variational_posterior}, one can prove that the LA would provide an accurate approximation of $\Pi\left(\cdot \mid \bD_{1:t} \right)$ in TV sense. Specifically, since we impose several regularity conditions on the local neighborhood of $\theta_0$ in (\textbf{A2}), we need to establish the concentration behaviors of $\fullthetapMLE$ and $\fullthetaMLE$ for suitable localization steps (e.g., $\fullthetapMLE, \fullthetaMLE \in \Theta(\theta_0, \bI_p, 1/2)$). For this purpose, we impose the following conditions, which are a version of (\textbf{A1}) adapted to the batch learning setup.

\bed
\item[(\textbf{A1$\ast$})] 
Assume that (\textbf{A1}) holds. Also, there exists $\{ \bV_{1:t} : t \in [T] \} \subset \symmPD$ such that
\begin{align} 
\begin{aligned} \label{assume:A1ast_1}
    \bbP_{0}^{(N)} \Bigg( 
    \left\| \FullFisherTilde[t]{\fullthetaBest[t]}^{-1/2} 
    \nabla \zeta_{1:t} \right\|_{2} \geq \widetilde{r}_{\eff, 1:t}
    \text{ for some } t \in [T] \Bigg) &\leq n^{-1}, \\
    \bbP_{0}^{(N)} \Bigg( 
    \left\| \FullFisher[t]{\theta_0}^{-1/2} 
    \nabla \zeta_{1:t} \right\|_{2} \geq r_{\eff, 1:t} \ 
    \text{ for some } t \in [T] \Bigg) &\leq n^{-1},
\end{aligned}    
\end{align}
and
\begin{align}
\begin{aligned}
    \label{assume:A1ast_2}
    \max_{t \in [T]} \sup_{\theta \in \Theta \left(\theta_0, \bI_{p}, 1/2 \right)} \left\| \FullFisher[t]{\theta}^{-1} \bV_{1:t} \right\|_{2} &\leq \dfrac{M_n^2}{9},
\end{aligned}
\end{align}
where 
\begin{align*}
    \widetilde{\lambda}_{1:t} = \left\| \FullFisherTilde[t]{\fullthetaBest[t]}^{-1} \bV_{1:t} \right\|_{2}, \quad
    \lambda_{1:t} = \left\| \FullFisher[t]{\theta_0}^{-1} \bV_{1:t} \right\|_{2},
\end{align*}
and
\begin{align} 
\begin{aligned} \label{def:batch_quantity}
    &\widetilde{r}_{\eff, 1:t} = \widetilde{p}_{\eff, 1:t}^{1/2} + \sqrt{2 \widetilde{\lambda}_{1:t} (\log n + \log T)}, &
    &\widetilde{p}_{\eff, 1:t} = \operatorname{tr}\left( \FullFisherTilde[t] {\fullthetaBest[t]}^{-1} \bV_{1:t} \right), \\
    &r_{\eff, 1:t} = p_{\eff, 1:t}^{1/2} + \sqrt{2\lambda_{1:t} (\log n + \log T)}, &
    &p_{\eff, 1:t} = \operatorname{tr}\left( \FullFisher[t]{\theta_0}^{-1} \bV_{1:t} \right), 
\end{aligned}    
\end{align}
and $M_n$ is the quantity specified in $(\textbf{A1})$.
\eed

Let $\scrE_{\est, 2}$ be the event on which the following inequalities hold: 
\begin{align} \label{def:Gamma_n2_event}
    \left\| \FullFisherTilde[t]{\fullthetaBest[t]}^{-1/2} 
    \nabla \zeta_{1:t} \right\|_{2} \leq \widetilde{r}_{\eff, 1:t}, \quad  
    \left\| \FullFisher[t]{\theta_0}^{-1/2} 
    \nabla \zeta_{1:t} \right\|_{2} \leq r_{\eff, 1:t} \quad 
    \text{ for all } t \in [T].
\end{align}
Then, by \eqref{assume:A1ast_1}, $\bbP_{0}^{(N)} \left( \scrE_{\est, 2} \right) \geq 1 - 2n^{-1}$ under (\textbf{A2}). On $\scrE_{\est, 2}$, we can prove that the following inequalities hold uniformly for all $t \in [T]$:
\begin{align}
\begin{aligned} \label{eqn:batch_bound}
    \left\| \FullFisherTilde[t]{\fullthetaBest}^{1/2} \big( \fullthetapMLE - \theta_0 \big) \right\|_2
    &\lesssim
    \left\| \FullFisherTilde[t]{\fullthetaBest}^{-1/2} \bOmega_0 \big( \theta_0 - \mu_0 \big) \right\|_2
    + M_n p_{\ast}^{1/2} \\
    \left\| \FullFisher[t]{\theta_0}^{1/2} \big( \fullthetaMLE - \theta_0 \big) \right\|_2
    &\lesssim
    M_n p_{\ast}^{1/2}.
\end{aligned}    
\end{align}
More detailed statements for the above inequalities are deferred to Appendix \ref{sec:Appendix_full_posterior}. Here, $\| \FullFisherTilde[t]{\fullthetaBest}^{-1/2}$ $\bOmega_0 ( \theta_0 - \mu_0 ) \|_2$ denotes the bias term arising from the quadratic penalization. Based on these results, we can prove Theorem \ref{thm:batch_posterior_LA}.
Note that the following theorem is the full posterior version of Theorem 1 in Section 3; hence, Theorem 2.2 in \citet{katsevich2025improved} yields an identical result.

\begin{theorem} \label{thm:batch_posterior_LA}
    Suppose that (\textbf{A0}), (\textbf{A1$\ast$}), (\textbf{A2}), (\textbf{S}) and (\textbf{P}) hold.
    Then, on $\scrE_{\est, 2}$, the following inequality holds uniformly for all $t \in [T]$:
    \begin{align*}
        d_{V} \bigg( \cN\left(\fullthetapMLE[t], \FullFisherTilde[t]{\fullthetapMLE[t]}^{-1} \right),  \Pi\left(\cdot \mid \bD_{1:t} \right) \bigg) 
        \leq K \left( \dfrac{p_{\ast}^2}{nt} \right)^{1/2}, 
    \end{align*}
    where $K = K(K_{\min}, K_{\max})$.
\end{theorem}

Although we conclude from Theorem \ref{thm:batch_posterior_LA} that $\cN (\fullthetapMLE[t], \FullFisherTilde[t]{\fullthetapMLE[t]}^{-1} ) \approx \Pi \left(\cdot \mid \bD_{1:t} \right)$, the relationship with $\cN ( \fullthetaMLE, \FullFisher[t]{\theta_0}^{-1} )$ remains unresolved. A standard BvM theorem states that $\cN ( \fullthetaMLE, \FullFisher[t]{\theta_0}^{-1} ) \approx \Pi \left(\cdot \mid \bD_{1:t} \right)$ in the TV sense. Importantly, $\cN ( \fullthetaMLE, \FullFisher[t]{\theta_0}^{-1} )$ does not exhibit the prior effects. Hence, we impose an additional assumption (\textbf{P$\ast$}) for $(\mu_0, \bOmega_0)$ so that the prior effects become asymptotically negligible. 

\bed
\item[(\textbf{P$\ast$})] Assume that $(\textbf{P})$ holds. Also, the initial prior parameters $\mu_0$ and $\bOmega_0$ satisfy 
\begin{align*} 
    \left\| \bOmega_{0}^{1/2} \big( \theta_0 - \mu_0 \big) \right\|_{2} \leq K_{\max} M_n p_{\ast}^{1/2}.
\end{align*}
\eed

\begin{theorem} \label{thm:batch_posterior_BvM}
    Suppose that (\textbf{A0}), (\textbf{A1$\ast$}), (\textbf{A2}), (\textbf{S}) and (\textbf{P$\ast$}) hold.
    Then, on $\scrE_{\est, 2}$, the following inequalities hold uniformly for all $t \in [T]$:
    \begin{align}
    \begin{aligned} \label{eqn:batch_posterior_BvM_claim}
    d_{V} \bigg( 
    \cN\left( \fullthetaMLE, \FullFisher[t]{\theta_0}^{-1} \right),  
    \cN\left( \fullthetapMLE, \FullFisherTilde[t]{\fullthetapMLE[t]}^{-1} \right) \bigg) 
    &\leq K M_n \left( \dfrac{p_{\ast}^{2}}{nt} \right)^{1/2}, \\
    d_{V} \bigg( 
    \cN\left( \fullthetaMLE, \FullFisher[t]{\theta_0}^{-1} \right),  
    \Pi\left(\cdot \mid \bD_{1:t} \right) \bigg) 
    &\leq K M_n \left( \dfrac{p_{\ast}^2}{nt} \right)^{1/2},
    \end{aligned}
    \end{align}
    where $K = K(K_{\min}, K_{\max})$.
\end{theorem}

We briefly introduce the key idea in our proof of Theorem \ref{thm:batch_posterior_BvM}. Our proof strategy is similar to that in \citet{spokoiny2025inexact} and \citet{katsevich2025improved} whose key idea is based on
\begin{align} 
\begin{aligned} \label{eqn:TV_triangular_ineq}
    &d_{V} \Big( 
    \cN\big( \fullthetaMLE, \FullFisher[t]{\theta_0}^{-1} \big),  
    \Pi\left(\cdot \mid \bD_{1:t} \right) \Big) \\
    &\leq 
    d_{V} \Big( \cN\big(\fullthetapMLE[t], \FullFisherTilde[t]{\fullthetapMLE[t]}^{-1} \big),  \Pi\left(\cdot \mid \bD_{1:t} \right) \Big)
    +
    d_{V} \Big( 
    \cN\big( \fullthetaMLE, \FullFisher[t]{\theta_0}^{-1} \big),  
    \cN\big( \fullthetapMLE, \FullFisherTilde[t]{\fullthetapMLE[t]}^{-1} \big) \Big).
\end{aligned}    
\end{align}
Recall the formula \eqref{eqn:TV_error_normals_intro} in Section \ref{sec:overview}.
If the initial prior assigns sufficient probability mass to a neighborhood of $\theta_0$ as described in (\textbf{P$\ast$}), it can be shown that
\begin{align} \label{eqn:efficient_bound}
    \left\| \FullFisherTilde[t]{\fullthetapMLE}^{1/2} \big( \fullthetaMLE - \fullthetapMLE \big) \right\|_2
    \vee
    \left\| \FullFisher[t]{\theta_0}^{-1/2}  \FullFisherTilde[t]{\fullthetapMLE} \FullFisher[t]{\theta_0}^{-1/2} - \bI_p \right\|_{\rm F}
    = 
    O\left( M_n \left( \dfrac{p_\ast^2}{nt} \right)^{1/2} \right),
\end{align}
which yields the first inequality in \eqref{eqn:batch_posterior_BvM_claim}.
Combining this with the result in Theorem \ref{thm:batch_posterior_LA}, one can easily check that the second inequality in \eqref{eqn:batch_posterior_BvM_claim} holds via \eqref{eqn:TV_triangular_ineq}. 

Since we assume the \textit{stochastic linear model} in Section \ref{sec:pMLE}, $\bF_{1:t, \theta_0}$ is non-random; indeed, it coincides with the Fisher information matrix. In more general cases, however, $\bF_{1:t, \theta_0}$ may exhibit randomness. Logistic regression with a random design serves as a representative example. 

To analyze such cases, it is natural to compare $\bF_{1:t, \theta_0}$ with its population counterpart.
For simplicity, suppose that $N_{t}^{-1} \mathbb{E}_{\bX}[\bF_{1:t,\theta_0}]$ is the same for all $t \in [T]$, and define $I_{\theta_0} = N_{t}^{-1} \mathbb{E}_{\bX}[\bF_{1:t,\theta_0}]$. Here, $\bbE_{\bX}$ denotes the expectation over the randomness of $\bX$. Under suitable regularity conditions (e.g. sub-Gaussian random design), we can obtain, for any ${\rm x} > 0$,
\begin{align} \label{eqn:Fisher}
\begin{aligned}
    \left\| \dfrac{1}{N_t} \bF_{1:t,\theta_0} - I_{\theta_0} \right\|_2 &\leq K \sqrt{\dfrac{p + {\rm x}}{N_t}}, \quad
    \left\| \dfrac{1}{N_t} \bF_{1:t,\theta_0} - I_{\theta_0} \right\|_{\rm F} &\leq K \sqrt{\dfrac{p(p + {\rm x})}{N_t}}
\end{aligned}       
\end{align}
with probability at least $1 - e^{-{\rm x}}$, where $K > 0$ is a quantity depending on some regularity conditions (e.g. sub-Gaussian parameter of each row of $\bX$). These bounds are standard; see \citep{tropp2012user, vershynin2018high}. 

This discrepancy is negligible in terms of the final TV convergence rate.
Specifically, combining \eqref{eqn:Fisher} with \eqref{eqn:TV_error_normals_intro},
it is not difficult to see that with probability at least $1 - e^{-{\rm x}}$
\begin{align*}
    d_V \left( \cN \big( \fullthetaMLE, \bF_{1:t, \theta_0}^{-1} \big), \cN \big( \fullthetaMLE,  [ N_tI_{\theta_0} ]^{-1} \big) \right) \lesssim K \sqrt{\dfrac{p(p + {\rm x})}{N_t}}
\end{align*}
provided that $\lambda_{\min}(I_{\theta_0}) \geq C$ for some constant $C > 0$.


\section{Online Bernstein--von Mises theorem} \label{sec:online_variational_posterior}
In this section, we present the main result of this paper, namely the \textit{online BvM theorem}. 
We elaborate on the high-level proof strategy in Section \ref{sec:overview}. 
For clarity, we now introduce explicit definition for $\widehat Q_t$ appearing in \eqref{eqn:overview_assume}. Let $Q_{\LA, t} = \cN( \fullthetapMLE, \FullFisherTilde[t]{\fullthetapMLE}^{-1})$.
By \eqref{eqn:intro_claim_2}, our claim reduces to obtaining upper bounds for the following two quantities:
\begin{align} \label{eqn:fin_claim_2}
    \Big\| \bOmega_{t}^{1/2} \big( \fullthetapMLE - \mu_{t}  \big) \Big\|_{2} \ \text{ and } \
    \Big\| \FullFisherTilde[t]{\fullthetapMLE[t]}^{-1/2} \bOmega_{t} \FullFisherTilde[t]{\fullthetapMLE[t]}^{-1/2} - \bI_{p} \Big\|_{\rm F}.
\end{align}
Propositions \ref{prop:similar_MAP} and \ref{prop:similar_variance} address the first and second terms, respectively.

To bound the first term in \eqref{eqn:fin_claim_2}, we need to ensure that the accumulated approximation errors are asymptotically ignorable. To quantify these errors, we introduce some quantity.
For $\theta \in \Theta$ and $t \in [T]$, define
\begin{align} \label{def:eta_approx}
    \eta_{t}(\theta) = \widetilde{L}_{t}(\theta) + \dfrac{1}{2} \left\| \bOmega_{t}^{1/2} \left( \theta - \mu_{t} \right) \right\|_{2}^{2}, \quad 
    \nabla \eta_{t}(\theta) = \nabla \widetilde{L}_{t}(\theta) + \bOmega_{t}\left( \theta - \mu_{t} \right).
\end{align}
Here, $\eta_{t}(\cdot)$ serves as the error term induced by the $t$-th variational approximation.
Specifically, for $\theta \in \Theta$ and $t \in [T]$, it is not difficult to see that
\begin{align} \label{eqn:eta_equation}
    \widetilde{L}_{1:t}(\theta) = \widetilde{L}_{t}(\theta) + \sum_{s=1}^{t - 1} \eta_{s}(\theta).
\end{align}

Based on \eqref{eqn:eta_equation}, we briefly present key ideas to bound the first term in \eqref{eqn:fin_claim_2}.
Recall that, the definitions of $\fullthetapMLE$ and $\thetaMAP$ imply the following equations
\begin{align*}
    \nabla \widetilde{L}_{1:t}(\fullthetapMLE) = 0, \quad \nabla \widetilde{L}_{t}(\thetaMAP) = 0, \quad  \forall t \in [T].
\end{align*}
By these two score equations and \eqref{eqn:eta_equation}, we have
\begin{align*}
    0 
    &= \nabla \widetilde{L}_{1:t}(\fullthetapMLE) 
    = \nabla \widetilde{L}_{t}(\fullthetapMLE) + \sum_{s=1}^{t - 1} \nabla \eta_{s}(\fullthetapMLE)
    = \nabla \widetilde{L}_{t}(\fullthetapMLE) - \nabla\widetilde{L}_{t}(\thetaMAP) + \sum_{s=1}^{t - 1} \nabla \eta_{s}(\fullthetapMLE) \\
    &= -\widetilde \bF_t(\fullthetapMLE, \thetaMAP) \big( \fullthetapMLE - \thetaMAP \big) + 
    \sum_{s=1}^{t - 1} \nabla \eta_{s}(\fullthetapMLE)
\end{align*}
for some $\widetilde \bF_t(\fullthetapMLE, \thetaMAP)$, depending on $\fullthetapMLE$ and $\thetaMAP$, by Taylor's theorem. If $\fullthetapMLE$ and $\thetaMAP$ are sufficiently close, one can replace $\widetilde \bF_t(\fullthetapMLE, \thetaMAP)$ with $\FisherMAP$ in the last display. Hence, we can obtain the following relation: 
\begin{align} \label{eqn:Taylor_approx_claim}
    \FisherMAP \big( \fullthetapMLE - \thetaMAP \big) \approx
    \widetilde \bF_t(\fullthetapMLE, \thetaMAP) \big( \fullthetapMLE - \thetaMAP \big) =
    \sum_{s=1}^{t - 1} \nabla \eta_{s}(\fullthetapMLE).
\end{align}
From this approximation, we have
\begin{align*}
    \left\| \FisherMAP^{1/2} \big( \fullthetapMLE - \thetaMAP \big) \right\|_{2}
    \approx
    \left\| \sum_{s=1}^{t - 1} \FisherMAP^{-1/2} \nabla \eta_{s}(\fullthetapMLE) \right\|_{2}
    \overset{\rm def}{=}
    \varrho_{n, t}.
\end{align*}
Also, recall from \eqref{eqn:VB_param_bounds} that on $\scrE_{\est, 1}$
\begin{align*}
    \left\| \FisherMAP^{1/2} \big( \thetaMAP - \mu_t \big) \right\|_{2} 
    \vee 
    \left\| \bOmega_{t}^{-1/2} \FisherMAP \bOmega_{t}^{-1/2} - \bI_{p} \right\|_{\rm F}
    \lesssim \epsilon_{n, t, \KL} = o(1)
\end{align*}
Combining these two facts, we can obtain the following rough relations:
\begin{align*}
    \left\| \bOmega_{t}^{1/2} \big( \fullthetapMLE - \mu_{t}  \big) \right\|_{2}
    \approx
    \left\| \FisherMAP^{1/2} \big( \fullthetapMLE - \mu_t \big) \right\|_{2}
    \approx
    \left\| \FisherMAP^{1/2} \big( \fullthetapMLE - \thetaMAP \big) \right\|_{2}
    \approx
    \left\| \sum_{s=1}^{t - 1} \FisherMAP^{-1/2} \nabla \eta_{s}(\fullthetapMLE) \right\|_{2}.
\end{align*}

The main challenge is establishing that $\varrho_{n, t} = o_P(1)$ for all $t \in [T]$. In Proposition \ref{prop:similar_MAP}, we prove that on $\scrE_{\est, 1} \cap \scrE_{\est, 2}$
\begin{align*} 
    \varrho_{n, t} = \left\| \sum_{s=1}^{t - 1} \FisherMAP^{-1/2} \nabla \eta_{s}(\fullthetapMLE) \right\|_{2} = O\left( M_n^{2} \sqrt{\dfrac{p_{\ast}^3}{n}} \right), \quad \forall t \in [T].
\end{align*}
We now formally state the results addressing the first term in \eqref{eqn:fin_claim_2}. 

\begin{proposition} \label{prop:similar_MAP}
    Suppose that (\textbf{A0}), (\textbf{A1$\ast$}), $(\textbf{A2})$, $(\textbf{S})$ and $(\textbf{P})$ hold. 
    Then, on $\scrE_{\est, 1} \cap \scrE_{\est, 2}$, the following inequalities hold uniformly for all $t \in [T]$:
    \begin{align} 
    \begin{aligned} \label{eqn:similar_MAP_claim}
        \left\| \FisherMAP[t]^{1/2} \big( \fullthetapMLE - \thetaMAP[t]  \big) \right\|_{2}
            \vee 
        \left\| \bOmega_{t}^{1/2} \big( \fullthetapMLE - \mu_{t}  \big) \right\|_{2}
            &\leq 
        K M_n^{2} \left( \dfrac{p_{\ast}^{3}}{n} \right)^{1/2}, \\
        \left\| \FisherMAP[t]^{1/2} \big( \theta_0 - \thetaMAP[t] \big) \right\|_{2}
            &\leq 
        K M_n \sqrt{p_{\ast}},
    \end{aligned}
    \end{align}
    where $K = K(K_{\rm min}, K_{\rm max})$.
\end{proposition}

From Proposition \ref{prop:similar_MAP}, we have $\| \thetaMAP - \theta_0 \|_2 \lesssim M_n (p_{\ast}/ N_t)^{1/2}$, where $N_t = nt$. Since Proposition \ref{prop:eigenvalue_order_main} implies that $\| \thetaMAP - \thetaBest \|_2 \lesssim M_n t^{-1/2} (p_{\ast}/ N_t)^{1/2}$, we have
\begin{align} \label{eqn:refined_bias}
    \left\| \thetaBest - \theta_0 \right\|_2 \lesssim 
    \| \thetaMAP - \thetaBest \|_2 \vee \| \thetaMAP - \theta_0 \|_2
    = O \left( M_n  (p_{\ast}/ N_t)^{1/2} \right),
\end{align}
This significantly improves the bound in \eqref{eqn:loose_bound_bias}. 

Under the setting in Proposition \ref{prop:similar_MAP}, one can argue the efficiency of the estimators $\mu_t$ and $\thetaMAP$. More precisely, we can show that
\begin{align*}
    \FullFisher[t]{\theta_0}^{1/2} \big( \thetaMAP - \theta_0 \big) &= \FullFisher[t]{\theta_0}^{-1/2} \nabla L_{1:t}(\theta_0) + o(1), \\
    \FullFisher[t]{\theta_0}^{1/2} \big( \mu_t - \theta_0 \big) &= \FullFisher[t]{\theta_0}^{-1/2} \nabla L_{1:t}(\theta_0) + o(1)
\end{align*}
with high probability. 
Therefore, $\thetaMAP$ and $\mu_t$ serve as asymptotically efficient estimators in the sense of \citet[][Section 8]{van2000asymptotic}. See Corollary \ref{coro:efficient_estimator} for detailed statements.

\begin{proposition} \label{prop:similar_variance}
    Suppose that (\textbf{A0}), (\textbf{A1$\ast$}), $(\textbf{A2})$, $(\textbf{S})$ and $(\textbf{P})$ hold.
    Then, on $\scrE_{\est, 1} \cap \scrE_{\est, 2}$, the following inequality holds uniformly for all $t \in [T]$:
    \begin{align} \label{eqn:claim_similar_variance}
        \left\| \FullFisherTilde[t]{\fullthetapMLE[t]}^{-1/2} \bOmega_{t} \FullFisherTilde[t]{\fullthetapMLE[t]}^{-1/2} - \bI_{p} \right\|_{\rm F} 
        \leq 
        K M_{n} \left( \dfrac{p_{\ast}^{2}}{n} \right)^{1/2}, 
    \end{align}
    where $K = K(K_{\min}, K_{\max})$.
\end{proposition}

We present a sketch of the proof of Proposition \ref{prop:similar_variance}.
Provided that $\lambda_{\min}(\FullFisherTilde[t]{\fullthetapMLE}) \asymp \lambda_{\max}(\FisherMAP)$, we can obtain the following bound:
\begin{align*}
    \left\| \FullFisherTilde[t]{\fullthetapMLE[t]}^{-1/2} \bOmega_{t} \FullFisherTilde[t]{\fullthetapMLE[t]}^{-1/2} - \bI_{p} \right\|_{\rm F} 
    &\lesssim
    \left\| \FisherMAP^{-1/2} \bOmega_{t} \FisherMAP^{-1/2} - \bI_{p} \right\|_{\rm F} 
    +
    \left\| \FisherMAP^{-1/2} \FullFisherTilde[t]{\fullthetapMLE[t]} \FisherMAP^{-1/2} - \bI_{p} \right\|_{\rm F} \\
    &\lesssim
    t^{-1} n^{-1/2} p_{\ast}
    +
    \left\| \FisherMAP^{-1/2} \FullFisherTilde[t]{\fullthetapMLE[t]} \FisherMAP^{-1/2} - \bI_{p} \right\|_{\rm F},
\end{align*}
where the second inequality holds by \eqref{eqn:VB_param_bounds}.
Hence, we only need to verify $\FullFisherTilde[t]{\fullthetapMLE[t]} \approx \FisherMAP$ in a certain sense. From the definitions in \eqref{def:Fisher_info_quantities} and \eqref{def:batch_quantities}, one can obtain the following equality:
\begin{align} \label{eqn:decomposition_Fisher}
    \FullFisherTilde[t]{\fullthetapMLE} - \FisherMAP 
    =  
    \sum_{s=1}^{t-1} \big( \FisherMAP[s] - \bOmega_s \big)
    +
    \sum_{s=1}^{t} \big( \Fisher[s]{\fullthetapMLE} - \Fisher[s]{\thetaMAP[s]} \big).
\end{align}
Then, under some conditions, we can show that
\begin{align*}
    \left\| \FisherMAP^{-1/2} \FullFisherTilde[t]{\fullthetapMLE[t]} \FisherMAP^{-1/2} - \bI_{p} \right\|_{\rm F}
    &\lesssim 
    t^{-1} \sum_{s=1}^{t-1} \big( s \cdot \epsilon_{s, 1} \big)
    +
    t^{-1} \sum_{s=1}^{t}  \epsilon_{s, 2}
    \lesssim
    \left( \dfrac{p_{\ast}^2 }{n} \right)^{1/2} + \max_{s \in [t]} \epsilon_{s, 2},
\end{align*}
where the second inequality holds by \eqref{eqn:VB_param_bounds}, and
\begin{align*}
    \epsilon_{s, 1} = \left\| \FisherMAP[s]^{-1/2} \bOmega_s \FisherMAP[s]^{-1/2} - \bI_p \right\|_{\rm F}, \quad 
    \epsilon_{s, 2} = \left\| \Fisher[s]{\thetaMAP[s]}^{-1/2} \Fisher[s]{\fullthetapMLE} \Fisher[s]{\thetaMAP[s]}^{-1/2} - \bI_p \right\|_{\rm F}.
\end{align*}

Bounding $\epsilon_{s, 2}$ requires that the map $\theta \mapsto \Fisher[s]{\theta}$ is sufficiently smooth. If $\fullthetapMLE$ is sufficiently close to $\thetaMAP[s]$, then we can expect that $\Fisher[s]{\fullthetapMLE} \approx \Fisher[s]{\thetaMAP[s]}$; see Lemma \ref{lemma:tech_Fisher_smooth} for a precise statement.   
Indeed, under some conditions, we can derive from \eqref{eqn:batch_bound} and \eqref{eqn:similar_MAP_claim} that 
\begin{align*}
    \big\| \fullthetapMLE - \thetaMAP[s] \big\|_2 
    \leq \big\| \fullthetapMLE - \theta_0 \big\|_2 + \big\| \thetaMAP[s] - \theta_0 \big\|_2 
    = O(M_n (p_{\ast}/ N_s)^{1/2} ).
\end{align*}
By applying Lemma \ref{lemma:tech_Fisher_smooth}, therefore, we can obtain $\max_{s \in [t]} \epsilon_{s, 2} = O(M_n p_{\ast} n^{-1/2})$, which results in \eqref{eqn:claim_similar_variance}.

\begin{theorem}[Online Bernstein--von Mises Theorem] \label{thm:onlineBvM}
    Suppose that (\textbf{A0}), (\textbf{A1$\ast$}), $(\textbf{A2})$, $(\textbf{S})$ and $(\textbf{P})$ hold.
    Then, 
    \begin{align*}
        \bbP_{0}^{(N)} \left(
        d_{V} \bigg( \Pi_{t}(\cdot),  \Pi\left(\cdot \mid \bD_{1:t} \right) \bigg) \leq K M_n^{2} \left( \dfrac{p_{\ast}^{3}}{n} \right)^{1/2} 
        \ \text{for all } t \in [T]
        \right)
        \geq 1 - 3n^{-1}.
    \end{align*}    
    If, in addition, (\textbf{P$\ast$}) holds, then
    \begin{align*}
        \bbP_0^{(N)} \left(
        d_{V} \bigg( \Pi_{t}, \cN\left( \fullthetaMLE, \FullFisher[t]{\theta_0}^{-1} \right) \bigg) 
        \leq  
        K M_n^{2} \left( \dfrac{p_{\ast}^{3}}{n} \right)^{1/2} \ \text{for all } t \in [T]
        \right)
        \geq 1 - 3n^{-1}.
    \end{align*}
    Here, $K = K(K_{\min}, K_{\max})$.
\end{theorem}


When $p$ is fixed, Theorem \ref{thm:onlineBvM} guarantees that the online BvM theorem holds even with a very small mini-batch size $n$. Specifically, if $n \gg (\log N)^4$ (as required by condition (\textbf{S})), the online BvM theorem holds. Moreover, Theorem \ref{thm:onlineBvM} further ensures that credible sets based on $\Pi_T$ yield asymptotically valid frequentist confidence sets. To be specific, for $\alpha \in (0, 1)$, consider the following Wald-type confidence and credible sets:
\begin{align*}
    \widehat{C}_{N}(\alpha) &= \left\{ \theta \in \Theta : \left\| \FullFisher[T]{\theta_0}^{1/2} \big( \theta - \fullthetaMLE[T] \big) \right\|_2^2 \leq \chi_{p, \alpha}^2 \right\},
    \\
    \widehat{C}_{n, T}(\alpha) &= \left\{ \theta \in \Theta : \left\| \bOmega_{T}^{1/2} \big( \theta - \mu_T \big) \right\|_2^2 \leq \chi_{p, \alpha}^2 \right\},
\end{align*}
where $\chi_{p, \alpha}^2$ denotes the $(1-\alpha)$ quantile of the $\chi_p^2$ distribution. Here, $\widehat{C}_{N}(\alpha)$ represents the standard frequentist confidence set based on batch MLE, while $\widehat{C}_{n, T}(\alpha)$ is the credible set based on the sequentially updated posterior $\Pi_T$.
Under the setting in Theorem \ref{thm:onlineBvM}, we have on $\scrE_{\est, 1} \cap \scrE_{\est, 2}$
\begin{align*}
    N \big\| \bOmega_T^{-1} - \FullFisher[t]{\theta_0}^{-1} \big\|_{\rm F} &= O\left( M_n \left( p_{\ast}^2 / n\right)^{1/2}  \right), \\
    \big\| \FullFisher[T]{\theta_0}^{1/2} \big( \theta_0 - \fullthetaMLE[T] \big) \big\|_2
    &=
    \big\| \bOmega_{T}^{1/2} \big( \theta_0 - \mu_T \big) \big\|_2
    + 
    O\left( M_n \left( p_{\ast}^3 / n\right)^{1/2}  \right).    
\end{align*}
See Corollaries \ref{coro:variance_consistency} and \ref{coro:CI_consistency} for precise statements.
Therefore, we can conclude that on $\scrE_{\est, 1} \cap \scrE_{\est, 2}$ the following inclusions hold:
\begin{align*}
  \widehat{C}_{N}(\alpha - \epsilon_n)
  \subset
  \widehat{C}_{n, T}(\alpha)
  \subset
  \widehat{C}_{N}(\alpha + \epsilon_n),
\end{align*}
where $\epsilon_n = C M_n \left( p_{\ast}^3 / n\right)^{1/2}$ for some constant $C > 0$. Hence, as long as $\widehat{C}_N$ is a valid confidence set, the credible set $\widehat{C}_{n, T}$ also provides valid frequentist coverage.

Our final bound in the online BvM theorem requires the condition 
$n \gg M_n^4 p_{\ast}^3$, which is stronger than the requirement $n \gg (M_n^2 \vee \log^2 T)\, p_{\ast}^2$ specified in the condition (\textbf{S}). This apparent discrepancy arises from a technical artifact in our proof. Specifically, in Proposition \ref{prop:similar_MAP}, we bound the spectral norm by the Frobenius norm:
    \begin{align*}
        \left\| \FisherTilde[t]{\thetaMAP[t]}^{-1/2} \bOmega_t \FisherTilde[t]{\thetaMAP[t]}^{-1/2} - \bI_p  \right\|_{2} 
        \leq 
        \left\| \FisherTilde[t]{\thetaMAP[t]}^{-1/2} \bOmega_t \FisherTilde[t]{\thetaMAP[t]}^{-1/2} - \bI_p  \right\|_{\rm F} 
        \lesssim
        \left( \dfrac{p_{\ast}^2}{nt^2} \right)^{1/2},
    \end{align*}
which introduces an additional factor of $p_{\ast}$ in the final bound.

A sharper analysis is possible by considering LA instead of VB. Under LA, we have $\mu_t = \thetaMAP[t]$ and $\bOmega_t = \FisherTilde[t]{\thetaMAP[t]}$, and thus
\begin{align*}
    \left\| \FisherTilde[t]{\thetaMAP[t]}^{-1/2} \bOmega_t \FisherTilde[t]{\thetaMAP[t]}^{-1/2} - \bI_p  \right\|_{2} 
    = 0.
\end{align*}
Consequently, the bounds in Proposition \ref{prop:similar_MAP} and Theorem \ref{thm:onlineBvM} can be improved from $M_n^2 (p_{\ast}^3 /n)^{1/2}$ to $M_n^2 (p_{\ast}^2 /n)^{1/2}$. 

However, it should not be interpreted as implying that LA-based online updates are strictly superior to VB-based ones. Importantly, this discrepancy does not reflect an inherent limitation of VB. Rather, the weaker dimension dependence arises from our technical analysis. Although we do not provide a formal proof, we believe that a refined analysis could also improve the dimension dependence under VB. We plan to investigate this direction in future work, especially in connection with one-pass algorithms and sharper dimension requirements.


\section{Numerical experiments} \label{sec:numerical_experiments}

In this section, we present small-scale numerical experiments to complement our theoretical findings. In particular, we observe that the Bayesian online estimator performs comparably to the batch estimator when the mini-batch size $n$ exceeds a certain threshold, while its performance deteriorates when $n$ falls below this threshold. Moreover, the performance gap between the online and batch estimators becomes larger as the number of mini-batches $T$ increases when $n$ is small, whereas the gap remains negligible for sufficiently large $n$ regardless of $T$.

For the simulation study, we consider a sequence of observations $Y_i \overset{\iid}{\sim} \operatorname{Bernoulli}(1/2)$ for all $i \in [N]$, which corresponds to one-dimensional logistic regression with the true parameter $\theta_0 = 0$. We generate $N = 1000$ samples and partition them into mini-batches of size $n = 1, 2, 4, 6, 8, 10, 20, 50, 200$, and $1000$.

For the online Bayesian updates, we solve the minimization problem \eqref{eq:KL-projection}. It is well known that minimizing the KL divergence in \eqref{eq:KL-projection} is equivalent to maximizing the evidence lower bound (ELBO), defined as
\begin{align} \label{max_ELBO}
    \max_{Q \in \cQ } \left\{ \mathbb{E}_{Q} \big[ L_t(\theta)\big] - K \big(Q ;\, \Pi_{t-1}\big) \right\},
\end{align}
where $\bbE_{Q}$ denotes the expectation with respect to the normal distribution $Q = \cN(\mu, \bOmega^{-1})$. Although the ELBO for logistic regression is concave in $(\mu, \bOmega)$, the expectation in \eqref{max_ELBO} is not available in closed form. We approximate it via Monte Carlo integration using $10^3$ samples, and then optimize \eqref{max_ELBO} using a gradient ascent algorithm with a suitably chosen learning rate. The initial prior is set as $\Pi_0 = \cN(0, 3^2)$.

\begin{figure}[t!] 
\centering
\begin{minipage}{0.58\textwidth} 
    \centering
    \vspace{1.8em}
    \includegraphics[width=\linewidth]{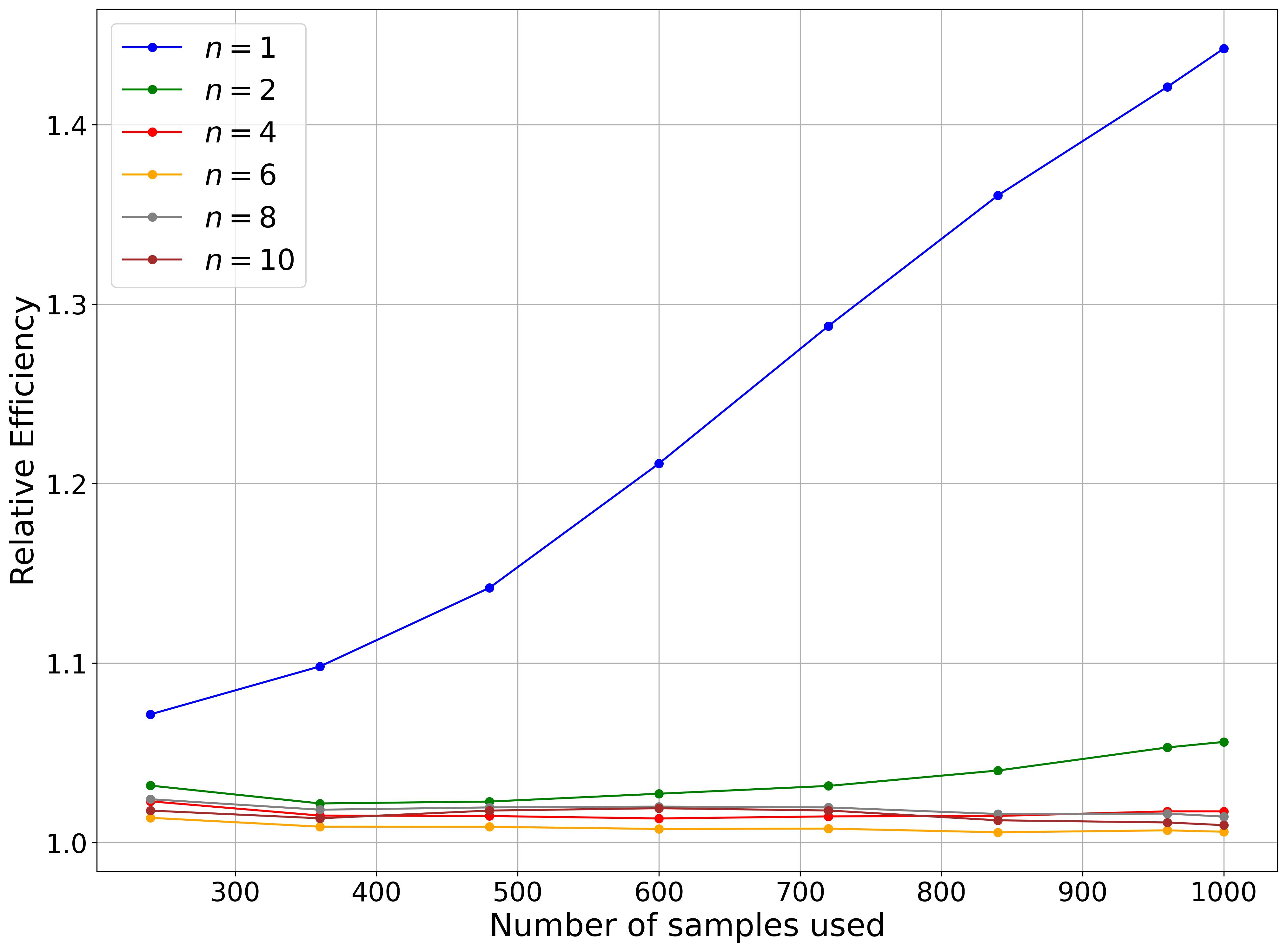}
    \vspace{-1.0em}
    \caption{Relative efficiency in terms of MSE across varying mini-batch size.}
    \label{fig:re_eff}
\end{minipage}
\hfill
\begin{minipage}{0.38\textwidth} 
    \centering
    \captionof{table}{Coverage probability (CP) and length of 95\% CIs over $M=500$ repetitions.}
    \renewcommand{\arraystretch}{1.15} 
    \begin{tabular}{c|cc}
        \hline
        \textbf{Method} & \textbf{CP} & \textbf{Length} \\
        \hline\hline
        $\widehat{\theta}_{1:t}^{\ML}$ & 0.944 & 0.248 \\
        $\mu_t$, $n=1$   & 0.986 & 0.398 \\
        $\mu_t$, $n=2$   & 0.974 & 0.301 \\
        $\mu_t$, $n=4$   & 0.956 & 0.269 \\
        $\mu_t$, $n=6$   & 0.950 & 0.259 \\
        $\mu_t$, $n=8$   & 0.948 & 0.255 \\
        $\mu_t$, $n=10$  & 0.942 & 0.253 \\
        $\mu_t$, $n=20$  & 0.942 & 0.250 \\
        $\mu_t$, $n=50$  & 0.938 & 0.248 \\
        $\mu_t$, $n=200$ & 0.938 & 0.248 \\
        $\mu_t$, $n=1000$& 0.942 & 0.248 \\
        \hline
    \end{tabular}
    \label{tab:comparison_transposed}
\end{minipage}
\end{figure}

We first compare online and batch learning in terms of point estimation performance. Specifically, we compare the online estimator $\mu_t$ with the batch estimator $\widehat{\theta}_{1:t}^{\ML}$ using the mean squared error (MSE) as the evaluation metric. The experiment was repeated $M = 500$ times, and for each replicate, we obtained an online estimator $\mu_{t, m}$ and a batch estimator $\widehat{\theta}_{1:t,m}^{\ML}$ for $m = 1, \ldots, M$. In all cases (even when $n = 1$), we observed that $|\mu_t - \theta_0| \to 0$ as $t \to \infty$, indicating that the point estimators are consistent regardless of the mini-batch size $n$. For further comparison, we also computed the \textit{relative efficiency}, defined as
\begin{align*}
    \operatorname{RE}_{t} = \frac{\sum_{m=1}^{M}|\mu_{t,m} - \theta_0 |^2}{\sum_{m=1}^{M}|\widehat{\theta}_{1:t,m}^{\ML} - \theta_0 |^2}.
\end{align*}

Figure \ref{fig:re_eff} shows the relative efficiency. For small mini-batch size (e.g., $n = 1$), the performance of the Bayesian online learning method deteriorates as the number of mini-batches $T$ increases. In this case, the accumulated error becomes non-negligible and the online estimator suffers a significant loss of statistical efficiency relative to the MLE. A similar, though less severe, pattern is observed for $n = 2$. However, as $n$ exceeds a certain threshold, the relative efficiency approaches 1, indicating that the online estimators perform comparably to the MLE.

We also assess the coverage probability of the credible/confidence intervals (CI) for $\theta_0$. 
Table \ref{tab:comparison_transposed} summarizes the coverage probability and average length of 95\% CIs over 500 repetitions for various mini-batch sizes, alongside those based on the MLE. As the mini-batch size increases, both the coverage probability and CI length tend to align with those of the MLE.

When the mini-batch size is small, however, our numerical results reveal an intriguing phenomenon: the online variational posterior exhibits conservative coverage. This behavior is not immediately apparent from standard intuition. The variance underestimation of mean-field VB is well known in the literature \citep{blei2017variational}.
Even when a full Gaussian variational family is employed, variance underestimation may still occur in small-sample regimes due to the directionality of the KL divergence, $\min_Q K(Q, \Pi_t)$, although this effect vanishes asymptotically (see Section \ref{sec:variational_posterior}). Furthermore, the posterior mean exhibits a non-negligible bias when the mini-batch size is small (e.g., $n=1,2$), as illustrated in Figure \ref{fig:re_eff}. Such bias would ordinarily lead to under-coverage of credible sets.

Contrary to this expectation, the empirical coverage becomes conservative in this small-$n$ regime.
Although a rigorous explanation is currently lacking, these findings suggest that online VB may implicitly incorporate a form of safety effect. We conjecture that variance inflation induced by sequential updates offsets the usual variance underestimation associated with KL-based variational approximations. A theoretical justification of this implicit regularization effect is deferred to future work.

\section{Discussion} \label{sec:discussion}

There are many papers on Bayesian online learning, but a large portion of them are based on heuristic ideas. Although our theoretical analysis is limited to regular parametric models, the online BvM theorem established in this paper provides strong theoretical justification for online Bayesian procedures and offers a useful mathematical tool for studying further theoretical properties in more complex models.

Remarkably, the online BvM theorem holds with a very small mini-batch size; specifically, $n \gg (\log N)^4$ is sufficient when $p$ is fixed. On the other hand, our simulation results indicate that if the mini-batch size is too small---say, $n = 1$---then although the online procedure still yields a consistent estimator, it loses statistical efficiency. We believe that this is due to the particular choice of approximation in our online learning procedure, which is based on variational approximation. Given that frequentist one-pass algorithms can yield efficient estimators, we believe that it is possible to develop an online Bayesian procedure with some algorithmic modifications that also achieves frequentist efficiency. We leave this as a direction for future work.


\acks{This work was supported by the National Research Foundation of Korea (NRF) grant funded by the Korea government (MSIT) (No. RS-2023-00240861), a Korea Institute for Advancement of Technology (KIAT) grant funded by the Korea Government (MOTIE) (RS-2024-00409092, 2024 HRD Program for Industrial Innovation), and Institute of Information \& Communications Technology Planning \& Evaluation(IITP)-Global Data-X Leader HRD program grant funded by the Korea government (MSIT) (IITP-2024-RS-2024-00441244).
}


\vskip 0.2in
\bibliography{Online_BvM}

\appendix

\section{Notations}

\renewcommand{\arraystretch}{1.2} 
\begin{table}[hbt!]
\caption{Table of Notations} \label{table:notations_app}
\centering
\begin{tabular}{c c c p{5cm}} 
\hline
Notation & Location \\ [0.5ex] 
\hline\hline
$D_{t}$, \ $N_{t,1}(\cdot), \ N_{t, 2}(\cdot), \ N_{t, 3}(\cdot)$ & \eqref{def:approximation_quantities} \\
$\Delta_{\local, t}(\cdot), \ \Delta_{\tail, \widetilde{\Pi}, t}(\cdot), \ \Delta_{\tail, \LA, t}(\cdot)$ & \eqref{def:integral_quantities} \\ 
$b_{n, t}$ & \eqref{def:bias_quantity} \\
$\tau_{3, t, \bias}$, \ $\tau_{4, t, \bias}$ & \eqref{def:bias_quantity} \\
$\Delta_{t}$ & \eqref{def:VB_quantities} \\
\hline
\end{tabular}
\end{table}

For $k \geq 2$, we say that a $k$-th order tensor $\bA = (A_{i_1, ..., i_k})_{i_1, ..., i_k \in [p]} \in \bbR^{p^{k}}$ is symmetric if $A_{i_1, ..., i_k} = A_{\sigma(i_1, ..., i_k)}$ for any permutation map $\sigma$ of the tuple $(i_1, ..., i_k)$.
For a symmetric $3$-order tensor $\bA = (A_{ijk})_{i,j,k \in [p]} \in \bbR^{p \times p \times p}$, $\mathbf{B} = (B_{jk})_{j, k \in [p]} \in \symmPD$ and $x \in \bbR^{p}$, let
\begin{align*}
    \langle \bA, \mathbf{B} \rangle = \left( \sum_{j, k \in [p]} A_{ijk} B_{jk} \right)_{i \in [p]} \in \bbR^{p}, \quad 
    \langle \bA, x \rangle = \left( \sum_{k \in [p]} A_{ijk} x_{k} \right)_{i, j \in [p]} \in \bbR^{p \times p}.
\end{align*}

For a function $f : \Theta \rightarrow \bbR$, let $\| f \|_{\infty} = \sup_{\theta \in \Theta} |f(\theta)|$.
For $A \subseteq \Theta$, let $\mathds{1}_{A}(\theta) : \Theta \rightarrow \{0, 1\}$ be the indicator function, defined as $1$ if $\theta \in A$ and $0$ otherwise.

\section{Proofs for Section \ref{sec:variational_posterior}} \label{sec:appendix_LA_VB_app}

Throughout this section, let $\Theta_{n, t} = \Theta (\FisherMAP, 4r_{\LA} )$. (Note that the notation $\Theta_{n,t}$ is reserved for defining other local sets in subsequent sections.)

\subsection{Laplace approximation}

For a function $g : \Theta \rightarrow \bbR$, let 
\begin{align} \label{def:TV_quantities}
\begin{aligned}
    \cI_{\widetilde{\Pi}, t}(g) 
    &\coloneqq 
    \dfrac{ \displaystyle \int g(\theta) \exp \left[ \widetilde{L}_{t}(\theta) \right] \rmd \theta  }{ \displaystyle \int \exp \left[ \widetilde{L}_{t}(\theta') \right] \rmd \theta' }
    =
    \dfrac{ \displaystyle \int g(u) e^{f_{t}(u; \thetaMAP)} \rmd u }{\displaystyle \int e^{f_{t}(u'; \thetaMAP)} \rmd u' }, 
    \quad \\    
    \cI_{\LA, t}(g) 
    &\coloneqq 
    \dfrac{\displaystyle \int g(\theta) \exp\left[  - \dfrac{1}{2} \left\| \FisherMAP[t]^{1/2} \left( \theta - \thetaMAP[t] \right) \right\|_{2}^{2} \right] \rmd \theta }
    {\displaystyle \int \exp\left[ - \dfrac{1}{2}\left\| \FisherMAP[t]^{1/2} \left( \theta' - \thetaMAP[t] \right) \right\|_{2}^{2} \right] \rmd \theta' } \\
    &=
    \dfrac{\displaystyle \int g(\theta) \exp\left[  - \dfrac{1}{2}\left\| \FisherMAP[t]^{1/2}u \right\|_{2}^{2} \right] \rmd u }
    {\displaystyle \int \exp\left[ - \dfrac{1}{2}\left\| \FisherMAP[t]^{1/2} u' \right\|_{2}^{2} \right] \rmd u' },
\end{aligned}
\end{align}
where $\int f(u) du = \int_{\Theta} f(u) du$ and
\begin{align} \label{def:f_t}
    f_{t}(u; \thetaMAP) 
    = \widetilde{L}_{t}(\thetaMAP + u) - \widetilde{L}_{t}(\thetaMAP)
    = \widetilde{L}_{t}(\thetaMAP + u) - \widetilde{L}_{t}(\thetaMAP) - \langle \nabla \widetilde{L}_{t}(\thetaMAP), u \rangle.
\end{align}
For $t \in [T]$, the total variation distanace between $\Pi_{t}^{\LA}$ and $\widetilde{\Pi}_{t}\left(\cdot \mid \bD_{t} \right)$ can be represented by
\begin{align*}
        d_{V} \left( \Pi_{t}^{\LA}(\cdot),  \widetilde{\Pi}_{t}\left(\cdot \mid \bD_{t} \right) \right) 
        = 
        \sup_{\|g\|_\infty \leq 1} \left| \cI_{\widetilde{\Pi}, t}(g) - \cI_{\LA, t}(g) \right|,        
\end{align*}
where the supremum is taken over every measurable function $g$ with $\| g \|_{\infty} \leq 1$.
Define the following quantities:
\begin{align} 
\begin{aligned} \label{def:integral_quantities}
    &\Delta_{\local, t}(g) = \left| \dfrac{N_{t, 1}(g)}{D_{t}} \right|, \quad 
    \Delta_{\tail, \widetilde{\Pi}, t}(g) = \left| \dfrac{N_{t, 2}(g)}{D_{t}} \right|, \quad  
    \Delta_{\tail, \LA, t}(g) = \left| \dfrac{N_{t, 3}(g)}{D_{t}} \right|, \\
    &\Delta_{\tail, t}(g) = \Delta_{\tail, \widetilde{\Pi}, t}(g) \vee \Delta_{\tail, \LA, t}(g),    
\end{aligned}
\end{align}
where
\begin{align} 
\begin{aligned} \label{def:approximation_quantities}
    D_{t} &= \int_{\Theta} \exp\left( -\dfrac{1}{2} \left\| \FisherMAP^{1/2} u \right\|_{2}^{2} \right) \rmd u, \\
    N_{t, 1}(g) &= \int_{\Theta_{n, t}} g(u) 
    \left[ \exp\left( f_{t}(u, \thetaMAP) \right) - \exp\left( -\dfrac{1}{2} \left\| \FisherMAP^{1/2} u \right\|_{2}^{2} \right) \right] \rmd u, \\
    N_{t, 2}(g) &= \int_{ \Theta_{n, t}^{\rm c} } |g(u)| \exp\left( f_{t}(u, \thetaMAP) \right) \rmd u, \\
    N_{t, 3}(g) &= \int_{ \Theta_{n, t}^{\rm c} } |g(u)| \exp\left( -\dfrac{1}{2} \left\| \FisherMAP^{1/2} u \right\|_{2}^{2} \right) \rmd u.
\end{aligned}
\end{align}

\begin{lemma} \label{lemma:LA_TV_integral}
    Suppose that (\textbf{A0}) holds.
    Also, assume that $\Delta_{\local, t}(\mathds{1}_{\Theta}) + \Delta_{\tail, t}(\mathds{1}_{\Theta}) \leq 1/2$ for all $t \in [T]$ on an event $\scrE$.
    Let $g : \Theta \rightarrow \bbR$ be a function satisfying $\| g \|_{\infty} \leq 1$. 
    Then, for all $t \in [T]$,
    \begin{align*}
        \left| \cI_{\widetilde{\Pi}, t}(g) - \cI_{\LA, t}(g) \right| \leq 
        2\bigg( \Delta_{\local, t}(g) + \Delta_{\local, t}(\mathds{1}_{\Theta}) + 2\Delta_{\tail, t}(g) + 2\Delta_{\tail, t}(\mathds{1}_{\Theta}) \bigg)
    \end{align*}
    on $\scrE$.
\end{lemma}

\begin{proof}
In this proof, we work on the event $\scrE$ without explicitly referring to it.
Let
\begin{align*}
({\rm i}) &= 
\dfrac{ \displaystyle \int  g(u) e^{f_{t}(u; \thetaMAP)} \rmd u  }{ \displaystyle \int  e^{f_{t}(u'; \thetaMAP)} \rmd u' }
-
\dfrac{ \displaystyle \int  g(u) \exp\left[ -\dfrac{1}{2}\left\| \FisherMAP[t]^{1/2}u \right\|_{2}^{2} \right]  \rmd u }{ \displaystyle \int  e^{f_{t}(u'; \thetaMAP)} \rmd u' }, \\
({\rm ii}) &=    
\dfrac{ \displaystyle \int  g(u) \exp\left[ -\dfrac{1}{2}\left\| \FisherMAP[t]^{1/2}u \right\|_{2}^{2} \right]  \rmd u }{ \displaystyle \int  e^{f_{t}( u'; \thetaMAP)} \rmd u' }
- \dfrac{ \displaystyle \int  g(u) \exp\left[ -\dfrac{1}{2}\left\| \FisherMAP[t]^{1/2}u \right\|_{2}^{2} \right]  \rmd u }
    { \displaystyle \int  \exp\left[ -\dfrac{1}{2}\left\| \FisherMAP[t]^{1/2}u' \right\|_{2}^{2} \right]  \rmd u' }.
\end{align*}
Note that
\begin{align} \label{eqn:TV_integral_eq1}
&\left| \cI_{\widetilde{\Pi}, t}(g) - \cI_{\LA, t}(g) \right|
\leq  \left| ({\rm i}) \right| + \left| ({\rm ii}) \right|.
\end{align}

Firstly, we will obtain an upper bound of $(\rm i)$. Note that
\begin{align*}
    \left| ({\rm i}) \right|
    \leq
    \left| \displaystyle \int e^{f_{t}(u; \thetaMAP)} \rmd u \right|^{-1} \times
    \left| \displaystyle \int g(u) \left[ e^{f_{t}(u; \thetaMAP)} - \exp\left( -\dfrac{1}{2}\left\| \FisherMAP[t]^{1/2}u \right\|_{2}^{2} \right) \right] \rmd u \right|.
\end{align*}
By the definitions in \eqref{def:integral_quantities}, note that
\begin{align*}
    D_{t} \Delta_{\local, t}(\mathds{1}_{\Theta}) 
    &= |N_{t, 1}(\mathds{1}_{\Theta})|
    = \left| \int_{\Theta_{n, t}} 
    \left[ \exp\left( f_{t}(u, \thetaMAP) \right) - \exp\left( -\dfrac{1}{2} \left\| \FisherMAP^{1/2} u \right\|_{2}^{2} \right) \right] \rmd u \right| \\
    &\geq 
    \int_{\Theta_{n, t}} 
    \exp\left( -\dfrac{1}{2} \left\| \FisherMAP^{1/2} u \right\|_{2}^{2} \right) \rmd u
    -
    \int_{\Theta_{n, t}} 
    \exp\left( f_{t}(u, \thetaMAP) \right)  \rmd u,
\end{align*}
which implies that
\begin{align*}
    &\int_{\Theta_{n, t}} 
    \exp\left( f_{t}(u, \thetaMAP) \right)  \rmd u \\
    &\geq 
    \int_{\Theta_{n, t}} 
    \exp\left( -\dfrac{1}{2} \left\| \FisherMAP^{1/2} u \right\|_{2}^{2} \right) \rmd u
    -
    D_{t} \Delta_{\local, t}(\mathds{1}_{\Theta}) \\
    &=
    \int_{\Theta_{n, t}} 
    \exp\left( -\dfrac{1}{2} \left\| \FisherMAP^{1/2} u \right\|_{2}^{2} \right) \rmd u
    -
    \Delta_{\local, t}(\mathds{1}_{\Theta})
    \int
    \exp\left( -\dfrac{1}{2} \left\| \FisherMAP^{1/2} u \right\|_{2}^{2} \right) \rmd u.
\end{align*}
It follows that
\begin{align}
\begin{aligned} \label{eqn:TV_integral_eq2}
    &\displaystyle \int   e^{f_{t}(u; \thetaMAP)} \rmd u 
    \geq 
    \displaystyle \int  _{\Theta_{n, t} } e^{f_{t}(u; \thetaMAP)} \rmd u \\
    &\geq 
    \displaystyle \int  _{\Theta_{n, t} } \exp\left( -\dfrac{1}{2}\left\| \FisherMAP[t]^{1/2}u \right\|_{2}^{2} \right) \rmd u
    - \Delta_{\local, t}(\mathds{1}_{\Theta}) \displaystyle \int   \exp\left( -\dfrac{1}{2}\left\| \FisherMAP[t]^{1/2}u \right\|_{2}^{2} \right) \rmd u \\
    &= 
    \bigg( 1 - \Delta_{\tail, \LA, t}(\mathds{1}_{\Theta}) - \Delta_{\local, t}(\mathds{1}_{\Theta}) \bigg) \displaystyle \int   \exp\left( -\dfrac{1}{2}\left\| \FisherMAP[t]^{1/2}u \right\|_{2}^{2} \right) \rmd u
\end{aligned}
\end{align}
and
\begin{align}
\begin{aligned} \label{eqn:TV_integral_eq3}
&\left| \displaystyle \int   g(u) \left[ e^{f_{t}(u; \thetaMAP)} - \exp\left( -\dfrac{1}{2}\left\| \FisherMAP[t]^{1/2}u \right\|_{2}^{2} \right) \right] \rmd u \right| \\
&\leq 
\left| \displaystyle \int  _{\Theta_{n, t} } g(u) \left[ e^{f_{t}(u; \thetaMAP)} - \exp\left( -\dfrac{1}{2}\left\| \FisherMAP[t]^{1/2}u \right\|_{2}^{2} \right) \right] \rmd u \right| \\
&\qquad +
\left| \displaystyle \int  _{\Theta_{n, t}^{\rm c} } g(u) \left[ e^{f_{t}(u; \thetaMAP)} - \exp\left( -\dfrac{1}{2}\left\| \FisherMAP[t]^{1/2}u \right\|_{2}^{2} \right) \right] \rmd u \right| \\
&\leq 
\bigg( \Delta_{\local, t}(g) + \Delta_{\tail, \widetilde{\Pi}, t}(g) + \Delta_{\tail, \LA, t}(g) \bigg) \displaystyle \int   \exp\left( -\dfrac{1}{2}\left\| \FisherMAP[t]^{1/2}u \right\|_{2}^{2} \right) \rmd u.
\end{aligned}
\end{align}
It follows that 
\begin{align} \label{eqn:TV_integral_eq4}
\begin{aligned}
    \left| ({\rm i}) \right| 
    &\leq
    \bigg( 1 - \Delta_{\tail, \LA, t}(\mathds{1}_{\Theta}) - \Delta_{\local, t}(\mathds{1}_{\Theta}) \bigg)^{-1} 
    \\
    &\qquad \times \bigg( \Delta_{\local, t}(g) + \Delta_{\tail, \widetilde{\Pi}, t}(g) + \Delta_{\tail, \LA, t}(g) \bigg).    
\end{aligned}    
\end{align}

Next, we will obtain an upper bound of $(\rm ii)$.
Let $\widetilde{Z} \sim \cN(0, \FisherMAP^{-1})$ and $\bbE_{\widetilde{Z}}$ denote the expectation with respect to the law of $\widetilde{Z}$.
Note that
\begin{align*}
    |({\rm ii})| 
    &\leq 
    \left| \dfrac{ \displaystyle \int  g(u) \exp\left[ -\dfrac{1}{2}\left\| \FisherMAP[t]^{1/2}u \right\|_{2}^{2} \right]  \rmd u }
    { \displaystyle \int  \exp\left[ -\dfrac{1}{2}\left\| \FisherMAP[t]^{1/2}u' \right\|_{2}^{2} \right]  \rmd u' } \right| 
    \times
    \left| 
    \dfrac{ \displaystyle \int  \exp\left[ -\dfrac{1}{2}\left\| \FisherMAP[t]^{1/2}u \right\|_{2}^{2} \right]  \rmd u }{ \displaystyle \int  e^{f_{t}( u'; \thetaMAP)} \rmd u'}   -  1
    \right| \\
    &=
    \left| \bbE_{\widetilde{Z}} g(\widetilde{Z}) \right| \times
    \left| 
    \dfrac{ \displaystyle \int  \exp\left[ -\dfrac{1}{2}\left\| \FisherMAP[t]^{1/2}u \right\|_{2}^{2} \right]  \rmd u }{ \displaystyle \int  e^{f_{t}( u'; \thetaMAP)} \rmd u'}   -  1
    \right| 
    \leq 
    \left| 
    \dfrac{ \displaystyle \int  \exp\left[ -\dfrac{1}{2}\left\| \FisherMAP[t]^{1/2}u \right\|_{2}^{2} \right]  \rmd u }{ \displaystyle \int  e^{f_{t}( u'; \thetaMAP)} \rmd u'}   -  1
    \right| \\
    &\leq
    \bigg( 1 - \Delta_{\tail, \LA, t}(\mathds{1}_{\Theta}) - \Delta_{\local, t}(\mathds{1}_{\Theta}) \bigg)^{-1} \bigg( \Delta_{\local, t}(\mathds{1}_{\Theta}) + \Delta_{\tail, \widetilde{\Pi}, t}(\mathds{1}_{\Theta})+ \Delta_{\tail, \LA, t}(\mathds{1}_{\Theta}) \bigg),        
\end{align*}
where the last inequality holds by \eqref{eqn:TV_integral_eq2} and \eqref{eqn:TV_integral_eq3}.
Combining the last display with \eqref{eqn:TV_integral_eq4}, the right hand side of \eqref{eqn:TV_integral_eq1} is bounded by
\begin{align*}
&\dfrac{
    \left[ \Delta_{\local, t}(g) + \Delta_{\tail, \widetilde{\Pi}, t}(g) + \Delta_{\tail, \LA, t}(g) \right]
    +
    \left[ \Delta_{\local, t}(\mathds{1}_{\Theta}) + \Delta_{\tail, \widetilde{\Pi}, t}(\mathds{1}_{\Theta})+ \Delta_{\tail, \LA, t}(\mathds{1}_{\Theta}) \right]
}
{\bigg( 1 - \Delta_{\tail, \LA, t}(\mathds{1}_{\Theta}) - \Delta_{\local, t}(\mathds{1}_{\Theta}) \bigg)}  \\       
&\leq 
2\left[
\Delta_{\local, t}(g) + \Delta_{\local, t}(\mathds{1}_{\Theta}) + 2\Delta_{\tail, t}(g) + 2\Delta_{\tail, t}(\mathds{1}_{\Theta})
\right],
\end{align*}
where the inequality holds by the assumption $\Delta_{\local, t}(\mathds{1}_{\Theta}) + \Delta_{\tail, t}(\mathds{1}_{\Theta}) \leq 1/2$. This completes the proof.
\end{proof}

\begin{lemma} \label{lemma:LA_tail_integral}
    Suppose that (\textbf{A0}) holds, and $N \geq 2$.
    Also, assume that $\widehat{\tau}_{3, t} r_{\LA} \leq 1/8$ for all $t \in [T]$ on an event $\scrE$. Then, on $\scrE$,
    \begin{align} \label{eqn:tail_posterior}
        \sup_{\| g \|_{\infty} \leq  1} \Delta_{\tail, \widetilde{\Pi}, t}(g) 
            &\leq 2e^{-8\log N - 8p} \quad {\rm and} \\
        \label{eqn:tail_LA}            
        \sup_{\| g \|_{\infty} \leq  1} \Delta_{\tail, \LA, t} (g)
            &\leq e^{-16\log N - 49p/2} \quad \forall t \in [T].
    \end{align}
\end{lemma}
\begin{proof}
In this proof, we work on the event $\scrE$ without explicitly referring to it.    
Note that
\begin{align} 
\begin{aligned} \label{eqn:tail_integral_proof_eq0}
\sup_{\| g \|_{\infty} \leq  1} \Delta_{\tail, \widetilde{\Pi}, t}(g) 
= 
\Delta_{\tail, \widetilde{\Pi}, t}(\mathds{1}_{\Theta})
=  
\dfrac{ \displaystyle \int_{\Theta_{n, t}^{\rm c} } e^{f_{t}(u; \thetaMAP)} \rmd u }
{ \displaystyle \int  \exp \left( -\dfrac{1}{2} \left\| \FisherMAP[t]^{1/2}u \right\|_{2}^{2} \right)  \rmd u }.       
\end{aligned}
\end{align}

To obtain an upper bound of the right-hand side in the last display, we will obtain an upper bound of $f_{t}(u; \thetaMAP)$ on $\Theta_{n, t}^{\rm c}$.
Let $u \in \bbR^{p}$ with $u \notin \Theta_{n, t}$. Then, we have
\begin{align} \label{eqn:tail_integral_proof_eq1}
\begin{aligned}
    u^{\circ} \overset{\rm def}{=} 
    4r_{\LA} \left\| \FisherMAP^{1/2} u \right\|_{2}^{-1}u \in \partial \Theta_{n, t}
    \overset{\rm def}{=} 
    \left\{ \theta \in \Theta : \left\| \FisherMAP^{1/2} \theta \right\|_{2} =  4r_{\LA} \right\}.    
\end{aligned}
\end{align}
By Taylor's theorem, note that
\begin{align} \label{eqn:tail_integral_proof_eq2}
\begin{aligned}
    &\widetilde{L}_{t}( \thetaMAP + u^{\circ} ) - \widetilde{L}_{t} ( \thetaMAP ) - \langle \nabla \widetilde{L}_{t} ( \thetaMAP ), u^{\circ} \rangle 
    \leq 
    \sup_{ \overline{u} \in \Theta_{n, t}} \bigg[ -\dfrac{1}{2} \left\| \FisherTilde[t]{\thetaMAP + \overline{u}}^{1/2} u^{\circ} \right\|_{2}^{2} \bigg]  \\
    &\leq
    -\dfrac{1}{2} \big( 1 - 4\widehat{\tau}_{3, t} r_{\LA} \big) \left\| \FisherMAP^{1/2} u^{\circ} \right\|_{2}^{2},    
\end{aligned}
\end{align}
where the second inequality holds by Lemma \ref{lemma:tech_Fisher_smooth}.
Also, Taylor's theorem gives
\begin{align} 
\begin{aligned} \label{eqn:tail_integral_proof_eq3}
    &\langle \nabla \widetilde{L}_{t} ( \thetaMAP + u^{\circ} ) - \nabla \widetilde{L}_{t}(\thetaMAP) , u - u^{\circ} \rangle 
    \leq 
    \sup_{ \overline{u} \in \Theta_{n, t} } 
    \left[ -\langle \FisherTilde[t]{\thetaMAP + \overline{u}} u^{\circ} , u - u^{\circ} \rangle \right]  \\
    \overset{\eqref{eqn:tail_integral_proof_eq1}}&{=}
    -\left( \left[ 4r_{\LA} \right]^{-1} \left\| \FisherMAP^{1/2} u \right\|_{2} - 1 \right)
    \inf_{ \overline{u} \in \Theta_{n, t}} 
    \langle \FisherTilde[t]{\thetaMAP + \overline{u}} u^{\circ} , u^{\circ} \rangle \\
    &\leq 
    -\left( \left[ 4r_{\LA} \right]^{-1} \left\| \FisherMAP^{1/2} u \right\|_{2} - 1 \right)
    \big( 1 - 4\widehat{\tau}_{3, t} r_{\LA} \big)
    \langle \FisherMAP u^{\circ} , u^{\circ} \rangle  \\
    &= 
    -\big( 1 - 4\widehat{\tau}_{3, t} r_{\LA} \big)
    \langle \FisherMAP u^{\circ} , u - u^{\circ} \rangle,     
\end{aligned}
\end{align}
where the second inequality holds by Lemma \ref{lemma:tech_Fisher_smooth}.

The concavity of $\widetilde{L}_{t}(\cdot)$ implies that:
\begin{align} \label{eqn:tail_integral_proof_eq4}
    \widetilde{L}_{t}(\thetaMAP + u) 
    \leq     
    \widetilde{L}_{t}(\thetaMAP + u^{\circ}) + \langle \nabla \widetilde{L}_{t} ( \thetaMAP + u^{\circ} ), u - u^{\circ} \rangle.
\end{align}
Also,
\begin{align*}
&f_{t}(u; \thetaMAP) 
= 
\widetilde{L}_{t}( \thetaMAP + u) - \widetilde{L}_{t} ( \thetaMAP ) - \langle \nabla \widetilde{L}_{t} ( \thetaMAP ), u \rangle  \\    
&=
\left[
\widetilde{L}_{t}( \thetaMAP + u) - \widetilde{L}_{t}(\thetaMAP + u^{\circ}) - \langle \nabla \widetilde{L}_{t} ( \thetaMAP + u^{\circ} ), u - u^{\circ} \rangle
\right] \\
&\qquad +
\widetilde{L}_{t}( \thetaMAP + u^{\circ} ) - \widetilde{L}_{t} ( \thetaMAP ) - \langle \nabla \widetilde{L}_{t} ( \thetaMAP ), u^{\circ} \rangle 
+ \langle \nabla \widetilde{L}_{t} ( \thetaMAP + u^{\circ} ) - \nabla \widetilde{L}_{t}(\thetaMAP) , u - u^{\circ} \rangle \\
\overset{\eqref{eqn:tail_integral_proof_eq4}}&{\leq}
\widetilde{L}_{t}( \thetaMAP + u^{\circ} ) - \widetilde{L}_{t} ( \thetaMAP ) - \langle \nabla \widetilde{L}_{t} ( \thetaMAP ), u^{\circ} \rangle +
\langle \nabla \widetilde{L}_{t} ( \thetaMAP + u^{\circ} ) - \nabla \widetilde{L}_{t}(\thetaMAP) , u - u^{\circ} \rangle \\
\overset{ \substack{\eqref{eqn:tail_integral_proof_eq2} \\ \eqref{eqn:tail_integral_proof_eq3} }  }&{\leq}
-\dfrac{1}{2} \big( 1 - 4\widehat{\tau}_{3, t} r_{\LA} \big) \left\| \FisherMAP^{1/2} u^{\circ} \right\|_{2}^{2} 
-\big( 1 - 4\widehat{\tau}_{3, t} r_{\LA} \big) \langle \FisherMAP u^{\circ} , u - u^{\circ} \rangle \\
\overset{\eqref{eqn:tail_integral_proof_eq1}}&{=}
-\dfrac{1}{2} \big( 1 - 4\widehat{\tau}_{3, t} r_{\LA} \big) \left\| \FisherMAP^{1/2} u^{\circ} \right\|_{2}^{2} 
-\big( 1 - 4\widehat{\tau}_{3, t} r_{\LA} \big) \left\| \FisherMAP^{1/2} u^{\circ} \right\|_{2} \left\| \FisherMAP^{1/2} u \right\|_{2} \\
&\qquad+ \big( 1 - 4\widehat{\tau}_{3, t} r_{\LA} \big) \left\| \FisherMAP^{1/2} u^{\circ} \right\|_{2}^{2} \\
&=
\dfrac{1}{2} \big( 1 - 4\widehat{\tau}_{3, t} r_{\LA} \big) \left\| \FisherMAP^{1/2} u^{\circ} \right\|_{2}^{2} 
-\big( 1 - 4\widehat{\tau}_{3, t} r_{\LA} \big) \left\| \FisherMAP^{1/2} u^{\circ} \right\|_{2} \left\| \FisherMAP^{1/2} u \right\|_{2} \\
&\leq 
-\dfrac{1}{2} \big( 1 - 4\widehat{\tau}_{3, t} r_{\LA} \big) \left\| \FisherMAP^{1/2} u^{\circ} \right\|_{2} \left\| \FisherMAP^{1/2} u \right\|_{2}
=
-2\big( 1 - 4\widehat{\tau}_{3, t} r_{\LA} \big) r_{\LA}  \left\| \FisherMAP^{1/2} u \right\|_{2},
\end{align*}
where the last line holds because $\| \FisherMAP^{1/2} u \|_{2} > \| \FisherMAP^{1/2} u^{\circ} \|_{2} = 4r_{\LA}$.
It follows that
\begin{align*}
&\displaystyle \int  _{\Theta_{n, t}^{\rm c} } e^{f_{t}(u; \thetaMAP)} \rmd u 
\leq 
\displaystyle \int  _{\Theta_{n, t}^{\rm c} } \exp\left[ -2\big( 1 - 4\widehat{\tau}_{3, t} r_{\LA} \big) r_{\LA}  \left\| \FisherMAP^{1/2} u \right\|_{2} \right]  \rmd u \\
&= 
\displaystyle \int  _{\Theta_{n, t}^{\rm c} } 
\exp\bigg[ -2\big( 1 - 4\widehat{\tau}_{3, t} r_{\LA} \big) r_{\LA}  \left\| \FisherMAP^{1/2} u \right\|_{2} 
+\dfrac{1}{2} \left\| \FisherMAP^{1/2} u \right\|_{2}^{2}  
\bigg] \exp \left( -\dfrac{1}{2} \left\| \FisherMAP^{1/2} u \right\|_{2}^{2} \right) \rmd u.
\end{align*}
Combining with the last display, the right-hand side of \eqref{eqn:tail_integral_proof_eq0} is bounded by
\begin{align*}
&\displaystyle \int  _{\Theta_{n, t}^{\rm c} } 
\exp\Bigg[ -2\big( 1 - 4\widehat{\tau}_{3, t} r_{\LA} \big) r_{\LA}  \left\| \FisherMAP^{1/2} u \right\|_{2} 
+\dfrac{1}{2} \left\| \FisherMAP^{1/2} u \right\|_{2}^{2}  
\Bigg]  
\dfrac{\exp \left( -\dfrac{1}{2} \left\| \FisherMAP^{1/2} u \right\|_{2}^{2} \right)}{ \displaystyle \int  \exp \left( -\dfrac{1}{2} \left\| \FisherMAP^{1/2} u' \right\|_{2}^{2} \right) \rmd u' } \rmd u  \\
&= 
\bbE_{Z} \Bigg( \exp \left[ -2\big( 1 - 4\widehat{\tau}_{3, t} r_{\LA} \big) r_{\LA}  \left\| Z \right\|_{2} 
+\dfrac{1}{2} \left\| Z \right\|_{2}^{2}  
\right] \mathds{1}\left\{ \left\| Z \right\|_{2} \geq 4r_{\LA} \right\}  \Bigg) \\
&\leq 
\bbE_{Z} \Bigg( \exp \left[ -r_{\LA}  \left\| Z \right\|_{2} 
+\dfrac{1}{2} \left\| Z \right\|_{2}^{2}  
\right] \mathds{1}\left\{ \left\| Z \right\|_{2} \geq 4r_{\LA} \right\}  \Bigg),
\end{align*}
where $Z \sim \cN(0, \bI_p)$ and the inequality holds by the assumption $\widehat{\tau}_{3, t} r_{\LA} \leq 1/8$. 
The right-hand side of the last display is equal to
\be \label{eqn:int}
\int_{4 r_{\LA}}^{\infty} \exp \left[ -r_{\LA} \omega + \dfrac{1}{2} \omega^{2} \right] p(\omega) \rmd \omega 
\ee
where $p(\cdot)$ denotes the density function of $\| Z \|_{2}$.
Note that $p(\cdot)$ is the derivative of the map $\omega \mapsto S(\omega) = - \bbP(\| Z \|_{2} > \omega)$ and
\bean
  \lim_{\omega \to \infty} \exp \left[ -r_{\LA} \omega + \dfrac{1}{2} \omega^{2} \right] 
 S(\omega) = 0
\eean
because we have, for any $\omega \geq \sqrt{p}$,
\begin{align*}
    \exp \left[ -r_{\LA} \omega + \dfrac{1}{2} \omega^{2} \right]
    \bbP\left( \| Z \|_{2} > \omega \right) 
    &\leq 
    \exp \left[ -r_{\LA} \omega + \dfrac{1}{2} \omega^{2} \right]
    \exp\left[ -\dfrac{(\omega - \sqrt{p})^{2}}{2} \right] \\
    &=
    \exp \left[ (\sqrt{p} - r_{\LA})\omega - \dfrac{p}{2} \right]
    =
    \exp \left[ -(\sqrt{p} + \sqrt{2\log N})\omega - \dfrac{p}{2} \right],
\end{align*}
where the inequality holds by the application of Lemma \ref{lemma:tech_Gaussian_Chaos} with $\bA = \mathbf{B} = \bI_{p}$.
Hence, integration by parts implies that \eqref{eqn:int} is equal to
\begin{align*}
\exp (4r_{\LA}^{2} )
\bbP \left( \left\| Z \right\|_{2} > 4r_{\LA} \right) 
+ 
\int_{4r_{\LA}}^{\infty} \left( \omega - r_{\LA} \right) \exp\left(  -r_{\LA} \omega  + \dfrac{1}{2} \omega^{2}  \right)
\bbP \left( \left\| Z \right\|_{2} > \omega \right) \rmd \omega.
\end{align*}
By Lemma \ref{lemma:tech_Gaussian_Chaos} with $\bA = \mathbf{B} = \bI_{p}$,
\begin{align*}
\begin{aligned} 
    \bbP \left( \left\| Z \right\|_{2} > 4r_{\LA} \right) 
    &\leq \exp\left( - \dfrac{1}{2} \left[ 7\sqrt{p} + 4\sqrt{2 \log N} \right]^{2} \right),
\end{aligned}    
\end{align*}
which implies that
\begin{align}
\begin{aligned} \label{eqn:tail_integral_proof_eq5}
    &\exp (4r_{\LA}^{2} ) \bbP \left( \left\| Z \right\|_{2} > 4r_{\LA} \right) \\
    &\leq 
    \exp\left( - \dfrac{1}{2} \left[ 7\sqrt{p} + 4\sqrt{2 \log N} \right]^{2} +  
    4\left[ 2\sqrt{p} + \sqrt{2 \log N} \right]^{2}
    \right) \\
    &\leq 
    \exp\left( - \dfrac{17}{2}p - 8 \log N \right).
\end{aligned}    
\end{align}
At the end of this proof, we will show that
\begin{align}
\begin{aligned} \label{eqn:tail_integral_proof_eq6}
\int_{4r_{\LA}}^{\infty} \left( \omega - r_{\LA} \right) \exp\left(  -r_{\LA} \omega + \dfrac{1}{2} \omega^{2}  \right)
\bbP \left( \left\| Z \right\|_{2} > \omega \right) \rmd \omega
\leq \exp\left( - 8p - 8\log N \right).
\end{aligned}    
\end{align}
By \eqref{eqn:tail_integral_proof_eq5} and \eqref{eqn:tail_integral_proof_eq6}, we have
\begin{align*}
\sup_{\| g \|_{\infty} \leq  1} \Delta_{\tail, \widetilde{\Pi}, t}(g) 
\leq 
2\exp\left( - 8p - 8\log N \right),
\end{align*}
which completes the proof of \eqref{eqn:tail_posterior}.

Similarly, one can bound the left hand side of \eqref{eqn:tail_LA} as
\begin{align*}
&\sup_{\| g \|_{\infty} \leq  1} \Delta_{\tail, \LA, t} (g)
=
\Delta_{\tail, \LA, t}(\mathds{1}_{\Theta})
= \displaystyle \int_{\Theta_{n, t}^{\rm c} } 
\dfrac{\exp \left( -\dfrac{1}{2} \left\| \FisherMAP^{1/2} u \right\|_{2}^{2} \right)}{ \displaystyle \int  \exp \left( -\dfrac{1}{2} \left\| \FisherMAP^{1/2} u' \right\|_{2}^{2} \right) \rmd u' } \rmd u \\
&=
\bbP \left( \left\| Z \right\|_{2} > 4r_{\LA} \right) 
\leq
\exp\left( - \dfrac{1}{2} \left[ 7\sqrt{p} + 4\sqrt{2 \log N} \right]^{2} \right)
\leq
\exp\left( - \dfrac{49}{2}p - 16\log N \right).
\end{align*}

To complete the proof, we only need to prove \eqref{eqn:tail_integral_proof_eq6}. Note that
\begin{align*}
&\int_{4r_{\LA}}^{\infty} \left( \omega - r_{\LA} \right) \exp\left(  -r_{\LA}  \omega + \dfrac{1}{2} \omega^{2}  \right)
\bbP \left( \left\| Z \right\|_{2} > \omega \right) \rmd \omega \\
\overset{\text{Lemma \ref{lemma:tech_Gaussian_Chaos}}}&{\leq} 
\int_{4r_{\LA}}^{\infty}  \left( \omega - r_{\LA} \right) \exp \left[  -r_{\LA}  \omega + \dfrac{1}{2} \omega^{2} 
- \dfrac{\left( \omega - \sqrt{p} \right)^2}{2}
\right] \rmd \omega \\
&=
\int_{0}^{\infty}  (\omega + 3r_{\LA}) \exp \left[  -r_{\LA} \left( \omega + 4r_{\LA} \right) + \dfrac{\left( \omega + 4r_{\LA} \right)^2}{2} 
- \dfrac{\left( \omega + 4r_{\LA} - \sqrt{p} \right)^2}{2}
\right] \rmd \omega \\
&=
\exp \left[ -4r_{\LA} \left( r_{\LA} - \sqrt{p} \right)  - \dfrac{p}{2} \right]
\int_{0}^{\infty}  (\omega + 3r_{\LA}) \exp \left[ -\left( r_{\LA} - \sqrt{p} \right)\omega \right] \rmd \omega \\
&= 
\exp \left[ -4r_{\LA} \left( r_{\LA} - \sqrt{p} \right)  - \dfrac{p}{2} \right]
\left( r_{\LA} - \sqrt{p} \right)^{-1} \big( \bbE(W) + 3 r_{\LA} \big),
\end{align*}
where $W \sim \operatorname{exp}(r_{\LA} - \sqrt{p})$ and density function of $\exp(\lambda)$ is given by $x \mapsto \lambda e^{-\lambda x}$.
By $\bbE W = (r_{\LA} - \sqrt{p})^{-1}$, the right-hand side of the last display is equal to
\begin{align*}
&\exp \bigg[ -4r_{\LA} \left( r_{\LA} - \sqrt{p} \right)  - \dfrac{p}{2} \bigg]
\bigg[ \dfrac{1}{(r_{\LA} - \sqrt{p})^{2}} + \dfrac{3r_{\LA}}{(r_{\LA} - \sqrt{p})} \bigg] \\
\overset{\eqref{def:hat_tau_t_radius}}&{=}
\exp \bigg[ -4 \left( 2\sqrt{p} + \sqrt{2\log N} \right) \left( \sqrt{p} + \sqrt{2\log N} \right)  - \dfrac{p}{2} \bigg] \\
&\qquad \times
\left[ \dfrac{1}{(\sqrt{p} + \sqrt{2\log N})^{2}} + \dfrac{6\sqrt{p} + 3\sqrt{2\log N} }{\sqrt{p} + \sqrt{2\log N}}  \right]  \\
&\leq 
\exp \left[ -4 \left( 2\sqrt{p} + \sqrt{2\log N} \right) \left( \sqrt{p} + \sqrt{2\log N} \right) \right]
e^{-p/2} \left( 1 + 6 \right)  \\
&= 
\exp \big( -8p - 8\log N  \big) 
\exp \big( -12\sqrt{2 p \log N} - p/2 + \log 7 \big) 
\leq
\exp \left( - 8p - 8\log N \right),
\end{align*}
where the last inequality holds by the assumption $N \geq 2$.
This completes the proof.
\end{proof}

\begin{lemma} \label{lemma:LA_tail_integral2}
    Suppose that conditions in Lemma \ref{lemma:LA_tail_integral} hold. 
    Then, for $0 \leq A_{n} \leq \sqrt{p} + \sqrt{2 \log N}$,
    \begin{align} \label{eqn:LA_tail_integral2_claim}
    \displaystyle \int_{ \Theta_{n, t}^{\rm c} } 
    \exp \left( A_n \left\| \FisherMAP^{1/2} u \right\|_{2} \right)
    \dfrac{\exp \left( -\dfrac{1}{2} \left\| \FisherMAP^{1/2} u \right\|_{2}^{2} \right)}{ \displaystyle \int \exp \left( -\dfrac{1}{2} \left\| \FisherMAP^{1/2} u' \right\|_{2}^{2} \right) \rmd u' } \rmd u 
    \leq 
    2e^{-8p - 8\log N}
    \end{align}
    on $\scrE$, where $\scrE$ is the event specified in Lemma \ref{lemma:LA_tail_integral}.
\end{lemma}
\begin{proof}
In this proof, we work on the event $\scrE$ without explicitly referring to it.
Let $\widetilde{Z} = \FisherMAP^{-1/2} Z$, where $Z \sim \cN(0, \bI_{p})$. With slight abuse of notations, let $\bbE$ be the corresponding expectation for $Z$.
Note that
\begin{align*}
\displaystyle
&\int_{\Theta_{n, t}^{\rm c} } 
\exp\left[ A_n \left\| \FisherMAP^{1/2} u \right\|_{2}
\right]  
\dfrac{\exp \left( -\dfrac{1}{2} \left\| \FisherMAP^{1/2} u \right\|_{2}^{2} \right)}{ \displaystyle \int  \exp \left( -\dfrac{1}{2} \left\| \FisherMAP^{1/2} u' \right\|_{2}^{2} \right) \rmd u' } \rmd u  \\
&= 
\bbE \Bigg( \exp \left[ A_n \left\| \FisherMAP^{1/2} \widetilde{Z} \right\|_{2}  
\right] \mathds{1}\left\{ \left\| Z \right\|_{2} \geq 4r_{\LA} \right\}  \Bigg) 
=
\bbE \Bigg( e^{ A_n \left\| Z \right\|_{2} }
\mathds{1}\left\{ \left\| Z \right\|_{2} \geq 4r_{\LA} \right\}  \Bigg).
\end{align*}
As in the proof of Lemma \ref{lemma:LA_tail_integral}, let $p(\cdot)$ be the density of $\|Z\|_2$, which is the derivative of the map $\omega \mapsto S(\omega) = - \bbP(\|Z\|_2 > \omega)$.
Then, the last display equals
\begin{align*}
\int_{4r_{\LA}}^\infty e^{ A_n \omega} p(\omega) \rmd \omega
=
\exp ( 4\delta_{n} r_{\LA}  )
\bbP \left( \left\| Z \right\|_{2} > 4r_{\LA} \right) 
+ 
\int_{4r_{\LA}}^{\infty} A_n \exp\left( A_n \omega  \right)
\bbP \left( \left\| Z \right\|_{2} > \omega \right) \rmd \omega
\end{align*}
by the integration by parts.
Note that
\begin{align*}
    \bbP \left( \left\| Z \right\|_{2} > 4r_{\LA} \right) 
    &= \bbP \left( \left\| Z \right\|_{2} > 8\sqrt{p} + 4\sqrt{2 \log N} \right) \\
    &\leq \bbP \left( \left\| Z \right\|_{2} > \sqrt{p} + \sqrt{\left[ 7\sqrt{p} + 4\sqrt{2 \log N} \right]^{2} } \right) \\
    \overset{\text{Lemma \ref{lemma:tech_Gaussian_Chaos}}}&{\leq} 
    \exp\left( - \dfrac{1}{2} \left[ 7\sqrt{p} + 4\sqrt{2 \log N} \right]^{2}  \right),
\end{align*}
which implies that
\begin{align}
\begin{aligned} \label{eqn:tail_integral2_proof_eq1}
    &\exp ( 4 A_{n} r_{\LA}  ) 
    \bbP \left( \left\| Z \right\|_{2} > 4r_{\LA} \right) \\
    &\leq 
    \exp\left( - \dfrac{1}{2} \left[ 7\sqrt{p} + 4\sqrt{2 \log N} \right]^{2} +  
    4 A_{n} \left[ 2\sqrt{p} + \sqrt{2 \log N} \right]
    \right) \\
    &\leq 
    \exp\left( - \dfrac{33}{2}p - 8 \log N \right),
\end{aligned}    
\end{align}
where the last inequality holds by the assumption $A_{n} \leq \sqrt{p} + \sqrt{2 \log N}$.
At the end of this proof, we will show that
\begin{align}
\begin{aligned} \label{eqn:tail_integral2_proof_eq2}
\int_{4r_{\LA}}^{\infty} A_n \exp\left(  A_n \omega  \right) \bbP \left( \left\| Z \right\|_{2} > \omega \right) \rmd \omega
\leq 
\exp\left( -8p - 8\log N \right).
\end{aligned}    
\end{align}
Therefore, \eqref{eqn:LA_tail_integral2_claim} holds by \eqref{eqn:tail_integral2_proof_eq1} and \eqref{eqn:tail_integral2_proof_eq2}.

To complete the proof, we only need to prove \eqref{eqn:tail_integral2_proof_eq2}. Note that
\begin{align*}
&\int_{4r_{\LA}}^{\infty} A_n \exp\left(  A_n \omega  \right) \bbP \left( \left\| Z \right\|_{2} > \omega \right) \rmd \omega \\
\overset{\text{Lemma \ref{lemma:tech_Gaussian_Chaos}}}&{\leq} 
\int_{4r_{\LA}}^{\infty}  
A_n \exp \left[  A_n \omega - \dfrac{\left( \omega - \sqrt{p} \right)^2}{2} \right] \rmd \omega 
=
\int_{4r_{\LA}}^{\infty}  
A_n \exp \left[  A_n \omega - \dfrac{\omega^2 -2\sqrt{p}\omega + p}{2} \right] \rmd \omega \\
&=
e^{A_n^{2}/2 + A_n\sqrt{p}} A_n
\int_{4r_{\LA}}^{\infty}  
\exp \bigg[  
-\dfrac{1}{2} \bigg( \omega - \left[ \sqrt{p} + A_n \right] \bigg)^{2} 
\bigg] \rmd \omega \\ 
&\leq 
e^{A_n^{2}/2 + 2A_n\sqrt{p}}
\int_{4r_{\LA}}^{\infty}  
\exp \bigg[  
-\dfrac{1}{2} \bigg( \omega - \left[ \sqrt{p} + A_n \right] \bigg)^{2} 
\bigg] \rmd \omega \\
&=
e^{A_n^{2}/2 + 2A_n\sqrt{p} + \log ( \sqrt{2\pi} )}
\int_{6\sqrt{p} + 3\sqrt{2 \log N}}^{\infty} \dfrac{1}{\sqrt{2\pi}} e^{-\omega^2/2} 
\rmd \omega \\
&\leq 
\exp \bigg[ 
\dfrac{A_n^{2}}{2} + 2A_n\sqrt{p} + \log ( \sqrt{2\pi} )
-\dfrac{1}{2} \bigg( 6\sqrt{p} + 3\sqrt{2 \log N} \bigg)^{2}
\bigg] \\
&\leq 
\exp \bigg[ 
-\dfrac{31}{2}p - 8 \log N - 15 \sqrt{2p \log N} + \log(\sqrt{2\pi})
\bigg] \\
&\leq 
\exp \left( -8p - 8\log N  \right),
\end{align*}
where the fourth inequality holds by $A_{n} \leq \sqrt{p} + \sqrt{2 \log N}$.
This completes the proof.
\end{proof}

For $\theta, u \in \Theta$, let
\begin{align} 
\begin{aligned} \label{def:remainder_3_4}
\cR_{t, 3} (\theta, u) 
    &= 
    \widetilde{L}_{t}(\theta + u) - \widetilde{L}_{t}(\theta) 
    - \langle \nabla \widetilde{L}_{t}(\theta), u \rangle 
    - \dfrac{ \langle \nabla^2 \widetilde{L}_{t}(\theta), u^{\otimes 2} \rangle }{2}, \\
\cR_{t, 4} (\theta, u) 
    &= 
    \cR_{t, 3} (\theta, u) 
    - \dfrac{ \langle \nabla^3 \widetilde{L}_{t}(\theta), u^{\otimes 3} \rangle }{6}. 
\end{aligned}
\end{align}
For simplicity, we often use the notations $\cR_{t, 3}(u) = \cR_{t, 3}(\thetaMAP, u)$ and $\cR_{t, 4}(u) = \cR_{t, 4}(\thetaMAP, u)$.

\begin{lemma} \label{lemma:LA_local_integral}
    Suppose that (\textbf{A0}) holds. Also, assume that
    \begin{align} \label{assume:LA_local_intgral}
        (\widehat{\tau}_{3, t} r_{\LA}^{2}) \vee (\widehat{\tau}_{4, t} p^{2}) \leq 1, \quad \forall t \in [T]
    \end{align}
    on an event $\scrE$.
    Then, on $\scrE$,
    \begin{align} \label{eqn:LA_error_general_g}
        \sup_{\| g \|_{\infty} \leq  1}
        \Delta_{\local, t}(g) 
        &\leq 
        K 
        \bigg( 
            \left( \widehat{\tau}_{4, t} + \widehat{\tau}_{3, t}^2 \right) p^2 
            +
            \widehat{\tau}_{3, t} p
            +
            \widehat{\tau}_{3, t}^{3} r_{\LA}^{6}
        \bigg) \\
        \label{eqn:LA_error_symm_g}
        \Delta_{\local, t}(\mathds{1}_{\Theta}) 
        &\leq 
        K 
        \bigg( 
            \left( \widehat{\tau}_{4, t} + \widehat{\tau}_{3, t}^2 \right) p^2 
            +
            \widehat{\tau}_{3, t}^{3} r_{\LA}^{6}
        \bigg)
    \end{align}
    for all $t \in [T]$, where $K > 1$ is a universal constant.
\end{lemma}

\begin{proof}
In this proof, we work on the event $\scrE$ without explicitly referring to it.
By the definitions,
\begin{align*}
f_{t}(u; \thetaMAP) 
= \cR_{t, 3}(u) - \dfrac{1}{2} \left\| \FisherMAP^{1/2} u \right\|_{2}^{2}
= \cR_{t, 4}(u) + \dfrac{1}{6} \langle \nabla^3 \widetilde{L}_{t} (\thetaMAP), u^{\otimes 3}  \rangle  - \dfrac{1}{2} \left\| \FisherMAP^{1/2} u \right\|_{2}^{2}.
\end{align*}
It follows that
\begin{align} \label{eqn:local_integral_eq0}
f_{t}(u; \thetaMAP) + \dfrac{1}{2} \left\| \FisherMAP^{1/2} u \right\|_{2}^{2}
= \cR_{t, 3}(u) = \cR_{t, 4}(u) + \dfrac{1}{6} \langle \nabla^3 \widetilde{L}_{t} (\thetaMAP), u^{\otimes 3} \rangle.
\end{align}
Hence,
\begin{align*}
\Delta_{\local, t}(g) 
&= \Bigg[ \int \exp \left( -\dfrac{1}{2} \left\| \FisherMAP^{1/2} u \right\|_{2}^{2}  \right) \rmd u  \Bigg]^{-1}     
\left| \int_{\Theta_{n, t}} g(u) \left[ e^{f_{t}(\thetaMAP, u)} - \exp \left( -\dfrac{1}{2} \left\| \FisherMAP^{1/2} u \right\|_{2}^{2}  \right)  \right] \rmd u  \right| \\
&=
\left| 
\displaystyle \int_{\Theta_{n, t}} g(u) \left[  e^{\cR_{t, 3}(u)} - 1 \right] 
\dfrac{\exp \left( -\dfrac{1}{2} \left\| \FisherMAP^{1/2} u \right\|_{2}^{2}  \right)}
{
\displaystyle \int  \exp \left( -\dfrac{1}{2} \left\| \FisherMAP^{1/2} u' \right\|_{2}^{2} \right) \rmd u'
}
\rmd u
\right|.
\end{align*}
Let $\widetilde{Z} = \FisherMAP^{-1/2} Z$, where $Z \sim \cN(0, \bI_{p})$. Let $\bbE_{\widetilde{Z}}$ be the corresponding expectation for $\widetilde{Z}$.
With slight abuse of notations, to simplify the notations, we further denote $\bbE_{A}(\cdot) = \bbE (\cdot \mathds{1}\{ \widetilde{Z} \in A \} )$ for a measurable set $A \subseteq \Theta$.
Then, the right-hand side of the last display is equal to
\begin{align}
\begin{aligned} \label{eqn:local_integral_eq0_2}
&\left| 
    \bbE_{\Theta_{n, t}} 
    \left[ 
    g(\widetilde{Z}) \bigg( e^{\cR_{t, 3}(\widetilde{Z})} - 1 \bigg)
    \right] 
\right| \\
&= 
\left|
    e^{\bbE_{\Theta_{n, t}} \cR_{t, 3}(\widetilde{Z}) } 
    \bbE_{\Theta_{n, t}} 
    \left[ 
    g(\widetilde{Z}) e^{\cR_{t, 3}(\widetilde{Z}) - \bbE_{\Theta_{n, t}} \cR_{t, 3}(\widetilde{Z}) } 
    \right] 
    -
    \bbE_{\Theta_{n, t}} 
    \left[ 
    g(\widetilde{Z}) 
    \right] 
\right| \\
&\leq 
\left|
    e^{ \bbE_{\Theta_{n, t}} \cR_{t, 3}(\widetilde{Z}) } 
    \bbE_{\Theta_{n, t}} 
    \left[ g(\widetilde{Z}) \bigg( e^{ \cR_{t, 3}(\widetilde{Z}) - \bbE_{\Theta_{n, t}} \cR_{t, 3}(\widetilde{Z}) }  -1 \bigg) \right] 
\right| \\
&\qquad +
\left|
    \bigg( e^{\bbE_{\Theta_{n, t}} \cR_{t, 3}(\widetilde{Z}) } -1 \bigg)
    \bbE_{\Theta_{n, t}} \left[ g(\widetilde{Z}) \right] 
\right| \\
&=
\left|
    e^{\bbE_{\Theta_{n, t}} \cR_{t, 3}(\widetilde{Z}) } 
    \bbE_{\Theta_{n, t}} 
    \left[ g(\widetilde{Z}) \bigg( e^{ \overline{\cR}_{t, 3}(\widetilde{Z}) }  -1 \bigg) \right] 
\right| \\
&\qquad +
\left|
    \bigg( e^{\bbE_{\Theta_{n, t}} \cR_{t, 3}(\widetilde{Z}) } -1 \bigg)
    \bbE_{\Theta_{n, t}} \left[ g(\widetilde{Z}) \right] 
\right| \\
&\leq 
\left|
    e^{\bbE_{\Theta_{n, t}} \cR_{t, 3}(\widetilde{Z}) } 
    \bbE_{\Theta_{n, t}} 
    \left[ g(\widetilde{Z}) \bigg( e^{ \overline{\cR}_{t, 3}(\widetilde{Z}) }  -1 \bigg) \right] 
\right|
+
\left|
    e^{\bbE_{\Theta_{n, t}} \cR_{t, 3}(\widetilde{Z}) } -1
\right|,     
\end{aligned}
\end{align}
where $\overline{\cR}_{t, 3}(\cdot) = \cR_{t, 3}(\cdot) - \bbE_{\Theta_{n, t}} \cR_{t, 3}(\widetilde{Z})$, and the last inequality holds by $\| g \|_{\infty} \leq 1$.
To prove \eqref{eqn:LA_error_general_g}, therefore, we only need to bound the following quantities:
\begin{align*}
    ({\rm i}) = \bbE_{\Theta_{n, t}} \cR_{t, 3}(\widetilde{Z}), \quad 
    ({\rm ii}) = \bbE_{\Theta_{n, t}} \left[ g(\widetilde{Z}) \bigg( e^{ \overline{\cR}_{t, 3}(\widetilde{Z}) }  -1 \bigg) \right].
\end{align*}

Firstly, we will obtain an upper bound of $({\rm i})$. 
Note that
\begin{align} \label{eqn:local_integral_eq1}
    \bbE_{ \Theta_{n, t} } \left[ \langle  \nabla^{3} \widetilde{L}_{t}(\thetaMAP), \widetilde{Z}^{\otimes 3} \rangle \right] = 0
\end{align}
by the symmetry of $\Theta_{n, t}$.
It follows that
\begin{align}
\begin{aligned} \label{eqn:local_integral_eq2}
({\rm i}) 
&= 
\bbE_{ \Theta_{n, t} } \left[ \cR_{3, t} ( \widetilde{Z} ) \right]
\overset{\eqref{def:remainder_3_4}}{=}
\bbE_{ \Theta_{n, t} } \left[ \dfrac{1}{6} \langle \nabla^3 \widetilde{L}_{t}(\thetaMAP), \widetilde{Z}^{\otimes 3} \rangle +  \cR_{4, t} ( \widetilde{Z} ) \right] 
\\
& \overset{\eqref{eqn:local_integral_eq1}}{=}
\bbE_{ \Theta_{n, t} } \left[ \cR_{4, t} ( \widetilde{Z} ) \right]
\leq
\sqrt{ \bbE_{ \Theta_{n, t} } \cR_{4, t}^{2} ( \widetilde{Z} ) }.    
\end{aligned}
\end{align}
By Taylor's theorem, for $u \in \Theta_{n, t}$, there exists $\widetilde{u} \in \Theta_{n, t}$ such that
\begin{align*}
\cR_{4, t}( u ) 
= \dfrac{1}{24} \langle \nabla^4 \widetilde{L}_{t}(\thetaMAP + \widetilde{u}), u^{\otimes 4}  \rangle
\overset{\eqref{def:hat_tau_t_radius}}&{\leq }
\dfrac{\widehat{\tau}_{4, t}}{24} \left\| \FisherMAP^{1/2} u \right\|_{2}^{4}.
\end{align*}
Hence,
\begin{align} 
\begin{aligned} \label{eqn:local_integral_eq2_2}
\bbE_{ \Theta_{n, t} } \cR_{4, t}^{2} ( \widetilde{Z} )
&\leq 
\bbE_{ \Theta_{n, t} } \left[  \dfrac{\widehat{\tau}_{4, t}^{2}}{24^{2}} \left\| \FisherMAP^{1/2} \widetilde{Z} \right\|_{2}^{8} \right]
=
\bbE_{ \Theta_{n, t} } \left[  \dfrac{\widehat{\tau}_{4, t}^{2}}{24^{2}} \left\| Z \right\|_{2}^{8} \right] \\
&\leq
\dfrac{\widehat{\tau}_{4, t}^{2}}{24^{2}} \bigg[ \operatorname{tr}(\bI_{p}) + 3 \| \bI_{p} \|_{2} \bigg]^{4}
= 
\dfrac{\widehat{\tau}_{4, t}^{2}}{24^{2}} \left( p + 3 \right)^{4},    
\end{aligned}
\end{align}
where the second inequality holds by Lemma \ref{lemma:tech_Gaussian_Moments}. 
Combining with \eqref{eqn:local_integral_eq2}, the last display implies that
\begin{align} \label{eqn:local_integral_eq2_3}
    ({\rm i})
    \leq  \sqrt{ \bbE_{ \Theta_{n, t} } \cR_{4, t}^{2} ( \widetilde{Z} ) } 
    \leq  \dfrac{\widehat{\tau}_{4, t}}{24} \left( p + 3 \right)^{2}.
\end{align}

Next, we will obtain an upper bound of $|({\rm ii})|$. By Lemma \ref{lemma:tech_mgf_approximation} with $X = \overline{\cR}_{t, 3}(\widetilde{Z})$ and $\epsilon = 8 \widehat{\tau}_{3, t} r_{\LA}^{2}$, we have
\begin{align*}
\left| 
    \bbE_{ \Theta_{n, t} }
    \left[
    \left( e^{\overline{\cR}_{t, 3}(\widetilde{Z})} - 1 - \overline{\cR}_{t, 3}(\widetilde{Z}) - \dfrac{\overline{\cR}_{t, 3}^{2}(u) }{2}  \right)
    g(\widetilde{Z})
    \right]
\right| 
&\leq 
\dfrac{5}{3} \left( 8 \widehat{\tau}_{3, t} r_{\LA}^{2} \right)^{3} e^{ 64 \widehat{\tau}_{3, t}^{2} r_{\LA}^{4} } \\
\overset{\eqref{assume:LA_local_intgral}}&{\leq}
c_1 \widehat{\tau}_{3, t}^{3} r_{\LA}^{6},
\end{align*}
for some universal constant $c_1 > 0$. It follows that
\begin{align*}
&\left| \bbE_{\Theta_{n, t}} \left[ g(\widetilde{Z}) \bigg( e^{ \overline{\cR}_{t, 3}(\widetilde{Z}) }  -1 \bigg) \right] \right| \\
&\leq 
\left| 
    \bbE_{ \Theta_{n, t} }
    \left[
    \left( \overline{\cR}_{t, 3}(\widetilde{Z}) + \dfrac{\overline{\cR}_{t, 3}^{2}(\widetilde{Z}) }{2}  \right)
    g(\widetilde{Z})
    \right]
\right| + c_1 \widehat{\tau}_{3, t}^{3} r_{\LA}^{6} \\
&\leq 
\bbE_{ \Theta_{n, t} }
\left|
    \overline{\cR}_{t, 3}(\widetilde{Z}) 
    g(\widetilde{Z})
\right| + 
\dfrac{1}{2} 
\bbE_{ \Theta_{n, t} }
\left|
    \overline{\cR}_{t, 3}^{2}(\widetilde{Z}) 
    g(\widetilde{Z})
\right| + 
c_1 \widehat{\tau}_{3, t}^{3} r_{\LA}^{6} \\
&\leq 
\bbE_{ \Theta_{n, t} }
\left|
    \overline{\cR}_{t, 3}(\widetilde{Z})
\right| + 
\dfrac{1}{2} 
\bbE_{ \Theta_{n, t} }
    \left[ \overline{\cR}_{t, 3}^{2}(\widetilde{Z}) \right]
+ 
c_1 \widehat{\tau}_{3, t}^{3} r_{\LA}^{6},
\end{align*}
where the last inequality holds by $\|g \|_{\infty} \leq 1$. Let
\begin{align*}
    ({\rm iii}) = \bbE_{ \Theta_{n, t} } \left| \overline{\cR}_{t, 3}(\widetilde{Z}) \right|, \quad
    ({\rm iv}) = \bbE_{ \Theta_{n, t} } \left[ \overline{\cR}_{t, 3}^{2}(\widetilde{Z}) \right].
\end{align*}
To bound $|({\rm ii})|$, we need to obtain upper bounds for $({\rm iii})$ and $({\rm iv})$. Firstly, we bound $({\rm iii})$.
Let $\overline{\cR}_{t, 4}(\cdot) = \cR_{t, 4}(\cdot) - \bbE_{\Theta_{n, t}} \cR_{t, 4}(\widetilde{Z})$. 
By \eqref{eqn:local_integral_eq0} and \eqref{eqn:local_integral_eq1}, we have 
\begin{align*}
    \bbE_{\Theta_{n, t}} \cR_{t, 4}(\widetilde{Z}) = \bbE_{\Theta_{n, t}} \cR_{t, 3}(\widetilde{Z}),
\end{align*}
which implies that, for every $u \in \Theta$, 
\begin{align} \label{eqn:local_integral_eq3}
    \overline{\cR}_{t, 3}(u) - \dfrac{1}{6} \langle \nabla^3 \widetilde{L}_{t} (\thetaMAP), u^{\otimes 3} \rangle = \overline{\cR}_{t, 4}(u).
\end{align}
It follows that
\begin{align*}
    \bbE_{ \Theta_{n, t} } \left| \overline{\cR}_{t, 3}(\widetilde{Z}) - \dfrac{1}{6} \langle \nabla^3 \widetilde{L}_{t} (\thetaMAP), \widetilde{Z}^{\otimes 3} \rangle \right|
    = 
    \bbE_{ \Theta_{n, t} } \left| \overline{\cR}_{t, 4}(\widetilde{Z})  \right|,
\end{align*}
which further implies that
\begin{align*} 
    ({\rm iii}) 
    \leq 
    \bbE_{ \Theta_{n, t} } \left| \overline{\cR}_{t, 4}(\widetilde{Z}) \right| 
    + \bbE_{ \Theta_{n, t} } \left| \dfrac{1}{6} \langle \nabla^3 \widetilde{L}_{t} (\thetaMAP), \widetilde{Z}^{\otimes 3} \rangle \right|.
\end{align*}
Note that
\begin{align*}
    \bbE_{ \Theta_{n, t} } \left| \overline{\cR}_{t, 4}(\widetilde{Z}) \right|
    \leq 
    \sqrt{\bbE_{ \Theta_{n, t} } \overline{\cR}_{t, 4}^{2}(\widetilde{Z})} 
    \leq 
    \sqrt{\bbE_{ \Theta_{n, t} } \cR_{t, 4}^{2}(\widetilde{Z})} 
    \overset{\eqref{eqn:local_integral_eq2_2}}{\leq}
    \dfrac{\widehat{\tau}_{4, t}}{24} \left( p + 3 \right)^{2}.
\end{align*}
Also,
\begin{align} 
\begin{aligned} \label{eqn:local_integral_eq4}
    &\bbE_{ \Theta_{n, t} } \left| \dfrac{1}{6} \langle \nabla^3 \widetilde{L}_{t} (\thetaMAP), \widetilde{Z}^{\otimes 3} \rangle \right|
    \leq
    \dfrac{1}{6} \sqrt{\bbE_{ \Theta_{n, t} } \left| \langle \nabla^3 \widetilde{L}_{t} (\thetaMAP), \widetilde{Z}^{\otimes 3} \rangle \right|^{2} } \\
    \overset{\text{Lemma \ref{lemma:tech_Gaussian_Tensor_expectation} }}&{\leq}
    \dfrac{1}{6} \sqrt{ 15 \widehat{\tau}_{3, t}^{2} \left\| \bI_{p} \right\|_{2} \operatorname{tr}^{2}(\bI_{p}) }
    = \dfrac{1}{6} \sqrt{ 15 \widehat{\tau}_{3, t}^{2} p^{2} }
    = \dfrac{\sqrt{15}}{6} \widehat{\tau}_{3, t} p.    
\end{aligned}
\end{align}
Combining the last three displays, we have
\begin{align*}
    ({\rm iii}) \leq \dfrac{\widehat{\tau}_{4, t}}{24} \left( p + 3 \right)^{2} + \dfrac{\sqrt{15}}{6} \widehat{\tau}_{3, t} p.
\end{align*}
To bound (iv), note that
\begin{align}
\begin{aligned} \label{eqn:local_integral_eq4_2}
({\rm iv}) 
&= 
\bbE_{ \Theta_{n, t} } \left[ \overline{\cR}_{t, 3}^{2}(\widetilde{Z}) \right]
\overset{\eqref{eqn:local_integral_eq3}}{=} 
\bbE_{ \Theta_{n, t} } \left[ \bigg( \dfrac{1}{6} \langle \nabla^3 \widetilde{L}_{t} (\thetaMAP), \widetilde{Z}^{\otimes 3} \rangle + \overline{\cR}_{t, 4}(\widetilde{Z}) \bigg)^{2} \right] \\
&\leq 
\dfrac{1}{18} \bbE_{ \Theta_{n, t} } \left[ \langle \nabla^3 \widetilde{L}_{t} (\thetaMAP), \widetilde{Z}^{\otimes 3} \rangle^{2} \right] +
2\bbE_{ \Theta_{n, t} } \left[ \overline{\cR}_{t, 4}^{2}(\widetilde{Z}) \right] \\
\overset{\substack{ \eqref{eqn:local_integral_eq2_2} \\ \eqref{eqn:local_integral_eq4} }}&{\leq }
\dfrac{15}{18} \widehat{\tau}_{3, t}^{2} p^{2} + 2\dfrac{\widehat{\tau}_{4, t}^{2}}{24^{2}} \left( p + 3 \right)^{4} 
\leq
\dfrac{5}{6}\widehat{\tau}_{3, t}^{2} p^{2} + \dfrac{8}{9} \widehat{\tau}_{4, t}^{2} p^{4},    
\end{aligned}
\end{align}
where the last inequality holds by $(p + 3)^{4} \leq 2^{8}p^{4}$.
Combining the last two displays, we have
\begin{align}
\begin{aligned} \label{eqn:local_integral_eq5}
|({\rm ii})| 
&\leq    
|({\rm iii})| + \dfrac{1}{2} |({\rm iv})| + c_1 \widehat{\tau}_{3, t}^{3} r_{\LA}^{6}  \\
&\leq 
\dfrac{\widehat{\tau}_{4, t}}{24} \left( p + 3 \right)^{2} 
+ 
\dfrac{\sqrt{15}}{6} \widehat{\tau}_{3, t} p 
+ 
\dfrac{5}{12}\widehat{\tau}_{3, t}^{2} p^{2}
+
\dfrac{4}{9} \widehat{\tau}_{4, t}^{2} p^{4}
+ 
c_1 \widehat{\tau}_{3, t}^{3} r_{\LA}^{6} \\
&\leq 
c_2 \bigg[ 
\left( \widehat{\tau}_{3, t}^{2} + \widehat{\tau}_{4, t} \right) p^{2}
+ \widehat{\tau}_{3, t} p
+ \widehat{\tau}_{3, t}^{3} r_{\LA}^{6}
\bigg]
\end{aligned}
\end{align}
for some universal constant $c_2 = c_{2}(c_{1}) > 0$, where the last inequality holds by $\widehat{\tau}_{4, t} p^{2} \leq 1$.

Now, we are ready to prove \eqref{eqn:LA_error_general_g}.
By \eqref{eqn:local_integral_eq2_3} and \eqref{eqn:local_integral_eq5}, the right-hand side of \eqref{eqn:local_integral_eq0_2} is bounded by
\begin{align*} 
&\left|
    e^{ \widehat{\tau}_{4, t} (p + 3)^{2}/24 } 
    \bbE_{\Theta_{n, t}} 
    \left[ g(\widetilde{Z}) \bigg( e^{ \overline{\cR}_{t, 3}(\widetilde{Z}) }  -1 \bigg) \right] 
\right|
+
\left|
    e^{ \widehat{\tau}_{4, t} (p + 3)^{2}/24 } -1
\right| \\
&\leq 
\left|
    e^{ 2\widehat{\tau}_{4, t} p^{2}/3 } 
    \bbE_{\Theta_{n, t}} 
    \left[ g(\widetilde{Z}) \bigg( e^{ \overline{\cR}_{t, 3}(\widetilde{Z}) }  -1 \bigg) \right] 
\right|
+
\left|
    e^{ 2\widehat{\tau}_{4, t} p^{2}/3 } -1
\right| \\
&\lesssim
    \left| \bbE_{\Theta_{n, t}} \bigg[ g(\widetilde{Z}) \bigg( e^{ \overline{\cR}_{t, 3}(\widetilde{Z}) }  -1 \bigg) \bigg] \right|
    +
    \widehat{\tau}_{4, t} p^{2} \\
\overset{\eqref{eqn:local_integral_eq5}}&{\leq}
    c_{2} \bigg[ 
    \left( \widehat{\tau}_{3, t}^{2} + \widehat{\tau}_{4, t} \right) p^{2}
    + \widehat{\tau}_{3, t} p
    + \widehat{\tau}_{3, t}^{3} r_{\LA}^{6}
    \bigg]
    + \widehat{\tau}_{4, t} p^{2} \\
&\leq  
c_{3}
\bigg[ 
\left( \widehat{\tau}_{3, t}^{2} + \widehat{\tau}_{4, t} \right) p^{2}
+ \widehat{\tau}_{3, t} p
+ \widehat{\tau}_{3, t}^{3} r_{\LA}^{6}
\bigg]
\end{align*}
for some positive constants $c_3 = c_3(c_2)$, where the second inequality holds because $\widehat{\tau}_{4, t} p^{2} \leq 1$ and $e^{x} \leq 1 +  2x$ for $x \leq 1.256$.
Therefore, we have
\begin{align*}
    \sup_{\| g \|_{\infty} \leq  1} \Delta_{\local, t}(g)
    &\leq 
    c_{3} \bigg[ 
    \widehat{\tau}_{4, t}^{2} p^{4} 
    + \widehat{\tau}_{3, t} \widehat{\tau}_{4, t} p^3 
    + \left( \widehat{\tau}_{3, t}^{2} + \widehat{\tau}_{4, t} \right) p^{2}
    + \widehat{\tau}_{3, t} p
    + \widehat{\tau}_{3, t}^{3} r_{\LA}^{6}
    \bigg],    
\end{align*}
which completes the proof of \eqref{eqn:LA_error_general_g}.

The proof of \eqref{eqn:LA_error_symm_g} is similar. However, there are some differences in obtaining the upper bound of $({\rm ii})$. 
Consider the case where $g(\cdot)$ is symmetric, meaning $g(u) = g(-u)$ for any $u \in \Theta_{n, t}$. 
Note that
\begin{align}
\begin{aligned} \label{eqn:local_integral_eq6}
&\left| \bbE_{\Theta_{n, t}} \left[ g(\widetilde{Z}) \bigg( e^{ \overline{\cR}_{t, 3}(\widetilde{Z}) }  -1 \bigg) \right] \right| \\
&\leq 
\left| 
    \bbE_{ \Theta_{n, t} }
    \left[
    \left( \overline{\cR}_{t, 3}(\widetilde{Z}) + \dfrac{\overline{\cR}_{t, 3}^{2}(\widetilde{Z}) }{2}  \right)
    g(\widetilde{Z})
    \right]
\right| + c_{1} \widehat{\tau}_{3, t}^{3} r_{\LA}^{6} \\
\overset{\eqref{eqn:local_integral_eq3}}&{=} 
\left| 
    \bbE_{ \Theta_{n, t} }
    \left[
    \dfrac{1}{6} \langle \nabla^3 \widetilde{L}_{t} (\thetaMAP), \widetilde{Z}^{\otimes 3} \rangle g(\widetilde{Z})
    + \overline{\cR}_{t, 4}(\widetilde{Z})g(\widetilde{Z})
    + \dfrac{\overline{\cR}_{t, 3}^{2}(\widetilde{Z}) }{2}g(\widetilde{Z})
    \right]
\right| + c_{1} \widehat{\tau}_{3, t}^{3} r_{\LA}^{6} \\
&=
\left| 
    \bbE_{ \Theta_{n, t} }
    \left[
    \overline{\cR}_{t, 4}(\widetilde{Z})g(\widetilde{Z})
    + \dfrac{ \overline{\cR}_{t, 3}^{2}(\widetilde{Z}) }{2}g(\widetilde{Z})
    \right]
\right| + c_{1} \widehat{\tau}_{3, t}^{3} r_{\LA}^{6} \\
&\leq 
    \bbE_{ \Theta_{n, t} }
    \left| 
    \overline{\cR}_{t, 4}(\widetilde{Z}) 
    \right|
    +
    \dfrac{1}{2}
    \bbE_{ \Theta_{n, t} }
    \overline{\cR}_{t, 3}^{2}(\widetilde{Z})
    + 
    c_1 \widehat{\tau}_{3, t}^{3} r_{\LA}^{6} \\
&\leq 
    \sqrt{\bbE_{ \Theta_{n, t} }
    \overline{\cR}_{t, 4}^{2}(\widetilde{Z}) 
    }
    +
    \dfrac{1}{2}
    \bbE_{ \Theta_{n, t} }
    \overline{\cR}_{t, 3}^{2}(\widetilde{Z})
    + 
    c_1 \widehat{\tau}_{3, t}^{3} r_{\LA}^{6} \\
\overset{\substack{\eqref{eqn:local_integral_eq2_2} \\ \eqref{eqn:local_integral_eq4_2}}}&{\lesssim}
    \widehat{\tau}_{4, t} p^{2} 
    + 
    \widehat{\tau}_{3, t}^{2} p^{2} 
    + 
    c_1 \widehat{\tau}_{3, t}^{3} r_{\LA}^{6} 
\leq 
c_{4} \bigg[ 
\left( \widehat{\tau}_{3, t}^{2} + \widehat{\tau}_{4, t} \right) p^{2}
+
\widehat{\tau}_{3, t}^{3} r_{\LA}^{6}
\bigg] 
\end{aligned}
\end{align}
for some universal constant $c_{4} = c_{4}(c_{1}) > 0$, where the second equality holds by the symmetry of $\Theta_{n, t}$ and $g(\cdot)$. 
For any symmetric $g(\cdot)$, therefore, we have
\begin{align*}
&\Delta_{\local, t}(g)    
\overset{\eqref{eqn:local_integral_eq0_2}}{\leq}
    \left|
        e^{\bbE_{\Theta_{n, t}} \cR_{t, 3}(\widetilde{Z}) } 
    \right|
    \left|    
        \bbE_{\Theta_{n, t}} 
        \left[ g(\widetilde{Z}) \bigg( e^{ \overline{\cR}_{t, 3}(\widetilde{Z}) }  -1 \bigg) \right] 
    \right|
    +
    \left|
        e^{\bbE_{\Theta_{n, t}} \cR_{t, 3}(\widetilde{Z}) } -1
    \right| \\
\overset{\eqref{eqn:local_integral_eq2_3}}&{\leq}
    \left|
        e^{ 2\widehat{\tau}_{4, t} p^{2}/3 } 
    \right|
    \left|  
        \bbE_{\Theta_{n, t}} 
        \left[ g(\widetilde{Z}) \bigg( e^{ \overline{\cR}_{t, 3}(\widetilde{Z}) }  -1 \bigg) \right] 
    \right|
    +
    \left|
        e^{ 2\widehat{\tau}_{4, t} p^{2}/3 } -1
    \right| \\
\overset{\eqref{eqn:local_integral_eq6}}&{\lesssim}
    c_{4} \bigg[ 
    \left( \widehat{\tau}_{3, t}^{2} + 
    \widehat{\tau}_{4, t} \right) p^{2} +
    \widehat{\tau}_{3, t}^{3} r_{\LA}^{6}
    \bigg] +
    \widehat{\tau}_{4, t} p^{2} 
\leq 
    c_{5} 
    \bigg[ 
    \left( \widehat{\tau}_{3, t}^{2} + 
    \widehat{\tau}_{4, t} \right) p^{2} +
    \widehat{\tau}_{3, t}^{3} r_{\LA}^{6}
    \bigg]
\end{align*}
for some universal constants $c_{5} = c_{5}(c_{4}) > 0$.  
This completes the proof of \eqref{eqn:LA_error_symm_g}.
\end{proof}

\subsection{Proofs of Theorem \ref{thm:LA_TV}, \ref{thm:LA_KL}}

\begin{proof}[Proof of Theorem \ref{thm:LA_TV}]
Since the total variation distance is bounded by $2$, note that the inequality in \eqref{claim:LA_TV} always holds provided that $(\widehat{\tau}_{3, t} r_{\LA}^{2}) \vee (\widehat{\tau}_{4, t} p^{2}) \geq C$ for some universal constant $C > 0$ and $K = K(C)$ is large enough.
Hence, we may assume that 
\begin{align}
\begin{aligned} \label{assume:LA_TV_acc}
    \left( \widehat{\tau}_{3, t} r_{\LA}^{2} \right)
    \vee
    \left( \widehat{\tau}_{4, t} p^{2} \right)
    \leq \delta, \quad \forall t \in [T]
\end{aligned}
\end{align}  
for some small enough constant $\delta$ on an event $\scrE_{1}$.
In this proof, we work on the event $\scrE_{1}$ without explicitly referring to it. 
By \eqref{assume:LA_TV_acc} and $r_{\LA} \geq 1$, we have
\begin{align*}
   \widehat{\tau}_{3, t} r_{\LA} \leq \widehat{\tau}_{3, t} r_{\LA}^{2} \leq \delta \leq 1/8, \quad N \geq 2,
\end{align*}
which allows us to utilize Lemma \ref{lemma:LA_tail_integral}.
By Lemma \ref{lemma:LA_tail_integral}, we have
\begin{align*}
    \sup_{\| g \|_{\infty} \leq  1} \Delta_{\tail, \widetilde{\Pi}, t}(g) 
        &\leq 2e^{-8\log N - 8p}, \\
    \sup_{\| g \|_{\infty} \leq  1} \Delta_{\tail, \LA, t} (g)
        &\leq e^{-16\log N - 49p/2},
\end{align*}
which, combining with $p \geq 1$, further implies that
\begin{align*}
    \Delta_{\tail, \widetilde{\Pi}, t}(\mathds{1}_{\Theta})\vee \Delta_{\tail, \LA, t}(\mathds{1}_{\Theta}) \leq 1/4.
\end{align*}
Also, by \eqref{assume:LA_TV_acc}, Lemma \ref{lemma:LA_local_integral} implies that
\begin{align*}
    \Delta_{\local, t}(\mathds{1}_{\Theta}) 
    &\leq 
    K
    \bigg( 
        \left[ \widehat{\tau}_{4, t} + \widehat{\tau}_{3, t}^2 \right] p^2 
        +
        \widehat{\tau}_{3, t}^{3} r_{\LA}^{6}
    \bigg) 
    \overset{\eqref{assume:LA_TV_acc}}{\leq}
    5 K_{1} \delta
    \leq 
    \dfrac{1}{4},
\end{align*}
where $K_{1}$ denotes the constant $K$ in Lemma \ref{lemma:LA_local_integral} and $\delta = \delta(K_{1})$ is small enough constant.
By Lemma \ref{lemma:LA_TV_integral}, it follows that
\begin{align*}
    &d_{V} \left( \Pi_{t}^{\LA}(\cdot),  \widetilde{\Pi}_{t}\left(\cdot \mid \bD_{t} \right) \right) 
    = 
    \sup_{\| g \|_{\infty} \leq 1} \left| \cI_{\widetilde{\Pi}, t}(g) - \cI_{\LA, t}(g) \right| \\
\overset{\text{Lemma \ref{lemma:LA_TV_integral}}}&{\leq}
    2 \sup_{\| g \|_{\infty} \leq 1} \bigg\{ \Delta_{\local, t}(g) + \Delta_{\local, t}(\mathds{1}_{\Theta}) + 2\Delta_{\tail, t}(g) + 2\Delta_{\tail, t}(\mathds{1}_{\Theta}) \bigg\} \\
\overset{\text{Lemma \ref{lemma:LA_tail_integral}}}&{\leq}
    \sup_{\| g \|_{\infty} \leq 1} \bigg\{ 2\Delta_{\local, t}(g) \bigg\} + 2\Delta_{\local, t}(\mathds{1}_{\Theta}) + 16e^{-8\log N - 8p} \\
\overset{\text{Lemma \ref{lemma:LA_local_integral}}}&{\leq}
        4K_{1}
        \bigg( 
            \left[ \widehat{\tau}_{4, t} + \widehat{\tau}_{3, t}^2 \right] p^2 
            +
            \widehat{\tau}_{3, t} p
            +
            \widehat{\tau}_{3, t}^{3} r_{\LA}^{6}
        \bigg)
        + 16e^{-8\log N - 8p} \\
&\leq 
        \left( 4K_{1} \vee 16 \right)
        \bigg( 
            \left[ \widehat{\tau}_{4, t} + \widehat{\tau}_{3, t}^2 \right] p^2 
            +
            \widehat{\tau}_{3, t} p
            +
            \widehat{\tau}_{3, t}^{3} r_{\LA}^{6}
            +
            e^{-8\log N - 8p}
        \bigg) \\
&\lesssim
        \left( 4K_{1} \vee 16 \right)
        \bigg( 
            \left[ \widehat{\tau}_{4, t} + \widehat{\tau}_{3, t}^2 \right] p^2 
            +
            \widehat{\tau}_{3, t} p
            +
            \widehat{\tau}_{3, t}^{3}\log^{3}N
            +
            e^{-8\log N - 8p}
        \bigg),
\end{align*}
which completes the proof.
\end{proof}

\begin{remark}
    The indicator function $\mathds{1}_{\Theta}$ in \eqref{eqn:LA_error_symm_g} can be replaced by any symmetric function $g$
\end{remark}

\begin{proof}[Proof of Theorem \ref{thm:LA_KL}]
In this proof, we will work on the event $\scrE_{1}$ without explicitly referring to it.
Let $\widetilde{Z} = \FisherMAP^{-1/2} Z$, where $Z \sim \cN(0, \bI_{p})$, and $\bbE_{\widetilde{Z}}$ be the corresponding expectation.
As in the proof of Lemma \ref{lemma:LA_local_integral}, we also use the notation $\bbE_{A}(\cdot) = \bbE_{\widetilde{Z}} (\cdot \mathds{1}\{ \widetilde{Z} \in A \} )$ for a measurable set $\bA \subseteq \Theta$.

For $u \in \Theta$, note that
\begin{align*}
&\log (\pi_{t}^{\LA}(\thetaMAP + u)/ \widetilde{\pi}_{t}(\thetaMAP + u \mid \bD_t)) \\
&= 
    -\log \left[ \dfrac{e^{\widetilde{L}_{t}(\thetaMAP + u)}}{\int e^{\widetilde{L}_{t}(\thetaMAP + u')} \rmd u' } \right]
    - \dfrac{1}{2} \left\| \FisherMAP^{1/2}\left( \thetaMAP + u - \thetaMAP \right) \right\|_{2}^{2} \\
    &\qquad - \log \left[ \int \exp \left( -\dfrac{1}{2} \left\| \FisherMAP^{1/2} \left( \thetaMAP + u - \thetaMAP \right) \right\|_{2}^{2} \right) \rmd u \right] \\
\overset{\eqref{def:f_t}}&{=}
    - \log \left[ \dfrac{e^{f_{t}(u; \thetaMAP)}}{\int e^{f_{t}(u'; \thetaMAP)} \rmd u' } \right]
    - \dfrac{1}{2} \left\| \FisherMAP^{1/2} u \right\|_{2}^{2}
    - \log \left[ \int \exp \left( -\dfrac{1}{2} \left\| \FisherMAP^{1/2} u \right\|_{2}^{2} \right) \rmd u \right] \\
&=
    - f_{t}(u; \thetaMAP)
    - \dfrac{1}{2} \left\| \FisherMAP^{1/2} u \right\|_{2}^{2}
    + \log \left[ \int e^{f_{t}(u; \thetaMAP)} \rmd u \right]
    - \log \left[ \int \exp \left( -\dfrac{1}{2} \left\| \FisherMAP^{1/2} u \right\|_{2}^{2} \right) \rmd u \right] \\
\overset{\eqref{def:remainder_3_4}}&{=}
    - \cR_{t, 3}(\thetaMAP, u)
    + \log \left[ \int e^{f_{t}(u; \thetaMAP)} \rmd u \right]
    - \log \left[ \int \exp \left( -\dfrac{1}{2} \left\| \FisherMAP^{1/2} u \right\|_{2}^{2} \right) \rmd u \right] \\
&= 
- \cR_{t, 3}(\thetaMAP, u) + W_{n, t},
\end{align*}
where 
\begin{align*}
W_{n, t} 
= \log \left[ \int e^{f_{t}(u; \thetaMAP)} \rmd u \right] - \log \left[ \int \exp \left( -\dfrac{1}{2} \left\| \FisherMAP^{1/2} u \right\|_{2}^{2} \right) \rmd u \right].     
\end{align*}
Hence,
\begin{align*}
K \left( \Pi_{t}^{\LA}(\cdot); \ \widetilde{\Pi}_{t}\left(\cdot \mid \bD_{t} \right) \right)
&= \displaystyle \int \log \left[ \dfrac{\pi_{t}^{\LA}(\thetaMAP + u)}{\widetilde{\pi}_{t}(\thetaMAP + u \mid \bD_t)} \right] \pi_{t}^{\LA}(\thetaMAP + u) \rmd u \\
&= \bbE_{\widetilde{Z}} \left[ - \cR_{t, 3}(\widetilde{Z}) \right] + W_{n, t}.
\end{align*}
We will obtain upper bounds for the following quantities:
\begin{align*}
    ({\rm i}) = \bbE_{\widetilde{Z}} \left[ - \cR_{t, 3}(\widetilde{Z}) \right], \quad 
    ({\rm ii}) = W_{n, t}. 
\end{align*}

Firstly, we will obtain an upper bound of $({\rm i})$. 
Note that
\begin{align*}
\bbE_{\widetilde{Z}} \left[ - \cR_{t, 3}(\widetilde{Z}) \right]
&= 
\bbE_{ \Theta_{n, t} } \left[ - \cR_{t, 3}(\widetilde{Z}) \right] + \bbE_{ \Theta_{n, t}^{\rm c} } \left[ - \cR_{t, 3}(\widetilde{Z}) \right] \\
\overset{\eqref{eqn:local_integral_eq0}}&{=}
\bbE_{ \Theta_{n, t} } \left[ - \cR_{t, 4}(\widetilde{Z}) \right] + \bbE_{ \Theta_{n, t}^{\rm c} } \left[ - \cR_{t, 3}(\widetilde{Z}) \right]
\end{align*}
For the first term in the right-hand side of the last display, note that
\begin{align} 
\begin{aligned} \label{eqn:LA_KL_error_eq1}
\bbE_{ \Theta_{n, t} } \left[ - \cR_{t, 4}(\widetilde{Z}) \right]
\leq 
\bbE_{ \Theta_{n, t} } \left| \cR_{t, 4}(\widetilde{Z}) \right|
\leq 
\sqrt{\bbE_{ \Theta_{n, t} } \cR_{t, 4}^{2}(\widetilde{Z})} 
\overset{\eqref{eqn:local_integral_eq2_3}}{\leq}
\dfrac{1}{24} \widehat{\tau}_{4, t} (p + 3)^2.    
\end{aligned}
\end{align}
Also, for $u \in \Theta_{n, t}^{\rm c}$, we have
\begin{align*}
- \cR_{t, 3}(u)
&=
- \left[ \widetilde{L}_{t}(\thetaMAP + u) - \widetilde{L}_{t}(\thetaMAP) - \langle \nabla \widetilde{L}_{t}(\thetaMAP), u \rangle + \dfrac{1}{2} \left\| \FisherMAP^{1/2} u \right\|_{2}^{2} \right] \\
&= 
\dfrac{1}{2} \left\| \FisherTilde[t]{\theta_{t}^{\circ}}^{1/2} u \right\|_{2}^{2} - 
\dfrac{1}{2} \left\| \FisherMAP^{1/2} u \right\|_{2}^{2} 
\leq 
\dfrac{1}{2} \left\| \FisherTilde[t]{\theta_{t}^{\circ}}^{1/2} u \right\|_{2}^{2}
\end{align*}
for some $\theta_{t}^{\circ} = \theta_{t}^{\circ}(\thetaMAP, u) \in \Theta (\thetaMAP, \FisherMAP, \| \FisherMAP^{1/2} u \|_{2})$ by Taylor's theorem.
Note that $\theta_{t}^{\circ}$ may not be located in $\Theta (\thetaMAP, \FisherMAP, 4r_{\LA})$.
Let $r = \| \FisherMAP^{1/2} u \|_{2} > 4 r_{\LA} > 1$.
Then,
\begin{align*}
\dfrac{1}{2} \left\| \FisherTilde[t]{\theta_{t}^{\circ}}^{1/2} u \right\|_{2}^{2}
&= \dfrac{1}{2} u^\top \FisherMAP^{1/2} \left(\bI_p + \FisherMAP^{-1/2} \FisherTilde[t]{\theta_{t}^{\circ}} \FisherMAP^{-1/2} - \bI_p \right) \FisherMAP^{1/2}  u
\\
&\leq 
\dfrac{1}{2} \left( 1 +  \left\| \FisherMAP^{-1/2} \FisherTilde[t]{\theta_{t}^{\circ}} \FisherMAP^{-1/2} - \bI_{p} \right\|_{2} \right) \left\| \FisherMAP^{1/2} u \right\|_{2}^{2} \\
\overset{\text{Lemma \ref{lemma:tech_Fisher_smooth}}}&{\leq}
\dfrac{1}{2} \left( 1 +  \widehat{\tau}_{3, t, r} \left\| \FisherMAP^{1/2} u \right\|_{2} \right) \left\| \FisherMAP^{1/2} u \right\|_{2}^{2} \\
&\leq 
\left( 1  \vee \widehat{\tau}_{3, t, r} \right) \left\| \FisherMAP^{1/2} u \right\|_{2}^{3}
=
\left( 1  \vee \widehat{\tau}_{3, t, r} \right) \exp \left( 3 \log \left\| \FisherMAP^{1/2} u \right\|_{2} \right) \\
&\leq 
\left( 1  \vee \widehat{\tau}_{3, t, r} \right) \exp \left( 3 \left\| \FisherMAP^{1/2} u \right\|_{2} \right) \\
&\leq 
N e^{8p} \exp\bigg( \left( \sqrt{p} + \sqrt{2 \log N} - 3 \right) \left\| \FisherMAP^{1/2} u \right\|_{2} \bigg)
\exp \left( 3 \left\| \FisherMAP^{1/2} u \right\|_{2} \right)
\end{align*}
where the last inequality holds by (\textbf{KL}).
It follows that
\begin{align}
\begin{aligned} \label{eqn:LA_KL_error_eq2}
&\bbE_{ \Theta_{n, t}^{\rm c} } \left[ - \cR_{t, 3}(\widetilde{Z}) \right]
\leq 
\bbE_{ \Theta_{n, t}^{\rm c} } \bigg[ 
    \dfrac{1}{2} \left\| \FisherTilde[t]{\theta_{t}^{\circ}}^{1/2} \widetilde{Z} \right\|_{2}^{2}
\bigg] \\
&\leq
N e^{8p}  \bbE_{ \Theta_{n, t}^{\rm c} } \Bigg[
\exp\bigg\{ \left( \sqrt{p} + \sqrt{2 \log N} - 3 \right) \left\| \FisherMAP^{1/2} \widetilde{Z} \right\|_{2} \bigg\}
\exp\left( 3 \left\| \FisherMAP^{1/2} \widetilde{Z} \right\|_{2} \right) \Bigg] \\
&=
N e^{8p} \bbE_{ \Theta_{n, t}^{\rm c} } \Bigg[
\exp\bigg( \left[ \sqrt{p} + \sqrt{2 \log N} \right] \left\| \FisherMAP^{1/2} \widetilde{Z} \right\|_{2} \bigg) \Bigg] \\
\overset{\substack{ \text{Lemma \ref{lemma:LA_tail_integral2}}}}&{\leq}
N e^{8p} e^{-8p - 8\log N} 
=
e^{- 7\log N}.
\end{aligned}
\end{align}
Therefore, \eqref{eqn:LA_KL_error_eq1} and \eqref{eqn:LA_KL_error_eq2} imply that
\begin{align} \label{eqn:LA_KL_error_eq3}
    ({\rm i}) = 
    \bbE_{ \Theta_{n, t} } \left[ - \cR_{t, 4}(\widetilde{Z}) \right] + \bbE_{ \Theta_{n, t}^{\rm c} } \left[ - \cR_{t, 3}(\widetilde{Z}) \right]
    \leq
    \dfrac{1}{24} \widehat{\tau}_{4, t} (p + 3)^2 + e^{- 7\log N}  
\end{align}

Next, we will obtain an upper bound of $({\rm ii})$. Note that
\begin{align*} 
&\left| e^{W_{n, t}} - 1 \right|
= 
    \left| \dfrac{ \displaystyle \int  \bigg( e^{f_{t}(u; \thetaMAP)} - \exp\left[ -\dfrac{1}{2}\left\| \FisherMAP[t]^{1/2}u \right\|_{2}^{2} \right]  \bigg)   \rmd u }{ \displaystyle \int  \exp\left[ -\dfrac{1}{2}\left\| \FisherMAP[t]^{1/2}u \right\|_{2}^{2} \right]  \rmd u }  \right| \\
&\leq 
    \left| \dfrac{ \displaystyle \int_{\Theta_{n, t}} \bigg( e^{f_{t}(u; \thetaMAP)} - \exp\left[ -\dfrac{1}{2}\left\| \FisherMAP[t]^{1/2}u \right\|_{2}^{2} \right]  \bigg)   \rmd u }{ \displaystyle \int  \exp\left[ -\dfrac{1}{2}\left\| \FisherMAP[t]^{1/2}u \right\|_{2}^{2} \right]  \rmd u }  \right|
    +
    \left| \dfrac{ \displaystyle \int_{\Theta_{n, t}^{\rm c}} e^{f_{t}(u; \thetaMAP)}  \rmd u }{ \displaystyle \int  \exp\left[ -\dfrac{1}{2}\left\| \FisherMAP[t]^{1/2}u \right\|_{2}^{2} \right]  \rmd u }  \right| \\
    &\qquad \qquad  +
    \left| \dfrac{ \displaystyle \int_{\Theta_{n, t}^{\rm c}} \exp\left[ -\dfrac{1}{2}\left\| \FisherMAP[t]^{1/2}u \right\|_{2}^{2} \right]  \rmd u }{ \displaystyle \int  \exp\left[ -\dfrac{1}{2}\left\| \FisherMAP[t]^{1/2}u \right\|_{2}^{2} \right]  \rmd u }  \right| \\
&\leq 
    \Delta_{\local, t}(\mathds{1}_{\Theta}) + \Delta_{\tail, \widetilde{\Pi}, t}(\mathds{1}_{\Theta})+ \Delta_{\tail, \LA, t}(\mathds{1}_{\Theta}) \\
&\leq 
    K_{1} 
        \bigg( 
            \left[ \widehat{\tau}_{4, t} + \widehat{\tau}_{3, t}^2 \right] p^2 
            +
            \widehat{\tau}_{3, t}^{3} r_{\LA}^{6}
        \bigg)   
    + 3e^{-8\log N - 8p},    
\end{align*}
where the last inequality holds by Lemmas \ref{lemma:LA_tail_integral} and \ref{lemma:LA_local_integral}, and $K_{1}$ denotes the constant $K$ in Lemma \ref{lemma:LA_local_integral}.
By $1 + x \leq e^{x}$ for $x \in \bbR$, it follows that
\begin{align} \label{eqn:LA_KL_error_eq4}
    ({\rm ii}) = W_{n, t} 
    \leq  
    K_{1} 
        \bigg( 
            \left[ \widehat{\tau}_{4, t} + \widehat{\tau}_{3, t}^2 \right] p^2 
            +
            \widehat{\tau}_{3, t}^{3} r_{\LA}^{6}
        \bigg)   
    + 3e^{-8\log N - 8p}.
\end{align}

By \eqref{eqn:LA_KL_error_eq3} and \eqref{eqn:LA_KL_error_eq4}, therefore, we have 
\begin{align*}
    &K \left( \Pi_{t}^{\LA}(\cdot); \ \widetilde{\Pi}_{t}\left(\cdot \mid \bD_{t} \right) \right) \\
    &\leq 
    \dfrac{1}{24} \widehat{\tau}_{4, t} (p + 3)^2 + e^{- 7\log N} 
    + K_{1} 
        \bigg( 
            \left[ \widehat{\tau}_{4, t} + \widehat{\tau}_{3, t}^2 \right] p^2 
            +
            \widehat{\tau}_{3, t}^{3} r_{\LA}^{6}
        \bigg)   
    + 3e^{-8\log N - 8p} \\
    &\leq  
    K_{2} 
        \bigg( 
            \left[ \widehat{\tau}_{4, t} + \widehat{\tau}_{3, t}^2 \right] p^2 
            +
            \widehat{\tau}_{3, t}^{3} r_{\LA}^{6}
            +
            e^{- 7\log N}
        \bigg)  \\
    &\lesssim
        K_{2}
        \bigg( 
            \left[ \widehat{\tau}_{4, t} + \widehat{\tau}_{3, t}^2 \right] p^2 
            +
            \widehat{\tau}_{3, t}^{3} \log^{3} N
            +
            e^{- 7\log N}
        \bigg),
\end{align*}
for some universal constant $K_{2} = K_{2}(K_{1}) > 0$.
This completes the proof.
\end{proof}

\subsection{Proof of Theorem \ref{thm:VB_KL}}

\begin{proof}[Proof of Theorem \ref{thm:VB_KL}]
In this proof, we will work on the event $\scrE_{1}$ without explicitly referring to it.
By Theorems \ref{thm:LA_TV} and \ref{thm:LA_KL}, we have
\begin{align*}
    d_{V} \left( \Pi_{t}^{\LA}(\cdot),  \widetilde{\Pi}_{t}\left(\cdot \mid \bD_{t} \right) \right) \leq K_1 \epsilon_{n, t, \TV}, \quad
    K \left( \Pi_{t}^{\LA}(\cdot),  \widetilde{\Pi}_{t}\left(\cdot \mid \bD_{t} \right) \right) \leq K_2 \epsilon_{n, t, \KL}^{2},
\end{align*}
where $K_1$ and $K_2$ denote the constant $K$ in Theorems \ref{thm:LA_TV} and \ref{thm:LA_KL}, respectively.
By the definition of $\Pi_{t}(\cdot)$, it follows that
\begin{align*}
    K \left( \Pi_{t}(\cdot); \ \widetilde{\Pi}_{t}\left(\cdot \mid \bD_{t} \right) \right) 
    \leq 
    K \left( \Pi_{t}^{\LA}(\cdot); \ \widetilde{\Pi}_{t}\left(\cdot \mid \bD_{t} \right) \right) 
    \leq K_2 \epsilon_{n, t, \KL}^{2}.
\end{align*}
By Pinsker's inequality, we have
\begin{align*}
    &d_{V}\left( \Pi_{t}(\cdot),  \widetilde{\Pi}_{t}\left(\cdot \mid \bD_{t} \right) \right)
    \leq 
    \sqrt{ \dfrac{1}{2} K \left( \Pi_{t}(\cdot),  \widetilde{\Pi}_{t}\left(\cdot \mid \bD_{t} \right) \right)}
    \leq 
    \sqrt{\dfrac{K_2}{2}} \epsilon_{n, t, \KL}, \\
    &d_{V}\left( \Pi_{t}^{\LA}(\cdot),  \widetilde{\Pi}_{t}\left(\cdot \mid \bD_{t} \right) \right)
    \leq 
    \sqrt{ \dfrac{1}{2} K \left( \Pi_{t}^{\LA}(\cdot),  \widetilde{\Pi}_{t}\left(\cdot \mid \bD_{t} \right) \right)}
    \leq 
    \sqrt{\dfrac{K_2}{2}} \epsilon_{n, t, \KL}.
\end{align*}
By taking the constant $K$ in Theorem \ref{thm:VB_KL} as $K = K_2 \vee \sqrt{K_2/2}$, we complete the proof.
\end{proof}

For $t \in [T]$, let
\begin{align}
\begin{aligned} \label{def:VB_quantities}
    \Delta_{t} &= \left\| \bOmega_{t}^{1/2} \left( \mu_t - \thetaMAP \right) \right\|_{2} \vee
    \left\| \FisherMAP^{1/2} \left( \mu_t - \thetaMAP \right) \right\|_{2} \\
    &\qquad \vee
    \left\| \bOmega_{t}^{-1/2} \FisherMAP \bOmega_{t}^{-1/2} - \bI_{p} \right\|_{\rm F} \vee
    \left\| \FisherMAP^{-1/2} \bOmega_{t} \FisherMAP^{-1/2} - \bI_{p} \right\|_{\rm F}.
\end{aligned}
\end{align}

\begin{corollary} \label{coro:VB_KL_Delta}
    Suppose that conditions in Theorem \ref{thm:VB_KL} hold. Also, assume that
    \begin{align*}
        \epsilon_{n, t, \KL} \leq (1200 K)^{-1}, \quad  \forall t \in [T]
    \end{align*}
    on $\scrE_{1}$ defined in Theorem \ref{thm:VB_KL}, where $K$ is the universal constant specified in Theorem \ref{thm:VB_KL}.
    Then, on $\scrE_{1}$, 
    \begin{align*}
        \Delta_{t} \leq 400 K \epsilon_{n, t, \KL}, \quad \forall t \in [T].
    \end{align*}
\end{corollary}
\begin{proof}
    From the proof of Theorem \ref{thm:VB_KL}, we have
    \begin{align*}
        d_{V}\left( \Pi_{t}(\cdot),  \widetilde{\Pi}_{t}\left(\cdot \mid \bD_{t} \right) \right) \
        &\leq
        K_1 \epsilon_{n, t, \KL}, \\  
        d_{V}\left( \Pi_{t}^{\LA}(\cdot), \widetilde{\Pi}_{t}\left(\cdot \mid \bD_{t} \right)  \right) 
        &\leq 
        K_1 \epsilon_{n, t, \KL},
    \end{align*}
    where $K_1$ denotes the constant $K$ in Theorem \ref{thm:VB_KL}. 
    It follows that
    \begin{align*}
    d_{V}\left( \Pi_{t}(\cdot),  \Pi_{t}^{\LA}(\cdot) \right)
    &\leq 
    d_{V}\left( \Pi_{t}(\cdot),  \widetilde{\Pi}_{t}\left(\cdot \mid \bD_{t} \right) \right) + d_{V}\left( \widetilde{\Pi}_{t}\left(\cdot \mid \bD_{t} \right),  \Pi_{t}^{\LA}(\cdot) \right) \\
    &\leq 
    K_1 \epsilon_{n, t, \KL} + K_1 \epsilon_{n, t, \KL} 
    \leq 
    2K_1 \epsilon_{n, t, \KL}.
    \end{align*}
    Since $2K_1 \epsilon_{n, t, \KL} \leq 1/600$, we can apply Lemma \ref{lemma:tech_Gaussian_comparison_lower}.
    By Lemma \ref{lemma:tech_Gaussian_comparison_lower} and symmetry of $d_{V}(\cdot, \cdot)$, therefore, we have
    \begin{align*}
        \dfrac{\Delta_{t}}{200} 
        \leq 
        d_{V}\left( \Pi_{t}(\cdot),  \Pi_{t}^{\LA}(\cdot) \right)
        \leq 
        2K_1 \epsilon_{n, t, \KL},
    \end{align*}
    which completes the proof.
\end{proof}


\section{Proofs for Section \ref{sec:pMLE}}

\begin{proof}[Proof of Theorem \ref{thm:penalized_estimation}]
In this proof, we work on the event $\scrE_{2} \cap \scrE_{\est, 1}$ without explicitly referring to it.
Let $t \in [T]$. 
For simplicity in notations, let $\Theta_{n, t} = \Theta (\thetaBest, \FisherBest, 4r_{\eff, t})$, $\overline{\Theta}_{n, t} = \Theta_{n, t}(\FisherBest, 4r_{\eff, t})$ in this proof, and 
\begin{align*}
    \partial \Theta_{n, t} = \left\{ \theta \in \Theta : \left\| \FisherBest^{1/2} \left( \theta - \thetaBest \right) \right\|_{2} = 4 r_{\eff, t} \right\}.
\end{align*}
For any $\theta \in \Theta$, let
\begin{align} \label{eqn:pen_estimation_eq0}
    g_{t}(\theta) = \bbE_{t} \widetilde{L}_{t}(\theta) + \langle \nabla \zeta_{t}, \theta \rangle = \bbE_{t} \widetilde{L}_{t}(\theta) + \langle \nabla \widetilde{L}_{t}(\thetaMAP) - \nabla \bbE_{t} \widetilde{L}_{t}(\thetaMAP), \theta \rangle.     
\end{align}
By the right-hand side of the last display and the strong concavity of $\theta \mapsto \bbE_{t} \widetilde{L}_{t}(\theta)$, note that $\nabla g_{t}(\thetaMAP) = 0$ and $\theta \mapsto g_{t}(\theta)$ is concave.
Hence, we only need to prove that the first-order stationary point $\thetaMAP$ of $g_{t}(\cdot)$ is located in $\Theta_{n, t}$.

By the concavity of $g_{t}(\cdot)$, for any $\theta \in \Theta_{n, t}^{\rm c}$, we have
\begin{align} \label{eqn:pen_estimation_eq1}
    g_{t}(\overline{\theta}) \geq \omega g_{t}(\theta) + (1 - \omega)g_{t}(\thetaBest),
\end{align}
where $\overline{\theta} = \omega \theta + (1- \omega)\thetaBest$ and $\omega = 4 r_{\eff, t} \| \FisherBest^{1/2}(\theta - \thetaBest) \|_{2}^{-1} \in (0, 1 )$.
At the end of this proof, we will show that 
\begin{align} \label{eqn:pen_estimation_eq2}
    \sup_{\theta^{\circ} \in \partial \Theta_{n, t}} g_{t}(\theta^{\circ}) - g_{t}(\thetaBest) \leq -2 r_{\eff, t}^{2} < 0.
\end{align}
It follows that, for any $\theta \in \Theta_{n, t}^{\rm c}$,
\begin{align*}
    0 > -2 r_{\eff, t}^{2} 
    \geq  \sup_{\theta^{\circ} \in \partial \Theta_{n, t}} g_{t}(\theta^{\circ}) - g_{t}(\thetaBest)
    \overset{\eqref{eqn:pen_estimation_eq1}}{\geq} \omega \bigg[ g_{t}(\theta) - g_{t}(\thetaBest) \bigg]
    \geq g_{t}(\theta) - g_{t}(\thetaBest),
\end{align*}
which implies that $\thetaMAP \in \Theta_{n, t}$. 

To complete the proof, we only need to prove \eqref{eqn:pen_estimation_eq2}. 
Let $\theta^{\circ} \in \partial \Theta_{n, t}$ and $u = \theta^{\circ} - \thetaBest$.  
By Taylor's theorem, there exists some $\widetilde{u} \in \overline{\Theta}_{n, t}$ such that, on $\scrE_{\est, 1}$,
\begin{align*}
    &g_{t}(\theta^{\circ}) - g_{t}(\thetaBest) 
    = \nabla g_{t}(\thetaBest)^{\top} u + \dfrac{1}{2} \langle \nabla^{2} g_{t}(\thetaBest + \widetilde{u}), u^{\otimes 2} \rangle \\
    &= \left[ \nabla \bbE_{t} \widetilde{L}_{t}(\thetaBest) + \nabla \zeta_{t} \right]^{\top} u + \dfrac{1}{2} \langle \nabla^{2} \bbE_{t} \widetilde{L}_{t}(\thetaBest + \widetilde{u}), u^{\otimes 2}  \rangle \\
    &= \nabla \zeta_{t}^{\top} u + \dfrac{1}{2} \langle \nabla^{2} \bbE_{t} \widetilde{L}_{t}(\thetaBest + \widetilde{u}), u^{\otimes 2}  \rangle \\
    &= \left[\FisherBest^{-1/2} \nabla \zeta_{t} \right]^{\top} \FisherBest^{1/2} u - \dfrac{1}{2} \langle \FisherTilde[t]{\thetaBest + \widetilde{u}}, u^{\otimes 2} \rangle \\
    \overset{\text{Lemma \ref{lemma:tech_Fisher_smooth}}}&{\leq}
    \bigg( \left\| \FisherBest^{-1/2} \nabla \zeta_{t} \right\|_{2}  - \dfrac{1}{2} \left( 1 -  4 \tau_{3, t}^{\ast} r_{\eff, t} \right) \left\| \FisherBest^{1/2} u \right\|_{2}  \bigg) 
    \left\| \FisherBest^{1/2} u \right\|_{2} \\
    \overset{(\textbf{A1})}&{\leq} 
    \left[ r_{\eff, t}  - 2\left( 1 -  4 \tau_{3, t}^{\ast} r_{\eff, t} \right) r_{\eff, t} \right] \times 4r_{\eff, t}, \quad (\because \left\| \FisherBest^{1/2} u \right\|_{2} = 4 r_{\eff, t})  \\
    &\leq - 2 r_{\eff, t}^{2},
\end{align*}
where the last inequality holds by $\tau_{3, t}^{\ast} r_{\eff, t} \leq 1/16$. This completes the proof.
\end{proof}

\begin{lemma} \label{lemma:radius_upper_bound}
    Suppose that (\textbf{A0}), (\textbf{A1}) hold. Also, assume that $\| \thetaBest - \theta_0 \|_{2} \leq 1/2$ for all $t \in [T]$ on an event $\scrE$. Then, 
    \begin{align*}
        r_{\eff, t} 
        \leq M_n \left[ \left\{ \dfrac{\lambda_{\max}\left( \Fisher[t]{\thetaBest} \right)}{\lambda_{\min}\big( \FisherBest \big)} \wedge 1 \right\} p_{\ast} \right]^{1/2}, \quad \forall t \in [T]
    \end{align*}
    on $\scrE$, where $p_{\ast} = p \vee \log n \vee \log T$.
\end{lemma}

\begin{proof}
Note that
\begin{align*}
    r_{\eff, t} 
    &= 
    p_{\eff, t}^{1/2} + \sqrt{2\lambda_{t} (\log n + \log T)}
    =
    \sqrt{\operatorname{tr}\left( \FisherBest^{-1} \bV_t \right)} + \sqrt{2\lambda_{t} (\log n + \log T)} \\
    &\leq
    \sqrt{ \left\| \FisherBest^{-1} \bV_t \right\|_2 p  } + \sqrt{2\lambda_{t} (\log n + \log T)} 
    =
    \lambda_{t}^{1/2} \left(  p^{1/2} + \sqrt{2 (\log n + \log T)} \right) \\
    &\leq 
    3 \lambda_{t}^{1/2} p_{\ast}^{1/2}.
\end{align*}
By the assumption $\| \thetaBest - \theta_0 \|_{2} \leq 1/2$ on $\scrE$, we have
\begin{align*}
    \lambda_{t} 
    &= \left\| \FisherBest^{-1} \bV_{t} \right\|_{2} 
    = \left\| \FisherBest^{-1} \Fisher[t]{\thetaBest} \Fisher[t]{\thetaBest}^{-1} \bV_{t} \right\|_{2} 
    \leq \left\| \FisherBest^{-1} \Fisher[t]{\thetaBest} \right\|_{2}  \left\| \Fisher[t]{\thetaBest}^{-1} \bV_{t} \right\|_{2}  \\
    \overset{(\textbf{A1})}&{\leq} 
    \dfrac{M_n^{2}}{9} \left\| \FisherBest^{-1} \Fisher[t]{\thetaBest} \right\|_{2}  
    \leq \dfrac{M_n^{2}}{9} \left[ 1 \wedge \dfrac{\lambda_{\max}\left( \Fisher[t]{\thetaBest} \right)}{\lambda_{\min} \big( \FisherBest \big)}\right].
\end{align*}
Combining the last two displays, we complete the proof.
\end{proof}



\section{Proofs for Section \ref{sec:eigenvalue_analysis}}

For $t \in [T]$, let
\begin{align}
\begin{aligned} \label{def:bias_quantity}
    b_{n, t} 
        &= \left\| \FisherBest^{-1/2} \bOmega_{t-1} \left( \theta_0 - \mu_{t-1} \right) \right\|_{2}, \\    
    \tau_{3, t, \bias} 
        &= \inf \left\{ 
            \tau_3 \in \bbR_{+} : 
            \sup_{u \in \Theta (\FisherBest, 4b_{n, t})} \sup_{z \in \bbR^{p}} 
            \dfrac{
            \left| \langle \nabla^3 \bbE_{t} \widetilde{L}_{t}(\thetaBest + u), z^{\otimes 3} \rangle \right|
            }{
            \left\| \FisherBest^{1/2} z \right\|_{2}^{3}
            }
            \leq 
            \tau_3
        \right\}, \\
    \tau_{4, t, \bias}
        &= \inf \left\{ 
            \tau_4 \in \bbR_{+} : 
            \sup_{u \in \Theta (\FisherBest, 4b_{n, t})} \sup_{z \in \bbR^{p}} 
            \dfrac{
            \left| \langle \nabla^4 \bbE_{t} \widetilde{L}_{t}(\thetaBest + u), z^{\otimes 4} \rangle \right|
            }{
            \left\| \FisherBest^{1/2} z \right\|_{2}^{4}
            }
            \leq 
            \tau_4
        \right\}.
\end{aligned}
\end{align}

\begin{lemma} \label{lemma:penalized_bias}
    Suppose that (\textbf{A0}) and (\textbf{A1}) hold. 
    Also, assume that $\tau_{3, t, \bias} \: b_{n, t} \leq 1/16$ on an event $\scrE$.
    Then, on $\scrE$,
    \begin{align} \label{eqn:penalized_bias_claim}
        \left\| \FisherBest^{1/2} \left( \theta_0 - \thetaBest \right) \right\|_{2} \leq {\rm (a)} + {\rm (b)} + {\rm (c)},
    \end{align}
    where
    \begin{align*}
        {\rm (a)} &= \dfrac{1}{6} \tau_{4, t, \bias} \bigg( 1 - 4\tau_{3, t, \bias} \: b_{n, t} \bigg)^{-1}
        \bigg( 1 + \dfrac{1}{2}\tau_{3, t, \bias} \: b_{n, t} \bigg)^3 b_{n, t}^3, \\
        {\rm (b)} &= \dfrac{1}{2} \tau_{3, t, \bias}^2\bigg( 1 - 4\tau_{3, t, \bias} \: b_{n, t} \bigg)^{-1}
        \bigg( 1 + \dfrac{1}{2}\tau_{3, t, \bias}b_{n, t} \bigg) b_{n, t}^3, \\
        {\rm (c)} &= \bigg( 1 + \dfrac{1}{2}\tau_{3, t, \bias}b_{n, t} \bigg) b_{n, t}.
    \end{align*}
\end{lemma}
\begin{proof}
In this proof, we work on the event $\scrE$ without explicitly referring to it.
Let $t \in [T]$. 
For simplicity in notations, let $\Theta_{n, t} = \Theta (\thetaBest, \FisherBest, 4b_{n, t})$, $\overline{\Theta}_{n, t} = \Theta (\FisherBest, 4b_{n, t})$ in this proof.
For $\theta \in \Theta$, let 
\begin{align} \label{eqn:penalized_bias_eq1}
    g_{t}(\theta) = \bbE_{t} \widetilde{L}_{t}(\theta) + \langle \bOmega_{t-1} \left( \theta_0 - \mu_{t-1} \right), \theta \rangle.
\end{align}
It follows that
\begin{align*}
    \nabla g_{t}(\theta_0) 
    \overset{\eqref{eqn:penalized_bias_eq1}}{=}
    \nabla \bbE_{t} \widetilde{L}_{t}(\theta_0) + \bOmega_{t-1} \left( \theta_0 - \mu_{t-1} \right)
    = \nabla \bbE_{t} L_{t}(\theta_0) = 0.
\end{align*}
Let 
\begin{align} \label{eqn:penalized_bias_eq1_2}
    \varphi_{t} = \bOmega_{t-1} \left( \theta_0 - \mu_{t-1} \right), \quad 
    \phi_{t} = \FisherBest^{-1} \left[ \varphi_{t} + \dfrac{1}{2} \left\langle \nabla^{3} \widetilde{L}_{t}(\thetaBest),  \left( \FisherBest^{-1} \varphi_{t} \right)^{\otimes 2} \right \rangle  \right].
\end{align}
By Taylor's theorem, we have
\begin{align}
\begin{aligned} \label{eqn:penalized_bias_eq2}
    &\FisherBest^{-1/2} \big[ \nabla g_{t} (\thetaBest + \phi_{t}) - \nabla g_{t} (\theta_0) \big] \\
    &=
    \left[ \int_{0}^{1} \FisherBest^{-1/2} \nabla^{2} g_{t} \left( \theta_0 + s\left\{ \thetaBest + \phi_{t} - \theta_0 \right\} \right) \FisherBest^{-1/2} \rmd s \right] \FisherBest^{1/2}\left( \thetaBest + \phi_{t} -\theta_0 \right) \\
    \overset{\eqref{eqn:penalized_bias_eq1}}&{=}
    \left[ \int_{0}^{1} \FisherBest^{-1/2} \nabla^{2} \bbE_{t} \widetilde{L}_{t} \left( \theta_0 + s\left\{ \thetaBest + \phi_{t} - \theta_0 \right\} \right) \FisherBest^{-1/2} \rmd s \right] \FisherBest^{1/2}\left( \thetaBest + \phi_{t} -\theta_0 \right).      
\end{aligned}
\end{align}
Later, we will prove that
\begin{align} \label{eqn:penalized_bias_eq2_2} 
    \left\| \FisherBest^{1/2} \left( \theta_0 - \thetaBest \right) \right\|_{2} \vee \left\| \FisherBest^{1/2} \phi_{t} \right\|_{2} \leq 4 b_{n, t},
\end{align}
which implies that $\theta_0, \thetaBest + \phi_{t} \in \Theta_{n, t}$. 
By \eqref{eqn:penalized_bias_eq2} and Lemma \ref{lemma:tech_Fisher_smooth} with $f(\cdot) = \bbE_t \widetilde{L}_{t}(\cdot)$ and $x = \thetaBest$, we have
\begin{align*}
    -\int_{0}^{1} \FisherBest^{-1/2} \nabla^{2} \bbE_{t} \widetilde{L}_{t} \left( \theta_0 + s\left\{ \thetaBest + \phi_{t} - \theta_0 \right\} \right) \FisherBest^{-1/2} \rmd s
    &\succeq 
    \int_{0}^{1} \FisherBest^{-1/2} \left[ \left( 1 - 4\tau_{3, t, \bias} b_{n, t} \right) \FisherBest \right] \FisherBest^{-1/2} \rmd s \\
    &= 
    \left( 1 - 4\tau_{3, t, \bias} b_{n, t} \right) \bI_{p},
\end{align*}
which implies that
\begin{align*}
    \left\| \FisherBest^{-1/2} \left[ \nabla g_{t} (\thetaBest + \phi_{t}) - \nabla g_{t} (\theta_0) \right] \right\|_{2}
    \geq 
    \left( 1 - 4\tau_{3, t, \bias} b_{n, t} \right) \left\| \FisherBest^{1/2}\left( \thetaBest + \phi_{t} -\theta_0 \right) \right\|_{2}.
\end{align*}
It follows that
\begin{align} 
\begin{aligned} \label{eqn:penalized_bias_eq3} 
    &\left\| \FisherBest^{1/2}\left( \thetaBest -\theta_0 \right) \right\|_{2} \\
    &\leq
    \left( 1 - 4\tau_{3, t, \bias} b_{n, t} \right)^{-1}
    \left\| \FisherBest^{-1/2} \left[ \nabla g_{t} (\thetaBest + \phi_{t}) - \nabla g_{t} (\theta_0) \right] \right\|_{2}
    +
    \left\| \FisherBest^{1/2}\phi_{t} \right\|_{2} \\
    &= 
    \left( 1 - 4\tau_{3, t, \bias} b_{n, t} \right)^{-1}
    ({\rm ii}) + 
    ({\rm i}),    
\end{aligned}
\end{align}
where
\begin{align*}
    ({\rm i}) = \left\| \FisherBest^{1/2}\phi_{t} \right\|_{2}, \quad     
    ({\rm ii}) = \left\| \FisherBest^{-1/2} \left[ \nabla g_{t} (\thetaBest + \phi_{t}) - \nabla g_{t} (\theta_0) \right] \right\|_{2}.
\end{align*}

In the remainder of this proof, we will prove \eqref{eqn:penalized_bias_eq2_2}, and obtain upper bounds of $({\rm i})$ and $({\rm ii})$. 
Firstly, we will prove \eqref{eqn:penalized_bias_eq2_2}, which encompasses an upper bound of $({\rm i})$. 
By Lemma \ref{lemma:tech_pre_bias_bound} with
\begin{align*}
    \tau_{3} = \tau_{3, t, \bias}, \quad 
    f(\cdot) = \bbE_{t} \widetilde{L}_{t} (\cdot), \quad 
    \theta = \thetaBest, \quad 
    \widetilde{\theta} = \theta_0, \quad 
    \beta = \varphi_{t}, \quad 
    r = b_{n, t},
\end{align*}
the condition $\tau_{3, t, \bias} b_{n, t} \leq 1/16$ implies that
\begin{align*}
     \left\| \FisherBest^{1/2} \left( \theta_0 - \thetaBest \right) \right\|_{2} \leq 4 b_{n, t}.
\end{align*}
Also, 
\begin{align}
\begin{aligned} \label{eqn:penalized_bias_eq3_2} 
    &\left\| \FisherBest^{1/2} \phi_{t} - \FisherBest^{-1/2} \varphi_{t} \right\|_{2}
    \overset{\eqref{eqn:penalized_bias_eq1_2}}{=} 
    \left\| \FisherBest^{-1/2} \times \dfrac{1}{2} \left\langle \nabla^{3} \widetilde{L}_{t}(\thetaBest),  \left( \FisherBest^{-1} \varphi_{t} \right)^{\otimes 2} \right \rangle \right\|_{2} \\
    &= 
    \sup_{u \in \bbR^{p} : \| u \|_{2} = 1}
    \dfrac{1}{2} \left| \left\langle \nabla^{3} \widetilde{L}_{t}(\thetaBest),  \left( \FisherBest^{-1} \varphi_{t} \right)^{\otimes 2} \otimes \left( \FisherBest^{-1/2}u \right)  \right \rangle  \right|  \\
    \overset{\eqref{def:3_4_smooth_equiv}}&{\leq}
    \sup_{u \in \bbR^{p} : \| u \|_{2} = 1}
    \dfrac{1}{2} \tau_{3, t, \bias} \left\| \FisherBest^{1/2} \FisherBest^{-1} \varphi_{t} \right\|_{2}^{2} \left\| \FisherBest^{1/2} \FisherBest^{-1/2} u \right\|_{2} 
    = \dfrac{1}{2} \tau_{3, t, \bias} b_{n, t}^{2}    
\end{aligned}
\end{align}
It follows that
\begin{align} \label{eqn:penalized_bias_eq4} 
    ({\rm i}) = \left\| \FisherBest^{1/2} \phi_{t} \right\|_{2} \leq \left( 1 + \dfrac{1}{2} \tau_{3, t, \bias} b_{n, t} \right) b_{n, t} \leq 4 b_{n, t},
\end{align}
which completes the proof of \eqref{eqn:penalized_bias_eq2_2}.

Next, we will obtain an upper bound of $({\rm ii})$. Since $\nabla g_{t}(\theta_0) = 0$, it suffices to obtain an upper bound of $\| \FisherBest^{-1/2} \nabla g_{t} (\thetaBest + \phi_{t}) \|_{2}$.
By \eqref{eqn:penalized_bias_eq1_2}, we have
\begin{align} \label{eqn:penalized_bias_eq5} 
    \FisherBest^{-1/2} \varphi_{t}
    =
    \FisherBest^{1/2} \phi_{t} 
    -
    \dfrac{1}{2} \FisherBest^{-1/2} \left\langle \nabla^{3} \widetilde{L}_{t}(\thetaBest),  \left( \FisherBest^{-1} \varphi_{t} \right)^{\otimes 2} \right \rangle.
\end{align}
Also,
\begin{align*}
    &\left\| \FisherBest^{-1/2} \nabla g_{t} (\thetaBest + \phi_{t}) \right\|_{2} 
    \overset{\eqref{eqn:penalized_bias_eq1}}{=}
    \left\| \FisherBest^{-1/2} \nabla \bbE_{t} \widetilde{L}_{t} (\thetaBest + \phi_{t}) + \FisherBest^{-1/2} \varphi_{t} \right\|_{2} \\
    \overset{\eqref{eqn:penalized_bias_eq5}}&{=}
    \left\| 
        \FisherBest^{-1/2} \nabla \bbE_{t} \widetilde{L}_{t} (\thetaBest + \phi_{t}) + \FisherBest^{1/2} \phi_{t} 
        -
        \dfrac{1}{2} \FisherBest^{-1/2} \left\langle \nabla^{3} \widetilde{L}_{t}(\thetaBest),  \left( \FisherBest^{-1} \varphi_{t} \right)^{\otimes 2} \right \rangle 
    \right\|_{2} \\
    &\leq 
    \left\| \FisherBest^{-1/2} \nabla \bbE_{t} \widetilde{L}_{t} (\thetaBest + \phi_{t}) + \FisherBest^{1/2} \phi_{t}  
    - \dfrac{1}{2} \FisherBest^{-1/2} \left\langle \nabla^{3} \widetilde{L}_{t}(\thetaBest),  \left( \phi_{t} \right)^{\otimes 2} \right \rangle   \right\|_{2} \\
    &\qquad +
    \left\| 
        \dfrac{1}{2} \FisherBest^{-1/2} \left\langle \nabla^{3} \widetilde{L}_{t}(\thetaBest),  \left( \phi_{t} \right)^{\otimes 2} \right \rangle
        - \dfrac{1}{2} \FisherBest^{-1/2} \left\langle \nabla^{3} \widetilde{L}_{t}(\thetaBest),  \left( \FisherBest^{-1} \varphi_{t} \right)^{\otimes 2} \right \rangle 
    \right\|_{2}.
\end{align*}
To bound $({\rm ii})$, we need to obtain an upper bounds of the following quantities:
\begin{align*} 
\begin{aligned} 
    ({\rm iii}) &= \left\| \FisherBest^{-1/2} \nabla \bbE_{t} \widetilde{L}_{t} (\thetaBest + \phi_{t}) + \FisherBest^{1/2} \phi_{t} - \dfrac{1}{2} \FisherBest^{-1/2} \left\langle \nabla^{3} \widetilde{L}_{t}(\thetaBest),  \left( \phi_{t} \right)^{\otimes 2} \right \rangle   \right\|_{2}, \\
    ({\rm iv}) &= \left\| 
        \dfrac{1}{2} \FisherBest^{-1/2} \left\langle \nabla^{3} \widetilde{L}_{t}(\thetaBest),  \left( \phi_{t} \right)^{\otimes 2} \right \rangle
        - \dfrac{1}{2} \FisherBest^{-1/2} \left\langle \nabla^{3} \widetilde{L}_{t}(\thetaBest),  \left( \FisherBest^{-1} \varphi_{t} \right)^{\otimes 2} \right \rangle 
    \right\|_{2}.    
\end{aligned}
\end{align*}
At the end of this proof, we will prove the following inequalities:
\begin{align} 
\begin{aligned} \label{eqn:penalized_bias_eq6}
    ({\rm iii}) &\leq \dfrac{1}{6} \tau_{4, t, \bias} \left( 1 + \dfrac{1}{2} \tau_{3, t, \bias} b_{n, t} \right)^{3} b_{n, t}^{3}
    \\ 
    ({\rm iv}) &\leq \dfrac{1}{2} \tau_{3, t, \bias}^{2} \left( 1 + \dfrac{1}{2} \tau_{3, t, \bias} b_{n, t} \right)b_{n, t}^{3}.    
\end{aligned}
\end{align}
By \eqref{eqn:penalized_bias_eq3}, we have
\begin{align*}
    \left\| \FisherBest^{1/2}\left( \thetaBest -\theta_0 \right) \right\|_{2}
    \leq
    \left( 1 - 4\tau_{3, t, \bias} b_{n, t} \right)^{-1} \big[ ({\rm iii}) + ({\rm iv}) \big] + ({\rm i}).
\end{align*}
By \eqref{eqn:penalized_bias_eq4} and \eqref{eqn:penalized_bias_eq6}, we complete the proof of \eqref{eqn:penalized_bias_claim}.

To complete the proof, we only need to prove \eqref{eqn:penalized_bias_eq6}.
First, we will prove the first inequality in \eqref{eqn:penalized_bias_eq6}.
By \eqref{eqn:penalized_bias_eq4}, note that $\thetaBest + \phi_{t} \in \Theta_{n, t}$. Also, the stochastic linearity in (\textbf{A1}) and the definition of $\thetaBest[t]$ imply that for $k \in \{2 ,3, 4\}$
\begin{align*}
    \nabla^{k} \bbE_{t} \widetilde{L}_{t}(\cdot) = \nabla^{k} \widetilde{L}_{t}(\cdot), \quad 
    \nabla \bbE_{t} \widetilde{L}_{t} (\thetaBest) = 0.
\end{align*}
By Taylor's theorem, we have
\begin{align*}
&\left\| \FisherBest^{-1/2} \nabla \bbE_{t} \widetilde{L}_{t} (\thetaBest + \phi_{t}) + \FisherBest^{1/2} \phi_{t} - \dfrac{1}{2} \FisherBest^{-1/2} \left\langle \nabla^{3} \widetilde{L}_{t}(\thetaBest),  \left( \phi_{t} \right)^{\otimes 2} \right \rangle   \right\|_{2}  \\
&=
\left\| \FisherBest^{-1/2} 
\left[
    \nabla \bbE_{t} \widetilde{L}_{t} (\thetaBest + \phi_{t}) 
    - \nabla \bbE_{t} \widetilde{L}_{t} (\thetaBest) 
    - \langle \nabla^{2} \bbE_{t} \widetilde{L}_{t} (\thetaBest), \phi_{t} \rangle
    - \dfrac{1}{2} \left\langle \nabla^{3} \bbE_{t} \widetilde{L}_{t}(\thetaBest),  \left( \phi_{t} \right)^{\otimes 2} \right \rangle  
\right]
\right\|_{2} \\
&\leq
\sup_{\widetilde{u} \in \overline{\Theta}_{n, t}}
\left\| \FisherBest^{-1/2} 
    \dfrac{1}{6} \left\langle \nabla^{4} \bbE_{t} \widetilde{L}_{t}(\thetaBest + \widetilde{u}),  \left( \phi_{t} \right)^{\otimes 3} \right \rangle  
\right\|_{2} \\
&= 
\sup_{\widetilde{u} \in \overline{\Theta}_{n, t}}
\left\| \FisherBest^{-1/2} 
    \dfrac{1}{6} \left\langle \nabla^{4} \widetilde{L}_{t}(\thetaBest + \widetilde{u}),  \left( \phi_{t} \right)^{\otimes 3} \right \rangle  
\right\|_{2}  \\
&=
\sup_{\widetilde{u} \in \overline{\Theta}_{n, t}}
\sup_{u \in \bbR^{p} : \| u \|_{2} = 1}
\dfrac{1}{6} 
\left\langle \nabla^{4} \widetilde{L}_{t}(\thetaBest + \widetilde{u}),  \left( \phi_{t} \right)^{\otimes 3} \otimes \left(\FisherBest^{-1/2} u \right) \right \rangle  \\
\overset{\eqref{def:3_4_smooth_equiv}}&{\leq}
\dfrac{1}{6} 
\tau_{4, t, \bias} \left\| \FisherBest^{1/2}\phi_{t} \right\|_{2}^{3}
\overset{\eqref{eqn:penalized_bias_eq4}}{\leq}
\dfrac{1}{6}
\tau_{4, t, \bias} \left( 1 + \dfrac{1}{2} \tau_{3, t, \bias} b_{n, t} \right)^{3} b_{n, t}^{3},
\end{align*}
which completes the proof of the first inequality in \eqref{eqn:penalized_bias_eq6}.

Next, we will prove the second inequality in \eqref{eqn:penalized_bias_eq6}. 
Let 
\begin{align*}
    \cT = (\cT_{ijk})_{i,j,k \in [p]} = \nabla^{3} \widetilde{L}_{t}(\thetaBest)/6 \in \bbR^{p \times p \times p}.
\end{align*}
With a slight abuse of notation, let $\cT: \bbR^p \to \bbR$ be the function defined as
\begin{align*}
    \cT(u) = \langle \cT, u^{\otimes 3} \rangle.
\end{align*}
Then, we have
\begin{align*}
    \nabla \cT(u) = \left( 3 \sum_{j, k}^{p} \cT_{ijk} u_{j} u_{k} \right)_{i \in [p]} =  3\left( \langle \cT_i, u^{\otimes 2} \rangle \right)_{i \in [p]}, \quad 
    \nabla^{2} \cT(u) = 6 \sum_{i = 1}^{p} u_{i} \cT_{i},
\end{align*}
where $\cT_{i} = (\cT_{ijk})_{j, k \in [p]} \in \bbR^{p \times p}$.
By using the above definitions, note that
\begin{align*}
    &\dfrac{1}{2}\left\| 
    \FisherBest^{-1/2} \left\langle \nabla^{3} \widetilde{L}_{t}(\thetaBest),  \left( \phi_{t} \right)^{\otimes 2} \right \rangle
    - \FisherBest^{-1/2} \left\langle \nabla^{3} \widetilde{L}_{t}(\thetaBest),  \left( \FisherBest^{-1} \varphi_{t} \right)^{\otimes 2} \right \rangle 
    \right\|_{2} \\
    &= 
    \left\| 
    \FisherBest^{-1/2} \nabla \cT \left( \phi_{t} \right) 
    -
    \FisherBest^{-1/2} \nabla \cT \left( \FisherBest^{-1} \varphi_{t} \right) 
    \right\|_{2} \\
    &\leq 
    \sup_{s \in [0, 1]}
    \left\| 
    \FisherBest^{-1/2} \nabla^{2} \cT \bigg( s\phi_{t} + (1-s) \FisherBest^{-1} \varphi_{t} \bigg) \left( \phi_{t} - \FisherBest^{-1} \varphi_{t} \right)
    \right\|_{2} \\    
    &\leq 
    \sup_{s \in [0, 1]}
    \left\| 
    \FisherBest^{-1/2} \nabla^{2} \cT \bigg( s\phi_{t} + (1-s) \FisherBest^{-1} \varphi_{t} \bigg) \FisherBest^{-1/2}
    \right\|_{2} 
    \left\| 
    \FisherBest^{1/2} \left( \phi_{t} - \FisherBest^{-1} \varphi_{t} \right)
    \right\|_{2}. 
\end{align*}
Therefore, we only need to bound the following quantities:
\begin{align*}
    \sup_{s \in [0, 1]}
    \left\| 
    \FisherBest^{-1/2} \nabla^{2} \cT \bigg( s\phi_{t} + (1-s) \FisherBest^{-1} \varphi_{t} \bigg) \FisherBest^{-1/2}
    \right\|_{2}, \quad 
    \left\| 
    \FisherBest^{1/2} \left( \phi_{t} - \FisherBest^{-1} \varphi_{t} \right)
    \right\|_{2}.
\end{align*}
For $u \in \Theta$, note that
\begin{align*}
    &\left\| \FisherBest^{-1/2} \nabla^{2} \cT (u) \FisherBest^{-1/2} \right\|_{2}
    = 
    \sup_{\widetilde{u} \in \bbR^{p} : \| \widetilde{u} \|_{2} = 1}
    6 \left| \left\langle \cT,  \left( \FisherBest^{-1/2} \widetilde{u} \right)^{\otimes 2} \otimes u \right\rangle \right| \\
    &= 
    \sup_{\widetilde{u} \in \bbR^{p} : \| \widetilde{u} \|_{2} = 1}
    \left| \left\langle \nabla^{3} \widetilde{L}_{t}(\thetaBest),  \left( \FisherBest^{-1/2} \widetilde{u} \right)^{\otimes 2} \otimes u \right\rangle \right|
    \overset{\eqref{def:3_4_smooth_equiv}}{\leq}
    \tau_{3, t, \bias} \left\| \FisherBest^{1/2} u \right\|_{2}.
\end{align*}
Also, for $u = (u_i)_{i \in [p]} \in \Theta$,
\begin{align*}
    u \mapsto \nabla^{2} \cT (u) = 6 \sum_{i = 1}^{p} u_{i} \cT_{i} 
\end{align*}
is a linear map. By the last two displays, we have
\begin{align*}
    &\sup_{s \in [0, 1]}
    \left\| 
    \FisherBest^{-1/2} \nabla^{2} \cT \bigg( s\phi_{t} + (1-s) \FisherBest^{-1} \varphi_{t} \bigg) \FisherBest^{-1/2}
    \right\|_{2} \\
    &=
    \sup_{s \in [0, 1]}
    \left\| 
    \FisherBest^{-1/2} \nabla^{2} \cT \left( (1-s) \FisherBest^{-1} \varphi_{t} \right) \FisherBest^{-1/2}
    +
    \FisherBest^{-1/2} \nabla^{2} \cT \left( s\phi_{t} \right) \FisherBest^{-1/2}
    \right\|_{2} \\
    &= 
    \left\| \FisherBest^{-1/2} \nabla^{2} \cT \left( \FisherBest^{-1} \varphi_{t} \right) \FisherBest^{-1/2} \right\|_{2}
    \vee
    \left\| \FisherBest^{-1/2} \nabla^{2} \cT \left( \phi_{t} \right) \FisherBest^{-1/2} \right\|_{2} \\
    &\leq 
    \tau_{3, t, \bias} \bigg[
    \left\|  \FisherBest^{-1/2} \varphi_{t}  \right\|_{2}  \vee   \left\|  \FisherBest^{1/2} \phi_{t} \right\|_{2}     
    \bigg]
    = 
    \tau_{3, t, \bias} \bigg[
    b_{n, t}  \vee   \left\|  \FisherBest^{1/2} \phi_{t} \right\|_{2}     
    \bigg] \\
    \overset{\eqref{eqn:penalized_bias_eq4}}&{\leq}
    \tau_{3, t, \bias} \left( 1 + \dfrac{1}{2} \tau_{3, t, \bias} b_{n, t} \right) b_{n, t}.
\end{align*}
Also, \eqref{eqn:penalized_bias_eq3_2} implies that
\begin{align*}
    \left\| 
    \FisherBest^{1/2} \left( \phi_{t} - \FisherBest^{-1} \varphi_{t} \right)
    \right\|_{2}
    \leq 
    \dfrac{1}{2} \tau_{3, t, \bias} b_{n, t}^{2}.
\end{align*}
By the last two displays, we complete the proof of the second inequality in \eqref{eqn:penalized_bias_eq6}.
\end{proof}

The quantities $(\widehat{\tau}_{3, t}, \widehat{\tau}_{4, t})$, $\widehat{\tau}_{3, t, r}$, $\tau_{3, t}^{\ast}$ and $(\tau_{3, t, \bias}, \tau_{4, t, \bias})$, which appear in the following lemma, are defined in \eqref{def:hat_tau_t_radius}, \eqref{def:tau_3tr}, \eqref{def:estimation_quantities} and \eqref{def:bias_quantity}, respectively.
By Lemma \ref{lemma:tech_smooth_tau_bound}, these quantities can be characterized by the smallest eigenvalue of $\FisherBest$ and $\FisherMAP$.

\begin{lemma} \label{lemma:penalized_bias_recursive}
    Suppose that (\textbf{A0}) and (\textbf{A1}) hold, and $N \geq 2$. Let $\alpha \in (0, 1]$ and $t \in \{1, 2, ..., T-1 \}$.
    Assume that 
    \begin{align} \label{assume:penalized_bias_recursive_C1}
        r_{\eff, t} &\leq C_1 M_n \sqrt{t^{-1} p_{\ast}}, \quad 
        \left\| \FisherBest[t]^{1/2} \left( \theta_0 - \thetaBest \right) \right\|_{2} \leq C_{2} M_n t^{\alpha} \sqrt{p_{\ast}}
    \end{align}
    on an event $\scrE$, where $C_1, C_2 > 0$ are constants.
    Also, assume that  
    \begin{align}
    \begin{aligned} \label{assume:penalized_bias_recursive_delta}
        \bigg[ \widehat{\tau}_{3, t} r_{\LA}^{2} \bigg] \vee \bigg[ \widehat{\tau}_{4, t} p^{2} \bigg] 
        \vee \bigg[ \tau_{3, t}^{\ast} r_{\eff, t} \bigg] \vee \bigg[ \epsilon_{n, t, \KL} \bigg]
        \leq \delta, 
    \end{aligned} \\
    \begin{aligned} \label{assume:penalized_bias_recursive_delta_2}
        \tau_{3, t+1, \bias} \big( M_n t^{1/2 + 2\alpha} \sqrt{p_{\ast}} \big) &\leq \delta, \\
        \big( \epsilon_{n, t, \KL} \vee \big[ \tau_{3, t}^{\ast} r_{\eff, t} \big] \big) t^{1/2 + \alpha} &\leq \delta, \\
        \left(  \tau_{4, t+1, \bias} \vee \tau_{3, t+1, \bias}^{2} \right) M_n^{2} t^{1/2 + 3\alpha} p_{\ast} &\leq \delta,
    \end{aligned}
    \end{align}
    on $\scrE$, where $\delta = \delta(C_1, C_2, \alpha) > 0$ is a small enough constant.
    Assume further that
    \begin{align} \label{assume:penalized_bias_recursive_tail}
        \widehat{\tau}_{3, t, r} \leq N e^{8p} \exp \left( \big[ \sqrt{p} + \sqrt{2 \log N} - 3 \big] r \right), \quad  \forall r > 4r_{\LA}
    \end{align}
    on $\scrE$.
    Then, on $\scrE_{\est, 1} \cap \scrE$, 
    \begin{align}
    \begin{aligned} \label{claim:penalized_bias_recursive}
        \left\| \FisherBest[t + 1]^{1/2} \left( \theta_0 - \thetaBest[t+1] \right) \right\|_{2} 
        \leq 
        K
        M_n (t^{1/2 - \alpha}+ t)^{\alpha} \sqrt{p_{\ast}},
    \end{aligned}
    \end{align}
    where $K = K(C_1, C_2, \alpha)$.
\end{lemma}
\begin{proof}
In this proof, we work on the event $\scrE \cap \scrE_{\est, 1}$ without explicitly referring to it, and assume that $\delta = \delta(C_1, C_2, \alpha)$ is small enough constant.

By Lemma \ref{lemma:penalized_bias}, if $\tau_{3, t+1, \bias} b_{n, t+1} \leq 1/16$, then
\begin{align} \label{eqn:penalized_bias_recursive_eq1}
        \left\| \FisherBest[t+1]^{1/2} \left( \theta_0 - \thetaBest[t+1] \right) \right\|_{2} \leq {\rm (a)} + {\rm (b)} + {\rm (c)},
\end{align}
where
\begin{align*}
    {\rm (a)} &= \dfrac{1}{6} \tau_{4, t+1, \bias} \bigg( 1 - 4\tau_{3, t+1, \bias} \: b_{n, t+1} \bigg)^{-1}
    \bigg( 1 + \dfrac{1}{2}\tau_{3, t+1, \bias} \: b_{n, t+1} \bigg)^3 b_{n, t+1}^3, \\
    {\rm (b)} &= \dfrac{1}{2} \tau_{3, t+1, \bias}^2\bigg( 1 - 4\tau_{3, t+1, \bias} \: b_{n, t+1} \bigg)^{-1}
    \bigg( 1 + \dfrac{1}{2}\tau_{3, t+1, \bias}b_{n, t+1} \bigg) b_{n, t+1}^3, \\
    {\rm (c)} &= \bigg( 1 + \dfrac{1}{2}\tau_{3, t+1, \bias}b_{n, t+1} \bigg) b_{n, t+1}.
\end{align*}
We will obtain an upper bound of $b_{n, t+1}$ first, and then bound ${\rm (a)} + {\rm (b)} + {\rm (c)}$.

By $N \geq 2$, \eqref{assume:penalized_bias_recursive_tail} and  \eqref{assume:penalized_bias_recursive_delta} with sufficiently small $\delta$, we can utilize the results in Theorem \ref{thm:penalized_estimation} and Corollary \ref{coro:VB_KL_Delta}.
Hence, we have
\begin{align}
\begin{aligned} \label{eqn:penalized_bias_recursive_eq1_2}
    &\left\| \FisherBest^{1/2} \big( \thetaBest - \thetaMAP \big) \right\|_{2} \overset{ \text{Theorem \ref{thm:penalized_estimation}} }{\leq} 4r_{\eff, t}, \\
    &\left\| \FisherBest^{-1/2} \FisherMAP \FisherBest^{-1/2} - \bI_{p} \right\|_{2} 
    \overset{ \text{Lemma \ref{lemma:tech_Fisher_smooth}} }{\leq} 4\tau_{3, t}^{\ast} r_{\eff, t}, \\
    &\left\| \FisherMAP^{-1/2} \bOmega_{t} \FisherMAP^{-1/2} - \bI_{p} \right\|_{2} \leq \left\| \FisherMAP^{-1/2} \bOmega_{t} \FisherMAP^{-1/2} - \bI_{p} \right\|_{\rm F} 
    \overset{ \text{Corollary \ref{coro:VB_KL_Delta}} }{\leq} c_1 \epsilon_{n, t, \KL}, \\
    &\left\| \bOmega_{t}^{1/2} \big( \thetaMAP - \mu_{t} \big) \right\|_{2} 
    \overset{ \text{Corollary \ref{coro:VB_KL_Delta}} }{\leq} c_1 \epsilon_{n, t, \KL}, 
\end{aligned}
\end{align}
for some universal constant $c_1 > 0$. 
Note that
\begin{align*}
    b_{n, t+1} 
        &= \left\| \FisherBest[t+1]^{-1/2} \bOmega_{t} \left( \theta_0 - \mu_{t} \right) \right\|_{2}
        \leq \left\| \bOmega_{t}^{-1/2} \bOmega_{t} \left( \theta_0 - \mu_{t} \right) \right\|_{2}
        = \left\| \bOmega_{t}^{1/2} \left( \theta_0 - \mu_{t} \right) \right\|_{2} \\
    &\leq 
        \left\| \bOmega_{t}^{1/2} \left( \theta_0 - \thetaBest \right) \right\|_{2} 
        + \left\| \bOmega_{t}^{1/2} \left( \thetaBest - \thetaMAP \right) \right\|_{2} 
        + \left\| \bOmega_{t}^{1/2} \left( \thetaMAP - \mu_{t} \right) \right\|_{2} \\
    &\leq
        \left\| \bOmega_{t}^{1/2} \FisherMAP^{-1/2} \right\|_{2}
        \left\| \FisherMAP^{1/2} \FisherBest^{-1/2}  \right\|_{2}
        \bigg(
        \left\| \FisherBest^{1/2} \left( \theta_0 - \thetaBest \right) \right\|_{2} 
        + \left\| \FisherBest^{1/2} \left( \thetaBest - \thetaMAP \right) \right\|_{2} \bigg) \\
        &\qquad + \left\| \bOmega_{t}^{1/2} \left( \thetaMAP - \mu_{t} \right) \right\|_{2} \\
    &\leq
        \left( 1 + \left\| \FisherMAP^{-1/2} \bOmega_{t} \FisherMAP^{-1/2} - \bI_{p} \right\|_{2} \right)^{1/2}
        \left( 1 + \left\| \FisherBest^{-1/2} \FisherMAP \FisherBest^{-1/2} - \bI_{p}  \right\|_{2}
        \right)^{1/2} \\
        &\qquad \times
        \bigg(
        \left\| \FisherBest^{1/2} \left( \theta_0 - \thetaBest \right) \right\|_{2} 
        + \left\| \FisherBest^{1/2} \left( \thetaBest - \thetaMAP \right) \right\|_{2} \bigg) 
        + \left\| \bOmega_{t}^{1/2} \left( \thetaMAP - \mu_{t} \right) \right\|_{2} \\
    \overset{\substack{\eqref{assume:penalized_bias_recursive_C1} \\ \eqref{eqn:penalized_bias_recursive_eq1_2}}}&{\leq}  
        \sqrt{\left( 1 + c_1 \epsilon_{n, t, \KL} \right)
        \left( 1 + 4\tau_{3, t}^{\ast} r_{\eff, t} \right)}
        \big( C_2 M_n t^{\alpha} \sqrt{p_{\ast}} + 4C_{1} M_n \sqrt{t^{-1} p_{\ast}} \big) 
        + c_1 \epsilon_{n, t, \KL} \\
    &\leq  
        \left( 1 + c_1 \epsilon_{n, t, \KL} \right)
        \left( 1 + 4\tau_{3, t}^{\ast} r_{\eff, t} \right)
        \big( C_2 M_n t^{\alpha} \sqrt{p_{\ast}} + 4C_{1} M_n \sqrt{t^{-1} p_{\ast}} + c_1 \epsilon_{n, t, \KL} \big).
\end{align*}
It follows that
\begin{align*}
    b_{n, t+1} 
    &\leq
    \left( 1 + c_1 \epsilon_{n, t, \KL} \right)
    \left( 1 + 4\tau_{3, t}^{\ast} r_{\eff, t} \right)
    \big( 
    C_2 M_n t^{\alpha} \sqrt{p_{\ast}} 
    + 4C_{1} M_n \sqrt{t^{-1} p_{\ast}} 
    + c_1 \epsilon_{n, t, \KL} \big) \\
    &\leq
    \left( 1 + c_1 \delta \right) \left( 1 + 4 \delta \right) 
    \big( C_2 M_n t^{\alpha} \sqrt{p_{\ast}} + 4C_{1} M_n \sqrt{t^{-1} p_{\ast}} + c_1 \delta \big) \\
    &\lesssim
    M_n t^{\alpha} \sqrt{p_{\ast}},
\end{align*}
where the second inequality holds by \eqref{assume:penalized_bias_recursive_delta}.
Consequently, we have
\begin{align} \label{eqn:penalized_bias_recursive_eq1_3_2}
    \tau_{3, t+1, \bias} b_{n, t+1} 
    \overset{\eqref{assume:penalized_bias_recursive_delta_2}}{\leq}
    1/16,
\end{align}
which implies that \eqref{eqn:penalized_bias_recursive_eq1} holds.

Next, we will obtain an upper bound of ${\rm (a)} + {\rm (b)} + {\rm (c)}$.
Since $(1-x)^{-1} \leq 1 + 2x$ for $x \leq 1/2$ and $(1+x)^{3} \leq 1 + 4x$ for $x \in [0, 0.3]$,
\begin{align*}
    {\rm (a)} 
    &= 
    \dfrac{1}{6} \tau_{4, t+1, \bias} \bigg( 1 - 4\tau_{3, t+1, \bias} \: b_{n, t+1} \bigg)^{-1} \bigg( 1 + \dfrac{1}{2}\tau_{3, t+1, \bias} \: b_{n, t+1} \bigg)^3 b_{n, t+1}^3 \\
    &\leq 
    \dfrac{1}{6} \tau_{4, t+1, \bias} \bigg( 1 + 8\tau_{3, t+1, \bias} \: b_{n, t+1} \bigg) \bigg( 1 + 2 \tau_{3, t+1, \bias} \: b_{n, t+1} \bigg) b_{n, t+1}^3 \\
    \overset{\eqref{eqn:penalized_bias_recursive_eq1_3_2}}&{\leq}
    \dfrac{1}{6} \tau_{4, t+1, \bias} \bigg( 1 + \dfrac{1}{2} \bigg) \bigg( 1 + \dfrac{1}{8} \bigg) b_{n, t+1}^3
    \leq 
    \tau_{4, t+1, \bias} b_{n, t+1}^3.
\end{align*}
Similarly, we have
\begin{align*}
    {\rm (b)} 
    &= 
    \dfrac{1}{2} \tau_{3, t+1, \bias}^2\bigg( 1 - 4\tau_{3, t+1, \bias} \: b_{n, t+1} \bigg)^{-1} \bigg( 1 + \dfrac{1}{2}\tau_{3, t+1, \bias}b_{n, t+1} \bigg) b_{n, t+1}^3 \\
    &\leq 
    \dfrac{1}{2} \tau_{3, t+1, \bias}^2\bigg( 1 + 8\tau_{3, t+1, \bias} \: b_{n, t+1} \bigg) \bigg( 1 + \dfrac{1}{2}\tau_{3, t+1, \bias}b_{n, t+1} \bigg) b_{n, t+1}^3 \\
    \overset{\eqref{eqn:penalized_bias_recursive_eq1_3_2}}&{\leq}
    \dfrac{1}{2} \tau_{3, t+1, \bias}^2 \bigg( 1 + \dfrac{1}{2} \bigg) \bigg( 1 + \dfrac{1}{32} \bigg) b_{n, t+1}^3 
    \leq 
    \tau_{3, t+1, \bias}^{2} b_{n, t+1}^3.
\end{align*}
By \eqref{assume:penalized_bias_recursive_delta}, \eqref{assume:penalized_bias_recursive_delta_2}, and the inequality $b_{n, t+1} \lesssim M_n t^{\alpha} \sqrt{p_{\ast}}$, we have
\begin{align}
\begin{aligned} \label{eqn:penalized_bias_recursive_eq1_4}
    \big( \tau_{3, t+1, \bias} b_{n, t+1} \big) \vee
    \big( 2\tau_{4, t+1, \bias} b_{n, t+1}^2 \big) \vee
    \big( 2\tau_{3, t+1, \bias}^{2} b_{n, t+1}^2 \big) 
    &< 1/6, \\
    \big( 4 c_1 \epsilon_{n, t, \KL} \big) \vee 
    \big( 8\tau_{3, t}^{\ast} r_{\eff, t} \big) \vee
    \big( 8C_1 K^{-1} \big)
    &< 1/6,
\end{aligned}
\end{align}
for some large enough constant $K = K(C_1, C_2, \alpha) > 0$. (Here, $K$ can be chosen depending only on $C_1$, but later, it will be taken as a larger constant that depends on $C_2$ and $\alpha$.) Also, we can further bound $b_{n, t+1}$ as follows:
\begin{align} 
\begin{aligned} \label{eqn:penalized_bias_recursive_eq1_3}
    &b_{n, t+1} \\
    &\leq 
    \left( 1 + c_1 \epsilon_{n, t, \KL} \right)
    \left( 1 + 4\tau_{3, t}^{\ast} r_{\eff, t} \right)
    \big( C_2 M_n t^{\alpha} \sqrt{p_{\ast}} + 4C_{1} M_n \sqrt{t^{-1} p_{\ast}} + c_1 \epsilon_{n, t, \KL} \big) \\
    &\leq
    \left( 1 + c_1 \epsilon_{n, t, \KL} \right)
    \left( 1 + 4\tau_{3, t}^{\ast} r_{\eff, t} \right)
    \big( K M_n t^{\alpha} \sqrt{p_{\ast}} + 4C_{1} M_n \sqrt{t^{-1} p_{\ast}} + c_1 \epsilon_{n, t, \KL} \big) \\
    &= 
    \left( 1 + c_1 \epsilon_{n, t, \KL} \right)
    \left( 1 + 4\tau_{3, t}^{\ast} r_{\eff, t} \right) \\
    &\qquad \times
    \big( 1 + 4 C_1 K^{-1} t^{-1/2-\alpha} + K^{-1} M_n^{-1} t^{-\alpha} p_{\ast}^{-1/2} c_1 \epsilon_{n, t, \KL} \big) 
    K M_n t^{\alpha} \sqrt{p_{\ast}} \\
    &\leq
    \left( 1 + c_1 \epsilon_{n, t, \KL} \right)
    \left( 1 + 4\tau_{3, t}^{\ast} r_{\eff, t} \right)
    \big( 1 + 4 C_1 K^{-1} t^{-1/2-\alpha} + c_1 \epsilon_{n, t, \KL} \big) 
    K M_n t^{\alpha} \sqrt{p_{\ast}},
\end{aligned}
\end{align}
where the second and last inequalities hold by taking sufficiently large $K$.
Let 
\begin{align*}
    \rho_{n, t} 
    &= \tau_{3, t+1, \bias}b_{n, t+1} 
    + 2\tau_{4, t+1, \bias} b_{n, t+1}^2 
    + 2\tau_{3, t+1, \bias}^{2} b_{n, t+1}^2 \\
    &\qquad + 4c_1 \epsilon_{n, t, \KL} 
    + 8\tau_{3, t}^{\ast} r_{\eff, t} 
    + 8C_1 K^{-1} t^{-1/2-\alpha} 
    \overset{\eqref{eqn:penalized_bias_recursive_eq1_4}}{<}
    1.
\end{align*}
Therefore, from \eqref{eqn:penalized_bias_recursive_eq1} and the bounds for (a) and (b) above, we have
\begin{align*}
&\left\| \FisherBest[t+1]^{1/2} \left( \theta_0 - \thetaBest[t+1] \right) \right\|_{2}  
\leq
\tau_{4, t+1, \bias} b_{n, t+1}^{3}
+ \tau_{3, t+1, \bias}^2 b_{n, t+1}^{3}
+ \left( 1 + \dfrac{1}{2}\tau_{3, t+1, \bias}b_{n, t+1} \right) b_{n, t+1} \\
&=
    \bigg( 1 + \dfrac{1}{2}\tau_{3, t+1, \bias}b_{n, t+1} + \tau_{4, t+1, \bias} b_{n, t+1}^2 + \tau_{3, t+1, \bias}^{2} b_{n, t+1}^2 \bigg) b_{n, t+1} \\
\overset{\eqref{eqn:penalized_bias_recursive_eq1_3}}&{\leq}
    \bigg( 1 + \dfrac{1}{2} \tau_{3, t+1, \bias}b_{n, t+1} + \tau_{4, t+1, \bias} b_{n, t+1}^2 + \tau_{3, t+1, \bias}^{2} b_{n, t+1}^2 \bigg)
    \left( 1 + c_1 \epsilon_{n, t, \KL} \right) \left( 1 + 4\tau_{3, t}^{\ast} r_{\eff, t} \right)
    \\
    &\qquad \times 
    \big( 1 + 4 C_1 K^{-1} t^{-1/2-\alpha} + c_1 \epsilon_{n, t, \KL} \big) 
    K M_n t^{\alpha} \sqrt{p_{\ast}} \\
&\leq
    \big( 1 
    + \tau_{3, t+1, \bias}b_{n, t+1} 
    + 2\tau_{4, t+1, \bias} b_{n, t+1}^2 
    + 2\tau_{3, t+1, \bias}^{2} b_{n, t+1}^2
    + 4c_1 \epsilon_{n, t, \KL} 
    + 8\tau_{3, t}^{\ast} r_{\eff, t} 
    + 8C_1 K^{-1} t^{-1/2-\alpha}
    \big) \\
    &\qquad \times
    K M_n t^{\alpha} \sqrt{p_{\ast}}, \\
&=
    \big( 1 + \rho_{n, t} \big) 
    K M_n t^{\alpha} \sqrt{p_{\ast}},
\end{align*}
where the third inequality holds because 
\begin{align*}
    (1 + x_{1})(1 + x_{2})(1 + x_{3})(1 + x_{4}) 
    \leq
    e^{x_1 + x_2 + x_3 + x_4}
    \leq
    1 + 2x_{1} + 2x_{2} + 2x_{3} + 2x_{4}
\end{align*}
for $x_{1}, x_{2},  x_{3}, x_{4} \in \bbR_{+}$ with $x_{1} + x_{2} + x_{3} + x_{4} \in (0, 1)$.

Suppose that the following inequality holds:
\begin{align} \label{eqn:claim_proof_penalized_bias_recursive}
    2^{1/\alpha} \cdot \rho_{n, t} \cdot t^{1/2+\alpha}
    \leq 
    1.
\end{align}
Since
\begin{align*}
   (1 + x)^{\omega} \leq 1 + 2^{\omega}x 
\end{align*}
for $x \in (0, 1)$ and $\omega \geq 1$, we have
\begin{align*} 
    \left( 1 + \rho_{n, t} \right)^{1/\alpha} 
    \leq 
    1 + 2^{1/\alpha} \rho_{n, t} 
    \overset{\eqref{eqn:claim_proof_penalized_bias_recursive}}{\leq}
    1 + t^{-1/2 - \alpha}. 
\end{align*}
It follows that
\begin{align*}
    \big( 1 + \rho_{n, t} \big) K M_n t^{\alpha} \sqrt{p_{\ast}} 
    &\leq
    \big( 1 + t^{-1/2 - \alpha} \big)^{\alpha} K M_n t^{\alpha} \sqrt{p_{\ast}} \\
    &= 
    \left( t \big[ 1 + t^{-1/2 - \alpha} \big] \right)^{\alpha} K M_n \sqrt{p_{\ast}} 
    =
    \big( t + t^{1/2 - \alpha} \big)^{\alpha} K M_n \sqrt{p_{\ast}},
\end{align*}
which completes the proof of \eqref{claim:penalized_bias_recursive}.
Therefore, we only need to prove \eqref{eqn:claim_proof_penalized_bias_recursive}.

By \eqref{eqn:penalized_bias_recursive_eq1_3} and \eqref{assume:penalized_bias_recursive_delta}, with a large enough $K$ and small enough $\delta$, we have
\begin{align*}
    b_{n, t+1} \leq 2K M_n t^{\alpha} \sqrt{p_{\ast}}.
\end{align*}
It follows that
\begin{align*}
   2^{1/\alpha} \big( \tau_{3, t+1, \bias} b_{n, t+1} \big) t^{1/2+\alpha}
   \leq c_2 \big( M_n t^{1/2 + 2\alpha} \sqrt{p_{\ast}} \big) \tau_{3, t+1, \bias}  
   \overset{\eqref{assume:penalized_bias_recursive_delta_2}}&{\leq} 
   1/6,
   \\
   2^{1/\alpha} \big( 2\tau_{4, t+1, \bias} b_{n, t+1}^2 \big) t^{1/2+\alpha}
   \leq c_2 \big( M_n^{2} t^{1/2 + 3\alpha} p_{\ast} \big) \tau_{4, t+1, \bias}
   \overset{\eqref{assume:penalized_bias_recursive_delta_2}}&{\leq}
   1/6,
   \\
   2^{1/\alpha} \big( 2\tau_{3, t+1, \bias}^{2} b_{n, t+1}^2 \big) t^{1/2+\alpha}
   \leq c_2 \big( M_n^{2} t^{1/2 + 3\alpha} p_{\ast} \big) \tau_{3, t+1, \bias}^{2}
   \overset{\eqref{assume:penalized_bias_recursive_delta_2}}&{\leq}
   1/6,
\end{align*}
where $c_2 = c_2(\alpha, K) > 0$.
Also, 
\begin{align*}
   2^{1/\alpha} \big( 4c_1 \epsilon_{n, t, \KL} \big) t^{1/2+\alpha}
   \leq c_3 t^{1/2+\alpha} \epsilon_{n, t, \KL}  
   \overset{\eqref{assume:penalized_bias_recursive_delta_2}}&{\leq} 
   1/6,
   \\
   2^{1/\alpha} \big( 8\tau_{3, t}^{\ast} r_{\eff, t} \big) t^{1/2+\alpha}
   \leq c_3 t^{1/2+\alpha} \big( \tau_{3, t}^{\ast} r_{\eff, t} \big)
   \overset{\eqref{assume:penalized_bias_recursive_delta_2}}&{\leq}
   1/6,
   \\
   2^{1/\alpha} \big( 8C_1 K^{-1} t^{-1/2-\alpha} \big) t^{1/2+\alpha}
   = 2^{1/\alpha} \big( 8C_1 K^{-1} \big) 
   &\leq
   1/6,
\end{align*}
where $c_3 = c_3(\alpha, K)$, and the last inequality holds by taking a sufficiently large $K$.
By the last two displays, we complete the proof of \eqref{eqn:claim_proof_penalized_bias_recursive}. 
\end{proof}

\begin{remark} 
    Note that the constant $K$ in \eqref{claim:penalized_bias_recursive} can be chosen as 
    $$
        K =  \big[ (2^{1/\alpha}) 48 C_1 \big] \vee C_2 \vee 1.
    $$
\end{remark}

\begin{lemma} \label{lemma:iterative_theta_consistency}
    Suppose that (\textbf{A0})-(\textbf{A2}) hold. Let $\alpha \in [1/2, 1]$ and $t \in \{1, 2, ..., T-1 \}$. Also, on an event $\scrE$, assume that 
    \begin{align} \label{assume:iterative_theta_consistency}
    \begin{aligned}
        \lambda_{\min}\left( \FisherBest \right) \wedge \lambda_{\min}\left( \FisherMAP \right) &\geq C_3 n t, \\
        \lambda_{\max}\left( \FisherBest \right) \vee \lambda_{\max}\left( \FisherMAP \right) &\leq C_4 n t, \\ 
        \left\| \FisherBest[t]^{1/2} \left( \theta_0 - \thetaBest[t] \right)\right\|_{2} &\leq C_5 M_n t^{\alpha} \sqrt{p_{\ast}}
    \end{aligned}
    \end{align}        
    for some constants $C_3$, $C_4$ and $C_5$ with
    \begin{align*}
        C_3 \in [K_{\min}/2, K_{\min}], \quad 
        C_4 \in [K_{\max}, 2K_{\max}], \quad 
        C_5 \geq (192 C_{3}^{-1/2} K_{\max}^{1/2}) \vee 1.
    \end{align*}
    Assume further that 
    \begin{align}
    \begin{aligned} \label{assume:iterative_theta_consistency_sample}
        &n \geq C M_n^{2} t^{4\alpha - 2} p_{\ast}^2
    \end{aligned}        
    \end{align}
    for a large enough constant $C = C(K_{\min}, K_{\max}, \alpha, C_5)$.
    Then, on $\scrE_{\est, 1} \cap \scrE$, 
    \begin{align}
    \begin{aligned} \label{eqn:iterative_theta_consistency_result1}
        \left\| \FisherBest^{1/2} \big( \thetaMAP - \thetaBest \big) \right\|_{2} 
        &\leq K M_n \sqrt{t^{-1} p_{\ast}}, \\
        \tau_{3, t}^{\ast} \vee \widehat{\tau}_{3, t} 
        &\leq K t^{-3/2} n^{-1/2}, \\ 
        \widehat{\tau}_{4, t} 
        &\leq K t^{-2} n^{-1}, \\
        \Delta_{t} \vee \epsilon_{n, t, \KL} 
        &\leq K t^{-1} n^{-1/2} p_{\ast},  \\
        \thetaMAP[t], \ \thetaBest[t] &\in \Theta \big( \theta_0, \bI_p, 1/2 \big),
    \end{aligned}
    \end{align}
    where $K = K(K_{\min}, K_{\max})$ is a large enough constant, and
    \begin{align}
    \begin{aligned} \label{eqn:iterative_theta_consistency_result2}
        \lambda_{\min}\left( \FisherBest[t+1] \right) \wedge \lambda_{\min}\left( \FisherMAP[t+1] \right)
        &\geq 
        C_3 (1-\Delta_t) n(t+1), \\ 
        \lambda_{\max}\left( \FisherBest[t+1] \right) \vee \lambda_{\max}\left( \FisherMAP[t+1] \right)
        &\leq 
        C_4 (1+\Delta_t) n(t+1), \\ 
        \thetaMAP[t+1], \ \thetaBest[t+1] &\in \Theta \big( \theta_0, \bI_p, 1/2 \big), \\
        \left\| \FisherBest[t+1]^{1/2} \left( \theta_0 - \thetaBest[t+1] \right)\right\|_{2} 
        &\leq 
        C_5 M_n (t+1)^{\alpha} \sqrt{p_{\ast}}.        
    \end{aligned}
    \end{align}    
    Furthermore, if $C_3 (1 - \Delta_{t}) \geq K_{\min}/2$ and $C_4 (1 + \Delta_{t}) \leq 2K_{\max}$, then
    \begin{align}
    \begin{aligned} \label{eqn:iterative_theta_consistency_result3}
        \left\| \FisherBest[t+1]^{1/2} \big( \thetaMAP[t+1] - \thetaBest[t+1] \big) \right\|_{2}
        &\leq K M_n \sqrt{(t+1)^{-1} p_{\ast}}, \\
        \tau_{3, t+1}^{\ast} \vee \widehat{\tau}_{3, t+1} &\leq K (t+1)^{-3/2} n^{-1/2}, \\
        \widehat{\tau}_{4, t+1} &\leq  K (t+1)^{-2} n^{-1}, \\
        \Delta_{t+1} \vee \epsilon_{n, t+1, \KL} &\leq K (t+1)^{-1} n^{-1/2} p_{\ast}
    \end{aligned}
    \end{align}
    on $\scrE_{\est, 1} \cap \scrE$.
\end{lemma}
\begin{proof}
In this proof, we work on the event $\scrE \cap \scrE_{\est, 1}$ without explicitly referring to it.
Also, the dependency on $C_3$ and $C_4$ is not explicitly expressed, as it can be replaced by the dependence on $K_{\min}$ and $K_{\max}$.
The constant $C = C(K_{\min}, K_{\max}, \alpha, C_5)$ in \eqref{assume:iterative_theta_consistency_sample} will be assumed to be large enough if necessary.
The proof is divided into several steps.

\noindent \textbf{Step 1: } $\tau_{3, t}^{\ast}$, $r_{\eff, t}$ \\
By \eqref{assume:iterative_theta_consistency}, we have
\begin{align}
\begin{aligned} \label{eqn:iterative_theta_eq1}
    \left\| \theta_0 - \thetaBest \right\|_{2} 
    &\leq 
    \lambda_{\min}^{-1/2}(\FisherBest) C_5 M_n t^{\alpha} \sqrt{p_{\ast}}
    \overset{\eqref{assume:iterative_theta_consistency}}{\leq} 
    \left( C_3 n t \right)^{-1/2} C_5 M_n t^{\alpha} \sqrt{p_{\ast}} \\
    &= 
    \big( C_{3}^{-1/2} C_5 \big) M_n t^{\alpha - 1/2} p_{\ast}^{1/2} n^{-1/2}
    \overset{\eqref{assume:iterative_theta_consistency_sample}}{\leq} 1/8.
\end{aligned}
\end{align}
Also,
\begin{align} 
\begin{aligned} \label{eqn:iterative_theta_eq1_1_2_1}
    r_{\eff, t} 
    \overset{\substack{\eqref{eqn:iterative_theta_eq1} \\ \text{Lemma \ref{lemma:radius_upper_bound}}}}&{\leq} 
    M_n  \lambda_{\min}^{-1/2}(\FisherBest) \lambda_{\max}^{1/2}(\Fisher[t]{\thetaBest})  p_{\ast}^{1/2} \\
    \overset{\substack{\eqref{assume:iterative_theta_consistency} \\ \eqref{eqn:iterative_theta_eq1}, (\textbf{A2})}}&{\leq} 
    M_n \left( C_3 n t \right)^{-1/2} ( K_{\max} n )^{1/2} \sqrt{p_{\ast}} 
    =  \big( C_3^{-1/2} K_{\max}^{1/2} \big) M_n \sqrt{t^{-1} p_{\ast}}.    
\end{aligned}
\end{align}
For $\theta \in \Theta (\thetaBest, \FisherBest, 4r_{\eff, t})$, we have
\begin{align}
\begin{aligned} \label{eqn:iterative_theta_eq1_1_2}
    \left\| \theta - \thetaBest \right\|_{2} 
    &\leq \lambda_{\min}^{-1/2}(\FisherBest) 4r_{\eff, t}
    \overset{\eqref{assume:iterative_theta_consistency}}{\leq} \left( C_3 n t \right)^{-1/2} 4r_{\eff, t} \\
    \overset{\eqref{eqn:iterative_theta_eq1_1_2_1}}&{\leq} 
    (4 C_3^{-1} K_{\max}^{1/2}) M_n p_{\ast}^{1/2} n^{-1/2}t^{-1} 
    \overset{\eqref{assume:iterative_theta_consistency_sample}}{\leq} 1/8.    
\end{aligned}
\end{align}
By \eqref{eqn:iterative_theta_eq1} and \eqref{eqn:iterative_theta_eq1_1_2}, 
we have
\begin{align*}
    \Theta (\thetaBest, \FisherBest, 4r_{\eff, t}) \subset \Theta (\theta_0, \bI_{p}, 1/2).
\end{align*}
Consequently, we have
\begin{align}
\begin{aligned} \label{eqn:1}
    \tau_{3, t}^{\ast} 
    \overset{ \substack{ (\textbf{A2}) \\ \text{Lemma \ref{lemma:tech_smooth_tau_bound}}  } }&{\leq}
    \left( K_{\max} n \right) \lambda_{\min}^{-3/2}(\FisherBest)
    \overset{ \eqref{assume:iterative_theta_consistency} }{\leq}
    \left( K_{\max} n \right) \left( C_3 n t \right)^{-3/2} \\
    &= \big( C_{3}^{-3/2} K_{\max} \big) t^{-3/2} n^{-1/2}. 
\end{aligned}    
\end{align}

\noindent \textbf{Step 2: } $\widehat{\tau}_{3, t}$, $\widehat{\tau}_{4, t}$ \\
By \eqref{eqn:iterative_theta_eq1_1_2_1} and \eqref{eqn:1}, we have
\begin{align}
\begin{aligned} \label{eqn:iterative_theta_eq1_2}
   \tau_{3, t}^{\ast} r_{\eff, t} 
   &\leq \left[ \dfrac{K_{\max}}{C_{3}^{3/2}}  t^{-3/2} n^{-1/2} \right] 
   \left[ \dfrac{K_{\max}^{1/2}}{C_3^{1/2}} M_n  \sqrt{t^{-1} p_{\ast}} \right]
   = \dfrac{ K_{\max}^{3/2}}{C_{3}^{2}} t^{-2} n^{-1/2} M_n p_{\ast}^{1/2} \\
   \overset{\eqref{assume:iterative_theta_consistency_sample}}&{\leq} 1/32.
\end{aligned}
\end{align}
Therefore, 
\begin{align} \label{eqn:3}
    \thetaMAP \in \Theta (\thetaBest, \FisherBest, 4r_{\eff, t})
\end{align}
by Theorem \ref{thm:penalized_estimation}.
It follows that
\begin{align*}
    \big\| \thetaMAP - \theta_0 \big\|_{2} 
    \leq \big\| \thetaMAP - \thetaBest \big\|_{2} + \big\| \thetaBest - \theta_0 \big\|_{2}
    \overset{ \substack{ \eqref{eqn:iterative_theta_eq1} \\ \eqref{eqn:iterative_theta_eq1_1_2} } }{\leq} 1/8 + 1/8 = 1/4.
\end{align*}
Also,
\begin{align} \label{eqn:2}
    r_{\LA} = 2\sqrt{p} + \sqrt{2 \log N} \leq 2\sqrt{p} + 2\sqrt{p_{\ast}} \leq 4\sqrt{p_{\ast}}.
\end{align}
For $\theta \in \Theta (\thetaMAP, \FisherMAP, 4r_{\LA})$, note that
\begin{align*}
    \big\| \theta - \thetaMAP \big\|_{2} 
    \overset{\eqref{assume:iterative_theta_consistency}}{\leq} \left( C_3 n t \right)^{-1/2} 4r_{\LA} 
    \leq \left( C_3 n t \right)^{-1/2} 16 p_{\ast}^{1/2}
    = \dfrac{16}{\sqrt{C_3}} \sqrt{\dfrac{p_{\ast}}{nt}}
    \overset{\eqref{assume:iterative_theta_consistency_sample}}{\leq} 1/4.
\end{align*}
It follows that $\Theta (\thetaMAP, \FisherMAP, 4r_{\LA}) \subset \Theta (\theta_0, \bI_{p}, 1/2)$.
By Lemma \ref{lemma:tech_smooth_tau_bound} and (\textbf{A2}), therefore, we have
\begin{align}
\begin{aligned} \label{eqn:iterative_theta_eq2}
    &\widehat{\tau}_{3, t}
    \leq \left( K_{\max} n \right) \lambda_{\min}^{-3/2}(\FisherMAP)
    \leq \left( K_{\max} n \right) \left( C_3 n t \right)^{-3/2}
    = \big( C_3^{-3/2} K_{\max} \big) t^{-3/2} n^{-1/2}, \\
    &\widehat{\tau}_{4, t}
    \leq \left( K_{\max} n \right) \lambda_{\min}^{-2}(\FisherMAP)
    \leq \left( K_{\max} n \right) \left( C_3 n t \right)^{-2}
    = \big( C_3^{-2} K_{\max} \big) t^{-2} n^{-1}. \\     
\end{aligned}  
\end{align}

\noindent \textbf{Step 3: } $\Delta_{t}, \epsilon_{n, t, \KL}$ \\
For $r > 4 r_{\LA}$ and $\theta \in \Theta (\thetaMAP, \FisherMAP, r)$, note that
\begin{align*}
    \big\| \theta - \theta_0 \big\|_{2} 
    \leq 
    \big\| \theta - \thetaMAP \big\|_{2} + \big\| \thetaMAP - \theta_0 \big\|_{2} 
    \overset{ \eqref{assume:iterative_theta_consistency} }{\leq}
    (C_3 nt)^{-1/2} r + \dfrac{1}{2}
    \overset{\eqref{assume:iterative_theta_consistency_sample}}{\leq} r + \dfrac{1}{2}.
\end{align*}
It follows that 
\begin{align*}
    \Theta (\thetaMAP, \FisherMAP, r) \subseteq \Theta (\theta_0, \bI_{p}, 1/2 + r).
\end{align*}
By \eqref{assume:iterative_theta_consistency}, Lemma \ref{lemma:tech_smooth_tau_bound} and \eqref{assume:A2_3} in (\textbf{A2}), the last display implies
\begin{align}
\begin{aligned} \label{eqn:8}
    \widehat{\tau}_{3, t, r}
    &\leq
    \left( C_3 n t \right)^{-3/2}
    K_{\max} N e^{8p} \exp\bigg( \left[ \sqrt{p} + \sqrt{2 \log N} - 3 \right] r \bigg) \\
    \overset{ \eqref{assume:iterative_theta_consistency_sample} }&{\leq}
    N e^{8p} \exp\bigg( \left[ \sqrt{p} + \sqrt{2 \log N} - 3 \right] r \bigg).
\end{aligned}
\end{align}
Also, 
\begin{align}
\begin{aligned} \label{eqn:9}
\widehat{\tau}_{3, t} r_{\LA}^{2} 
\overset{ \substack{ \eqref{eqn:2}, \eqref{eqn:iterative_theta_eq2}  }  }&{\leq}
\big( C_3^{-3/2} K_{\max} \big) t^{-3/2} n^{-1/2} (16 p_{\ast})
= 
\big( 16 C_3^{-3/2} K_{\max} \big) t^{-3/2} n^{-1/2} p_{\ast} \\
\overset{\eqref{assume:iterative_theta_consistency_sample}}&{\leq} 1/8,
\\
\widehat{\tau}_{4, t} p^{2} 
\overset{ \substack{ \eqref{eqn:iterative_theta_eq2}  }  }&{\leq}
\left( C_3^{-2} K_{\max} \right) t^{-2} n^{-1} p^{2}
\overset{\eqref{assume:iterative_theta_consistency_sample}}{\leq} 1/8.    
\end{aligned}
\end{align}
By the last two displays, we can apply Theorem \ref{thm:LA_KL}.
Recall that
\begin{align*}
\epsilon_{n, t, \KL}^{2} 
=       
\left[ \widehat{\tau}_{4, t} + \widehat{\tau}_{3, t}^2 \right] p^2 
+
\widehat{\tau}_{3, t}^{3} \log^{3} N
+
e^{- 7\log N}.
\end{align*}
By \eqref{eqn:iterative_theta_eq2}, $\epsilon_{n, t, \KL}^{2}$ is bounded by
\begin{align*}
    &\dfrac{K_{\max}}{C_3^{2}} t^{-2} n^{-1} p^2
    + \dfrac{K_{\max}^{2} }{C_3^{3} } t^{-3} n^{-1} p^2 
    + \dfrac{K_{\max}^{3} }{C_3^{9/2} } t^{-9/2} n^{-3/2} \log^{3} N
    + e^{- 7\log N} \\
    &\leq 
    c_1 \bigg[ 1 + t^{-1} + t^{-5/2} n^{-1/2} p_{\ast}^{-2} \log^{3}N  + e^{- 5\log N} \bigg] 
    t^{-2} n^{-1} p_{\ast}^{2} \\
    &\lesssim
    c_1 \bigg[ 1 + t^{-1} + t^{-5/2} n^{-1/2} p_{\ast}^{-2} (\log^{3}n + \log^{3}T)  + e^{- 5\log N} \bigg] t^{-2} n^{-1} p_{\ast}^{2} \\
    &\lesssim
    c_1 \bigg[ 1 + t^{-1} + t^{-5/2} n^{-1/2} (\log n + \log T)  + e^{- 5\log N} \bigg] t^{-2} n^{-1} p_{\ast}^{2} \\
    \overset{\eqref{assume:iterative_theta_consistency_sample}}&{\lesssim}
    c_1 t^{-2} n^{-1} p_{\ast}^{2}.
\end{align*}
for some constant $c_1 = c_1(K_{\min}, K_{\max}) > 0$.
It follows that
\begin{align*}
    \epsilon_{n, t, \KL} \overset{\eqref{assume:iterative_theta_consistency_sample}}{\leq} \delta
\end{align*}
for a small enough constant $\delta > 0$, which further implies that 
\begin{align} \label{eqn:iterative_theta_eq3}
    \Delta_{t} 
    \overset{ \text{ Corollary \ref{coro:VB_KL_Delta}} }{\leq} 
    c_1' \epsilon_{n, t, \KL} 
    \leq 
    c_2 t^{-1} n^{-1/2} p_{\ast},
\end{align}
where $c_1'$ is the universal constant in Corollary \ref{coro:VB_KL_Delta} and $c_2 = c_2(K_{\min}, K_{\max}) > 0$.
By the results in \textbf{Step 1-3}, we completes the proof of \eqref{eqn:iterative_theta_consistency_result1}.

\noindent \textbf{Step 4: } $\lambda_{\min}(\bOmega_{t})$, $\lambda_{\max}(\bOmega_{t})$ \\
By \eqref{eqn:iterative_theta_eq3}, we have
\begin{align*}
    \left\| \FisherMAP^{-1/2} \bOmega_{t} \FisherMAP^{-1/2} - \bI_{p} \right\|_{2} 
    \overset{ \eqref{def:VB_quantities} }{\leq} \Delta_{t}
    \leq c_2 t^{-1} n^{-1/2} p_{\ast}.
\end{align*}
It follows that
\begin{align}
\begin{aligned} \label{eqn:5}
    &\lambda_{\min}(\bOmega_{t}) 
    \geq \left( 1 - \Delta_{t} \right) \lambda_{\min}(\FisherMAP) 
    \overset{\eqref{assume:iterative_theta_consistency}}{\geq} C_3 \left( 1 - \Delta_{t} \right) n t, \\
    &\lambda_{\max}(\bOmega_{t}) 
    \leq \left( 1 + \Delta_{t} \right) \lambda_{\max}(\FisherMAP) 
    \overset{\eqref{assume:iterative_theta_consistency}}{\leq} C_4 \left( 1 + \Delta_{t} \right) n t.
\end{aligned}
\end{align}

\noindent \textbf{Step 5: } $\| \thetaBest[t+1] - \theta_0 \|_{2}$, $\lambda_{\min}(\FisherBest[t+1])$, $\lambda_{\max}(\FisherBest[t+1])$, $\lambda_{\min}(\FisherMAP[t+1])$, $\lambda_{\max}(\FisherMAP[t+1])$ \\
The results in \textbf{Step 1-4} and the assumptions can be summarized as follows:
\begin{align} 
\begin{aligned} \label{eqn:4}
    &\left\| \FisherBest[t]^{1/2} \left( \theta_0 - \thetaBest[t] \right)\right\|_{2} 
    \overset{ \eqref{assume:iterative_theta_consistency} }{\leq} 
    C_5 M_n t^{\alpha} \sqrt{p_{\ast}}, \\
    &\left\| \FisherBest[t]^{1/2} \left( \thetaMAP[t] - \thetaBest[t] \right)\right\|_{2}
    \overset{\eqref{eqn:3}}{\leq}
    4r_{\eff, t} 
    \overset{\eqref{eqn:iterative_theta_eq1_1_2_1}}{\leq}
    \big( 4 C_3^{-1/2} K_{\max}^{1/2} \big)  M_n \sqrt{t^{-1} p_{\ast}}, \\
    & \left\| \bOmega_{t}^{1/2} \left( \thetaMAP[t] - \mu_{t} \right)\right\|_{2}
    \vee \left\| \FisherMAP^{-1/2} \bOmega_{t} \FisherMAP^{-1/2} - \bI_{p}  \right\|_{2} 
    \overset{ \eqref{def:VB_quantities} }{\leq} 
    \Delta_{t} 
    \overset{ \eqref{eqn:iterative_theta_eq3} }{\leq} 
    c_2 t^{-1} n^{-1/2} p_{\ast} 
    \overset{ \eqref{assume:iterative_theta_consistency_sample} }{\leq} 
    1/4, \\
    &\lambda_{\min}\left( \FisherBest \right) \overset{\eqref{assume:iterative_theta_consistency}}{\geq} C_3 n t, \\
    &\left\| \FisherBest^{-1/2} \FisherMAP \FisherBest^{-1/2} - \bI_{p} \right\|_{2} 
    \overset{ \substack{ \eqref{eqn:3} \\ \text{Lemma \ref{lemma:tech_Fisher_smooth}} } }{\leq}
    4 \tau_{3, t}^{\ast} r_{\eff, t} 
    \overset{ \eqref{eqn:iterative_theta_eq1_2}, \eqref{assume:iterative_theta_consistency_sample} }{\leq}
    1/8, \\
    &\left\| \thetaBest - \theta_0 \right\|_{2} 
    \overset{ \eqref{eqn:iterative_theta_eq1} }{\leq}    
    1/8.
\end{aligned}
\end{align}
By the last display, we can apply Lemma \ref{lemma:ignorable_update2} with
\begin{align*}
    D_{1} = C_5, \quad 
    D_{2} = 4 C_3^{-1/2} K_{\max}^{1/2}, \quad 
    D_{3} = c_2, \quad
    D_{4} = C_{3},
\end{align*}
where the choice $D_{3} = c_2$ is justified by $p_{\ast} \leq n$.
By Lemma \ref{lemma:ignorable_update2}, we have
\begin{align*}
    \left\| \bF_{t+1, \thetaBest}^{1/2} \left( \thetaBest[t+1] - \thetaBest \right) \right\|_{2} \leq 
    C' M_n t^{\alpha - 1/2} \sqrt{p_{\ast}}.
\end{align*}
where $C' = C'(K_{\min}, K_{\max}, C_5) > 0$. 
Consequently, we have
\begin{align}
\begin{aligned} \label{eqn:6}
    \left\| \thetaBest[t+1] - \theta_0 \right\|_{2} 
    &\leq 
    \left\| \thetaBest[t+1] - \thetaBest[t] \right\|_{2} + \left\| \thetaBest[t] - \theta_0 \right\|_{2} \\
    \overset{ \substack{ \eqref{eqn:4}, (\textbf{A2}) } }&{\leq}
    \big( K_{\min} n \big)^{-1/2} C' M_n t^{\alpha - 1/2} \sqrt{p_{\ast}} + 1/8 \\
    &=
    \big( K_{\min}^{-1/2} C' \big)\big( M_n^2 t^{2\alpha - 1} p_{\ast} n^{-1} \big)^{1/2} + 1/8 
    \overset{ \eqref{assume:iterative_theta_consistency_sample} }{\leq}
    1/4.
\end{aligned}    
\end{align}
Since $\thetaBest[t+1] \in \Theta(\theta_0, \bI_{p}, 1/2)$, combining the results in \textbf{Step 4} and (\textbf{A2}), we have
\begin{align}
\begin{aligned} \label{eqn:7}
    \lambda_{\min} \big(\FisherBest[t+1] \big) 
    &\geq 
    \lambda_{\min}(\bOmega_{t}) + \lambda_{\min}\big( \Fisher[t+1]{\thetaBest[t+1]} \big)
    \overset{ \substack{ \eqref{eqn:5} \\ (\textbf{A2}) } }{\geq} 
    C_3 \left( 1 - \Delta_{t} \right) n t + K_{\min} n \\
    &\geq c_{3, t} n(t+1), \\
    \lambda_{\max} \big(\FisherBest[t+1] \big) 
    &\leq \lambda_{\max} (\bOmega_{t}) + \lambda_{\max}\big( \Fisher[t+1]{\thetaBest[t+1]} \big)
    \overset{ \substack{ \eqref{eqn:5} \\ (\textbf{A2}) } }{\leq} 
    C_4 \left( 1 + \Delta_{t} \right) n t + K_{\max} n \\
    &\leq c_{4, t} n(t+1),
\end{aligned}
\end{align}
where
\begin{align}
\begin{aligned} \label{eqn:19} 
    c_{3, t} &=  C_3\left( 1 - \Delta_{t} \right) \in \left[ 3K_{\min}/8, K_{\min} \right], \quad (\because \Delta_{t} \leq 1/4) \\ 
    c_{4, t} &=  C_4 \left( 1 + \Delta_{t} \right) \in \left[K_{\max}, 5K_{\max}/2 \right], \quad (\because \Delta_{t} \leq 1/4).
\end{aligned}    
\end{align}
The proofs of $\lambda_{\min}(\FisherMAP[t+1])$ and $\lambda_{\max}(\FisherMAP[t+1])$ are similar to those in \textbf{Step 1-2}. Hence, we give a sketch of the proof.
Since $c_{3, t} \geq 3K_{\min}/8$, we have
\begin{align*}
    r_{\eff, t+1} 
    \overset{\substack{\eqref{eqn:6} \\ \text{Lemma \ref{lemma:radius_upper_bound}}}}&{\leq} 
    M_n  \lambda_{\min}^{-1/2}(\FisherBest[t+1]) \lambda_{\max}^{1/2}(\Fisher[t+1]{\thetaBest[t+1]})  p_{\ast}^{1/2} \\
    \overset{\substack{\eqref{eqn:7} \\ \eqref{eqn:6}, (\textbf{A2})}}&{\leq} 
    M_n \big( c_{3, t} n [t+1] \big)^{-1/2} ( K_{\max} n )^{1/2} p_{\ast}^{1/2}
    =  \left( c_{3, t}^{-1/2} K_{\max}^{1/2} \right) M_n \sqrt{(t+1)^{-1} p_{\ast}}, \\
    \sup_{\theta \in \Theta_{n, t+1}}
    \left\| \theta - \theta_0  \right\|_2
    &\leq 
    \sup_{\theta \in \Theta_{n, t+1}}
    \left\| \theta - \thetaBest[t+1]  \right\|_2 + \left\| \thetaBest[t+1] - \theta_0 \right\|_2 
    \overset{\eqref{eqn:6}}{\leq}
    \lambda_{\min}^{-1/2}(\FisherBest[t+1]) 4r_{\eff, t+1} + 1/4 \\
    \overset{\eqref{eqn:7}}&{\leq} 
    (4 c_{3, t}^{-1} K_{\max}^{1/2}) M_n p_{\ast}^{1/2} n^{-1/2} (t+1)^{-1} + 1/4
    \overset{\eqref{assume:iterative_theta_consistency_sample}}{\leq} 1/2,    \\    
    \tau_{3, t+1}^{\ast} 
    \overset{ \substack{ (\textbf{A2}) \\ \text{Lemma \ref{lemma:tech_smooth_tau_bound}}  } }&{\leq}
    \left( K_{\max} n \right) \lambda_{\min}^{-3/2}(\FisherBest[t+1])
    \overset{ \eqref{assume:iterative_theta_consistency} }{\leq}
    \left( K_{\max} n \right) \left( c_{3, t} n [t+1] \right)^{-3/2} \\
    &= \big( c_{3, t}^{-3/2} K_{\max} \big) (t+1)^{-3/2} n^{-1/2}, \\
    \tau_{3, t+1}^{\ast} r_{\eff, t+1} 
    &\leq
    \big( c_{3, t}^{-2} K_{\max}^{3/2} \big)  M_n p_{\ast}^{1/2} (t+1)^{-2} n^{-1/2}
    \overset{\eqref{assume:iterative_theta_consistency_sample}}{\leq}
    1/16, \\
    \thetaMAP[t+1] \overset{ \text{Theorem \ref{thm:penalized_estimation}} }&{\in} \Theta_{n, t+1}, \quad 
    \big( \because \tau_{3, t+1}^{\ast} r_{\eff, t+1} \leq 1/16 \big)
\end{align*}
where $\Theta_{n, t+1} = \Theta (\thetaBest[t+1], \FisherBest[t+1], 4r_{\eff, t+1})$.
Since the last display gives
\begin{align*}
    \big\| \thetaMAP[t+1] - \theta_0 \big\|_{2} \leq 1/2,
\end{align*}
we have
\begin{align*}
    \lambda_{\min} \big(\FisherMAP[t+1] \big) 
    &\geq 
    \lambda_{\min}(\bOmega_{t}) + \lambda_{\min} \big( \Fisher[t+1]{\thetaMAP[t+1]} \big)
    \overset{ \substack{ \eqref{eqn:5} \\ (\textbf{A2}) } }{\geq} 
    C_3 \left( 1 - \Delta_{t} \right) n t + K_{\min} n \\
    &\geq c_{3, t} n(t+1), \\
    \lambda_{\max} \big(\FisherMAP[t+1] \big) 
    &\leq \lambda_{\max}(\bOmega_{t}) + \lambda_{\max} \big(\Fisher[t+1]{\thetaMAP[t+1]} \big)
    \overset{ \substack{ \eqref{eqn:5} \\ (\textbf{A2}) } }{\leq} 
    C_4 \left( 1 + \Delta_{t} \right) n t + K_{\max} n \\
    &\leq c_{4, t} n(t+1),
\end{align*}
which completes the proof of the first three assertions in \eqref{eqn:iterative_theta_consistency_result2}.

\noindent \textbf{Step 6: } $\tau_{3, t+1, \bias}$, $\tau_{4, t+1, \bias}$ \\
Note that
\begin{align*}
    &b_{n, t+1} 
    = 
    \left\| \FisherBest[t+1]^{-1/2} \bOmega_{t} \left( \theta_0 - \mu_{t} \right) \right\|_{2}
    \leq
    \left\| \bOmega_{t}^{-1/2} \bOmega_{t} \left( \theta_0 - \mu_{t} \right) \right\|_{2}
    = 
    \left\| \bOmega_{t}^{1/2} \left( \theta_0 - \mu_{t} \right) \right\|_{2} \\
    &\leq 
    \left\| \bOmega_{t}^{1/2} \left( \theta_0 - \thetaBest \right) \right\|_{2} + \left\| \bOmega_{t}^{1/2} \left( \thetaBest - \thetaMAP \right) \right\|_{2} + \left\| \bOmega_{t}^{1/2} \left( \thetaMAP - \mu_{t} \right) \right\|_{2} \\
    &\leq
        \left\| \bOmega_{t}^{1/2} \FisherMAP^{-1/2} \right\|_{2}
        \left\| \FisherMAP^{1/2} \FisherBest^{-1/2}  \right\|_{2}
        \bigg[
        \left\| \FisherBest^{1/2} \left( \theta_0 - \thetaBest \right) \right\|_{2} 
        + \left\| \FisherBest^{1/2} \left( \thetaBest - \thetaMAP \right) \right\|_{2} \bigg]
        + \left\| \bOmega_{t}^{1/2} \left( \thetaMAP - \mu_{t} \right) \right\|_{2} \\
    &\leq
        \left( 1 + \left\| \FisherMAP^{-1/2} \bOmega_{t} \FisherMAP^{-1/2} - \bI_{p} \right\|_{2} \right)^{1/2}
        \left( 1 + \left\| \FisherBest^{-1/2} \FisherMAP \FisherBest^{-1/2} - \bI_{p}  \right\|_{2}
        \right)^{1/2} \\
        &\qquad \times
        \bigg[
        \left\| \FisherBest^{1/2} \left( \theta_0 - \thetaBest \right) \right\|_{2} 
        + \left\| \FisherBest^{1/2} \left( \thetaBest - \thetaMAP \right) \right\|_{2} \bigg] 
        + \left\| \bOmega_{t}^{1/2} \left( \thetaMAP - \mu_{t} \right) \right\|_{2} \\
    \overset{ \eqref{eqn:4}, \eqref{assume:iterative_theta_consistency} }&{\leq}
    \left( 1 + \Delta_{t} \right)^{1/2}
    \left( 1 + 4\tau_{3, t}^{\ast} r_{\eff, t} \right)^{1/2}
    \big( C_5 M_n t^{\alpha} \sqrt{p_{\ast}} + 4r_{\eff, t} \big) 
    + \Delta_{t} \\
    \overset{\eqref{eqn:iterative_theta_eq1_1_2_1}, \eqref{eqn:4} }&{\leq} 
    2\left( C_5 M_n t^{\alpha} \sqrt{p_{\ast}} + \big( 4C_3^{-1/2} K_{\max}^{1/2} \big) M_n \sqrt{t^{-1} p_{\ast}} \right)
    + c_2 t^{-1} n^{-1/2} p_{\ast} \\    
    \overset{\eqref{assume:iterative_theta_consistency_sample} }&{\leq} 
    c_5 M_n t^{\alpha} \sqrt{p_{\ast}},
\end{align*}
where $c_5 = c_5(K_{\min}, K_{\max}, C_5) > 0$.
For $\theta \in \Theta (\thetaBest[t+1], \FisherBest[t+1], 4b_{n, t+1})$, we have
\begin{align*}
    \left\| \theta - \thetaBest[t+1] \right\|_{2} 
    \overset{ \eqref{eqn:7} }&{\leq}
    \big( c_{3, t} n (t+1) \big)^{-1/2} 4 b_{n, t+1}
    \leq \big( c_{3, t} n (t+1) \big)^{-1/2} 4c_5 M_n t^{\alpha} \sqrt{p_{\ast}} \\
    &\leq \big( 4c_{5} (c_{3, t})^{-1/2} \big) \sqrt{ M_n^{2} t^{2\alpha-1} p_{\ast} n^{-1}}  
    \overset{\eqref{assume:iterative_theta_consistency_sample}}{\leq} 1/4.
\end{align*}
Combining with \eqref{eqn:6}, we have
\begin{align*}
    \Theta_{n} \left( \thetaBest[t+1], \FisherBest[t+1], 4b_{n, t+1} \right) \subseteq \Theta (\theta_0, \bI_{p}, 1/2).
\end{align*}
By Lemma \ref{lemma:tech_smooth_tau_bound} and (\textbf{A2}), we have
\begin{align}
\begin{aligned} \label{eqn:10}
    \tau_{3, t+1, \bias}
    &\leq \left( K_{\max} n \right) \lambda_{\min}^{-3/2}(\FisherBest[t+1])
    \leq \left( K_{\max} n \right) \big( c_{3, t} n (t+1) \big)^{-3/2} \\
    &=  K_{\max} (c_{3, t})^{-3/2} (t+1)^{-3/2} n^{-1/2}. 
\end{aligned}    
\end{align}
Similarly,
\begin{align}
\begin{aligned} \label{eqn:11}
    \tau_{4, t+1, \bias}
    &\leq \left( K_{\max} n \right) \lambda_{\min}^{-2}(\FisherBest[t+1])
    \leq \left( K_{\max} n \right) \big( c_{3, t} n (t+1) \big)^{-2} \\
    &= K_{\max} (c_{3, t})^{-2}  (t+1)^{-2} n^{-1}.
\end{aligned}    
\end{align}

\noindent \textbf{Step 7: $\| \FisherBest[t+1]^{1/2} \left( \theta_0 - \thetaBest[t+1] \right)\|_{2}$}  \\
To complete the proof of the last inequality in \eqref{eqn:iterative_theta_consistency_result2}, we will show that the assumptions in Lemma \ref{lemma:penalized_bias_recursive} are satisfied.
By the results in \textbf{Step 1-6}, we have 
\begin{align*}
    \forall r > 4r_{\LA}, \quad 
    \widehat{\tau}_{3, t, r} 
    \overset{ \eqref{eqn:8} }&{\leq}
    N e^{8p} \exp\bigg( \left[ \sqrt{p} + \sqrt{2 \log N} - 3 \right] r \bigg),   \\
    \widehat{\tau}_{3, t} r_{\LA}^{2} 
    \overset{\eqref{eqn:9}}{\leq}
    c_6 t^{-3/2} n^{-1/2} p_{\ast}, \quad
    \widehat{\tau}_{4, t} p^{2} \overset{\eqref{eqn:9}}&{\leq} c_6 t^{-2} n^{-1} p^{2}, \\
    \tau_{3, t}^{\ast} r_{\eff, t} 
    \overset{ \eqref{eqn:iterative_theta_eq1_2} }{\leq}
    c_6 t^{-2} n^{-1/2} M_n p_{\ast}^{1/2}, \quad 
    \epsilon_{n, t, \KL} \vee \Delta_{t} 
    \overset{ \eqref{eqn:iterative_theta_eq3} }&{\leq}
    c_6 t^{-1} n^{-1/2} p_{\ast}, \\    
    \tau_{3, t+1, \bias} (M_n t^{1/2 + 2\alpha} \sqrt{p_{\ast}})
    \overset{ \eqref{eqn:10} }&{\leq}
    c_6 M_n t^{2\alpha - 1} \sqrt{p_{\ast}} n^{-1/2}, \\
    \tau_{4, t+1, \bias} (M_n^{2} t^{1/2 + 3\alpha} p_{\ast}) 
    \overset{ \eqref{eqn:11} }&{\leq}
    c_6 M_n^{2} t^{3\alpha - 3/2} p_{\ast} n^{-1}, \\
    \tau_{3, t+1, \bias}^{2} (M_n^{2} t^{1/2 + 3\alpha} p_{\ast}) 
    \overset{ \eqref{eqn:10} }&{\leq} 
    c_6 M_n^{2} t^{3\alpha - 5/2} p_{\ast} n^{-1}, \\
    \left( \epsilon_{n, t, \KL} \vee \big[ \tau_{3, t}^{\ast} r_{\eff, t} \big] \right) t^{1/2 + \alpha} \overset{ \substack{ \eqref{eqn:iterative_theta_eq1_2} \\ \eqref{eqn:iterative_theta_eq3} } }&{\leq}
    c_6 M_n t^{\alpha - 1/2} p_{\ast} n^{-1/2},  
\end{align*}
where $c_6 = c_6(K_{\min}, K_{\max})$.
By the last display and \eqref{assume:iterative_theta_consistency_sample},
the assumptions \eqref{assume:penalized_bias_recursive_delta}, \eqref{assume:penalized_bias_recursive_delta_2} and \eqref{assume:penalized_bias_recursive_tail} in Lemma \ref{lemma:penalized_bias_recursive} are satisfied. 
Also, we have
\begin{align*}
r_{\eff, t} 
\overset{ \eqref{eqn:iterative_theta_eq1_1_2_1} }{\leq} 
\big( C_3^{-1/2} K_{\max}^{1/2} \big) M_n \sqrt{t^{-1} p_{\ast}}, \quad 
\left\| \FisherBest[t]^{1/2} \left( \theta_0 - \thetaBest[t] \right)\right\|_{2} 
\overset{ \eqref{assume:iterative_theta_consistency} }{\leq}
C_5 M_n t^{\alpha} \sqrt{p_{\ast}}.
\end{align*}
Hence, the assumption \eqref{assume:penalized_bias_recursive_C1} in Lemma \ref{lemma:penalized_bias_recursive} are satisfied with $C_1 = C_3^{-1/2} K_{\max}^{1/2}$ and $C_2 = C_5$.
Recall that 
\begin{align*}
C_5 \geq 192 C_{3}^{-1/2} K_{\max}^{1/2} \geq (2^{1/\alpha})48(C_{3}^{-1/2} K_{\max}^{1/2}), \quad
C_5 \geq 1,
\end{align*}
which is chosen to be sufficiently large so that we can replace $K$ in \eqref{claim:penalized_bias_recursive} with $C_5$; see the remark following Lemma \ref{lemma:penalized_bias_recursive}.
By Lemma \ref{lemma:penalized_bias_recursive} with $K = C_5$, we have
\begin{align*}
    \left\| \FisherBest[t+1]^{1/2} \left( \theta_0 - \thetaBest[t+1] \right)\right\|_{2} 
    &\leq C_5 M_n (t^{1/2 - \alpha} + t)^{\alpha} \sqrt{p_{\ast}} \\
    &\leq C_5 M_n (t+1)^{\alpha} \sqrt{p_{\ast}}, \quad (\because \alpha \in [1/2, 1]),
\end{align*}
which completes the proof of \eqref{eqn:iterative_theta_consistency_result2}.

All remaining proofs for \eqref{eqn:iterative_theta_consistency_result3} are similar to those in \textbf{Step 1-4}, but replace $c_{3, t}$, $c_{4, t}$ and $t$ with $K_{\min}/2$, $2K_{\max}$ and $t+1$, respectively.
\end{proof}

\begin{proposition} \label{prop:eigenvalue_order}
    Suppose that (\textbf{A0}), (\textbf{A1}), (\textbf{A2}), (\textbf{S}) and (\textbf{P}) hold.
    Then, on $\scrE_{\est, 1}$, the following inequalities hold uniformly for all $t \in [T]$:
    \begin{align} \label{eqn:eigenvalue_order_claim}
    \begin{aligned}
        &\left\| \FisherBest[t]^{1/2} \left( \theta_0 - \thetaBest \right) \right\|_{2} \leq K_{\rm up} M_n \sqrt{t \ p_{\ast}},  &
        &\lambda_{\min} \left( \FisherBest \right) \wedge \lambda_{\min} \left( \FisherMAP \right) \geq K_{\rm low} nt, \\  
        &\left\| \FisherBest[t]^{1/2} \big( \thetaMAP - \thetaBest \big) \right\|_{2} \leq K_{\rm up} M_n \sqrt{t^{-1} p_{\ast}},  &
        &\lambda_{\max} \left( \FisherBest \right)  \vee  \lambda_{\max} \left( \FisherMAP \right) \leq K_{\rm up} nt, 
    \end{aligned}
    \end{align}    
    and
    \begin{align} 
    \begin{aligned} \label{eqn:eigenvalue_order_claim_2}
        \widehat{\tau}_{3, t} \vee \tau_{3, t}^{\ast} 
        &\leq
        K_{\rm up} t^{-3/2} n^{-1/2}, \\
        \widehat{\tau}_{4, t} 
        &\leq    
        K_{\rm up} t^{-2} n^{-1}, \\
        \Delta_{t} \vee \epsilon_{n, t, \KL} 
        &\leq K_{\rm up} t^{-1} n^{-1/2} p_{\ast},
    \end{aligned}
    \end{align}    
    where $K_{\rm up} = K_{\rm up}(K_{\min}, K_{\max})$ and $K_{\rm low} = K_{\rm low}(K_{\min})$.
\end{proposition}
\begin{proof}
In this proof, we will work on the event $\scrE_{\est, 1}$ without explicitly mentioning it.
Also, although the value of $\alpha$ in Lemma \ref{lemma:iterative_theta_consistency} can range over the interval $[1/2, 1]$, we fix $\alpha = 1/2$ in this proof. 
Later, we specifiy $C_5$ in Lemma \ref{lemma:iterative_theta_consistency} as a constant depending only on $K_{\min}$ and $K_{\max}$.
Consequently, (\textbf{S}) implies the condition \eqref{assume:iterative_theta_consistency_sample} with $\alpha = 1/2$.

\noindent \textbf{Step 1: Proof strategy} \\
Suppose that the following inequalities hold for some $t \in [T]$:
\begin{align}
\begin{aligned} \label{eqn:23} 
    \lambda_{\min}\left( \FisherBest \right) \wedge \lambda_{\min}\left( \FisherMAP \right) &\geq \dfrac{K_{\min}}{2}  n t, \\
    \lambda_{\max}\left( \FisherBest \right) \vee \lambda_{\max}\left( \FisherMAP \right) &\leq 2 K_{\max} n t, \\ 
    \left\| \FisherBest[t]^{1/2} \left( \theta_0 - \thetaBest[t] \right)\right\|_{2} &\leq \big( 4 \vee [192\sqrt{2} K_{\min}^{-1/2} K_{\max}^{1/2}] \big) M_n \sqrt{t \ p_{\ast}}.
\end{aligned}    
\end{align}
By Lemma \ref{lemma:iterative_theta_consistency} with 
\begin{align*}
    C_3 = \dfrac{K_{\min}}{2}, \quad
    C_4 = 2 K_{\max} n, \quad
    C_5 = 4 \vee ( 192\sqrt{2} K_{\min}^{-1/2} K_{\max}^{1/2}),
\end{align*}
we have
\begin{align*}
    \left\| \FisherBest[t]^{1/2} \big( \thetaMAP[t] - \thetaBest[t] \big) \right\|_{2} 
    &\leq K_1 M_n \sqrt{t^{-1} p_{\ast}}, \\
    \tau_{3, t}^{\ast} \vee \widehat{\tau}_{3, t} 
    &\leq K_1  t^{-3/2} n^{-1/2}, \\
    \widehat{\tau}_{4, t} 
    &\leq K_1 t^{-2} n^{-1}, \\
    \Delta_t \vee \epsilon_{n, t, \KL} 
    &\leq K_1 t^{-1} n^{-1/2} p_{\ast},
\end{align*}
where $K_1 = K_1(K_{\min}, K_{\max})$ denotes the constant $K$ specified in Lemma \ref{lemma:iterative_theta_consistency}.
If \eqref{eqn:23} holds for all $t \in [T]$, then we complete the proof by taking
\begin{align*}
    K_{\rm up} = 2K_{\max} \vee K_1 \vee 4 \vee \big(192\sqrt{2} K_{\min}^{-1/2} K_{\max}^{1/2} \big), \quad
    K_{\rm low} = K_{\min}/2.
\end{align*}
Therefore, it suffices to show that \eqref{eqn:23} holds for all $t \in [T]$.
The proof is divided into several steps.
We first establish---through \textbf{Step 2-4}--- that \eqref{eqn:23} holds for $t = 1$. Subsequently, by using induction, we will prove in \textbf{Step 5} that \eqref{eqn:23} holds for all $t \in [T]$.

\noindent \textbf{Step 2: } $\lambda_{\min}(\FisherBest[1])$ and $\lambda_{\max}(\FisherBest[1])$ \\
By (\textbf{P}) with a small enough $\delta = \delta(K_{\min}, K_{\max})$, the assumption \eqref{assume:pre_theta_consistency_sample} in Lemma \ref{lemma:pre_theta_consistency} is satisfied for $t = 0$. By Lemma \ref{lemma:pre_theta_consistency}, we have
\begin{align*}
\left\| \Fisher[1]{\theta_0}^{1/2} \left( \theta_0 - \thetaBest[1] \right) \right\|_{2} 
\leq \sqrt{2} \left\| \bOmega_{0}^{1/2} \big( \theta_0 - \mu_0 \big) \right\|_{2}  
\overset{(\textbf{P})}{\leq}
( \sqrt{2} \delta ) n^{1/2},
\end{align*}
which further implies that
\begin{align}
\begin{aligned} \label{eqn:eigenvalue_order_eq1}
&\left\| \theta_0 - \thetaBest[1] \right\|_{2} 
\leq
\lambda_{\min}^{-1/2} \big( \Fisher[1]{\theta_0} \big) ( \sqrt{2} \delta ) n^{1/2} 
\overset{(\textbf{A2})}{\leq} 
(K_{\min} n)^{-1/2} ( \sqrt{2} \delta ) n^{1/2}
\overset{(\textbf{P})}{\leq} 1/8, \\
&\lambda_{\min}(\FisherBest[1]) 
\geq 
\lambda_{\min}(\Fisher[1]{\thetaBest[1]}) + \lambda_{\min}(\bOmega_0)
\overset{(\textbf{A2})}{\geq} K_{\min} n, \\
&\lambda_{\max}(\FisherBest[1]) 
\leq \lambda_{\max}(\bOmega_{0}) + \lambda_{\max}(\Fisher[1]{\thetaBest[1]}) 
\overset{(\textbf{A2}), (\textbf{P})}{\leq} 
K_{\max}p_{\ast} + K_{\max} n 
\overset{(\textbf{S})}{\leq} 
\dfrac{4}{3}K_{\max} n.
\end{aligned}
\end{align}

\noindent \textbf{Step 3: } $\| \FisherBest[1]^{1/2} \left( \theta_0 - \thetaBest[1] \right) \|_{2}$ \\
Note that
\begin{align} \label{eqn:12} 
\begin{aligned}
    b_{n, 1} 
    &= \left\| \FisherBest[1]^{-1/2} \bOmega_{0} \big( \theta_0 - \mu_0 \big) \right\|_{2} 
    \leq 
    \lambda_{\min}^{-1/2} \big( \FisherBest[1] \big)
    \left\| \bOmega_0 \right\|_{2}^{1/2}
    \left\| \bOmega_{0}^{1/2} \big( \theta_0 - \mu_0 \big) \right\|_{2} \\
    &\leq 
    \big( K_{\min} n \big)^{-1/2} 
    \big( K_{\max} p_{\ast} \big)^{1/2}
    \big( \delta n^{1/2} \big)
    =
    \big( K_{\min}^{-1/2} K_{\max}^{1/2} \delta \big) p_{\ast}^{1/2}
    \leq 
    M_n p_{\ast}^{1/2},
\end{aligned}    
\end{align}
where the last inequality holds by $M_n \geq 1$ and small enough $\delta$.
Also, for $\theta \in \Theta (\thetaBest[1], \FisherBest[1], 4b_{n, 1})$, we have
\begin{align*}
    \left\| \theta - \thetaBest[1] \right\|_{2} 
    \leq
    \lambda_{\min}^{-1/2} \big( \FisherBest[1] \big)
    4 b_{n, 1}
    \overset{\eqref{eqn:eigenvalue_order_eq1}, \eqref{eqn:12} }{\leq}
    (K_{\min} n)^{-1/2} 4 M_n \sqrt{p_{\ast}} 
    \overset{(\textbf{S})}{\leq} 1/4,
\end{align*}
which, combining with $\left\| \theta_0 - \thetaBest[1] \right\|_{2} \leq 1/8$, implies that $\Theta (\thetaBest[1], \FisherBest[1], 4b_{n, 1}) \subset \Theta (\theta_0, \bI_{p}, 1/2)$.
Hence, Lemma \ref{lemma:tech_smooth_tau_bound} implies that
\begin{align*}
    \tau_{3, 1, \bias} \overset{(\textbf{A2})}&{\leq} \left( K_{\max} n \right) \left( K_{\min} n \right)^{-3/2} = \big( K_{\max} K_{\min}^{-3/2} \big) n^{-1/2}, \\
    \tau_{4, 1, \bias} \overset{(\textbf{A2})}&{\leq} \left( K_{\max} n \right) \left( K_{\min} n \right)^{-2} = \big( K_{\max} K_{\min}^{-2} \big) n^{-1}.
\end{align*}
It follows that
\begin{align*}
    \tau_{3, 1, \bias} b_{n, 1} 
    \leq
    \left( \big( K_{\max} K_{\min}^{-3/2} \big) n^{-1/2} \right) \big( M_n \sqrt{p_{\ast}} \big)
    \overset{ (\textbf{S}) }{\leq} 1/16,
\end{align*}
which allows us to utilize Lemma \ref{lemma:tech_pre_bias_bound}.
By Lemma \ref{lemma:tech_pre_bias_bound} with 
\begin{align*}
    \tau_{3} = \tau_{3, 1, \bias}, \quad 
    f(\theta) = \bbE_{1} \widetilde{L}_{1} (\theta), \quad 
    \theta = \thetaBest[1], \quad 
    \widetilde{\theta} = \theta_0, \quad 
    \beta = \bOmega_{0} \big( \theta_0 - \mu_0 \big), \quad 
    r = b_{n, 1},
\end{align*}
we have
\begin{align} \label{eqn:eigenvalue_order_eq2}
    \left\| \FisherBest[1]^{1/2} \left( \theta_0 - \thetaBest[1] \right) \right\|_{2} 
    \leq 4 b_{n, 1}
    \overset{ \eqref{eqn:12} }{\leq} 4 M_n \sqrt{p_{\ast}}.
\end{align}

\noindent \textbf{Step 4: } $\lambda_{\min}(\FisherMAP[1])$ and $\lambda_{\max}(\FisherMAP[1])$ \\
Combining with $\| \theta_0 - \thetaBest[1] \|_2 \leq 1/8$, we have
\begin{align}
\begin{aligned} \label{eqn:13} 
    r_{\eff, 1}
    \overset{\text{Lemma \ref{lemma:radius_upper_bound}}}&{\leq}
    M_n \lambda_{\min}^{-1/2}\big( \FisherBest[1] \big) \lambda_{\max}^{1/2}\big( \Fisher[1]{\thetaBest[1]} \big) p_{\ast}^{1/2} \\
    \overset{ \eqref{eqn:eigenvalue_order_eq1}, (\textbf{A2}) }&{\leq}
    M_n \big( K_{\min} n \big)^{-1/2} \big( K_{\max} n \big)^{1/2} p_{\ast}^{1/2}
    = 
    \big( K_{\min}^{-1} K_{\max} \big)^{1/2} M_n p_{\ast}^{1/2}.
\end{aligned}    
\end{align}
Also, for $\theta \in \Theta (\thetaBest[1], \FisherBest[1], 4r_{\eff, 1})$,
\begin{align} 
\begin{aligned} \label{eqn:14}
    \left\| \theta - \thetaBest[1] \right\|_{2} 
    &\leq
    \lambda_{\min}^{-1/2} \big( \FisherBest[1] \big) 4r_{\eff, 1}
    \overset{ \eqref{eqn:eigenvalue_order_eq1}, \eqref{eqn:13} }{\leq} 
    (K_{\min} n)^{-1/2} 4 \big( K_{\min}^{-1} K_{\max} \big)^{1/2} M_n p_{\ast}^{1/2} \\
    &=
    \big( 4K_{\min}^{-1} K_{\max}^{1/2} \big) M_n p_{\ast}^{1/2} n^{-1/2} 
    \overset{(\textbf{S})}{\leq} 1/4.
\end{aligned}    
\end{align}
It follows that $\Theta (\thetaBest[1], \FisherBest[1], 4r_{\eff, 1}) \subset \Theta (\theta_0, \bI_{p}, 1/2)$. Hence, we have
\begin{align*}
    \tau_{3, 1}^{\ast} 
    \overset{ \substack{ \text{Lemma \ref{lemma:tech_smooth_tau_bound}} \\ (\textbf{A2}), \eqref{eqn:eigenvalue_order_eq1} }}&{\leq} 
    \big( K_{\min} n \big)^{-3/2}
    \big( K_{\max} n \big)
    \leq
    \big( K_{\max} K_{\min}^{-3/2} \big) n^{-1/2}, \\ 
    \tau_{3, 1}^{\ast} r_{\eff, 1} 
    \overset{\eqref{eqn:13} }&{\leq} 
    \big( K_{\max}^{3/2} K_{\min}^{-2} \big) M_n p_{\ast}^{1/2} n^{-1/2}
    \overset{(\textbf{S}) }{\leq} 
    1/16.
\end{align*}
By the last display, Theorem \ref{thm:penalized_estimation} implies that
\begin{align*}
    \left\| \FisherBest[1]^{1/2} \big( \thetaMAP[1] - \thetaBest[1] \big) \right\|_{2}
    \leq 4r_{\eff, 1}, \quad
    \left\| \thetaMAP[1] - \thetaBest[1] \right\|_{2}  
    \overset{\eqref{eqn:14}}{\leq} 1/4,
\end{align*}
which further implies that $\thetaMAP[1] \in \Theta(\theta_0, \bI_{p}, 1/2)$.
It follows that
\begin{align}
\begin{aligned} \label{eqn:eigenvalue_order_eq2_2}
&\lambda_{\min}(\FisherMAP[1]) 
\geq \lambda_{\min}(\Fisher[1]{\thetaMAP[1]}) + \lambda_{\min}(\bOmega_0)
\overset{(\textbf{A2})}{\geq} 
K_{\min} n, \\
&\lambda_{\max}(\FisherMAP[1]) 
\leq \lambda_{\max}(\bOmega_{0}) + \lambda_{\max}(\Fisher[1]{\thetaMAP[1]}) 
\overset{(\textbf{A2}), (\textbf{P})}{\leq} 
K_{\max}p_{\ast} + K_{\max} n 
\overset{(\textbf{S})}{\leq} 
\dfrac{4}{3}K_{\max} n.
\end{aligned}
\end{align}
Combining \eqref{eqn:eigenvalue_order_eq1}, \eqref{eqn:eigenvalue_order_eq2} and \eqref{eqn:eigenvalue_order_eq2_2}, one can check that \eqref{eqn:23} holds for $t = 1$.

\noindent \textbf{Step 5: Inductive argument}  \\
We complete this proof by employing an inductive argument.
Let $t_0 \in \{2, 3, ..., T\}$. Suppose that the following inequalities hold:
\begin{align}
\begin{aligned} \label{eqn:eigenvalue_order_induction}
    \lambda_{\min}(\FisherBest[s]) \wedge \lambda_{\min}(\FisherMAP[s]) &\geq c_{1} n s, \\
    \lambda_{\max}(\FisherBest[s]) \vee \lambda_{\max}(\FisherMAP[s]) &\leq c_{2} n s, \\
    \left\| \FisherBest[s]^{1/2} \left( \theta_0 - \thetaBest[s] \right) \right\|_{2} &\leq c_{3} M_n s^{1/2} \sqrt{p_{\ast}}
\end{aligned}
\end{align}
for all $s \in [t_0-1]$, where 
\begin{align*}
    c_{1} \in [K_{\min}/2, K_{\min}], \quad 
    c_{2} \in [K_{\max}, 2K_{\max}], \quad 
    c_{3} = 4 \vee \big( 192 \sqrt{2} K_{\min}^{-1/2} K_{\max}^{1/2} \big).
\end{align*}
By applying Lemma \ref{lemma:iterative_theta_consistency} with $\alpha = 1/2$, it follows from \eqref{eqn:iterative_theta_consistency_result2} that
\begin{align}
\begin{aligned} \label{eqn:15}
    \lambda_{\min}(\FisherBest[t_0]) \wedge \lambda_{\min}(\FisherMAP[t_0]) &\geq c_{t_0} n t_0, \\
    \lambda_{\max}(\FisherBest[t_0]) \vee \lambda_{\max}(\FisherMAP[t_0]) &\leq c'_{t_0} n t_0, \\
    \left\| \FisherBest[t_0]^{1/2} \left( \theta_0 - \thetaBest[t_0] \right) \right\|_{2} &\leq c_{3} M_n t_0^{1/2} \sqrt{p_{\ast}}
\end{aligned}    
\end{align}
for some constants $c_{t_0}, c'_{t_0} > 0$.
Next, we prove the following inequalities:
\begin{align} \label{eqn:eigenvalue_order_eq3}
    c_{t_0} \geq \dfrac{K_{\min}}{2}, \quad 
    c'_{t_0} \leq 2K_{\max}, \quad 
    \forall t_0 \in \{2, 3, ..., T \},
\end{align}
which, combining with the result in \textbf{Step 0}, completes the proof by induction.

By \eqref{eqn:iterative_theta_consistency_result1}, \eqref{eqn:iterative_theta_consistency_result2} in Lemma \ref{lemma:iterative_theta_consistency}, \eqref{eqn:eigenvalue_order_induction} implies that 
\begin{align}
\begin{aligned} \label{eqn:20}
    \lambda_{\min}\big( \Fisher[t]{\thetaMAP[t]} \big) \wedge \lambda_{\min} \big( \Fisher[t]{\thetaBest[t]} \big) 
    \overset{ (\textbf{A2}) }&{\geq} K_{\min} n, \quad &\forall t \in [t_0], \\
    \lambda_{\max}\big( \Fisher[t]{\thetaMAP[t]} \big) \wedge \lambda_{\max} \big( \Fisher[t]{\thetaBest[t]} \big) 
    \overset{ (\textbf{A2}) }&{\leq} K_{\max} n, \quad &\forall t \in [t_0], \\
    \Delta_{s} 
    \leq K_1 s^{-1} n^{-1/2}  p_{\ast} 
    \overset{ (\textbf{S}) }&{\leq} 1/4, \quad &\forall s \in [t_0 -1],
\end{aligned}    
\end{align}
where the first two inequalities hold because $\thetaBest[t], \thetaMAP[t] \in \Theta(\theta_0, \bI_p, 1/2)$ for all $t \in [t_0]$.
Also, for any $t \in [T]$, we have
\begin{align}
\begin{aligned} \label{eqn:24}
    \FisherMAP[t] 
    &= 
    \bOmega_{t-1} + \Fisher[t]{\thetaMAP[t]} \\
    &\succeq
    (1 - \Delta_{t-1}) \FisherMAP[t-1] + \Fisher[t]{\thetaMAP[t]} \\
    &\succeq
    (1 - \Delta_{t-2}) (1 - \Delta_{t-1}) \FisherMAP[t-2] 
    + (1 - \Delta_{t-1}) \Fisher[t-1]{\thetaMAP[t-1]} 
    + \Fisher[t]{\thetaMAP[t]} \\
    &\succeq
    \sum_{s = 1}^{t-1} \left(
    \bigg[ \prod_{r = 1}^{s} \left( 1 - \Delta_{t-r} \right) \bigg] \Fisher[t-s]{\thetaMAP[t-s]}
    \right)
    + \Fisher[t]{\thetaMAP[t]}, \\
    \FisherBest[t] 
    &= 
    \bOmega_{t-1} + \Fisher[t]{\thetaBest[t]} \\
    &\succeq
    \sum_{s = 1}^{t-1} \left(
    \bigg[ \prod_{r = 1}^{s} \left( 1 - \Delta_{t-r} \right) \bigg] \Fisher[t-s]{\thetaMAP[t-s]}
    \right)
    + \Fisher[t]{\thetaBest[t]}.
\end{aligned}    
\end{align}
It follows that
\begin{align*}
    &\lambda_{\min} \big( \FisherBest[t_0] \big) \wedge \lambda_{\min} \big( \FisherMAP[t_0] \big) \\
    \overset{ \substack{\eqref{eqn:20} \\  \eqref{eqn:24} } }&{\geq}
    \sum_{s = 1}^{t_0-1} \left( \left[ \prod_{r = 1}^{s} \left( 1 - \Delta_{t_0-r} \right) \right] K_{\min} n \right) + K_{\min} n \\
    &\geq 
    \sum_{s = 1}^{t_0-1} \left( \exp \left[ - 2\sum_{r = 1}^{s} \Delta_{t_0-r} \right] K_{\min} n \right) + K_{\min} n \\
    &\geq \sum_{s = 1}^{t_0-1} \left( \left[1 - 2\sum_{r = 1}^{s} \Delta_{t_0-r} \right] K_{\min} n \right) + K_{\min} n \\
    &\geq \sum_{s = 1}^{t_0-1} \left( \left[1 - 2\sum_{r = 1}^{t_0 -1} \Delta_{r} \right] K_{\min} n \right) + K_{\min} n, \quad ( \because \Delta_r \geq 0, \quad \forall r \in [t_0 - 1])
\end{align*}
where the second and third inequalities hold by $1 - x \geq e^{-2x}$ for $x \in [0, 0.795]$ and $e^{-x} \geq 1 - x$ for $x \in \bbR$, respectively. Suppose that the following inequality holds:
\begin{align} \label{eqn:25}
    \sum_{s = 1}^{t_0-1} \Delta_{s} \leq 1/4, \quad \forall t_0 \in \{2, 3, ..., T\}.
\end{align}
Then, for any $t_0 \in \{2, 3, ..., T\}$, we have
\begin{align*}
    \sum_{s = 1}^{t_0-1} \left( \left[1 - 2\sum_{r = 1}^{t_0 -1} \Delta_{t_0-r} \right] K_{\min} n \right) + K_{\min} n
    \geq
    \dfrac{K_{\min}}{2} n t_0.
\end{align*}
Hence, we need to show \eqref{eqn:25}.
Note that
\begin{align*}
    \sum_{s = 1}^{t_0-1} \Delta_{s}
    \overset{ \eqref{eqn:20} }&{\leq} 
    \sum_{s = 1}^{t_0-1} K_1 p_{\ast} n^{-1/2} s^{-1}
    = K_1 p_{\ast} n^{-1/2} \sum_{s = 1}^{t_0-1} s^{-1} 
    \leq K_1 p_{\ast} n^{-1/2} \sum_{s = 1}^{T} s^{-1} \\
    &\leq K_1 p_{\ast} n^{-1/2} \left( \log T + 1 \right)
    \leq 2K_1 p_{\ast} n^{-1/2} \log (T \vee 3) 
    \overset{(\textbf{S})}{\leq} 1/4.
\end{align*}

As in \eqref{eqn:24}, for any $t \in [T]$, we have
\begin{align*}
    \FisherMAP[t] 
    &\preceq
    \bigg[ \prod_{r = 1}^{t-1} \left( 1 + \Delta_{r} \right) \bigg] \bOmega_0 +
    \sum_{s = 1}^{t-1} \left(
    \bigg[ \prod_{r = 1}^{s} \left( 1 + \Delta_{t-r} \right) \bigg] \Fisher[t-s]{\thetaMAP[t-s]}
    \right)
    + \Fisher[t]{\thetaMAP[t]}, \\
    \FisherBest[t] 
    &\preceq
    \bigg[ \prod_{r = 1}^{t-1} \left( 1 + \Delta_{r} \right) \bigg] \bOmega_0 +
    \sum_{s = 1}^{t-1} \left(
    \bigg[ \prod_{r = 1}^{s} \left( 1 + \Delta_{t-r} \right) \bigg] \Fisher[t-s]{\thetaMAP[t-s]}
    \right)
    + \Fisher[t]{\thetaBest[t]}.
\end{align*}
Also, note that
\begin{align*}
    \left\| \bOmega_0 \right\|_2 
    \overset{(\textbf{P})}{\leq}
    K_{\max} p_{\ast} 
    \overset{(\textbf{S})}{\leq}
    \dfrac{K_{\max}}{3} n.
\end{align*}
It follows that
\begin{align*}
    &\lambda_{\max}(\FisherBest[t_0]) \vee \lambda_{\max}(\FisherMAP[t_0]) \\
    \overset{ \substack{ \eqref{eqn:20} } }&{\leq}
    \bigg[ \prod_{r = 1}^{t_0-1} \left( 1 + \Delta_{r} \right) \bigg] \dfrac{K_{\max}}{3}n + 
    \sum_{s = 1}^{t_0-1} \left( \left[ \prod_{r = 1}^{s} \left( 1 + \Delta_{t_0-r} \right) \right] K_{\max} n \right) + K_{\max} n \\
    &\leq 
    \exp \left( \sum_{r = 1}^{t_0-1} \Delta_{r} \right) \dfrac{K_{\max}}{3}n +
    \sum_{s = 1}^{t_0-1} \left( \exp \left[ \sum_{r = 1}^{s} \Delta_{t_0-r} \right] K_{\max} n \right) + K_{\max} n \\
    &\leq 
    \left( 1 + 2\sum_{r = 1}^{t_0-1} \Delta_{r} \right) \dfrac{K_{\max}}{3}n +
    \sum_{s = 1}^{t_0-1} \left( \left[1 + 2\sum_{r = 1}^{s} \Delta_{t_0-r} \right] K_{\max} n \right) + K_{\max} n \\
    &\leq 
    \left( 1 + 2\sum_{r = 1}^{t_0-1} \Delta_{r} \right) \dfrac{K_{\max}}{3}n +
    \sum_{s = 1}^{t_0-1} \left( \left[1 + 2\sum_{r = 1}^{t_0-1} \Delta_{r} \right] K_{\max} n \right) + K_{\max} n \\
    &\leq 
    \dfrac{3}{2} \cdot \dfrac{K_{\max}}{3}n
    + \dfrac{3}{2} K_{\max} n(t_0 - 1)
    + K_{\max}n 
    = \dfrac{3}{2} K_{\max} n t_0
    \\
    &\leq 2 K_{\max} n t_0,
\end{align*}
where the second and the third inequalities hold because $1+ x \leq e^{x} \leq 1 + 2x$ for $x \in [0, 1/4]$ and $\sum_{r= 1}^{t_0 - 1} \Delta_{r} \leq 1/4$.
This completes the proof of \eqref{eqn:eigenvalue_order_eq3}. 
\end{proof}


\section{Proofs for Section \ref{sec:full_posterior}} \label{sec:Appendix_full_posterior}

\begin{lemma} \label{lemma:full_posterior_optimal_parameter}
    Suppose that (\textbf{A0}), (\textbf{A1}), (\textbf{A2}), (\textbf{S}) and (\textbf{P}) hold.
    Then, for all $t \in [T]$,
    \begin{align} \label{eqn:full_posterior_optimal_parameter_claim}
        \left\| \FullFisherTilde[t]{\fullthetaBest[t]}^{1/2} \left( \theta_0 - \fullthetaBest[t] \right) \right\|_{2} 
        \leq 4\left\| \FullFisherTilde[t]{\fullthetaBest[t]}^{-1/2} \bOmega_{0} \big( \theta_0 - \mu_0 \big) \right\|_{2}, 
    \end{align}
    and
    \begin{align} \label{eqn:full_posterior_optimal_parameter_claim2}
    \lambda_{\min}(\FullFisherTilde[t]{\fullthetaBest[t]}) \geq K_{\min} nt, \quad 
    \lambda_{\max}(\FullFisherTilde[t]{\fullthetaBest[t]}) \leq \lambda_{\max}(\bOmega_0 ) + K_{\max} nt.
    \end{align}
\end{lemma}
\begin{proof}
Let $t \in [T]$.
By the definition of $\fullthetaBest[t]$, we have 
\begin{align*}
    \bbE L_{1:t}(\theta_0) - \dfrac{1}{2} \left\| \bOmega_{0}^{1/2} \big( \theta_0 - \mu_0 \big) \right\|_{2}^{2} 
    &= \bbE \widetilde{L}_{1:t}(\theta_0)  
    \leq 
    \bbE \widetilde{L}_{1:t}(\fullthetaBest[t]) \\
    &= \bbE L_{1:t}(\fullthetaBest[t]) - \dfrac{1}{2} \left\| \bOmega_{0}^{1/2}  \big( \fullthetaBest[t] - \mu_0 \big)\right\|_{2}^{2}
    \leq 
    \bbE L_{1:t}(\fullthetaBest[t]).
\end{align*}
It follows that 
\begin{align} \label{eqn:full_posterior_optimal_parameter_eq1}
    \bbE L_{1:t}(\fullthetaBest[t]) - \bbE L_{1:t}(\theta_0) \geq - \dfrac{1}{2} \left\| \bOmega_{0}^{1/2} \big( \theta_0 - \mu_0 \big) \right\|_{2}^{2}.
\end{align}
In this proof, we denote $\Theta_{n, t} = \Theta (\theta_0, \FullFisher[t]{\theta_0}, \sqrt{2} \| \bOmega_{0}^{1/2} ( \theta_0 - \mu_0 ) \|_{2} )$. 
For $\theta \in \Theta_{n, t}$, we have
\begin{align*}
    \left\| \theta - \theta_0 \right\|_{2} 
    &\leq \lambda_{\min}^{-1/2} (\FullFisher[t]{\theta_0}) \sqrt{2}\left\| \bOmega_{0}^{1/2} \big( \theta_0 - \mu_0 \big) \right\|_{2} \\
    \overset{(\textbf{A2})}&{\leq}
    \left( K_{\min} nt \right)^{-1/2} \sqrt{2} \left\| \bOmega_{0}^{1/2} \big( \theta_0 - \mu_0 \big) \right\|_{2} 
    \overset{(\textbf{P})}{\leq}
    \big( K_{\min}^{-1/2} \sqrt{2} \delta \big) t^{-1/2} 
    \leq \dfrac{1}{4},
\end{align*}
where the last inequality holds by a small enough $\delta$.
It follows that
\begin{align} \label{eqn:full_posterior_optimal_parameter_eq2}
    \Theta_{n, t} \subset \left\{ \theta \in \Theta : \left\| \theta - \theta_0 \right\|_{2} \leq 1/2 \right\}.
\end{align}
By Lemma \ref{lemma:tech_smooth_tau_bound} and (\textbf{A2}), $\bbE L_{1:t}(\theta)$ satisfies the third order smoothness at $\theta_0$ with parameter
\begin{align*}
    \left( K_{\max} K_{\min}^{-3/2} (nt)^{-1/2}, \FullFisher[t]{\theta_0}, \: \sqrt{2} \left\| \bOmega_{0}^{1/2} \big( \theta_0 - \mu_0 \big) \right\|_{2} \right).
\end{align*}

By contradiction, we will prove that
\begin{align} \label{eqn:full_posterior_optimal_parameter_eq3}
    \left\| \FullFisher[t]{\theta_0}^{1/2} \left( \fullthetaBest[t] - \theta_0 \right) \right\|_{2} \leq \sqrt{2} \left\| \bOmega_{0}^{1/2} \big( \theta_0 - \mu_0 \big) \right\|_{2}.
\end{align}
Suppose $\fullthetaBest[t] \notin \Theta_{n, t}$.
Let 
\begin{align*}
    \partial \Theta_{n, t} = \left\{ \theta \in \Theta : \left\| \FullFisher[t]{\theta_0}^{1/2} (\theta - \theta_0) \right\|_{2}  = \sqrt{2} \left\| \bOmega_{0}^{1/2} \big( \theta_0 - \mu_0 \big) \right\|_{2} \right\}.
\end{align*}
For $\theta^{\circ} \in \partial \Theta_{n, t}$, we have
\begin{align*}
&\bbE L_{1:t}(\theta^{\circ}) - \bbE L_{1:t}(\theta_0) 
\leq \nabla \bbE L_{1:t}(\theta_0)  - 
\dfrac{1}{2} \inf_{\theta \in \Theta_{n, t}}  \left\| \FullFisher[t]{\theta}^{1/2} (\theta^{\circ} - \theta_0) \right\|_{2}^{2} \\
&= - \dfrac{1}{2} \inf_{\theta \in \Theta_{n, t}}  \left\| \FullFisher[t]{\theta}^{1/2} (\theta^{\circ} - \theta_0) \right\|_{2}^{2} \\
\overset{\text{Lemma \ref{lemma:tech_Fisher_smooth}}}&{\leq}
-\dfrac{1}{2} \left( 1 -  K_{\max} K_{\min}^{-3/2} (nt)^{-1/2} \sqrt{2} \left\| \bOmega_{0}^{1/2} \big( \theta_0 - \mu_0 \big) \right\|_{2} \right) 
\left\| \FullFisher[t]{\theta_0}^{1/2} (\theta^{\circ} - \theta_0) \right\|_{2}^{2} \\
&=
-\dfrac{1}{2} \left( 1 -  \big( \sqrt{2} K_{\max} K_{\min}^{-3/2} \big) (nt)^{-1/2}  \left\| \bOmega_{0}^{1/2} \big( \theta_0 - \mu_0 \big) \right\|_{2} \right) \cdot
2 \left\| \bOmega_{0}^{1/2} \big( \theta_0 - \mu_0 \big) \right\|_{2}^{2} \\
\overset{(\textbf{P})}&{\leq}
-\left( 1 -  \big( \sqrt{2} K_{\max} K_{\min}^{-3/2} \delta \big) t^{-1/2}  \right) 
\left\| \bOmega_{0}^{1/2} \big( \theta_0 - \mu_0 \big) \right\|_{2}^{2} \\
&<
-\dfrac{1}{2} \left\| \bOmega_{0}^{1/2} \big( \theta_0 - \mu_0 \big) \right\|_{2}^{2}.
\end{align*}
where the last inequality holds by a small enough $\delta$.
Consequently, 
\begin{align*}
\bbE L_{1:t}(\fullthetaBest[t]) - \bbE L_{1:t}(\theta_0) 
< 
-\dfrac{1}{2} \left\| \bOmega_{0}^{1/2} \big( \theta_0 - \mu_0 \big) \right\|_{2}^{2}
\end{align*}
by the concavity of the map $\theta \mapsto \bbE L_{1:t}(\theta)$, which contradicts to \eqref{eqn:full_posterior_optimal_parameter_eq1}. This completes the proof of \eqref{eqn:full_posterior_optimal_parameter_eq3}.

By \eqref{eqn:full_posterior_optimal_parameter_eq2} and \eqref{eqn:full_posterior_optimal_parameter_eq3}, we have
\begin{align*}
    &\lambda_{\min}(\FullFisherTilde[t]{\fullthetaBest[t]}) \geq \lambda_{\min}(\FullFisher[t]{\fullthetaBest[t]}) \geq K_{\min} nt, \\ 
    &\lambda_{\max}(\FullFisherTilde[t]{\fullthetaBest[t]}) \leq \lambda_{\max}(\bOmega_0 ) + \lambda_{\max}(\FullFisher[t]{\fullthetaBest[t]}) \leq \lambda_{\max}(\bOmega_0 ) + K_{\max} nt, 
\end{align*}
which completes the proof of \eqref{eqn:full_posterior_optimal_parameter_claim2}.

Next, we prove \eqref{eqn:full_posterior_optimal_parameter_claim}.
Let
\begin{align*}
    \rho_{n, t} = \left\| \FullFisherTilde[t]{\fullthetaBest[t]}^{-1/2} \bOmega_{0} \big( \theta_0 - \mu_0 \big) \right\|_{2}, \quad 
    \Theta_{t, \bias} = \Theta (\fullthetaBest[t], \FullFisherTilde[t]{\fullthetaBest[t]}, 4\rho_{n, t}).
\end{align*}
Note that
\begin{align}
\begin{aligned} \label{eqn:27} 
    \rho_{n, t} 
    &\leq 
    \lambda_{\min}^{-1/2} \big( \FullFisherTilde[t]{\fullthetaBest[t]} \big)
    \left\| \bOmega_0 \right\|_{2}^{1/2}
    \left\| \bOmega_0^{1/2} \big( \theta_0 - \mu_0 \big) \right\|_{2} \\
    \overset{(\textbf{P})}&{\leq}
    \big( K_{\min} nt \big)^{-1/2} \big( K_{\max} p_{\ast} \big)^{1/2} \big( \delta n^{1/2} \big) 
    =
    \big( K_{\min}^{-1/2} K_{\max}^{1/2} \delta \big) t^{-1/2} p_{\ast}^{1/2} \\
    &\leq 
    t^{-1/2} p_{\ast}^{1/2},
\end{aligned}    
\end{align}
where the last inequality holds by a small enough $\delta$.
For $\theta \in \Theta_{t, \bias}$, note that
\begin{align*}
    \left\| \theta - \fullthetaBest[t] \right\|_{2} \leq  (K_{\min} nt)^{-1/2} 4 \rho_{n, t} 
    \leq
    \big( 4 K_{\min}^{-1/2} \big) n^{-1/2} t^{-1} p_{\ast}^{1/2}
    \overset{(\textbf{S})}{\leq}
    \dfrac{1}{4}.
\end{align*}
Also, by $\fullthetaBest[t] \in \Theta_{n, t}$, we have $\| \fullthetaBest[t] - \theta_0 \| \leq 1/4$.
It follows that $\Theta_{t, \bias} \subseteq \Theta (\theta_0, \bI_{p}, 1/2)$, which, combining (\textbf{A2}) and Lemma \ref{lemma:tech_smooth_tau_bound}, implies that
$\bbE \widetilde{L}_{1:t}(\cdot)$ satisfies the third order smoothness at $\fullthetaBest[t]$ with 
\begin{align*}
\left( K_{\max} K_{\min}^{-3/2} (nt)^{-1/2}, \:  \FullFisherTilde[t]{\fullthetaBest[t]}, \: 4\rho_{n, t} \right). 
\end{align*}
Let $\tau_{n, t} = K_{\max} K_{\min}^{-3/2} (nt)^{-1/2}$.
Then, by (\textbf{S}), one can easily check that $\tau_{n, t} \rho_{n, t} \leq 1/16$.
By Lemma \ref{lemma:tech_pre_bias_bound} with
\begin{align*}
    \tau_{3} = \tau_{n, t}, \quad 
    f(\theta) = \bbE \widetilde{L}_{1:t} (\theta), \quad 
    \theta = \fullthetaBest[t], \quad 
    \widetilde{\theta} = \theta_0, \quad 
    \beta = \bOmega_{0} \big( \theta_0 - \mu_0 \big), \quad 
    r = \rho_{n, t},
\end{align*}
$\tau_{n, t} \rho_{n, t} \leq 1/16$ implies that
\begin{align*}
     \left\|  \FullFisherTilde[t]{\fullthetaBest[t]}^{1/2} \left( \theta_0 - \fullthetaBest[t] \right) \right\|_{2} \leq 4 \rho_{n, t},
\end{align*}
which completes the proof of \eqref{eqn:full_posterior_optimal_parameter_claim}.
\end{proof}

\begin{lemma} \label{lemma:full_posterior_parameter}
    Suppose that (\textbf{A0}), (\textbf{A1$\ast$}), (\textbf{A2}), (\textbf{S}) and (\textbf{P}) hold.
    Then, on $\scrE_{\est, 2}$,
    \begin{align*}
        \left\| \FullFisherTilde[t]{\fullthetaBest[t]}^{1/2} \left( \fullthetapMLE[t] - \fullthetaBest[t] \right) \right\|_{2} \leq 4M_n p_{\ast}^{1/2}, \quad \forall t \in [T].
    \end{align*}    
\end{lemma}

\begin{proof}
    In this proof, we work on the event $\scrE_{\est, 2}$ without explicitly mentioning it.
    The proof of this lemma is similar to that of Theorem \ref{thm:penalized_estimation}. Hence, we provide a sketch of the proof. 
    
    Let $t \in [T]$. 
    By Lemma \ref{lemma:full_posterior_optimal_parameter}, we have
    \begin{align*}
        \left\| \FullFisherTilde[t]{\fullthetaBest[t]}^{1/2} \left( \theta_0 - \fullthetaBest[t] \right) \right\|_{2} 
        &\leq 4 \left\| \FullFisherTilde[t]{\fullthetaBest[t]}^{-1/2} \bOmega_0 \big( \theta_0 - \mu_0 \big) \right\|_{2}
        \overset{ \eqref{eqn:27} }{\leq}
        4 t^{-1/2} p_{\ast}^{1/2}, \\ 
        \lambda_{\min} \big( \FullFisherTilde[t]{\fullthetaBest[t]} \big) &\geq K_{\min} (nt),
    \end{align*}
    which implies that
    \begin{align*}
        \left\| \theta_0 - \fullthetaBest \right\|_{2} 
        &\leq (K_{\min} nt)^{-1/2} 4t^{-1/2} p_{\ast}^{1/2}
        \overset{ (\textbf{S}) }{\leq}
        1/4, \\
        \widetilde{\lambda}_{1:t} 
        &= \left\| \FullFisherTilde[t]{\fullthetaBest[t]}^{-1} \bV_{1:t} \right\|_{2} 
        \leq \left\| \FullFisher[t]{\fullthetaBest[t]}^{-1} \bV_{1:t} \right\|_{2} \overset{(\textbf{A1$\ast$})}{\leq} 
        \dfrac{M_n^{2}}{9}, \\
        \widetilde{r}_{\eff, 1:t} 
        &\leq \widetilde{\lambda}_{1:t}^{1/2} \left[ p^{1/2} + \sqrt{2 (\log n + \log T)} \right] \leq 3 \widetilde{\lambda}_{1:t}^{1/2} p_{\ast}^{1/2}
        \leq M_n p_{\ast}^{1/2}.
    \end{align*}
    For all $\theta \in \Theta ( \fullthetaBest, \FullFisherTilde[t]{\fullthetaBest[t]}, 4\widetilde{r}_{\eff, 1:t} )$, it follows from the last display that
    \begin{align*}
        \left\| \theta - \fullthetaBest \right\|_{2} 
        \leq 
        (K_{\min} nt)^{-1/2} 4M_n p_{\ast}^{1/2} \overset{ (\textbf{S}) }{\leq}  1/4.
    \end{align*}
    Hence, we have
    \begin{align*}
        \Theta \left( \fullthetaBest, \FullFisherTilde[t]{\fullthetaBest[t]}, 4\widetilde{r}_{\eff, 1:t} \right) \subseteq \Theta \left( \theta_0, \bI_{p}, 1/2 \right),
    \end{align*}
    which, combining with Lemma \ref{lemma:tech_smooth_tau_bound}, implies that $\bbE \widetilde{L}_{1:t}(\theta)$ satisfies the third order smoothness at $\fullthetaBest$ with parameters
    \begin{align*}
        \bigg( K_{\max} K_{\min}^{-3/2} (nt)^{-1/2}, \: \FullFisherTilde[t]{\fullthetaBest[t]}, \: 4\widetilde{r}_{\eff, 1:t}  \bigg).
    \end{align*}
    Also,
    \begin{align*} 
        \left( K_{\max} K_{\min}^{-3/2} (nt)^{-1/2} \right) \widetilde{r}_{\eff, 1:t} 
        \leq \left( K_{\max} K_{\min}^{-3/2} (nt)^{-1/2} \right) \big( M_n p_{\ast}^{1/2} \big) 
        \overset{ (\textbf{S}) }{\leq} 1/16,
    \end{align*}
    which allows us to utilize Lemma \ref{lemma:tech_pre_bias_bound}.
    By Lemma \ref{lemma:tech_pre_bias_bound} with
    \begin{align*}
        \tau_{3} = K_{\max} K_{\min}^{-3/2} (nt)^{-1/2}, \quad 
        f(\theta) = \bbE \widetilde{L}_{1:t} (\theta), \quad 
        \theta = \fullthetaBest, \quad 
        \widetilde{\theta} = \fullthetapMLE, \quad 
        \beta = \nabla \zeta_{1:t}, \quad 
        r = \widetilde{r}_{\eff, 1:t},
    \end{align*}
    we have on $\scrE_{\est, 2}$
    \begin{align*}
         \left\|  \FullFisherTilde[t]{\fullthetaBest[t]}^{1/2} \left( \fullthetapMLE - \fullthetaBest \right) \right\|_{2} \leq 4 \widetilde{r}_{\eff, 1:t} \leq 4M_n p_{\ast}^{1/2}, \quad \forall t \in [T],
    \end{align*}        
    which completes the proof. 
\end{proof}

\begin{proof}[Proof of Theorem \ref{thm:batch_posterior_LA}]
    In this proof, we work on the event $\scrE_{\est, 2}$ without explicitly mentioning it. The proof of this theorem is similar to the proofs of Theorem \ref{thm:LA_TV} and Proposition \ref{prop:eigenvalue_order}. Hence, we provide a sketch of the proof.
    
    Let $t \in [T]$. 
    Note that
    \begin{align}
    \begin{aligned} \label{eqn:batch_posterior_LA_eq0}
        \left\| \FullFisherTilde[t]{\fullthetaBest[t]}^{1/2} \big( \theta_0 - \fullthetaBest[t] \big) \right\|_{2} 
        &\leq 4 \left\| \FullFisherTilde[t]{\fullthetaBest[t]}^{-1/2} \bOmega_0 \big( \theta_0 - \mu_0 \big) \right\|_{2}
        \overset{ \eqref{eqn:27} }{\leq}
        4 t^{-1/2} p_{\ast}^{1/2}, \\ 
        \left\| \FullFisherTilde[t]{\fullthetaBest[t]}^{1/2} \big( \fullthetapMLE - \fullthetaBest \big) \right\|_{2} 
        &\leq 4M_n p_{\ast}^{1/2},
    \end{aligned}
    \end{align}
    which, combining with (\textbf{S}), implies that
    \begin{align} \label{eqn:26}
        \big\| \theta_0 - \fullthetaBest[t] \big\|_{2} \leq 1/8, \quad 
        \big\| \fullthetapMLE - \fullthetaBest \big\|_{2} \leq 1/8.
    \end{align}
    From the last display, we have $\| \theta_0 - \fullthetapMLE \|_{2} \leq 1/4$, which implies that
    \begin{align} \label{eqn:batch_posterior_LA_eq1}
        \lambda_{\min} ( \FullFisherTilde[t]{\fullthetapMLE[t]} ) 
        \geq \lambda_{\min} ( \bOmega_0 ) + \lambda_{\min} ( \FullFisher[t]{\fullthetapMLE[t]} ) 
        \geq \lambda_{\min} ( \FullFisher[t]{\fullthetapMLE[t]} ) \overset{(\textbf{A2})}{\geq} K_{\min} nt.
    \end{align}
    Note that $r_{\LA} = 2\sqrt{p} + \sqrt{2 \log N} \leq  4 \sqrt{p_{\ast}}$.
    Also, we have
    \begin{align*}
        \sup_{\theta \in \Theta (\fullthetapMLE, \FullFisherTilde[t]{\fullthetapMLE[t]}, 4r_{\LA} )}
        \big\| \theta - \fullthetapMLE \big\|_{2} 
        \leq \lambda_{\min}^{-1/2} ( \FullFisherTilde[t]{\fullthetapMLE[t]} )  4 r_{\LA}
        \leq ( K_{\min} nt )^{-1/2} 16 \sqrt{p_{\ast}} 
        \overset{(\textbf{S})}{\leq} 1/4, 
    \end{align*}
    which implies that $\Theta (\fullthetapMLE, \FullFisherTilde[t]{\fullthetapMLE[t]}, 4r_{\LA} ) \subseteq \Theta (\theta_0, \bI_{p}, 1/2)$. 
    By Lemma \ref{lemma:tech_smooth_tau_bound} and (\textbf{A2}), $\widetilde{L}_{1:t}(\cdot)$ satisfies the third and fourth order smoothness at $\fullthetapMLE$ with parameter
    \begin{align*} 
        \left( K_{\max} K_{\min}^{-3/2} (nt)^{-1/2}, \FullFisherTilde[t]{\fullthetapMLE[t]}, 4 r_{\LA} \right) \quad \text{ and } \quad 
        \left( K_{\max} K_{\min}^{-2} (nt)^{-1}, \FullFisherTilde[t]{\fullthetapMLE[t]}, 4 r_{\LA} \right), 
    \end{align*}
    respectively.
    In this proof, let
    \begin{align*}
        \tau_{3, t} = K_{\max} K_{\min}^{-3/2} (nt)^{-1/2}, \quad 
        \tau_{4, t} = K_{\max} K_{\min}^{-2} (nt)^{-1}.
    \end{align*}
    Then, we can apply the proof strategy in Theorem \ref{thm:LA_TV}, which implies that
    \begin{align*}
        &d_{V} \bigg( \cN\left(\fullthetapMLE, \FullFisherTilde[t]{\fullthetapMLE[t]} \right),  \Pi\left(\cdot \mid \bD_{1:t} \right) \bigg) \\
        &\leq 
        c_1
        \bigg( 
            \left[ \tau_{4, t} + \tau_{3, t}^2 \right] p^2 
            +
            \tau_{3, t} p
            +
            \tau_{3, t}^{3} \log^{3}N
            +
            e^{-8\log N - 8p}
        \bigg) \\
        &\leq 
        c_2
        \bigg( 
            (nt)^{-1} p^2 
            +
            (nt)^{-1/2} p
            +
            (nt)^{-3/2} \log^{3} N
            +
            e^{-8\log N - 8p}
        \bigg) \\       
        \overset{(\textbf{S})}&{\leq}
        c_3 \sqrt{\dfrac{p_{\ast}^{2}}{nt}}
    \end{align*}
    for some constants $c_1, c_2, c_3 > 0$, depending only on $(K_{\min}, K_{\max})$.
\end{proof}

\begin{proof}[Proof of Theorem \ref{thm:batch_posterior_BvM}]
In this proof, we work on the event $\scrE_{\est, 2}$ without explicitly mentioning it.
Let $t \in [T]$.
To complete this proof, we utilize Lemma \ref{lemma:tech_Gaussian_comparison}.
Hence, we need to obtain upper bounds of the following quantities:
\begin{align*}
    {\rm (i)} = \left\| \FullFisherTilde[t]{\fullthetapMLE}^{1/2} \left( \fullthetaMLE - \fullthetapMLE \right) \right\|_{2}, \quad 
    {\rm (ii)} = \left\| \FullFisher[t]{\theta_{0}}^{-1/2} \FullFisherTilde[t]{\fullthetapMLE} \FullFisher[t]{\theta_{0}}^{-1/2} - \bI_{p} \right\|_{\rm F}. 
\end{align*}
\textbf{Step 1: ${\rm (i)}$} \\
Firstly, we will obtain an upper bound of ${\rm (i)}$. To complete the proof, we will show that
\begin{align*} 
     \left\|  \FullFisher[t]{\theta_0}^{1/2} \left( \fullthetaMLE - \theta_0 \right) \right\|_{2} 
     \leq 4M_n p_{\ast}^{1/2}
\end{align*} 
by utilizing Lemma \ref{lemma:tech_pre_bias_bound}. Note that required proof for the last display is similar to that of Theorem \ref{thm:penalized_estimation}. Hence, in \textbf{Step 1}, we provide a sketch of the proof.
Note that
\begin{align*}
        \left\| \FullFisher[t]{\theta_{0}}^{-1} \bV_{1:t} \right\|_{2} \overset{(\textbf{A1$\ast$})}&{\leq} \dfrac{M_n^{2}}{9}, \\ 
        \lambda_{\min}(\FullFisher[t]{\theta_0}) \overset{(\textbf{A2})}&{\geq} K_{\min} (nt) \\
        r_{\eff, 1:t} &\leq \lambda_{1:t}^{1/2}\left[ p^{1/2} + \sqrt{2 (\log n + \log T)} \right] \leq 3(\lambda_{1:t} p_{\ast})^{1/2} 
        \leq M_n p_{\ast}^{1/2},
\end{align*}
which, combining with (\textbf{S}), implies that
\begin{align*}
    \Theta \left( \theta_0, \FullFisher[t]{\theta_0}, 4r_{\eff, 1:t} \right) \subseteq \Theta \left( \theta_0, \bI_{p}, 1/2 \right).
\end{align*}
By Lemma \ref{lemma:tech_smooth_tau_bound}, the last display implies that $\bbE L_{1:t}(\theta)$ satisfies the third order smoothness at $\theta_0$ with parameters
\begin{align*}
    \left( K_{\max} K_{\min}^{-3/2} (nt)^{-1/2}, \: \FullFisher[t]{\theta_0}, \: 4r_{\eff, 1:t}  \right).
\end{align*}
Also,
\begin{align}
\begin{aligned} \label{eqn:batch_posterior_BvM_eq1}
    \big( K_{\max} K_{\min}^{-3/2} (nt)^{-1/2} \big) r_{\eff, 1:t} 
    &\leq 
    \big( K_{\max} K_{\min}^{-3/2} \big) \big( M_n p_{\ast}^{1/2} (nt)^{-1/2} \big) 
    \overset{(\textbf{S})}{\leq} 1/16, \\
    \left\| \bF_{1:t, \theta_0}^{-1/2} \nabla \zeta_{1:t} \right\|_{2} &\leq r_{\eff, 1:t},
\end{aligned}
\end{align}
where the second inequality holds on $\scrE_{\eff, 2}$ by the condition \eqref{assume:A1ast_1}.
Note that \eqref{eqn:batch_posterior_BvM_eq1} allows us to utilize Lemma \ref{lemma:tech_pre_bias_bound}.
By Lemma \ref{lemma:tech_pre_bias_bound} with
\begin{align*}
    \tau_{3} = K_{\max} K_{\min}^{-3/2} (nt)^{-1/2}, \quad 
    f(\theta) = \bbE L_{1:t} (\theta), \quad 
    \theta = \theta_0, \quad 
    \widetilde{\theta} = \fullthetaMLE, \quad 
    \beta = \nabla \zeta_{1:t}, \quad 
    r = r_{\eff, 1:t},
\end{align*}
we have
\begin{align} \label{eqn:batch_posterior_BvM_eq2} 
     \left\|  \FullFisher[t]{\theta_0}^{1/2} \big( \fullthetaMLE - \theta_0 \big) \right\|_{2} 
     \leq 4 r_{\eff, 1:t} 
     \leq 4M_n p_{\ast}^{1/2},
\end{align}        
which implies $\| \fullthetaMLE - \theta_0 \|_{2} \leq 1/2$ by (\textbf{S}).

By \eqref{eqn:full_posterior_optimal_parameter_claim2}, \eqref{eqn:batch_posterior_LA_eq1} and (\textbf{A2}), we have
\begin{align} \label{eqn:batch_posterior_BvM_eq3}
\lambda_{\min} (\FullFisherTilde[t]{\fullthetapMLE}) \wedge \lambda_{\min} (\FullFisherTilde[t]{\fullthetaBest}) \wedge \lambda_{\min} (\FullFisher[t]{\theta_0}) \wedge \lambda_{\min} (\FullFisher[t]{\fullthetaMLE})
\geq
K_{\min} (nt). 
\end{align}
It follows that
\begin{align}
\begin{aligned} \label{eqn:batch_posterior_BvM_eq3_2}
&\left\| \FullFisherTilde[t]{\fullthetapMLE}^{-1/2} \bOmega_{0} \big( \fullthetaMLE - \mu_0 \big) \right\|_{2} \\
&\leq 
    \left\| \FullFisherTilde[t]{\fullthetapMLE}^{-1/2} \bOmega_{0} \big( \fullthetaMLE - \theta_0 \big) \right\|_{2} 
    + \left\| \FullFisherTilde[t]{\fullthetapMLE}^{-1/2} \bOmega_{0} \big( \theta_0 - \mu_0 \big) \right\|_{2} \\
&=
    \left\| \FullFisherTilde[t]{\fullthetapMLE}^{-1/2} \bOmega_{0} \FullFisher[t]{\theta_0}^{-1/2}
    \FullFisher[t]{\theta_0}^{1/2} \big( \fullthetaMLE - \theta_0 \big) \right\|_{2} 
    + \left\| \FullFisherTilde[t]{\fullthetapMLE}^{-1/2} \bOmega_{0} \big( \theta_0 - \mu_0 \big) \right\|_{2} \\
&\leq  
    \left( K_{\min} nt \right)^{-1} 
    \left\| \bOmega_0 \right\|_{2} 
    \left\| \FullFisher[t]{\theta_0}^{1/2} \big( \fullthetaMLE - \theta_0 \big) \right\|_{2} \\
&\qquad + 
    \left( K_{\min} nt \right)^{-1/2} 
    \left\| \bOmega_0 \right\|_{2}^{1/2}
    \left\| \bOmega_{0}^{1/2} \big( \theta_0 - \mu_0 \big) \right\|_{2} \\
&\leq
    \left( K_{\min} nt \right)^{-1} 
    \left( K_{\max} p_{\ast} \right)
    \left( 4 M_n p_{\ast}^{1/2} \right) \\
    &\qquad + 
    \left( K_{\min} nt \right)^{-1/2} 
    \left( K_{\max} p_{\ast} \right)^{1/2}
    \big( K_{\max} M_n p_{\ast}^{1/2} \big)
    \\    
\overset{(\textbf{S})}&{\leq}
    \big( 1 +  K_{\min}^{-1/2} K_{\max}^{3/2} \big)
    M_n \left( \dfrac{p_{\ast}^{2}}{nt} \right)^{1/2} = c_1 M_n \left( \dfrac{p_{\ast}^{2}}{nt} \right)^{1/2},
\end{aligned}
\end{align}
where $c_1 = 1 +  K_{\min}^{-1/2} K_{\max}^{3/2}$, and the third inequality holds by \eqref{eqn:batch_posterior_BvM_eq2}, \eqref{eqn:batch_posterior_BvM_eq3} and (\textbf{P$\ast$}).
In this proof, let 
$$
b_{t} = \left\| \FullFisherTilde[t]{\fullthetapMLE}^{-1/2} \bOmega_{0} \big( \fullthetaMLE - \mu_0 \big) \right\|_{2}.
$$
Then, for $\theta \in \Theta ( \fullthetapMLE, \FullFisherTilde[t]{\fullthetapMLE}, 4b_{t})$, note that
\begin{align*}
    \big\| \theta - \theta_0 \big\|_{2} 
    &\leq \big\| \theta - \fullthetapMLE \big\|_{2} 
        + \big\| \fullthetapMLE - \theta_0 \big\|_{2} 
    \leq \lambda_{\min}^{-1/2} \big( \FullFisherTilde[t]{\fullthetapMLE} \big) 4 b_t
        + \big\| \fullthetapMLE - \theta_0 \big\|_{2} \\
    &\leq (K_{\min} nt)^{-1/2} \bigg[ 4c_1 M_n \left( \dfrac{p_{\ast}^{2}}{nt} \right)^{1/2} \bigg]
    + 1/4
    \overset{(\textbf{S})}{\leq} 1/2,
\end{align*}
where the third inequality holds by \eqref{eqn:26}, \eqref{eqn:batch_posterior_BvM_eq3} and \eqref{eqn:batch_posterior_BvM_eq3_2}. 
It follows that $\Theta ( \fullthetapMLE, \FullFisherTilde[t]{\fullthetapMLE}, 4b_{t}) \subseteq \Theta (\theta_0, \bI_{p}, 1/2)$. 
By Lemma \ref{lemma:tech_smooth_tau_bound}, $\widetilde{L}_{1:t} (\theta)$ satisfies the third order smoothness at $\fullthetapMLE$ with parameters
\begin{align*}
    \left( K_{\max} K_{\min}^{-3/2} (nt)^{-1/2}, \: \FullFisherTilde[t]{\fullthetapMLE}, \: 4b_{t} \right).
\end{align*}
Also,
\begin{align*} 
    \big( K_{\max} K_{\min}^{-3/2} (nt)^{-1/2} \big) b_{t} 
    \leq 
    K_{\max} K_{\min}^{-3/2} (nt)^{-1/2} \bigg[ c_1 M_n\left( \dfrac{p_{\ast}^{2}}{nt} \right)^{1/2} \bigg]   
    \overset{(\textbf{S})}{\leq}
    1/16,
\end{align*} 
which allows us to utilize Lemma \ref{lemma:tech_pre_bias_bound}.
By Lemma \ref{lemma:tech_pre_bias_bound} with
\begin{align*}
    &\tau_{3} = K_{\max} K_{\min}^{-3/2} (nt)^{-1/2}, \quad 
    f(\theta) = \widetilde{L}_{1:t} (\theta), \quad 
    \theta = \fullthetapMLE, \quad 
    \widetilde{\theta} = \fullthetaMLE, \\ 
    &\beta = \bOmega_{0} \big( \fullthetaMLE - \mu_0 \big), \quad 
    r = b_{t},
\end{align*}
we have
\begin{align} \label{eqn:batch_posterior_BvM_eq4_2_1}
     \left\| \FullFisherTilde[t]{\fullthetapMLE}^{1/2} \left( \fullthetaMLE - \fullthetapMLE \right) \right\|_{2} 
     \leq 
     4 b_{t} \leq 4c_1 M_n\left( \dfrac{p_{\ast}^{2}}{nt} \right)^{1/2}.
\end{align}     

\noindent \textbf{Step 2: ${\rm (ii)}$} \\
Next, we will obtain an upper bound of ${\rm (ii)}$. 
By Lemmas \ref{lemma:full_posterior_optimal_parameter} and \ref{lemma:full_posterior_parameter}, we have
\begin{align}
\begin{aligned} \label{eqn:batch_posterior_BvM_eq4_2}
    \left\| \FullFisherTilde[t]{\fullthetaBest}^{1/2} \big( \theta_0 - \fullthetaBest \big) \right\|_{2}
    &\leq 
    4 \left\| \FullFisherTilde[t]{\fullthetaBest}^{-1/2} \bOmega_{0} \big( \theta_0 - \mu_0 \big) \right\|_{2} 
    \overset{\eqref{eqn:27}}{\leq} 
    4 t^{-1/2} p_{\ast}^{1/2} \\
    &\leq 
    4 M_n p_{\ast}^{1/2}, \\
    \left\| \FullFisherTilde[t]{\fullthetaBest}^{1/2} \big( \fullthetapMLE - \fullthetaBest \big) \right\|_{2}
    &\leq 4 M_n p_{\ast}^{1/2}.
\end{aligned}
\end{align}
Combining (\textbf{S}), \eqref{eqn:batch_posterior_BvM_eq3} and \eqref{eqn:batch_posterior_BvM_eq4_2}, one can easily check that
\begin{align*}
    \Theta (\fullthetaBest, \FullFisherTilde[t]{\fullthetaBest}, 4M_n p_{\ast}^{1/2}) &\subseteq \Theta(\theta_{0}, \bI_{p}, 1/2), \\
    \Theta (\fullthetaBest, \FullFisher[t]{\fullthetaBest}, 10M_n p_{\ast}^{1/2}) &\subseteq \Theta(\theta_{0}, \bI_{p}, 1/2).
\end{align*}
Hence, by Lemma \ref{lemma:tech_smooth_tau_bound}, $\widetilde{L}_{1:t}(\cdot)$ satisfies the third order smoothness at $\fullthetaBest$ with parameters
\begin{align*}
    \left( K_{\max} K_{\min}^{-3/2} (nt)^{-1/2}, \: \FullFisherTilde[t]{\fullthetaBest}, \: 4M_n p_{\ast}^{1/2} \right).
\end{align*}
Since $\fullthetapMLE \in \Theta (\fullthetaBest, \FullFisherTilde[t]{\fullthetaBest}, 4M_n p_{\ast}^{1/2})$ by \eqref{eqn:batch_posterior_BvM_eq4_2}, we have
\begin{align*}
    &\left\| \FullFisherTilde[t]{\fullthetapMLE}^{1/2} \big( \theta_0 - \fullthetapMLE \big) \right\|_{2}
    \leq 
    \left\| \FullFisherTilde[t]{\fullthetapMLE}^{1/2} \FullFisherTilde[t]{\fullthetaBest}^{-1/2}  \right\|_{2}
    \left\| \FullFisherTilde[t]{\fullthetaBest}^{1/2} \big( \theta_0 - \fullthetapMLE \big) \right\|_{2} \\
    \overset{\text{Lemma \ref{lemma:tech_Fisher_smooth}}}&{\leq} 
    \bigg( 1 +  K_{\max} K_{\min}^{-3/2} (nt)^{-1/2} (4M_n p_{\ast}^{1/2}) \bigg)^{1/2}
    \left\| \FullFisherTilde[t]{\fullthetaBest}^{1/2} \big( \theta_0 - \fullthetapMLE \big) \right\|_{2} \\    
    \overset{\eqref{eqn:batch_posterior_BvM_eq4_2}}&{\leq}
    \bigg( 1 +  K_{\max} K_{\min}^{-3/2} (nt)^{-1/2} (4M_n p_{\ast}^{1/2}) \bigg)^{1/2} 
    8M_n p_{\ast}^{1/2}
    \overset{(\textbf{S})}{\leq} 10 M_n p_{\ast}^{1/2},
\end{align*}
which implies that
\begin{align*}
    \theta_0 
    \in 
    \Theta \left( \fullthetapMLE, \FullFisherTilde[t]{\fullthetapMLE}, 10 M_n p_{\ast}^{1/2} \right)
    \subset
    \Theta \left( \fullthetapMLE, \FullFisher[t]{\fullthetapMLE}, 10 M_n p_{\ast}^{1/2} \right)
    \subseteq
    \Theta \left( \theta_0, \bI_{p}, 1/2 \right).
\end{align*}
Consequently, by Lemma \ref{lemma:tech_smooth_tau_bound}, $L_{1:t}(\cdot)$ satisfies the third order smoothness at $\fullthetapMLE$ with parameters
\begin{align*}
    \left( K_{\max} K_{\min}^{-3/2} (nt)^{-1/2}, \: \FullFisher[t]{\fullthetapMLE}, \: 10 M_n p_{\ast}^{1/2} \right).
\end{align*}
Hence, combining with the last two displays, Lemma \ref{lemma:tech_Fisher_smooth} and (\textbf{A2}) give that
\begin{align} \label{eqn:batch_posterior_BvM_eq5}
\begin{aligned}
    \left\| \FullFisher[t]{\fullthetapMLE[t]}^{-1/2} \FullFisher[t]{\theta_0} \FullFisher[t]{\fullthetapMLE[t]}^{-1/2} - \bI_{p} \right\|_{2} 
    &\leq 
    K_{\max} K_{\min}^{-3/2} (nt)^{-1/2} (10 M_n p_{\ast}^{1/2}), \\
    \left\| \FullFisher[t]{\fullthetapMLE[t]} \right\|_{2} 
    &\leq 
    K_{\max} (nt).
\end{aligned}
\end{align}
Note that
\begin{align*}
    &\left\| \FullFisher[t]{\theta_0} - \FullFisherTilde[t]{\fullthetapMLE}  \right\|_{2}
    \leq 
    \left\| \bOmega_{0} \right\|_{2} + \left\| \FullFisher[t]{\theta_0} - \FullFisher[t]{\fullthetapMLE[t]}  \right\|_{2} \\
    &\leq 
    \left\| \bOmega_{0} \right\|_{2} + 
    \left\| \FullFisher[t]{\fullthetapMLE[t]} \right\|_{2} 
    \left\| \FullFisher[t]{\fullthetapMLE[t]}^{-1/2} \FullFisher[t]{\theta_0} \FullFisher[t]{\fullthetapMLE[t]}^{-1/2} - \bI_{p} \right\|_{2} \\
    \overset{\substack{\eqref{eqn:batch_posterior_BvM_eq5}, (\textbf{P$\ast$})}}&{\leq}
    K_{\max} p_{\ast} + 
    \left( K_{\max} nt \right) \left( K_{\max} K_{\min}^{-3/2} (nt)^{-1/2} ( 10 M_n p_{\ast}^{1/2}) \right) \\
    \overset{(\textbf{S})}&{\leq} 
    \big( 1 + 10 K_{\max}^{2} K_{\min}^{-3/2} \big) M_n (nt)^{1/2} p_{\ast}^{1/2}.
\end{align*}
It follows that
\begin{align*}
    \left\| \FullFisher[t]{\theta_0}^{-1/2} \FullFisherTilde[t]{\fullthetapMLE} \FullFisher[t]{\theta_0}^{-1/2} - \bI_{p} \right\|_{2}
    &\leq 
    \lambda_{\min}^{-1} (\FullFisher[t]{\theta_0})
    \left\| \FullFisher[t]{\theta_0} - \FullFisherTilde[t]{\fullthetapMLE} \right\|_{2} \\
    &\leq 
    (K_{\min} nt)^{-1}
    \bigg[ 1 + 10 K_{\max}^{2} K_{\min}^{-3/2} \bigg] M_n (nt)^{1/2} p_{\ast}^{1/2} \\
    &=
    \big( K_{\min}^{-1} + 10 K_{\max}^{2} K_{\min}^{-5/2} \big) M_n \left( \dfrac{p_{\ast}}{nt} \right)^{1/2},
\end{align*}
which further implies that
\begin{align}
\begin{aligned} \label{eqn:batch_posterior_BvM_eq6}
    \left\| \FullFisher[t]{\theta_0}^{-1/2} \FullFisherTilde[t]{\fullthetapMLE} \FullFisher[t]{\theta_0}^{-1/2} - \bI_{p} \right\|_{\rm F}
    &\leq 
    \sqrt{p} \left\| \FullFisher[t]{\theta_0}^{-1/2} \FullFisherTilde[t]{\fullthetapMLE} \FullFisher[t]{\theta_0}^{-1/2} - \bI_{p} \right\|_{2} \\
    &\leq 
    \big( K_{\min}^{-1} + 10 K_{\max}^{2} K_{\min}^{-5/2} \big) M_n \left( \dfrac{p_{\ast}^{2}}{nt} \right)^{1/2} \\
    &= 
    c_2 M_n \left( \dfrac{p_{\ast}^{2}}{nt} \right)^{1/2},
\end{aligned}
\end{align}
where $c_2 = K_{\min}^{-1} + 10 K_{\max}^{2} K_{\min}^{-5/2}$.

\noindent \textbf{Step 3: Applying Lemma \ref{lemma:tech_Gaussian_comparison}} \\
By (\textbf{S}), we have
\begin{align*}
    \left\| \FullFisher[t]{\theta_0}^{-1/2} \FullFisherTilde[t]{\fullthetapMLE} \FullFisher[t]{\theta_0}^{-1/2} - \bI_{p} \right\|_{2} 
    \leq 0.684.
\end{align*}
By Lemma \ref{lemma:tech_Gaussian_comparison}, we have
\begin{align*}
    &d_{V} \bigg( \cN\left( \fullthetaMLE, \FullFisher[t]{\theta_0}^{-1} \right),  
    \cN \left(\fullthetapMLE, \FullFisherTilde[t]{\fullthetapMLE[t]}^{-1} \right) \bigg) \\
    &\leq 
    \dfrac{1}{2}
    \bigg(
    \left\| \FullFisherTilde[t]{\fullthetapMLE}^{1/2} \left( \fullthetaMLE - \fullthetapMLE \right) \right\|_{2}^{2}
    +
    \left\| \FullFisher[t]{\theta_0}^{-1/2} \FullFisherTilde[t]{\fullthetapMLE} \FullFisher[t]{\theta_0}^{-1/2} - \bI_{p} \right\|_{\rm F}^{2}
    \bigg)^{1/2} \\
    \overset{\substack{ \eqref{eqn:batch_posterior_BvM_eq4_2_1} \\ \eqref{eqn:batch_posterior_BvM_eq6}}}&{\leq}
    \dfrac{1}{2} \big( 16c_1^2 + c_2^2 \big)^{1/2} 
    M_n \left( \dfrac{p_{\ast}^{2}}{nt} \right)^{1/2} 
    =
    c_3 
    M_n \left( \dfrac{p_{\ast}^{2}}{nt} \right)^{1/2}
\end{align*}
for some constant $c_3 = c_3(K_{\min}, K_{\max}) > 0$.
This completes the proof of the first assertion in \eqref{eqn:batch_posterior_BvM_claim}.

Combining with Theorem \ref{thm:batch_posterior_LA}, we have
\begin{align*}
&d_{V} \bigg( \cN\left(\fullthetaMLE, \FullFisher[t]{\theta_0}^{-1} \right),  \Pi\left(\cdot \mid \bD_{1:t} \right) \bigg) \\
&\leq 
d_{V} \bigg( \cN\left( \fullthetaMLE, \FullFisher[t]{\theta_0}^{-1} \right),  
\cN\left(\fullthetapMLE, \FullFisherTilde[t]{\fullthetapMLE[t]}^{-1} \right) \bigg)
+
d_{V} \bigg( \cN\left(\fullthetapMLE, \FullFisherTilde[t]{\fullthetapMLE[t]}^{-1} \right),  
\Pi\left(\cdot \mid \bD_{1:t} \right) \bigg) \\
&\leq 
c_3 M_n \left( \dfrac{p_{\ast}^{2}}{nt} \right)^{1/2}
+
c_4 \left( \dfrac{p_{\ast}^2}{nt} \right)^{1/2} 
\leq 
c_5 M_n \left( \dfrac{p_{\ast}^{2}}{nt} \right)^{1/2}
\end{align*}
for some constants $c_4, c_5 > 0$, depending only on $(K_{\min}, K_{\max})$.
This completes the proof of the second assertion in \eqref{eqn:batch_posterior_BvM_claim}.
\end{proof}


\section{Proofs for Section \ref{sec:online_variational_posterior}}

\begin{proof}[Proof of Proposition \ref{prop:similar_MAP}]
In this proof, we work on the event $\scrE_{\est, 1} \cap \scrE_{\est, 2}$ without explicitly referring to it. 
By Lemmas \ref{lemma:full_posterior_optimal_parameter}, \ref{lemma:full_posterior_parameter} and Proposition \ref{prop:eigenvalue_order}, we have
\begin{align}
\begin{aligned} \label{eqn:similar_MAP_regularity}
    &\left\| \FullFisherTilde[t]{\fullthetaBest[t]}^{1/2} \big( \theta_{0} - \fullthetaBest[t] \big) \right\|_{2} 
    \leq 4M_n p_{\ast}^{1/2}, &\quad
    &\lambda_{\min}(\FullFisherTilde[t]{\fullthetaBest[t]}) 
    \geq K_{\min} nt \\
    &\left\| \FullFisherTilde[t]{\fullthetaBest[t]}^{1/2} \big( \fullthetapMLE[t] - \fullthetaBest[t] \big) \right\|_{2} 
    \leq 4M_n p_{\ast}^{1/2}, &\quad
    &\lambda_{\max}(\FullFisherTilde[t]{\fullthetaBest[t]}) 
    \leq \dfrac{4}{3} K_{\max} nt, \\
    &\lambda_{\min}(\FisherMAP[t]) \wedge \lambda_{\min}(\FisherBest[t]) 
    \geq K_{\rm low} nt, &\quad
    &\lambda_{\max}(\FisherMAP[t]) \wedge \lambda_{\max}(\FisherBest[t]) 
    \leq K_{\rm up} nt, \\
    &\thetaMAP, \thetaBest, \fullthetaBest, \fullthetapMLE 
    \in \Theta(\theta_{0}, \bI_{p}, 1/4), &\quad
    &\Delta_{t} \vee \epsilon_{n, t, \KL} 
    \leq K_{\rm up} t^{-1} n^{-1/2} p_{\ast}
\end{aligned}
\end{align}
for all $t \in [T]$.

By utilizing Lemma \ref{lemma:tech_pre_bias_bound} and an inductive argument, we prove the following inequalities with some constants $D_{1}, D_{2} > 0$:
\begin{align} \label{eqn:similar_MAP_inner_claim_0}
    \left\| \FisherMAP[t]^{1/2} \big( \fullthetapMLE - \thetaMAP[t] \big) \right\|_{2}
    \leq D_{1} M_n^{2} \left( \dfrac{p_{\ast}^{3}}{n} \right)^{1/2}, \quad
    \left\| \FisherMAP[t]^{1/2} \big( \theta_0 - \thetaMAP[t] \big) \right\|_{2}
    \leq D_{2} M_n \sqrt{p_{\ast}}
\end{align}
for all $t \in [T]$. Based on \eqref{eqn:similar_MAP_inner_claim_0}, we subsequently prove the following inequality:
\begin{align*}
    \left\| \bOmega_{t}^{1/2} \big( \fullthetapMLE - \mu_{t} \big) \right\|_{2} \leq D_{3} M_n^{2} \left( \dfrac{p_{\ast}^{3}}{n} \right)^{1/2} \quad 
    \text{ for all } t \in [T]
\end{align*}
for some constant $D_{3} = D_{3}(D_{1})$.

\noindent \textbf{Step 1: Inductive argument} \\
We will first show that \eqref{eqn:similar_MAP_inner_claim_0} holds for $t = 1$.
Note that
\begin{align*}
    \left\| \FisherMAP[1]^{1/2} \big( \fullthetapMLE[1] - \thetaMAP[1] \big) \right\|_{2} = 0
\end{align*}
because $\fullthetapMLE[t] = \thetaMAP[t]$ at $t = 1$. 
Also, we have
\begin{align*}
    \left\| \FisherMAP[1]^{1/2} \big( \theta_0 - \thetaMAP[1] \big) \right\|_{2}
    &=
    \left\| \FisherMAP[1]^{1/2} \FullFisherTilde[1]{\fullthetaBest[1]}^{-1/2} \FullFisherTilde[1]{\fullthetaBest[1]}^{1/2} \big( \theta_0 - \fullthetapMLE[1] \big) \right\|_{2} \\
    &\leq 
    \left\| \FisherMAP[1]^{1/2} \FullFisherTilde[1]{\fullthetaBest[1]}^{-1/2} \right\|_{2}
    \Bigg(
    \left\| \FullFisherTilde[1]{\fullthetaBest[1]}^{1/2} \big( \theta_0 - \fullthetaBest[1] \big) \right\|_{2} 
    +
    \left\| \FullFisherTilde[1]{\fullthetaBest[1]}^{1/2} \big( \fullthetaBest[1] - \fullthetapMLE[1] \big) \right\|_{2} 
    \Bigg) \\
    \overset{ \eqref{eqn:similar_MAP_regularity} }&{\leq}
    (K_{\min} n)^{-1/2} (K_{\rm up} n)^{1/2} \big( 4M_n \sqrt{p_{\ast}} + 4M_n \sqrt{p_{\ast}} \big) \\
    &= 
    \big( 8 K_{\min}^{-1/2} K_{\rm up}^{1/2} \big) M_n \sqrt{p_{\ast}}.
\end{align*}
Hence, for $t = 1$, the inequalities in \eqref{eqn:similar_MAP_inner_claim_0} hold with $D_1 = 0$ and $D_2 = 8 K_{\min}^{-1/2} K_{\rm up}^{1/2}$. 

Let $t_0 \in \{ 2, 3, ..., T \}$.
To prove \eqref{eqn:similar_MAP_inner_claim_0} by induction, suppose that 
\begin{align}
\begin{aligned} \label{eqn:similar_MAP_induction}
    \left\| \FisherMAP[t]^{1/2} \big( \theta_0 - \thetaMAP[t] \big) \right\|_{2}
    \leq 
    \big( 8K_{\min}^{-1/2} K_{\rm up}^{1/2} + 1 \big) M_n \sqrt{p_{\ast}}, \quad 
    \forall t \in [t_0 - 1].
\end{aligned}    
\end{align}
Based on \eqref{eqn:similar_MAP_induction}, we will show that
\begin{align}
\begin{aligned} \label{eqn:similar_MAP_inner_claim}
    \left\| \FisherMAP[t_0]^{1/2} \big( \fullthetapMLE[t_0] - \thetaMAP[t_0] \big) \right\|_{2}
    &\leq 
    D_{1} M_n^{2} \left( \dfrac{p_{\ast}^{3}}{n} \right)^{1/2}, \\
    \left\| \FisherMAP[t_0]^{1/2} \big( \theta_0 - \thetaMAP[t_0] \big) \right\|_{2}
    &\leq 
    \big( 8K_{\min}^{-1/2} K_{\rm up}^{1/2} + 1 \big) M_n \sqrt{p_{\ast}},
\end{aligned}    
\end{align}
where $D_{1} = D_{1}(K_{\min}, K_{\max}, K_{\rm low}, K_{\rm up})$.
It then follows by induction that \eqref{eqn:similar_MAP_inner_claim_0} holds for all $t \in [T]$.

The proof is divided into several steps. In \textbf{Step 2}, we introduce the theoretical framework needed to prove \eqref{eqn:similar_MAP_inner_claim}. 
Then, in \textbf{Step 3-6}, we will prove that \eqref{eqn:similar_MAP_inner_claim} holds for every $t_0 \in \{ 2, 3, ..., T \}$.

\noindent \textbf{Step 2: Framework for applying Lemma \ref{lemma:tech_pre_bias_bound}} \\
To prove \eqref{eqn:similar_MAP_inner_claim}, we will utilize Lemma \ref{lemma:tech_pre_bias_bound}. In \textbf{Step 2}, therefore, we introduce some theoretical preliminaries needed to apply Lemma \ref{lemma:tech_pre_bias_bound}.

Recall the definition of $\nabla \eta_{t}(\theta)$ given in \eqref{def:eta_approx}.
For $\theta \in \Theta$ and $t \in [T]$, note that
\begin{align} \label{eqn:similar_MAP_eq1}
\begin{aligned}
&\nabla \eta_{t}(\theta) \\
&= \nabla \widetilde{L}_{t}(\theta) + \bOmega_{t}\left( \theta - \mu_{t} \right)
= \nabla \widetilde{L}_{t}(\theta) + \FisherMAP (\theta - \thetaMAP) - \FisherMAP (\theta - \thetaMAP) + \bOmega_{t}\left( \theta - \mu_{t} \right) \\
&= \nabla \widetilde{L}_{t}(\theta) - \nabla \widetilde{L}_{t}(\thetaMAP) + \FisherMAP (\theta - \thetaMAP) - \FisherMAP (\theta - \thetaMAP) + \bOmega_{t}\left( \theta - \mu_{t} \right) \\
&= \dot{\cR}_{t, 3}(\thetaMAP, \theta - \thetaMAP) - \FisherMAP (\theta - \thetaMAP)  + \bOmega_{t}\left( \theta - \mu_{t} \right) \\
&= \dot{\cR}_{t, 3}(\thetaMAP, \theta - \thetaMAP) + \left(\bOmega_{t} - \FisherMAP \right) \big(\theta - \thetaMAP \big)  + \bOmega_{t} \big( \thetaMAP - \mu_{t} \big),
\end{aligned}
\end{align}
where $\dot{\cR}_{t, 3}(\cdot, \cdot)$ is defined as
\begin{align*}
    \dot{\cR}_{t, 3}(\theta, u) 
    = \nabla \widetilde{L}_{t}(\theta + u) - \nabla \widetilde{L}_{t}(\theta) - \nabla^{2} \widetilde{L}_{t}(\theta) u, \quad \forall \theta, u \in \Theta.
\end{align*}
For $\theta \in \Theta$, define a linear perturbation version of $\widetilde{L}_{t_0}(\theta)$ by
\begin{align*}
    g_{n, t_0}(\theta) = \widetilde{L}_{t_0}(\theta) + \left\langle \sum_{t=1}^{t_0 - 1} \nabla \eta_{t}(\fullthetapMLE[t_0]), \theta \right\rangle.
\end{align*}
Note that
\begin{align*}
    \nabla g_{n, t_0} (\fullthetapMLE[t_0]) 
    =  \nabla \widetilde{L}_{t_0}(\fullthetapMLE[t_0]) + \sum_{t=1}^{t_0 - 1} \nabla \eta_{t}(\fullthetapMLE[t_0])
    \overset{\eqref{eqn:eta_equation}}{=} \nabla \widetilde{L}_{1:t_{0}} (\fullthetapMLE[t_0]) = 0.
\end{align*}
If $\widetilde{L}_{t_0}(\cdot)$ satisfies the third order smoothness at $\thetaMAP[t_0]$ with parameter 
\begin{align*}
    \left( \tau_{n, t_0}, \ \FisherMAP[t_0], \ 4 \left\| \sum_{t=1}^{t_0 - 1} \FisherMAP[t_0]^{-1/2} \nabla \eta_{t}(\fullthetapMLE[t_0]) \right\|_{2}  \right)
\end{align*}
for some $\tau_{n, t_0} \geq 0$ and 
\begin{align*}
    \tau_{n, t_0} \left\|  \sum_{t=1}^{t_0 - 1}  \FisherMAP[t_0]^{-1/2} \nabla \eta_{t}(\fullthetapMLE[t_0])  \right\|_{2} \leq \dfrac{1}{16},
\end{align*}
one can apply Lemma \ref{lemma:tech_pre_bias_bound} and then obtain the following inequality:
\begin{align*}
    \left\| \FisherMAP[t_{0}]^{1/2} \big( \fullthetapMLE[t_{0}] - \thetaMAP[t_{0}] \big) \right\|_{2}
    \leq 
    4 \left\| \sum_{t=1}^{t_0 - 1} \FisherMAP[t_0]^{-1/2} \nabla \eta_{t}(\fullthetapMLE[t_0]) \right\|_{2}.
\end{align*}
Later, we will show that $\tau_{n, t_0}$ can be chosen as $\tau_{n, t_0} = K_{\max} K_{\rm low}^{-3/2} t_{0}^{-3/2} n^{-1/2}$.
Therefore, to apply Lemma \ref{lemma:tech_pre_bias_bound}, we need to obtain an upper bound of 
$$
\left\| \sum_{t=1}^{t_0 - 1}  \FisherMAP[t_0]^{-1/2} \nabla \eta_{t}(\fullthetapMLE[t_0]) \right\|_{2}.
$$
By \eqref{eqn:similar_MAP_eq1}, we have
\begin{align*}
    &\left\|  \sum_{t=1}^{t_0 - 1}  \FisherMAP[t_0]^{-1/2} \nabla \eta_{t}(\fullthetapMLE[t_0])  \right\|_{2} \\
    &=
    \left\|  \sum_{t=1}^{t_0 - 1} \FisherMAP[t_0]^{-1/2} \bigg[ \dot{\cR}_{t, 3}(\thetaMAP, \fullthetapMLE[t_0] - \thetaMAP) + \left( \bOmega_{t} - \FisherMAP \right) \big( \fullthetapMLE[t_0] - \thetaMAP \big)  + \bOmega_{t} \big( \thetaMAP - \mu_{t} \big) \bigg]  \right\|_{2} \\
    &\leq 
    \left\|  \sum_{t=1}^{t_0 - 1} \FisherMAP[t_0]^{-1/2} \dot{\cR}_{t, 3}(\thetaMAP, \fullthetapMLE[t_0] - \thetaMAP) \right\|_{2} 
    +
    \left\|  \sum_{t=1}^{t_0 - 1} \FisherMAP[t_0]^{-1/2} \left( \bOmega_{t} - \FisherMAP \right) \big( \fullthetapMLE[t_0] - \thetaMAP \big)  \right\|_{2}  \\
    &\qquad \qquad +
    \left\|  \sum_{t=1}^{t_0 - 1} \FisherMAP[t_0]^{-1/2} \bOmega_{t}\big( \thetaMAP - \mu_{t} \big) \right\|_{2} \\
    &\leq 
    \sum_{t=1}^{t_0 - 1} \Bigg[
    \left\|  \FisherMAP[t_0]^{-1/2} \dot{\cR}_{t, 3}(\thetaMAP, \fullthetapMLE[t_0] - \thetaMAP) \right\|_{2} 
    +
    \left\|  \FisherMAP[t_0]^{-1/2} \left( \bOmega_{t} - \FisherMAP \right) \big( \fullthetapMLE[t_0] - \thetaMAP \big)  \right\|_{2}  \\
    &\qquad \qquad +
    \left\|  \FisherMAP[t_0]^{-1/2} \bOmega_{t}\big( \thetaMAP - \mu_{t} \big) \right\|_{2} \Bigg].
\end{align*}
For $t \in [t_0 - 1]$, let 
\begin{align*}
    {\rm (i)_{t}} &= \left\|  \FisherMAP[t_0]^{-1/2} \dot{\cR}_{t, 3}(\thetaMAP, \fullthetapMLE[t_0] - \thetaMAP) \right\|_{2}, \\
    {\rm (ii)_{t}} &= \left\|  \FisherMAP[t_0]^{-1/2} \left( \bOmega_{t} - \FisherMAP \right) \big( \fullthetapMLE[t_0] - \thetaMAP \big)  \right\|_{2}, \\  
    {\rm (iii)_{t}} &= \left\|  \FisherMAP[t_0]^{-1/2} \bOmega_{t}\big( \thetaMAP - \mu_{t} \big) \right\|_{2}.
\end{align*}
We will obtain upper bounds of these quantities throughout \textbf{Step 3-5}.

To bound ${\rm (i)_{t}}, {\rm (ii)_{t}}$ and ${\rm (iii)_{t}}$,
$\fullthetapMLE[t_0]$ should be located in a sufficiently small neighborhood of $\thetaMAP[t]$ for all $t \in [t_0-1]$.
Let $t \in [t_0-1]$. Note that
\begin{align}
\begin{aligned} \label{eqn:similar_MAP_consistent_estimator}
    &\left\| \FisherMAP[t]^{1/2} \big( \fullthetapMLE[t_0] - \thetaMAP[t] \big) \right\|_{2} \\
    &\leq 
    \left\| \FisherMAP[t]^{1/2} \big( \fullthetapMLE[t_0] - \theta_0 \big) \right\|_{2}
    +
    \left\| \FisherMAP[t]^{1/2} \big( \theta_0 - \thetaMAP[t] \big) \right\|_{2} \\
    &\leq
    \left\| \FisherMAP[t]^{1/2} \FullFisherTilde[t_0]{\fullthetaBest[t_0]}^{-1/2} \right\|_{2}
    \left\| \FullFisherTilde[t_0]{\fullthetaBest[t_0]}^{1/2} \big( \fullthetapMLE[t_0] - \theta_0 \big) \right\|_{2}
    +
    \left\| \FisherMAP[t]^{1/2} \big( \theta_0 - \thetaMAP[t] \big) \right\|_{2} \\    
    &\leq
    \left\| \FisherMAP[t]^{1/2} \FullFisherTilde[t_0]{\fullthetaBest[t_0]}^{-1/2} \right\|_{2}
    \Bigg[ 
    \left\| \FullFisherTilde[t_0]{\fullthetaBest[t_0]}^{1/2} \big( \fullthetapMLE[t_0] - \fullthetaBest[t_0] \big) \right\|_{2}
    +
    \left\| \FullFisherTilde[t_0]{\fullthetaBest[t_0]}^{1/2} \big( \fullthetaBest[t_0] - \theta_0 \big) \right\|_{2}
    \Bigg] \\
    &\qquad +
    \left\| \FisherMAP[t]^{1/2} \big( \theta_0 - \thetaMAP[t] \big) \right\|_{2} \\    
    &\leq
    (K_{\min} nt_0)^{-1/2} (K_{\rm up}nt)^{1/2} \big( 4M_n \sqrt{p_{\ast}} + 4M_n \sqrt{p_{\ast}} \big) + 
    \big( 8K_{\min}^{-1/2} K_{\rm up}^{1/2} + 1 \big) M_n \sqrt{p_{\ast}} \\
    &\leq 
    \big( 16K_{\min}^{-1/2} K_{\rm up}^{1/2} + 1 \big) M_n \sqrt{p_{\ast}}.
\end{aligned}
\end{align}
where the fourth inequality holds by \eqref{eqn:similar_MAP_regularity} and \eqref{eqn:similar_MAP_induction}.
It follows that
\begin{align*} 
    \fullthetapMLE[t_0] \in \Theta_{n, t} \coloneqq \Theta \left( \thetaMAP, \FisherMAP, \big[ 16K_{\min}^{-1/2} K_{\rm up}^{1/2} + 1 \big] M_n \sqrt{p_{\ast}} \right).
\end{align*}
For $\theta \in \Theta_{n, t}$, we have
\begin{align*}
    \big\| \theta - \thetaMAP \big\|_{2} 
    \overset{\eqref{eqn:similar_MAP_regularity}}{\leq}
    (K_{\rm low} nt)^{-1/2} \big( 16K_{\min}^{-1/2} K_{\rm up}^{1/2} + 1 \big) M_n \sqrt{p_{\ast}}
    \overset{(\textbf{S})}{\leq}
    1/4.
\end{align*}
It follows that $\Theta_{n, t} \subseteq \Theta (\theta_0, \bI_{p}, 1/2)$ because $\| \theta_0 - \thetaMAP \|_{2} \leq 1/4$ by \eqref{eqn:similar_MAP_regularity}.
Combining $(\textbf{A2})$ and Lemma \ref{lemma:tech_smooth_tau_bound}, $\widetilde{L}_{t}(\cdot)$ satisfies the third order smoothness at $\thetaMAP$ with parameters
\begin{align} \label{eqn:similar_MAP_eq3}
    \bigg( K_{\max} K_{\rm low}^{-3/2} t^{-3/2} n^{-1/2}, \: \FisherMAP, \: 
    \big( 16K_{\min}^{-1/2} K_{\rm up}^{1/2} + 1 \big) M_n \sqrt{p_{\ast}} \bigg).
\end{align}
Now, we are ready to obtain upper bounds of ${\rm (i)_{t}}$, ${\rm (ii)_{t}}$ and ${\rm (iii)_{t}}$.

\noindent \textbf{Step 3:}  ${\rm (i)_{t}}$ \\
Note that
\begin{align}
\begin{aligned} \label{eqn:similar_eq4}
\left\|  \FisherMAP[t_0]^{-1/2} \dot{\cR}_{t, 3}(\thetaMAP, \fullthetapMLE[t_0] - \thetaMAP) \right\|_{2}
&\leq 
\left\|  \FisherMAP[t_0]^{-1/2} \FisherMAP^{1/2} \right\|_{2}
\left\|  \FisherMAP^{-1/2} \dot{\cR}_{t, 3}(\thetaMAP, \fullthetapMLE[t_0] - \thetaMAP) \right\|_{2} \\
\overset{\eqref{eqn:similar_MAP_regularity}}&{\leq} 
\big( K_{\rm low}^{-1/2} K_{\rm up}^{1/2} \big) \sqrt{\dfrac{t}{t_0}}
\left\|  \FisherMAP^{-1/2} \dot{\cR}_{t, 3}(\thetaMAP, \fullthetapMLE[t_0] - \thetaMAP) \right\|_{2}.    
\end{aligned}
\end{align}
Also, by $\fullthetapMLE[t_0] \in \Theta_{n, t}$ and Taylor's theorem, we have
\begin{align*}
&\left\|  \FisherMAP^{-1/2} \dot{\cR}_{t, 3}(\thetaMAP, \fullthetapMLE[t_0] - \thetaMAP) \right\|_{2} \\
&\leq
\dfrac{1}{2}
\sup_{\widetilde{u} \in \overline \Theta_{n, t}}
\sup_{u \in \bbR^{p} : \| u \|_{2} = 1}
\left| 
\left\langle \nabla^{3} \widetilde{L}_{t} (\thetaMAP + \widetilde{u}), \: \left(\fullthetapMLE[t_0] - \thetaMAP \right)^{\otimes 2} \otimes \left(\FisherMAP^{-1/2} u \right) \right\rangle
\right| \\
\overset{\eqref{eqn:similar_MAP_eq3}}&{\leq} 
\dfrac{1}{2}
\big( K_{\max} K_{\rm low}^{-3/2} t^{-3/2} n^{-1/2} \big) 
\left\| \FisherMAP^{1/2} \left(\fullthetapMLE[t_0] - \thetaMAP \right) \right\|_{2}^{2} 
\sup_{u \in \bbR^{p} : \| u \|_{2} = 1}
\left\| \FisherMAP^{1/2} \FisherMAP^{-1/2} u \right\|_{2} \\
\overset{\eqref{eqn:similar_MAP_consistent_estimator}}&{\leq} 
\dfrac{1}{2}
\big( K_{\max} K_{\rm low}^{-3/2} t^{-3/2} n^{-1/2} \big) 
\big( 16K_{\min}^{-1/2} K_{\rm up}^{1/2} + 1 \big)^{2} 
M_n^{2} p_{\ast} \\
&=
\bigg( \dfrac{K_{\max}}{2 K_{\rm low}^{3/2}} \big( 16K_{\min}^{-1/2} K_{\rm up}^{1/2} + 1 \big)^{2} \bigg) 
M_n^{2} t^{-3/2} \left( \dfrac{p_{\ast}^{2}}{n} \right)^{1/2},
\end{align*}
where
\begin{align*}
    \overline \Theta_{n, t} = 
    \Theta \Big( \FisherMAP, \big[ 16K_{\min}^{-1/2} K_{\rm up}^{1/2} + 1 \big] M_n \sqrt{p_{\ast}} \Big).
\end{align*}
Combining with \eqref{eqn:similar_eq4}, therefore, we have
\begin{align*}
\left\|  \FisherMAP[t_0]^{-1/2} \dot{\cR}_{t, 3}(\thetaMAP, \fullthetapMLE[t_0] - \thetaMAP) \right\|_{2} 
\leq 
c_1 M_n^{2} t_{0}^{-1/2} t^{-1} \left( \dfrac{p_{\ast}^{2}}{n} \right)^{1/2}
\end{align*}
for some positive constant $c_1 = c_1(K_{\min}, K_{\max}, K_{\rm low}, K_{\rm up})$.
Consequently, we have
\begin{align*}
    &\sum_{t=1}^{t_0 - 1} 
    \left\| \FisherMAP[t_0]^{-1/2} \dot{\cR}_{t, 3}(\thetaMAP, \fullthetapMLE[t_0] - \thetaMAP) \right\|_{2} 
    \leq 
    \sum_{t=1}^{t_0 - 1} 
    c_1   
    M_n^{2} t_{0}^{-1/2} t^{-1} \left( \dfrac{p_{\ast}^{2}}{n} \right)^{1/2} \\
    &= 
    c_1   
    M_n^{2} t_{0}^{-1/2} \left( \dfrac{p_{\ast}^{2}}{n} \right)^{1/2} 
    \sum_{t=1}^{t_0 - 1} t^{- 1} 
    \leq 
    c_1 
    M_n^{2} \left( \log t_0 + 1 \right) t_{0}^{-1/2} \left( \dfrac{p_{\ast}^{2}}{n} \right)^{1/2} \\
    &\leq
    2 c_1 M_n^{2} \left( \dfrac{p_{\ast}^{2}}{n} \right)^{1/2},
\end{align*}
where the last two inequalities hold by $\sum_{t=1}^{t_0 - 1} t^{- 1} \leq \log t_0 + 1$ and $t_{0}^{-1/2}(\log t_0 + 1) \leq 2$.

\noindent \textbf{Step 4:}  ${\rm (ii)_{t}}$ \\
Next, we will obtain an upper bound of ${\rm (ii)_{t}}$. Note that
\begin{align*}
    &\left\|  \FisherMAP[t_0]^{-1/2} \left( \bOmega_{t} - \FisherMAP \right)( \fullthetapMLE[t_0] - \thetaMAP)  \right\|_{2}
    \leq 
    \left\|  \FisherMAP[t_0]^{-1/2} \left( \bOmega_{t} - \FisherMAP \right) \FisherMAP^{-1/2}  \right\|_{2}
    \left\|  \FisherMAP^{1/2} \left( \fullthetapMLE[t_0] - \thetaMAP \right)  \right\|_{2} \\
    &=
    \left\|  \FisherMAP[t_0]^{-1/2} \FisherMAP^{1/2} \left( \FisherMAP^{-1/2} \bOmega_{t} \FisherMAP^{-1/2} - \bI_{p} \right) \right\|_{2}
    \left\|  \FisherMAP^{1/2} \left( \fullthetapMLE[t_0] - \thetaMAP \right)  \right\|_{2} \\
    &\leq 
    \left\|  \FisherMAP[t_0]^{-1/2} \FisherMAP^{1/2} \right\|_{2} 
    \left\| \FisherMAP^{-1/2} \bOmega_{t} \FisherMAP^{-1/2} - \bI_{p} \right\|_{2}
    \left\|  \FisherMAP^{1/2} \left( \fullthetapMLE[t_0] - \thetaMAP \right)  \right\|_{2}.
\end{align*}
Note that
\begin{align*}
    \left\|  \FisherMAP[t_0]^{-1/2} \FisherMAP^{1/2} \right\|_{2} 
    \overset{\eqref{eqn:similar_MAP_regularity}}&{\leq} 
    \big( K_{\rm low}^{-1/2} K_{\rm up}^{1/2} \big) \sqrt{\dfrac{t}{t_0}}, \\
    \left\| \FisherMAP^{-1/2} \bOmega_{t} \FisherMAP^{-1/2} - \bI_{p} \right\|_{2}
    \overset{\eqref{eqn:similar_MAP_regularity}}&{\leq} 
    K_{\rm up} \left( \dfrac{p_{\ast}^{2}}{n t^{2}} \right)^{1/2}.
\end{align*}
Also,
\begin{align*}
    \left\|  \FisherMAP^{1/2} \big( \fullthetapMLE[t_0] - \thetaMAP \big) \right\|_{2}    
    \overset{\eqref{eqn:similar_MAP_consistent_estimator}}{\leq}
    \big( 16K_{\min}^{-1/2} K_{\rm up}^{1/2} + 1 \big) M_n \sqrt{p_{\ast}}.
\end{align*}
Combining the above three inequalities, we have
\begin{align*}
    &\left\|  \FisherMAP[t_0]^{-1/2} \left( \bOmega_{t} - \FisherMAP \right)( \fullthetapMLE[t_0] - \thetaMAP) \right\|_{2} \\
    &\leq 
    \bigg( K_{\rm low}^{-1/2} K_{\rm up}^{3/2} \big( 16K_{\min}^{-1/2} K_{\rm up}^{1/2} + 1 \big) \bigg) 
    M_n t^{-1/2} t_{0}^{-1/2} \left( \dfrac{p_{\ast}^{3}}{n} \right)^{1/2} \\
    &= c_2 M_n t^{-1/2} t_{0}^{-1/2} \left( \dfrac{p_{\ast}^{3}}{n} \right)^{1/2},
\end{align*}
where $c_2 = K_{\rm low}^{-1/2} K_{\rm up}^{3/2} ( 16K_{\min}^{-1/2} K_{\rm up}^{1/2} + 1 )$.
Consequently, we have
\begin{align*}
&\sum_{t=1}^{t_0 - 1} 
\left\|  \FisherMAP[t_0]^{-1/2} \left( \bOmega_{t} - \FisherMAP \right)( \fullthetapMLE[t_0] - \thetaMAP)  \right\|_{2} 
\leq 
\sum_{t=1}^{t_0 - 1} 
c_2
M_n t^{-1/2} t_0^{-1/2} \left( \dfrac{p_{\ast}^{3}}{n} \right)^{1/2} \\
&=
c_2
M_n t_0^{-1/2} \left( \dfrac{p_{\ast}^{3}}{n} \right)^{1/2}
\sum_{t=1}^{t_0 - 1} t^{-1/2}  
\leq 
c_2
M_n t_{0}^{-1/2} \left( \dfrac{p_{\ast}^{3}}{n} \right)^{1/2}
\left( 2 t_{0}^{1/2} - 1 \right)  \\
&\leq 
2c_2
M_n  \left( \dfrac{p_{\ast}^{3}}{n} \right)^{1/2}.
\end{align*}
\noindent \textbf{Step 5:}  ${\rm (iii)_{t}}$ \\
Next, we will obtain an upper bound of ${\rm (iii)_{t}}$. Note that
\begin{align*}
\left\|  \FisherMAP[t_0]^{-1/2} \bOmega_{t} \big( \thetaMAP - \mu_{t} \big) \right\|_{2} 
\leq 
\left\|  \FisherMAP[t_0]^{-1/2} \FisherMAP^{1/2} \right\|_{2} 
\left\|  \FisherMAP^{-1/2} \bOmega_{t}^{1/2} \right\|_{2} 
\left\|  \bOmega_{t}^{1/2} \big( \thetaMAP - \mu_{t} \big) \right\|_{2}. 
\end{align*}
Note that
\begin{align*}
    \left\|  \FisherMAP[t_0]^{-1/2} \FisherMAP^{1/2} \right\|_{2}
    \overset{\eqref{eqn:similar_MAP_regularity}}&{\leq} 
    \big( K_{\rm low}^{-1/2} K_{\rm up}^{1/2} \big) \sqrt{\dfrac{t}{t_0}}, \\
    \left\|  \FisherMAP^{-1/2} \bOmega_{t}^{1/2} \right\|_{2} 
    &\leq 
    \left( \left\|  \FisherMAP^{-1/2} \bOmega_{t} \FisherMAP^{-1/2} - \bI_{p} \right\|_{2} + 1 \right)^{1/2} 
    \overset{\eqref{eqn:similar_MAP_regularity}}{\leq} 
    \left( K_{\rm up} \left( \dfrac{p_{\ast}^{2}}{n t^{2}} \right)^{1/2} + 1 \right)^{1/2}
    \overset{(\textbf{S})}{\leq}
    2, \\
    \left\|  \bOmega_{t}^{1/2} \big( \thetaMAP - \mu_{t} \big) \right\|_{2}
    \overset{\eqref{eqn:similar_MAP_regularity}}&{\leq} 
    K_{\rm up} \left( \dfrac{p_{\ast}^{2}}{n t^{2}} \right)^{1/2}.
\end{align*}
By the last display, we have
\begin{align*}
    \left\|  \FisherMAP[t_0]^{-1/2} \bOmega_{t} \big( \thetaMAP - \mu_{t} \big) \right\|_{2}  
    \leq 
    \big( 2K_{\rm low}^{-1/2} K_{\rm up}^{3/2} \big) 
    t^{-1/2} t_{0}^{-1/2} \left( \dfrac{p_{\ast}^{2}}{n} \right)^{1/2}
    = 
    c_3 t^{-1/2} t_{0}^{-1/2} \left( \dfrac{p_{\ast}^{2}}{n} \right)^{1/2},
\end{align*}
where $c_3 = 2K_{\rm low}^{-1/2} K_{\rm up}^{3/2}$.
Consequently, we have
\begin{align*}
&\sum_{t=1}^{t_0 - 1} 
\left\|  \FisherMAP[t_0]^{-1/2} \bOmega_{t}\left( \thetaMAP - \mu_{t} \right) \right\|_{2}  
\leq 
\sum_{t=1}^{t_0 - 1} 
c_3
t^{-1/2} t_{0}^{-1/2} \left( \dfrac{p_{\ast}^{2}}{n} \right)^{1/2} \\
&=
c_3
t_{0}^{-1/2} \left( \dfrac{p_{\ast}^{2}}{n} \right)^{1/2} 
\sum_{t=1}^{t_0 - 1} t^{-1/2}  
\leq 
c_3
t_{0}^{-1/2}
\left( \dfrac{p_{\ast}^{2}}{n} \right)^{1/2} \left( 2 t_{0}^{1/2} - 1 \right)  
\leq 
2c_3  \left( \dfrac{p_{\ast}^{2}}{n} \right)^{1/2}.
\end{align*}

\noindent \textbf{Step 6:} Upper bound of $\| \sum_{t=1}^{t_0 - 1}  \FisherMAP[t_0]^{-1/2} \nabla \eta_{t}(\fullthetapMLE[t_0]) \|_{2}$ \\
Let
\begin{align*}
    \varrho_{n, t_0} = \left\|  \sum_{t=1}^{t_0 - 1}  \FisherMAP[t_0]^{-1/2} \nabla \eta_{t}(\fullthetapMLE[t_0])  \right\|_{2}.
\end{align*}
Combining the results in \textbf{Step 3-5}, note that
\begin{align}
\begin{aligned} \label{eqn:similar_MAP_eq5}
\varrho_{n, t_0}
&\leq 
2 c_1 M_n^{2} \left( \dfrac{p_{\ast}^{2}}{n} \right)^{1/2} + 
2c_2 M_n  \left( \dfrac{p_{\ast}^{3}}{n} \right)^{1/2} +
2c_3 \left( \dfrac{p_{\ast}^{2}}{n} \right)^{1/2} \\
&\leq
2\big( c_1 + c_2 + c_3 \big) M_n^{2} \left( \dfrac{p_{\ast}^{3}}{n} \right)^{1/2} 
= c_4 M_n^{2} \left( \dfrac{p_{\ast}^{3}}{n} \right)^{1/2}
\end{aligned}
\end{align}
where $c_4 = 2( c_1 + c_2 + c_3)$, and the second inequality holds by $M_n \wedge p_{\ast} \geq 1$.
In this proof, we denote 
\begin{align*}
    \widetilde{\Theta}_{n, t_0} = \Theta (\thetaMAP[t_0], \FisherMAP[t_0], 4\varrho_{n, t_0}).  
\end{align*}
By \eqref{eqn:similar_MAP_regularity}, we have $\| \thetaMAP[t_0] - \theta_0 \|_{2} \leq 1/4$.
Also, for $\theta \in \widetilde{\Theta}_{n, t_0}$, we have
\begin{align*}
    \big\| \theta - \thetaMAP[t_0] \big\|_{2} 
    \overset{ \eqref{eqn:similar_MAP_regularity} }&{\leq} 
    \left( K_{\rm low} n t_0 \right)^{-1/2} 4\varrho_{n, t_0} \\
    \overset{\eqref{eqn:similar_MAP_eq5}}&{\leq}
    \left( K_{\rm low} n t_0 \right)^{-1/2} 
    4c_4 M_n^{2} \left( \dfrac{p_{\ast}^{3}}{n} \right)^{1/2} \\ 
    &= 
    (4K_{\rm low}^{-1/2} c_4)
    M_n^{2} p_{\ast}^{3/2} n^{-1} t_{0}^{-1/2}
    \overset{(\textbf{S})}{\leq}
    1/4.
\end{align*}
It follows that $\widetilde{\Theta}_{n, t_0} \subseteq \Theta (\theta_0, \bI_{p}, 1/2)$.
Then, by Lemma \ref{lemma:tech_smooth_tau_bound} and \eqref{eqn:similar_MAP_regularity}, 
$\widetilde{L}_{t_0}(\cdot)$ satisfies the third order smoothness at $\thetaMAP[t_0]$ with parameters
\begin{align} 
    \left( K_{\max} K_{\rm low}^{-3/2} t_{0}^{-3/2} n^{-1/2}, \: \FisherMAP[t_0], \: 4\varrho_{n, t_0} \right).
\end{align}

\noindent \textbf{Step 7:} \eqref{eqn:similar_MAP_inner_claim} and \eqref{eqn:similar_MAP_claim} \\
By \textbf{Step 1-6}, we are ready to prove \eqref{eqn:similar_MAP_inner_claim}. 
Note that
\begin{align} \label{eqn:similar_MAP_eq6}
    K_{\max} K_{\rm low}^{-3/2} t_0^{-3/2} n^{-1/2} \varrho_{n, t_0} 
    \leq 
    (K_{\max} K_{\rm low}^{-3/2} c_4)
    M_n^{2} p_{\ast}^{3/2} n^{-1} t_0^{-3/2}
    \overset{(\textbf{S})}{\leq} 
    1/16,
\end{align}
which allows us to apply Lemma \ref{lemma:tech_pre_bias_bound}.
By Lemma \ref{lemma:tech_pre_bias_bound} with
\begin{align*}
    &\tau_{3} = K_{\max} K_{\rm low}^{-3/2} t_{0}^{-3/2} n^{-1/2}, \quad 
    f(\theta) = \widetilde{L}_{t_0} (\theta), \quad 
    \theta = \thetaMAP[t_0], \quad 
    \widetilde{\theta} = \fullthetapMLE[t_0], \\ 
    &\beta = \sum_{t=1}^{t_0 - 1} \nabla \eta_{t}(\fullthetapMLE[t_0]), \quad 
    r = \varrho_{n, t_0},
\end{align*}
we have
\begin{align}
\begin{aligned} \label{eqn:similar_MAP_K_app}
     \left\| \FisherMAP[t_0]^{1/2} \big( \fullthetapMLE[t_0] - \thetaMAP[t_0] \big) \right\|_{2}
     \leq 4\varrho_{n, t_0} 
     \leq 
     4 c_4 M_n^{2} \left( \dfrac{p_{\ast}^{3}}{n} \right)^{1/2},
\end{aligned}
\end{align}
which completes the proof of the first assertion in \eqref{eqn:similar_MAP_inner_claim}.
It follows that
\begin{align*}
    &\left\| \FisherMAP[t_0]^{1/2} \big( \theta_0 - \thetaMAP[t_0] \big) \right\|_{2}
    \leq 
    \left\| \FisherMAP[t_0]^{1/2} \big( \theta_0 - \fullthetapMLE[t_0]  \big) \right\|_{2}
    +
    \left\| \FisherMAP[t_0]^{1/2} \big( \fullthetapMLE[t_0] - \thetaMAP[t_0]  \big) \right\|_{2} \\
    &\leq 
    \left\| \FisherMAP[t_0]^{1/2} \FullFisherTilde[t_0]{\fullthetaBest[t_0]}^{-1/2} \right\|_{2}
    \left\| \FullFisherTilde[t_0]{\fullthetaBest[t_0]}^{1/2} \big( \theta_0 - \fullthetapMLE[t_0] \big) \right\|_{2}
    +
    \left\| \FisherMAP[t_0]^{1/2} \big( \fullthetapMLE[t_0] - \thetaMAP[t_0] \big) \right\|_{2} \\    
    \overset{\eqref{eqn:similar_MAP_regularity}}&{\leq}
    (K_{\min} n t_0 )^{-1/2} (K_{\rm up} n t_0)^{1/2} \big( 4M_n \sqrt{p_{\ast}} + 4M_n \sqrt{p_{\ast}} \big)
    + 
    4c_4 M_n^{2} \left( \dfrac{p_{\ast}^{3}}{n} \right)^{1/2} \\
    &= 
    \big( 8 K_{\min}^{-1/2} K_{\rm up}^{1/2} + 4c_4 M_n p_{\ast} n^{-1/2} \big) M_n \sqrt{p_{\ast}} \\
    \overset{(\textbf{S})}&{\leq}
    \big( 8 K_{\min}^{-1/2} K_{\rm up}^{1/2} + 1 \big) M_n \sqrt{p_{\ast}} = c_5 M_n \sqrt{p_{\ast}}
\end{align*}
for some constant $c_5 = 8 K_{\min}^{-1/2} K_{\rm up}^{1/2} + 1$.
This completes the proof of \eqref{eqn:similar_MAP_inner_claim}. 

It follows by induction that
\begin{align*}
    \left\| \FisherMAP^{1/2} \big( \fullthetapMLE - \thetaMAP \big) \right\|_{2}
    \leq 
    4c_4 M_n^{2} \left( \dfrac{p_{\ast}^{3}}{n} \right)^{1/2}, \quad 
    \left\| \FisherMAP^{1/2} \big( \theta_0 - \thetaMAP \big) \right\|_{2}
    \leq 
    c_5 M_n \sqrt{p_{\ast}}, \quad \forall t \in [T].
\end{align*}
Also, for any $t \in [T]$,
\begin{align*}
&\left\| \bOmega_{t}^{1/2} \big( \fullthetapMLE - \mu_{t}  \big) \right\|_{2}
\leq 
\left\| \bOmega_{t}^{1/2} \big( \fullthetapMLE - \thetaMAP[t]  \big) \right\|_{2}
+
\left\| \bOmega_{t}^{1/2} \big( \thetaMAP[t] - \mu_{t} \big) \right\|_{2} \\
\overset{\text{Corollary \ref{coro:VB_KL_Delta}}}&{\leq}
\big( 1 + \Delta_{t} \big)^{1/2} \left\| \FisherMAP[t]^{1/2} \big( \fullthetapMLE - \thetaMAP[t] \big) \right\|_{2} + \Delta_{t}
\leq
\big( 1 + \Delta_{t} \big) \left\| \FisherMAP[t]^{1/2} \big( \fullthetapMLE - \thetaMAP[t] \big) \right\|_{2} + \Delta_{t}
\\
\overset{\eqref{eqn:similar_MAP_regularity}}&{\leq}
\big( 1 + K_{\rm up} t^{-1} n^{-1/2} p_{\ast} \big)
4c_4 M_n^{2} \left( \dfrac{p_{\ast}^{3}}{n} \right)^{1/2}  
+
K_{\rm up} t^{-1} n^{-1/2} p_{\ast} \\
&=
\bigg[
4c_4\big( 1 + K_{\rm up} t^{-1} n^{-1/2} p_{\ast} \big) 
+
K_{\rm up} M_n^{-2} t^{-1} p_{\ast}^{-1/2}
\bigg]
M_n^{2} \left( \dfrac{p_{\ast}^{3}}{n} \right)^{1/2} \\
\overset{(\textbf{S})}&{\leq}
c_6 M_n^{2} \left( \dfrac{p_{\ast}^{3}}{n} \right)^{1/2} 
\end{align*}
for some positive constant $c_6 = c_6(c_4, K_{\rm up})$.
This completes the proof of \eqref{eqn:similar_MAP_claim}.
\end{proof}

\begin{proof}[Proof of Proposition \ref{prop:similar_variance}]
In this proof, we work on the event $\scrE_{\est, 1} \cap \scrE_{\est, 2}$ without explicitly referring to it. Given the conditions in Proposition \ref{prop:similar_MAP}, the conditions required for Lemmas \ref{lemma:full_posterior_optimal_parameter}, \ref{lemma:full_posterior_parameter} and Proposition \ref{prop:eigenvalue_order} are satisfied.
It follows that 
\begin{align}
\begin{aligned} \label{eqn:similar_variance_regularity}
    &\left\| \FullFisherTilde[t]{\fullthetaBest[t]}^{1/2} \big( \theta_{0} - \fullthetaBest[t] \big) \right\|_{2} 
    \leq 4M_n p_{\ast}^{1/2}, &\quad
    &\lambda_{\min}(\FullFisherTilde[t]{\fullthetaBest[t]}) 
    \geq K_{\min} nt \\
    &\left\| \FullFisherTilde[t]{\fullthetaBest[t]}^{1/2} \big( \fullthetapMLE[t] - \fullthetaBest[t] \big) \right\|_{2} 
    \leq 4M_n p_{\ast}^{1/2}, &\quad
    &\lambda_{\max}(\FullFisherTilde[t]{\fullthetaBest[t]}) 
    \leq \dfrac{4}{3} K_{\max} nt, \\
    &\lambda_{\min}(\FisherMAP[t]) \wedge \lambda_{\min}(\FisherBest[t]) 
    \geq K_{\rm low} nt, &\quad
    &\lambda_{\max}(\FisherMAP[t]) \wedge \lambda_{\max}(\FisherBest[t]) 
    \leq K_{\rm up} nt, \\
    &\thetaMAP, \thetaBest, \fullthetaBest, \fullthetapMLE 
    \in \Theta(\theta_{0}, \bI_{p}, 1/4), &\quad
    &\Delta_{t} \vee \epsilon_{n, t, \KL} 
    \leq K_{\rm up} t^{-1} n^{-1/2} p_{\ast}
\end{aligned}
\end{align}
for all $t \in [T]$.

\noindent \textbf{Step 1: Proof framework} \\
For $t \in [T]$, note that
\begin{align*}
\begin{aligned} 
&\left\| \FullFisherTilde[t]{\fullthetapMLE[t]}^{-1/2} \bOmega_{t} \FullFisherTilde[t]{\fullthetapMLE[t]}^{-1/2} - \bI_{p} \right\|_{\rm F}
\leq 
\lambda_{\min}^{-1} \left( \FullFisherTilde[t]{\fullthetapMLE[t]} \right)
\left\| \bOmega_{t} - \FullFisherTilde[t]{\fullthetapMLE[t]} \right\|_{\rm F} \\
&\leq 
\lambda_{\min}^{-1} \left( \FullFisherTilde[t]{\fullthetapMLE[t]} \right)
\bigg( \left\| \bOmega_{t} - \FisherMAP[t] \right\|_{\rm F} + \left\| \FisherMAP[t] - \FullFisherTilde[t]{\fullthetapMLE[t]} \right\|_{\rm F}  \bigg) \\
&\leq 
\lambda_{\min}^{-1} \left( \FullFisherTilde[t]{\fullthetapMLE[t]} \right)
\left\| \FisherMAP[t] \right\|_{2}
\Bigg[   \left\| \FisherMAP[t]^{-1/2} \bOmega_{t} \FisherMAP[t]^{-1/2} - \bI_{p} \right\|_{\rm F}  
+ \left\| \FisherMAP[t]^{-1/2} \FullFisherTilde[t]{\fullthetapMLE[t]} \FisherMAP[t]^{-1/2} - \bI_{p} \right\|_{\rm F} \Bigg]  \\
\overset{ \eqref{eqn:similar_variance_regularity} }&{\leq}
K_{\min}^{-1} K_{\rm up} 
\bigg( \left\| \FisherMAP[t]^{-1/2} \bOmega_{t} \FisherMAP[t]^{-1/2} - \bI_{p} \right\|_{\rm F} 
    + \left\| \FisherMAP[t]^{-1/2} \FullFisherTilde[t]{\fullthetapMLE[t]} \FisherMAP[t]^{-1/2} - \bI_{p} \right\|_{\rm F} \bigg) \\
\overset{ \eqref{eqn:similar_variance_regularity} }&{\leq}    
K_{\min}^{-1} K_{\rm up} 
\Bigg[ 
K_{\rm up} t^{-1} \left( \dfrac{p_{\ast}^{2}}{n} \right)^{1/2} + 
\left\| \FisherMAP[t]^{-1/2} \FullFisherTilde[t]{\fullthetapMLE[t]} \FisherMAP[t]^{-1/2} - \bI_{p} \right\|_{\rm F} \Bigg].    
\end{aligned}
\end{align*}
Hence, we only need to obtain an upper bound of 
\begin{align*}
    \left\| \FisherMAP[t]^{-1/2} \FullFisherTilde[t]{\fullthetapMLE[t]} \FisherMAP[t]^{-1/2} - \bI_{p} \right\|_{\rm F}.
\end{align*}
Note that
\begin{align*}
    \FullFisherTilde[t]{\fullthetapMLE}
    = 
    \FisherMAP[t] + 
    \sum_{s=1}^{t - 1} 
    \bigg[ \left( \FisherMAP[s] -  \bOmega_{s} \right) + \left( \Fisher[s]{\fullthetapMLE[t]} - \Fisher[s]{\thetaMAP[s]} \right) 
    \bigg] + 
    \left( \Fisher[t]{\fullthetapMLE[t]} -  \Fisher[t]{\thetaMAP[t]} \right)
\end{align*}
It follows that
\begin{align*}
    &\FisherMAP[t]^{-1/2} \FullFisherTilde[t]{ \fullthetapMLE } \FisherMAP[t]^{-1/2} - \bI_{p} \\
    &= 
    \sum_{s=1}^{t - 1} 
    \FisherMAP[t]^{-1/2} \bigg[ \left( \FisherMAP[s] -  \bOmega_{s} \right) + \left( \Fisher[s]{\fullthetapMLE[t]} - \Fisher[s]{\thetaMAP[s]} \right) 
    \bigg]
    \FisherMAP[t]^{-1/2} + 
    \FisherMAP[t]^{-1/2} 
    \left( \Fisher[t]{\fullthetapMLE} - \Fisher[t]{\thetaMAP[t]} \right) 
    \FisherMAP[t]^{-1/2} \\
    &= 
    \left( -\sum_{s=1}^{t - 1} \FisherMAP[t]^{-1/2} \FisherMAP[s]^{1/2} \bigg[ \FisherMAP[s]^{-1/2} \bOmega_{s} \FisherMAP[s]^{-1/2} - \bI_{p} \bigg] \FisherMAP[s]^{1/2} \FisherMAP[t]^{-1/2} \right)  \\
    &\qquad + \left( \sum_{s=1}^{t} \FisherMAP[t]^{-1/2} \Fisher[s]{\thetaMAP[s]}^{1/2} 
    \bigg[ \Fisher[s]{\thetaMAP[s]}^{-1/2} \Fisher[s]{\fullthetapMLE[t]} \Fisher[s]{\thetaMAP[s]}^{-1/2} - \bI_{p}  \bigg] \Fisher[s]{\thetaMAP[s]}^{1/2} \FisherMAP[t]^{-1/2} \right).    
\end{align*}
Let
\begin{align*}
    {\rm (i)} &= 
    \left\| \sum_{s=1}^{t - 1} \FisherMAP[t]^{-1/2} \FisherMAP[s]^{1/2} \bigg[ \FisherMAP[s]^{-1/2} \bOmega_{s} \FisherMAP[s]^{-1/2} - \bI_{p} \bigg] \FisherMAP[s]^{1/2} \FisherMAP[t]^{-1/2} \right\|_{\rm F}, \\
    {\rm (ii)} &= 
    \left\| \sum_{s=1}^{t} \FisherMAP[t]^{-1/2} \Fisher[s]{\thetaMAP[s]}^{1/2} 
    \bigg[ \Fisher[s]{\thetaMAP[s]}^{-1/2} \Fisher[s]{\fullthetapMLE[t]} \Fisher[s]{\thetaMAP[s]}^{-1/2} - \bI_{p}  \bigg] \Fisher[s]{\thetaMAP[s]}^{1/2} \FisherMAP[t]^{-1/2} \right\|_{\rm F}.
\end{align*}
Note that 
\begin{align*}
    \left\| \FisherMAP[t]^{-1/2} \FullFisherTilde[t]{ \fullthetapMLE[t] } \FisherMAP[t]^{-1/2} - \bI_{p} \right\|_{\rm F}
    \leq 
    {\rm (i)} + {\rm (ii)}.
\end{align*}

\noindent \textbf{Step 2:} ${\rm (i)}$ \\
Let $s \in [t-1]$. Note that
\begin{align*}
    &\left\| \FisherMAP[t]^{-1/2} \FisherMAP[s]^{1/2} \bigg[ \FisherMAP[s]^{-1/2} \bOmega_{s} \FisherMAP[s]^{-1/2} - \bI_{p} \bigg] \FisherMAP[s]^{1/2} \FisherMAP[t]^{-1/2} \right\|_{\rm F} \\
    &\leq 
    \lambda_{\min}^{-1} \big( \FisherMAP[t] \big)
    \lambda_{\max} \big( \FisherMAP[s] \big)
    \left\| \FisherMAP[s]^{-1/2} \bOmega_{s} \FisherMAP[s]^{-1/2} - \bI_{p} \right\|_{\rm F} \\
    \overset{ \substack{ \eqref{eqn:similar_variance_regularity} \\ \text{Corollary \ref{coro:VB_KL_Delta}} }}&{\leq}    
    \bigg[ \dfrac{K_{\rm up} s}{K_{\rm low} t} \bigg]  \bigg[ K_{\rm up} s^{-1} \left( \dfrac{p_{\ast}^{2}}{n} \right)^{1/2} \bigg]
    = 
    \big( K_{\rm up}^{2} K_{\rm low}^{-1} \big) t^{-1} \left( \dfrac{p_{\ast}^{2}}{n} \right)^{1/2}.
\end{align*}
By the last display, we have
\begin{align*}
    {\rm (i)} 
    \leq 
    \sum_{s=1}^{t - 1} \big( K_{\rm up}^{2} K_{\rm low}^{-1} \big)
    t^{-1} \left( \dfrac{p_{\ast}^{2}}{n} \right)^{1/2}
    = 
    \big( K_{\rm up}^{2} K_{\rm low}^{-1} \big)  t^{-1} \left( \dfrac{p_{\ast}^{2}}{n} \right)^{1/2}
    \sum_{s=1}^{t - 1} 1 
    \leq  
    c_1 \left( \dfrac{p_{\ast}^{2}}{n} \right)^{1/2},
\end{align*}
where $c_1 = K_{\rm up}^{2} K_{\rm low}^{-1}$.

\noindent \textbf{Step 3:} ${\rm (ii)}$ \\
Let $s \in [t]$. Note that
\begin{align} 
\begin{aligned} \label{eqn:similar_variance_eq1}
    &\left\| \FisherMAP[t]^{-1/2} \Fisher[s]{\thetaMAP[s]}^{1/2} 
    \bigg[ \Fisher[s]{\thetaMAP[s]}^{-1/2} \Fisher[s]{\fullthetapMLE[t]} \Fisher[s]{\thetaMAP[s]}^{-1/2} - \bI_{p}  \bigg] \Fisher[s]{\thetaMAP[s]}^{1/2} \FisherMAP[t]^{-1/2} \right\|_{\rm F} \\
    &\leq 
    \lambda_{\min}^{-1} \big( \FisherMAP[t] \big)
    \lambda_{\max} \big( \Fisher[s]{\thetaMAP[s]} \big)
    \left\| \Fisher[s]{\thetaMAP[s]}^{-1/2} \Fisher[s]{\fullthetapMLE[t]} \Fisher[s]{\thetaMAP[s]}^{-1/2} - \bI_{p} \right\|_{\rm F} \\
    &\leq
    \bigg( \dfrac{ K_{\max} }{K_{\rm low}} t^{-1} \bigg) 
    \sqrt{p} 
    \left\| \Fisher[s]{\thetaMAP[s]}^{-1/2} \Fisher[s]{\fullthetapMLE[t]} \Fisher[s]{\thetaMAP[s]}^{-1/2} - \bI_{p} \right\|_{2}. 
\end{aligned}
\end{align}
where the last inequality holds by \eqref{eqn:similar_variance_regularity} and (\textbf{A2}).
Also, by Proposition \ref{prop:similar_MAP}, we have
\begin{align*}
    \left\| \FisherMAP[t]^{1/2} \big( \fullthetapMLE[t] - \thetaMAP[t] \big) \right\|_{2}
    \leq  
    K M_n^{2} \left( \dfrac{p_{\ast}^3}{n} \right)^{1/2}, \quad 
    \left\| \FisherMAP[t]^{1/2} \big( \theta_0 - \thetaMAP[t] \big) \right\|_{2}
    \leq  
    K M_n p_{\ast}^{1/2}, \quad     
    \forall t \in [T],
\end{align*}
where $K = K(K_{\rm low}, K_{\rm up}, K_{\min}, K_{\max}) > 0$.
It follows that
\begin{align*}
    &\left\| \Fisher[s]{\thetaMAP[s]}^{1/2} \big( \fullthetapMLE[t] - \thetaMAP[s] \big) \right\|_{2} \\
    &\leq 
    \left\| \Fisher[s]{\thetaMAP[s]}^{1/2} \big( \fullthetapMLE[t] - \thetaMAP[t] \big) \right\|_{2}
    +
    \left\| \Fisher[s]{\thetaMAP[s]}^{1/2} \big( \thetaMAP[t] - \theta_0 \big) \right\|_{2}
    +
    \left\| \Fisher[s]{\thetaMAP[s]}^{1/2} \big( \theta_0 - \thetaMAP[s] \big) \right\|_{2} \\
    &\leq 
    \left\| \Fisher[s]{\thetaMAP[s]}^{1/2} \FisherMAP[t]^{-1/2} \right\|_{2}
    \left\| \FisherMAP[t]^{1/2} \big( \fullthetapMLE[t] - \thetaMAP[t] \big) \right\|_{2}
    +
    \left\| \Fisher[s]{\thetaMAP[s]}^{1/2} \FisherMAP[t]^{-1/2} \right\|_{2}
    \left\| \FisherMAP[t]^{1/2} \big( \thetaMAP[t] - \theta_0 \big) \right\|_{2} 
    + 
    \left\| \FisherTilde[s]{\thetaMAP[s]}^{1/2} \left( \theta_0 - \thetaMAP[s] \right) \right\|_{2} \\
    &\leq
    \left( \dfrac{K_{\max}}{ K_{\rm low} t} \right)^{1/2}\big( K M_n^{2} p_{\ast}^{3/2} n^{-1/2} \big)
    +
    \left( \dfrac{K_{\max}}{ K_{\rm low} t} \right)^{1/2}\big( K M_n p_{\ast}^{1/2} \big)
    + 
    K M_n p_{\ast}^{1/2} \\
    &= 
    \bigg( K K_{\max}^{1/2}K_{\rm low}^{-1/2}t^{-1/2} M_n p_{\ast} n^{-1/2} +  K_{\max}^{1/2} K_{\rm low}^{-1/2} t^{-1/2} K + K \bigg) M_n p_{\ast}^{1/2} \\
    \overset{(\textbf{S})}&{\leq}
    K\big( K_{\max}^{1/2}K_{\rm low}^{-1/2} + 2 \big)  M_n p_{\ast}^{1/2} 
    = c_2 M_n p_{\ast}^{1/2},
\end{align*}
where the third inequality holds by \eqref{eqn:similar_variance_regularity} and $\Fisher[s]{\thetaMAP[s]} \preceq \FisherTilde[s]{\thetaMAP[s]}$, and $c_2 = K\big( K_{\max}^{1/2}K_{\rm low}^{-1/2} + 2 \big)$.
It follows that
\begin{align*}
    &\big\| \theta_0 - \thetaMAP[s] \big\|_{2} 
    \overset{\eqref{eqn:similar_variance_regularity}}{\leq}
    1/4, \\ 
    &\big\| \theta - \thetaMAP[s] \big\|_{2} 
    \overset{(\textbf{A2})}{\leq}
    (K_{\min} n)^{-1/2} c_2 M_n p_{\ast}^{1/2}
    \overset{(\textbf{S})}{\leq} 1/4, \quad \forall \theta \in \Theta \big( \thetaMAP[s], \Fisher[s]{\thetaMAP[s]}, c_2 M_n p_{\ast}^{1/2} \big) \\
    &\fullthetapMLE[t] \in \Theta \big( \thetaMAP[s], \Fisher[s]{\thetaMAP[s]}, c_2 M_n p_{\ast}^{1/2} \big) 
    \subseteq \Theta (\theta_0, \bI_p, 1/2).
\end{align*}
which, combining with Lemma \ref{lemma:tech_smooth_tau_bound}, implies that $L_{s}(\cdot)$ satisfies the third order smoothness at $\thetaMAP[s]$ with parameters
\begin{align*}
    \left( K_{\max} K_{\min}^{-3/2} n^{-1/2}, \: \Fisher[s]{\thetaMAP[s]}, \: 
    c_2 M_n p_{\ast}^{1/2} \right).
\end{align*}
By Lemma \ref{lemma:tech_Fisher_smooth}, it follows that
\begin{align*}
    \left\| \Fisher[s]{\thetaMAP[s]}^{-1/2} \Fisher[s]{\fullthetapMLE} \Fisher[s]{\thetaMAP[s]}^{-1/2} - \bI_{p}  \right\|_{2}
    \leq 
    \big( K_{\max} K_{\min}^{-3/2} n^{-1/2} \big)
    \big( c_2 M_n p_{\ast}^{1/2} \big) 
    =
    c_3 M_n  \left( \dfrac{p_{\ast}}{n} \right)^{1/2},
\end{align*}
where $c_3 = K_{\max} K_{\min}^{-3/2} c_2$.
Hence, the right-hand side of \eqref{eqn:similar_variance_eq1} is bounded by
\begin{align*}
\bigg[ \dfrac{ K_{\max} }{K_{\rm low}} t^{-1} \bigg] \sqrt{p}
\left( c_3 M_n  \left( \dfrac{p_{\ast}}{n} \right)^{1/2} \right) 
\leq 
c_4
M_n t^{-1} \left( \dfrac{p_{\ast}^{2}}{n} \right)^{1/2}
\end{align*}
where $c_4 = K_{\rm low}^{-1} K_{\max} c_3$.
Consequently, we have
\begin{align*}
{\rm (ii)} 
\leq 
\sum_{s=1}^{t}
c_4 M_n t^{-1} \left( \dfrac{p_{\ast}^{2}}{n} \right)^{1/2} 
=
c_4 M_n \left( \dfrac{p_{\ast}^{2}}{n} \right)^{1/2}.
\end{align*}

\noindent \textbf{Step 4} \\
By \textbf{Step 2-3}, we have
\begin{align*}
    \left\| \FisherMAP[t]^{-1/2} \FullFisherTilde[t]{ \fullthetapMLE[t] } \FisherMAP[t]^{-1/2} - \bI_{p} \right\|_{\rm F}
    &\leq 
    {\rm (i)} + {\rm (ii)} 
    \leq 
    c_1 \left( \dfrac{p_{\ast}^{2}}{n} \right)^{1/2}
    + 
    c_4
    M_n \left( \dfrac{p_{\ast}^{2}}{n} \right)^{1/2} \\
    &\leq 
    \left( c_1 + c_4 \right) M_n \left( \dfrac{p_{\ast}^{2}}{n} \right)^{1/2}.
\end{align*}
By \textbf{Step 1} and the last display, we have
\begin{align*}
    &\left\| \FullFisherTilde[t]{\fullthetapMLE[t]}^{-1/2} \bOmega_{t} \FullFisherTilde[t]{\fullthetapMLE[t]}^{-1/2} - \bI_{p} \right\|_{\rm F} \\
    &\leq 
    K_{\min}^{-1} K_{\rm up} 
    \Bigg[ 
    K_{\rm up} t^{-1} \left( \dfrac{p_{\ast}^{2}}{n} \right)^{1/2} + 
    \left\| \FisherMAP[t]^{-1/2} \FullFisherTilde[t]{\fullthetapMLE[t]} \FisherMAP[t]^{-1/2} - \bI_{p} \right\|_{\rm F} \Bigg]  \\
    &\leq
    K_{\min}^{-1} K_{\rm up} 
    \Bigg[ 
    K_{\rm up} t^{-1} \left( \dfrac{p_{\ast}^{2}}{n} \right)^{1/2} + \left( c_1 + c_4 \right) M_n \left( \dfrac{p_{\ast}^{2}}{n} \right)^{1/2}
    \Bigg] \\
    &\leq 
    c_5 M_n \left( \dfrac{p_{\ast}^{2}}{n} \right)^{1/2},
\end{align*}
where $c_5 = c_5(K_{\min}, K_{\rm up}, c_1, c_4) > 0$. This completes the proof.
\end{proof}

\begin{proof}[Proof of Theorem \ref{thm:onlineBvM}]
In this proof, we work on the event $\scrE_{\est, 1} \cap \scrE_{\est, 2}$ without explicitly referring to it. 
Note that $\bbP_{0}^{(N)}(\scrE_{\est, 1} \cap \scrE_{\est, 2}) \geq 1 - 3n^{-1}$.
For all $t \in [T]$, we have
\begin{align*}
    &d_{V} \bigg( \Pi_{t}(\cdot),  \Pi\left(\cdot \mid \bD_{1:t} \right) \bigg) \\
    &\leq 
    d_{V} \bigg( \Pi_{t}(\cdot),  \cN\left(\fullthetapMLE, \FullFisherTilde[t]{\fullthetapMLE[t]}^{-1} \right) \bigg)
    +
    d_{V} \bigg( \cN\left(\fullthetapMLE, \FullFisherTilde[t]{\fullthetapMLE[t]}^{-1} \right),  \Pi\left(\cdot \mid \bD_{1:t} \right) \bigg) \\
    \overset{\text{Theorem \ref{thm:batch_posterior_LA}}}&{\leq}
    d_{V} \bigg( \Pi_{t}(\cdot),  \cN\left(\fullthetapMLE, \FullFisherTilde[t]{\fullthetapMLE[t]}^{-1} \right) \bigg) +
    K_1 \left( \dfrac{p_{\ast}^2}{nt} \right)^{1/2},
\end{align*}
where $K_1 = K_1(K_{\min}, K_{\max}) > 0$ is the constant $K$ specified in Theorem \ref{thm:batch_posterior_LA}.
Hence, we only need to obtain an upper bound of
\begin{align*}
    d_{V} \bigg( \Pi_{t}(\cdot),  \cN\left(\fullthetapMLE, \FullFisherTilde[t]{\fullthetapMLE[t]}^{-1} \right) \bigg).
\end{align*}
By Propositions \ref{prop:similar_MAP} and \ref{prop:similar_variance}, we have
\begin{align*}
    \left\| \bOmega_{t}^{1/2} \big( \fullthetapMLE - \mu_{t}  \big) \right\|_{2} 
    &\leq K_{2} M_n^{2} \left( \dfrac{p_{\ast}^{3}}{n} \right)^{1/2}, \\
    \left\| \FullFisherTilde[t]{\fullthetapMLE[t]}^{-1/2} \bOmega_{t} \FullFisherTilde[t]{\fullthetapMLE[t]}^{-1/2} - \bI_{p} \right\|_{\rm F} 
    &\leq K_{3} M_{n} \left( \dfrac{p_{\ast}^{2}}{n} \right)^{1/2},
\end{align*}
where $K_{2}$ and $K_{3}$, depending only on $(K_{\min}, K_{\max}, K_{\rm low}, K_{\rm up})$, are the constants specified in Propositions \ref{prop:similar_MAP} and \ref{prop:similar_variance}, respectively.
Also,
\begin{align*}
    \left\| \FullFisherTilde[t]{\fullthetapMLE[t]}^{-1/2} \bOmega_{t} \FullFisherTilde[t]{\fullthetapMLE[t]}^{-1/2} - \bI_{p} \right\|_{2} 
    \leq
    K_{3} M_{n} \left( \dfrac{p_{\ast}^{2}}{n} \right)^{1/2} 
    \overset{(\textbf{S})}{\leq}
    0.684,
\end{align*}
which, combining with Lemma \ref{lemma:tech_Gaussian_comparison}, implies that
\begin{align*}
    &d_{V} \bigg( \Pi_{t}(\cdot),  \cN\left(\fullthetapMLE, \FullFisherTilde[t]{\fullthetapMLE[t]}^{-1} \right) \bigg) \\
    &\leq 
    \dfrac{1}{2} \left( 
    \left\| \bOmega_{t}^{1/2} \big( \fullthetapMLE - \mu_{t}  \big) \right\|_{2}^{2}
    + 
    \left\| \FullFisherTilde[t]{\fullthetapMLE[t]}^{-1/2} \bOmega_{t} \FullFisherTilde[t]{\fullthetapMLE[t]}^{-1/2} - \bI_{p} \right\|_{\rm F}^{2} 
    \right)^{1/2} \\
    &\leq
    \dfrac{1}{2} \Bigg( 
    K_{2} M_n^{2} \left( \dfrac{p_{\ast}^{3}}{n} \right)^{1/2}
    + 
    K_{3} M_{n} \left( \dfrac{p_{\ast}^{2}}{n} \right)^{1/2}
    \Bigg) \\
    &\leq
    \dfrac{1}{2} \big( K_{2} + K_{3} \big) M_n^{2} \left( \dfrac{p_{\ast}^{3}}{n} \right)^{1/2},
\end{align*}
where the second inequality holds by $x^{2} + y^{2} \leq (x + y)^{2}$ for $x, y \geq 0$.
Therefore, we have
\begin{align*}
    d_{V} \bigg( \Pi_{t}(\cdot),  \Pi\left(\cdot \mid \bD_{1:t} \right) \bigg)
    &\leq 
    \dfrac{ K_{2} + K_{3} }{2} M_n^{2} \left( \dfrac{p_{\ast}^{3}}{n} \right)^{1/2} + 
    K_{1} \left( \dfrac{p_{\ast}^2}{nt} \right)^{1/2} \\
    &\leq 
    \dfrac{ K_{2} + K_{3} + 2K_{1} }{2} M_n^{2} \left( \dfrac{p_{\ast}^{3}}{n} \right)^{1/2}
    = 
    K_4 M_n^{2} \left( \dfrac{p_{\ast}^{3}}{n} \right)^{1/2},
\end{align*}
where $K_4 = (K_{2} + K_{3} + 2K_{1})/2$.

If we further assume (\textbf{P$\ast$}), we can employ Theorem \ref{thm:batch_posterior_BvM}.
For all $t \in [T]$, it holds that
\begin{align*}
    d_{V} \bigg( \cN\left(\fullthetaMLE, \FullFisher[t]{\theta_0}^{-1} \right),  \Pi \left(\cdot \mid \bD_{1:t} \right) \bigg) 
    &\leq K_5 M_n \left( \dfrac{p_{\ast}^2}{nt} \right)^{1/2}, 
\end{align*}
where $K_5 = K_5(K_{\min}, K_{\max})$ is the constant specified in Theorem \ref{thm:batch_posterior_BvM}.
It follows that
\begin{align*}
    d_{V} \left( \Pi_{t}, \cN\left( \fullthetaMLE, \FullFisher[t]{\theta_0}^{-1} \right) \right) 
    &\leq 
    d_{V} \Big( \Pi_{t}(\cdot),  \Pi\left(\cdot \mid \bD_{1:t} \right) \Big) 
    +
    d_{V} \Big( \cN\left(\fullthetaMLE, \FullFisher[t]{\theta_0}^{-1} \right),  \Pi \left(\cdot \mid \bD_{1:t} \right) \Big) \\
    &\leq
    K_4 M_n^{2} \left( \dfrac{p_{\ast}^{3}}{n} \right)^{1/2}
    +
    K_5 M_n \left( \dfrac{p_{\ast}^2}{nt} \right)^{1/2}  \\
    &\leq 
    \big( K_4 + K_5 \big) M_n^{2} \left( \dfrac{p_{\ast}^{3}}{n} \right)^{1/2},
\end{align*}
which completes the proof.
\end{proof}

\begin{corollary} \label{coro:MLE_estimation}
    Suppose that (\textbf{A0}), (\textbf{A1$\ast$}), (\textbf{A2}), (\textbf{S}) and (\textbf{P$\ast$}) hold.
    Then, on $\scrE_{\est, 2}$, for all $t \in [T]$, $\bbE L_{1:t}(\theta)$ satisfies the third order smoothness at $\theta_0$ with parameters
    \begin{align*}
    \left( K_{\max} K_{\min}^{-3/2} (nt)^{-1/2}, \: \FullFisher[t]{\theta_0}, \: 4r_{\eff, 1:t}  \right).
    \end{align*}
    Furthermore, the following inequalities hold uniformly for all $t \in [T]$: 
    \begin{align*} 
     \left\|  \FullFisher[t]{\theta_0}^{1/2} \big( \fullthetaMLE - \theta_0 \big) \right\|_{2} 
     &\leq 4r_{\eff, 1:t}
     \leq 4M_n p_{\ast}^{1/2}, \\
     \left\| \FullFisherTilde[t]{\fullthetapMLE}^{1/2} \big( \fullthetaMLE - \fullthetapMLE \big) \right\|_{2} 
     &\leq K M_n\left( \dfrac{p_{\ast}^{2}}{nt} \right)^{1/2},
    \end{align*}   
    where $K = K(K_{\min}, K_{\max})$.
\end{corollary}

\begin{proof}
    See \textbf{Step 1} in the proof of Theorem \ref{thm:batch_posterior_BvM}.
\end{proof}

\begin{corollary} \label{coro:efficient_estimator} 
    Suppose that (\textbf{A0}), (\textbf{A1$\ast$}), $(\textbf{A2})$, $(\textbf{S})$ and (\textbf{P$\ast$}) hold.
    Then, on $\scrE_{\est, 1} \cap \scrE_{\est, 2}$, the following inequalities holds uniformly for all $t \in [T]$:
    \begin{align*}
    \left\| \FullFisher[t]{\theta_0}^{1/2} \big( \fullthetaMLE - \theta_0 \big) - \FullFisher[t]{\theta_0}^{-1/2} \nabla L_{1:t}(\theta_0) \right\|_2
    &\leq 
    K M_n^2 \left( \dfrac{p_\ast^2}{n} \right)^{1/2}, \\
    \left\| \FullFisher[t]{\theta_0}^{1/2} \big( \thetaMAP - \theta_0 \big) - \FullFisher[t]{\theta_0}^{-1/2} \nabla L_{1:t}(\theta_0) \right\|_2
    &\leq 
    K M_n^2 \left( \dfrac{p_\ast^3}{n} \right)^{1/2}, \\
    \left\| \FullFisher[t]{\theta_0}^{1/2} \big( \mu_t - \theta_0 \big) - \FullFisher[t]{\theta_0}^{-1/2} \nabla L_{1:t}(\theta_0) \right\|_2
    &\leq 
    K M_n^2 \left( \dfrac{p_\ast^3}{n} \right)^{1/2}, 
    \end{align*}    
    where $K = K(K_{\min}, K_{\max}, K_{\rm low}, K_{\rm up}) > 0$.
\end{corollary}

\begin{proof}
    In this proof, we work on the event $\scrE_{\est, 1} \cap \scrE_{\est, 2}$ without explicitly referring to it. By Lemmas \ref{lemma:full_posterior_optimal_parameter}, \ref{lemma:full_posterior_parameter} and Proposition \ref{prop:eigenvalue_order}, for all $t \in [T]$, we have
    \begin{align}
    \begin{aligned} \label{eqn:efficient_estimator_regularity}
        \lambda_{\min}(\FisherMAP[t]) &\geq K_{\rm low} nt, \\
        \thetaMAP, \fullthetapMLE 
        &\in \Theta(\theta_{0}, \bI_{p}, 1/4), \\
        \Delta_{t} 
        &\leq K_{\rm up} t^{-1} n^{-1/2} p_{\ast}.
    \end{aligned}
    \end{align}
    Let $t \in [T]$. By Taylor's theorem, we have
    \begin{align*}
        - \nabla L_{1:t}(\theta_0) 
        &= \nabla L_{1:t}(\fullthetaMLE) - \nabla L_{1:t}(\theta_0)
        = - \overline \bF_{1:t} \big( \fullthetaMLE - \theta_0 \big) \\
        &= - \FullFisher[t]{\theta_0} \big( \fullthetaMLE - \theta_0 \big) 
           - \big( \overline \bF_{1:t} - \FullFisher[t]{\theta_0} \big) \big( \fullthetaMLE - \theta_0 \big).
    \end{align*}
    where
    \begin{align*}
        \overline \bF_{1:t}
         = \overline \bF_{1:t}(\fullthetaMLE, \theta_0) 
         = - \int_{0}^{1} \nabla^2 L_{1:t}\big( s \fullthetaMLE + (1-s) \theta_0 \big) \rmd s.
    \end{align*}
    It follows that
    \begin{align*}
        \FullFisher[t]{\theta_0}^{1/2} \big( \fullthetaMLE - \theta_0 \big)
        &=
        \FullFisher[t]{\theta_0}^{-1/2} \nabla L_{1:t}(\theta_0) 
        +
        \FullFisher[t]{\theta_0}^{-1/2} \big( \overline \bF_{1:t} - \FullFisher[t]{\theta_0} \big) \FullFisher[t]{\theta_0}^{-1/2} \FullFisher[t]{\theta_0}^{1/2} \big( \fullthetaMLE - \theta_0 \big) \\
        &=
        \FullFisher[t]{\theta_0}^{-1/2} \nabla L_{1:t}(\theta_0) 
        +
        \big( \FullFisher[t]{\theta_0}^{-1/2} \overline \bF_{1:t} \FullFisher[t]{\theta_0}^{-1/2} - \bI_p \big) \FullFisher[t]{\theta_0}^{1/2} \big( \fullthetaMLE - \theta_0 \big).
    \end{align*}
    By Corollary \ref{coro:MLE_estimation} and Lemma \ref{lemma:tech_Fisher_smooth}, we have
    \begin{align*}
        \left\|  \FullFisher[t]{\theta_0}^{1/2} \big( \fullthetaMLE - \theta_0 \big) \right\|_{2} 
        \leq 4M_n p_{\ast}^{1/2},
    \end{align*}
    and
    \begin{align*}
        \left\| \FullFisher[t]{\theta_0}^{-1/2} \overline \bF_{1:t} \FullFisher[t]{\theta_0}^{-1/2} - \bI_p \right\|_2
        &\leq 
        \big( K_{\max} K_{\min}^{-3/2} (nt)^{-1/2} \big)
        \big( 4M_n p_{\ast}^{1/2} \big) \\
        &= c_1 M_n p_{\ast}^{1/2} (nt)^{-1/2},
    \end{align*}
    where $c_1 = 4K_{\max} K_{\min}^{-3/2}$. Consequently, we have
    \begin{align*}
        &\left\|
        \FullFisher[t]{\theta_0}^{1/2} \big( \fullthetaMLE - \theta_0 \big)
        - 
        \FullFisher[t]{\theta_0}^{-1/2} \nabla L_{1:t}(\theta_0)
        \right\|_2 \\
        &\leq 
        \left\| \FullFisher[t]{\theta_0}^{-1/2} \overline \bF_{1:t} \FullFisher[t]{\theta_0}^{-1/2} - \bI_p \right\|_2 
        \left\| \FullFisher[t]{\theta_0}^{1/2} \big( \fullthetaMLE - \theta_0 \big) \right\|_2 
        \leq 
        4c_1 M_n^2 p_{\ast} (nt)^{-1/2}.
    \end{align*}
    Note that
    \begin{align*}
        &\left\| \FullFisher[t]{\theta_0}^{1/2} \big( \fullthetaMLE - \thetaMAP \big)  \right\|_2
        \leq 
        \left\| \FullFisher[t]{\theta_0}^{1/2} \big( \fullthetaMLE - \fullthetapMLE \big)  \right\|_2
        +
        \left\| \FullFisher[t]{\theta_0}^{1/2} \big( \fullthetapMLE - \thetaMAP \big)  \right\|_2 \\
        &\leq 
        \left\| \FullFisher[t]{\theta_0}^{1/2} \FullFisherTilde[t]{\fullthetapMLE}^{-1/2}\right\|_2
        \left\| \FullFisherTilde[t]{\fullthetapMLE}^{1/2} \big( \fullthetaMLE - \fullthetapMLE \big)  \right\|_2
        +
        \left\| \FullFisher[t]{\theta_0}^{1/2} \FisherMAP^{-1/2}\right\|_2
        \left\| \FisherMAP^{1/2} \big( \fullthetapMLE - \thetaMAP \big)  \right\|_2 \\
        \overset{ \substack{(\textbf{A2}) \\ \eqref{eqn:efficient_estimator_regularity} }  }&{\leq}
        K_{\max}^{1/2} K_{\min}^{-1/2}
        \left\| \FullFisherTilde[t]{\fullthetapMLE}^{1/2} \big( \fullthetaMLE - \fullthetapMLE \big)  \right\|_2
        +
        K_{\max}^{1/2} K_{\rm low}^{-1/2}
        \left\| \FisherMAP^{1/2} \big( \fullthetapMLE - \thetaMAP \big)  \right\|_2 \\
        \overset{ \substack{ \text{Corollary \ref{coro:MLE_estimation}} \\ \text{Proposition \ref{prop:similar_MAP}}  }  }&{\leq}
        c_2 \left( M_n \sqrt{\dfrac{p_\ast^2}{nt}} + M_n^2 \sqrt{\dfrac{p_\ast^3}{n}}  \right) 
        \leq 
        2c_2 M_n^2 \sqrt{\dfrac{p_\ast^3}{n}}
    \end{align*}
    for some constant $c_2 = c_2(K_{\min}, K_{\max}, K_{\rm low}, K_{\rm up}) > 0$.
    Also, we have
    \begin{align*}
        &\left\| \FullFisher[t]{\theta_0}^{1/2} \big( \fullthetaMLE - \mu_t \big)  \right\|_2
        \leq 
        \left\| \FullFisher[t]{\theta_0}^{1/2} \big( \fullthetaMLE - \thetaMAP \big)  \right\|_2
        +
        \left\| \FullFisher[t]{\theta_0}^{1/2} \big( \thetaMAP - \mu_t \big)  \right\|_2 \\
        &\leq 
        \left\| \FullFisher[t]{\theta_0}^{1/2} \big( \fullthetaMLE - \thetaMAP \big)  \right\|_2
        +
        \left\| \FullFisher[t]{\theta_0}^{1/2} \FisherMAP^{-1/2}\right\|_2
        \left\| \FisherMAP^{1/2} \big( \thetaMAP - \mu_t \big)  \right\|_2 \\
        \overset{ \substack{\eqref{eqn:efficient_estimator_regularity} }  }&{\leq}
        2c_2 M_n^2 \sqrt{\dfrac{p_\ast^3}{n}}
        +
        K_{\max}^{1/2} K_{\rm low}^{-1/2} K_{\rm up} t^{-1} n^{-1/2} p_{\ast} 
        \leq
        c_3 M_n^2 \sqrt{\dfrac{p_\ast^3}{n}}
    \end{align*}
    for some constant $c_3 = c_3(K_{\min}, K_{\max}, K_{\rm low}, K_{\rm up}) > 0$.
    Therefore, we have
    \begin{align*}
        &\left\| \FullFisher[t]{\theta_0}^{1/2} \big( \thetaMAP - \theta_0 \big) - \FullFisher[t]{\theta_0}^{-1/2} \nabla L_{1:t}(\theta_0) \right\|_2 \\
        &\leq 
        \left\| \FullFisher[t]{\theta_0}^{1/2} \big( \fullthetaMLE - \theta_0 \big) - \FullFisher[t]{\theta_0}^{-1/2} \nabla L_{1:t}(\theta_0) \right\|_2
        +
        \left\| \FullFisher[t]{\theta_0}^{1/2} \big( \fullthetaMLE - \thetaMAP \big)  \right\|_2 \\
        &\leq 
        4c_1 M_n^2 p_{\ast} (nt)^{-1/2} + 2c_2 M_n^2 \sqrt{\dfrac{p_\ast^3}{n}}
        \leq 
        \big( 4c_1 + 2c_2 \big) M_n^2 \sqrt{\dfrac{p_\ast^3}{n}}.
    \end{align*}
    Similarly, 
    \begin{align*}
        &\left\| \FullFisher[t]{\theta_0}^{1/2} \big( \mu_t - \theta_0 \big) - \FullFisher[t]{\theta_0}^{-1/2} \nabla L_{1:t}(\theta_0) \right\|_2 \\
        &\leq 
        \left\| \FullFisher[t]{\theta_0}^{1/2} \big( \fullthetaMLE - \theta_0 \big) - \FullFisher[t]{\theta_0}^{-1/2} \nabla L_{1:t}(\theta_0) \right\|_2
        +
        \left\| \FullFisher[t]{\theta_0}^{1/2} \big( \fullthetaMLE - \mu_t \big)  \right\|_2 \\
        &\leq 
        4c_1 M_n^2 p_{\ast} (nt)^{-1/2} + c_3 M_n^2 \sqrt{\dfrac{p_\ast^3}{n}}
        \leq 
        \big( 4c_1 + c_3 \big) M_n^2 \sqrt{\dfrac{p_\ast^3}{n}}.
    \end{align*}    
    This completes the proof.
\end{proof}

\begin{corollary} \label{coro:variance_estimation}
    Suppose that (\textbf{A0}), (\textbf{A1$\ast$}), $(\textbf{A2})$, $(\textbf{S})$ and (\textbf{P$\ast$}) hold.
    Then, on $\scrE_{\est, 1} \cap \scrE_{\est, 2}$, the following inequalities holds uniformly for all $t \in [T]$:
    \begin{align*} 
        \lambda_{\min} \big( \bOmega_{t} \big) \wedge \lambda_{\min} \big( \FullFisher[t]{\theta_0} \big) &\geq K_1 nt, \\
        \lambda_{\max} \big( \FullFisherTilde[t]{\fullthetapMLE} \big) \wedge \lambda_{\max} \big( \FullFisher[t]{\theta_0} \big) &\leq K_2 nt, \\
        \left\| \FullFisher[t]{\theta_0}^{-1/2} \FullFisherTilde[t]{\fullthetapMLE} \FullFisher[t]{\theta_0}^{-1/2} - \bI_p \right\|_{\rm F} &\leq K_2 M_n \left( \dfrac{p_{\ast}^2}{nt} \right)^{1/2}, \\
        \left\| \FullFisher[t]{\theta_0}^{1/2} \big( \fullthetaMLE - \mu_t \big)  \right\|_2
        &\leq K_2 M_n^2 \left( \dfrac{p_{\ast}^3}{n} \right)^{1/2},
    \end{align*}   
    where $K_1$ and $K_2$ are positive constants depending only on $(K_{\min}, K_{\max})$.
\end{corollary}

\begin{proof}
    The first two assertions directly follow from Proposition \ref{prop:eigenvalue_order} and assumption (\textbf{A2}). For the proof of the third assertion, see \textbf{Step 2} in Theorem \ref{thm:batch_posterior_BvM}. The last assertion follows from the proof of Corollary \ref{coro:efficient_estimator}.
\end{proof}

\begin{corollary} \label{coro:variance_consistency}
    Suppose that (\textbf{A0}), (\textbf{A1$\ast$}), $(\textbf{A2})$, $(\textbf{S})$ and (\textbf{P$\ast$}) hold.
    Then, on $\scrE_{\est, 1} \cap \scrE_{\est, 2}$, the following inequalities holds uniformly for all $t \in [T]$:
    \begin{align*} 
        nt \left\| \bOmega_t^{-1} - \FullFisher[t]{\theta_0}^{-1} \right\|_{\rm F} 
        &\leq K M_n \left( \dfrac{p_{\ast}^2}{n} \right)^{1/2}, \\
        \big\| \FullFisher[t]{\theta_0}^{-1/2} \bOmega_t \FullFisher[t]{\theta_0}^{-1/2} - \bI_p \big\|_{\rm F} 
        &\leq K M_n \left( \dfrac{p_{\ast}^2}{n} \right)^{1/2}, \\
        \left\| \bOmega_t^{1/2} \big( \fullthetaMLE - \mu_t \big)  \right\|_2
        &\leq K M_n^2 \left( \dfrac{p_{\ast}^3}{n} \right)^{1/2}
    \end{align*}   
    where $K = K(K_{\min}, K_{\max})$.
\end{corollary}

\begin{proof}
    In this proof, we work on the event $\scrE_{\est, 1} \cap \scrE_{\est, 2}$ without explicitly referring to it. Let $t \in [T]$ and $N_t = nt$ in this proof. Note that
    \begin{align*}
        \big\| \bOmega_t^{-1} - \FullFisher[t]{\theta_0}^{-1} \big\|_{\rm F} 
        &\leq 
        \big\| \bOmega_t^{-1} \big\|_2
        \big\| \FullFisher[t]{\theta_0}^{-1/2} \bOmega_t \FullFisher[t]{\theta_0}^{-1/2} - \bI_p \big\|_{\rm F}  \\ 
        &\leq
        \big\| \bOmega_t^{-1} \big\|_2
        \big\| \FullFisher[t]{\theta_0}^{-1} \big\|_2
        \big\| \bOmega_t - \FullFisher[t]{\theta_0} \big\|_{\rm F} \\
        \overset{ \text{Corollary \ref{coro:variance_estimation}} }&{\leq}
        K_1^{-2} N_t^{-2} \big\| \bOmega_t - \FullFisher[t]{\theta_0} \big\|_{\rm F},  
    \end{align*}
    where $K_1$ is the constant specified in Corollary \ref{coro:variance_estimation}. Hence, we only need to obtain an upper bound of $\| \bOmega_t - \FullFisher[t]{\theta_0} \|_{\rm F}$.
    
    Note that
    \begin{align*}
        \left\| \bOmega_{t} - \FullFisherTilde[t]{\fullthetapMLE} \right\|_{\rm F}
        &\leq 
        \left\| \FullFisherTilde[t]{\fullthetapMLE} \right\|_2
        \left\| \FullFisherTilde[t]{\fullthetapMLE}^{-1/2} \bOmega_{t} \FullFisherTilde[t]{\fullthetapMLE}^{-1/2} - \bI_p \right\|_{\rm F} \\
        \overset{ \substack{ \text{Corollary \ref{coro:variance_estimation}} \\ \text{Proposition \ref{prop:similar_variance}} } }&{\leq}
        \big( K_2 N_t \big) K_3 M_n \left( \dfrac{p_{\ast}^2 }{n} \right)^{1/2} 
        =
        \big( K_2 K_3 \big) N_t M_n \left( \dfrac{p_{\ast}^2 }{n} \right)^{1/2}, 
    \end{align*}
    where $K_2$ is the constant specified in Corollary \ref{coro:variance_estimation}, and $K_3$ denotes the constant $K$ in Proposition \ref{prop:similar_variance}. Also,
    \begin{align*}
        \left\| \FullFisherTilde[t]{\fullthetapMLE} - \FullFisher[t]{\theta_0} \right\|_{\rm F}
        &\leq 
        \Big\| \FullFisher[t]{\theta_0} \Big\|_2
        \left\| \FullFisher[t]{\theta_0}^{-1/2} \FullFisherTilde[t]{\fullthetapMLE} \FullFisher[t]{\theta_0}^{-1/2} - \bI_p \right\|_{\rm F} \\
        \overset{ \substack{ \text{Corollary \ref{coro:variance_estimation}} } }&{\leq}
        \big( K_2 N_t \big) K_2 M_n \left( \dfrac{p_{\ast}^2 }{N_t} \right)^{1/2} 
        =
        K_2^2  M_n p_{\ast} N_t^{1/2}.
    \end{align*}
    Consequently, we have
    \begin{align*}
        \big\| \bOmega_t - \FullFisher[t]{\theta_0} \big\|_{\rm F}
        &\leq 
        \left\| \bOmega_{t} - \FullFisherTilde[t]{\fullthetapMLE} \right\|_{\rm F}
        +
        \left\| \FullFisherTilde[t]{\fullthetapMLE} - \FullFisher[t]{\theta_0} \right\|_{\rm F} \\
        &\leq 
        \big( K_2 K_3 \big) N_t M_n \left( \dfrac{p_{\ast}^2 }{n} \right)^{1/2}
        +
        K_2^2  M_n p_{\ast} N_t^{1/2} \\
        &\leq 
        \big( K_2 K_3 + K_2^2 \big) N_t M_n \left( \dfrac{p_{\ast}^2 }{n} \right)^{1/2},
    \end{align*}
    which implies that
    \begin{align*}
        \big\| \bOmega_t^{-1} - \FullFisher[t]{\theta_0}^{-1} \big\|_{\rm F} 
        &\leq 
        K_1^{-2} N_t^{-2} 
        \big( K_2 K_3 + K_2^2 \big) N_t M_n \left( \dfrac{p_{\ast}^2 }{n} \right)^{1/2} \\
        &= 
        K_1^{-2} \big( K_2 K_3 + K_2^2 \big) N_t^{-1} M_n \left( \dfrac{p_{\ast}^2 }{n} \right)^{1/2}.
    \end{align*}
    Also,
    \begin{align*}
        \big\| \FullFisher[t]{\theta_0}^{-1/2} \bOmega_t \FullFisher[t]{\theta_0}^{-1/2} - \bI_p \big\|_{\rm F}
        &\leq
        \big\| \FullFisher[t]{\theta_0}^{-1} \big\|_2
        \big\| \bOmega_t - \FullFisher[t]{\theta_0} \big\|_{\rm F} \\
        \overset{ \text{Corollary \ref{coro:variance_estimation}} }&{\leq}
        K_1^{-1} N_t^{-1} \big\| \bOmega_t - \FullFisher[t]{\theta_0} \big\|_{\rm F} \\
        &\leq 
        K_1^{-1} \big( K_2 K_3 + K_2^2 \big) M_n \left( \dfrac{p_{\ast}^2 }{n} \right)^{1/2},
    \end{align*}
    which, combining with the last assertion in Corollary \ref{coro:variance_estimation}, implies that 
    \begin{align*}
        \left\| \bOmega_t^{1/2} \big( \fullthetaMLE - \mu_t \big)  \right\|_2
        &\leq K_4 M_n^2 \left( \dfrac{p_{\ast}^3}{n} \right)^{1/2}
    \end{align*}
    for some positive constant $K_4$ depending only on $(K_{\min}, K_{\max})$.
\end{proof}

\begin{corollary} \label{coro:CI_consistency}
    Suppose that (\textbf{A0}), (\textbf{A1$\ast$}), $(\textbf{A2})$, $(\textbf{S})$ and (\textbf{P$\ast$}) hold.
    Then, on $\scrE_{\est, 1} \cap \scrE_{\est, 2}$, the following inequalities holds uniformly for all $t \in [T]$:
    \begin{align*} 
        \left| 
            \big\| \bOmega_t^{1/2} \big( \theta_0 - \mu_t \big) \big\|_2
            - 
            \big\| \FullFisher[t]{\theta_0}^{1/2} \big( \theta_0 - \fullthetaMLE \big) \big\|_2
        \right|
        \leq 
        K M_n \left( \dfrac{p_{\ast}^3}{n} \right)^{1/2},
    \end{align*}   
    where $K = K(K_{\min}, K_{\max})$.
\end{corollary}
\begin{proof}
    In this proof, we work on the event $\scrE_{\est, 1} \cap \scrE_{\est, 2}$ without explicitly referring to it. Let $t \in [T]$. By Corollaries \ref{coro:MLE_estimation} and \ref{coro:variance_consistency}, we have
    \begin{align*}
        \big\|  \FullFisher[t]{\theta_0}^{1/2} \big( \fullthetaMLE - \theta_0 \big) \big\|_{2} 
        &\leq 4M_n p_{\ast}^{1/2}, \\
        \big\| \FullFisher[t]{\theta_0}^{-1/2} \bOmega_t \FullFisher[t]{\theta_0}^{-1/2} - \bI_p \big\|_{\rm F} 
        &\leq K_1 M_n \left( \dfrac{p_{\ast}^2}{n} \right)^{1/2}, \\
        \big\| \bOmega_t^{1/2} \big( \fullthetaMLE - \mu_t \big)  \big\|_2
        &\leq K_1 M_n^2 \left( \dfrac{p_{\ast}^3}{n} \right)^{1/2},
    \end{align*}
    where $K_1 = K_1(K_{\min}, K_{\max})$ denotes the constant $K$ in Corollary \ref{coro:variance_consistency}.   
    Let $\epsilon_{n, 2} = K_1 M_n \left( p_{\ast}^2 / n \right)^{1/2}$ and $\epsilon_{n, 3} = K_1 M_n^2 \left( p_{\ast}^3 / n \right)^{1/2}$ in this proof.
    Note that
    \begin{align*}
        \big\| \bOmega_t^{1/2} \big( \theta_0 - \mu_t \big) \big\|_2
        &\leq 
        \big\| \bOmega_t^{1/2} \big( \theta_0 - \fullthetaMLE \big) \big\|_2
        +
        \big\| \bOmega_t^{1/2} \big( \fullthetaMLE - \mu_t \big) \big\|_2 \\
        &\leq 
        \big( 1 + \epsilon_{n, 2} \big)^{1/2}
        \big\| \FullFisher[t]{\theta_0}^{1/2} \big( \theta_0 - \fullthetaMLE \big) \big\|_2
        +
        \big\| \bOmega_t^{1/2} \big( \fullthetaMLE - \mu_t \big) \big\|_2 \\
        &\leq 
        \big( 1 + \epsilon_{n, 2} \big)
        \big\| \FullFisher[t]{\theta_0}^{1/2} \big( \theta_0 - \fullthetaMLE \big) \big\|_2
        +
        \big\| \bOmega_t^{1/2} \big( \fullthetaMLE - \mu_t \big) \big\|_2 \\
        &\leq 
        \big\| \FullFisher[t]{\theta_0}^{1/2} \big( \theta_0 - \fullthetaMLE \big) \big\|_2
        +
        \epsilon_{n, 2} \big( 4M_n p_{\ast}^{1/2} \big)
        +
        \epsilon_{n, 3},
    \end{align*}
    which implies that
    \begin{align*}
        \big\| \bOmega_t^{1/2} \big( \theta_0 - \mu_t \big) \big\|_2
        - 
        \big\| \FullFisher[t]{\theta_0}^{1/2} \big( \theta_0 - \fullthetaMLE \big) \big\|_2
        \leq 
        5 \epsilon_{n, 3}. 
    \end{align*}
    Also,  
    \begin{align*}
        \big\| \bOmega_t^{1/2} \big( \theta_0 - \mu_t \big) \big\|_2
        &\geq 
        \big\| \bOmega_t^{1/2} \big( \theta_0 - \fullthetaMLE \big) \big\|_2
        -
        \big\| \bOmega_t^{1/2} \big( \fullthetaMLE - \mu_t \big) \big\|_2 \\
        &\geq 
        \big( 1 - \epsilon_{n, 2} \big)^{1/2}
        \big\| \FullFisher[t]{\theta_0}^{1/2} \big( \theta_0 - \fullthetaMLE \big) \big\|_2
        -
        \big\| \bOmega_t^{1/2} \big( \fullthetaMLE - \mu_t \big) \big\|_2 \\
        &\geq 
        \big( 1 - \epsilon_{n, 2} \big)
        \big\| \FullFisher[t]{\theta_0}^{1/2} \big( \theta_0 - \fullthetaMLE \big) \big\|_2
        -
        \big\| \bOmega_t^{1/2} \big( \fullthetaMLE - \mu_t \big) \big\|_2 \\
        &\geq 
        \big\| \FullFisher[t]{\theta_0}^{1/2} \big( \theta_0 - \fullthetaMLE \big) \big\|_2
        -
        \epsilon_{n, 2} \big( 4M_n p_{\ast}^{1/2} \big)
        -
        \epsilon_{n, 3},
    \end{align*}
    which implies that
    \begin{align*}
        \big\| \bOmega_t^{1/2} \big( \theta_0 - \mu_t \big) \big\|_2
        - 
        \big\| \FullFisher[t]{\theta_0}^{1/2} \big( \theta_0 - \fullthetaMLE \big) \big\|_2
        \geq 
        -5 \epsilon_{n, 3}. 
    \end{align*}
    This completes the proof.
\end{proof}

\section{Logistic regression with Gaussian design} \label{sec:logit_example_app}

In this section, we demonstrate that the main results in Section \ref{sec:online_variational_posterior} hold under the logistic regression model.
First, we introduce some notations needed for the theoretical verifications of this model.
Let $\bY = (Y_i)_{i \in [N]} \in \bbR^{N}$ be the response vector and $\bX = (X_{ij})_{i \in [N], j \in [p]} \in \bbR^{N \times p}$ be the design matrix.
Also, for $t \in [T]$, let
\begin{align*}
    &I_{t} = \{ n(t-1) + 1, n(t-1) + 2, ..., nt \}, \quad 
    I_{1:t} = \cup_{s=1}^{t} I_{s}, \\
    &\bY_{t} = (Y_{i})_{i \in I_{t}}, \quad
    \bX_{t} = (X_{ij})_{i \in I_{t}, j \in [p]}, \quad 
    \bX_{1:t} = (X_{ij})_{i \in I_{1:t}, j \in [p]} \in \bbR^{N_t \times p}.
\end{align*}
With slight abuse of notation, we denote $\bD_{t} = (\bX_{t}, \bY_{t})$ in this section.

For the logistic regression model, the likelihood function is given by
\begin{align} \label{def:logit_likelihood}
    L_{t}(\theta) = \sum_{i \in I_{t}} \left[ Y_i X_{i}^{\top} \theta - b(X_{i}^{\top} \theta) \right],
\end{align}
where $b(\cdot) = \log (1 + \exp(\cdot))$. 
Note that $b(\cdot)$ is four times differentiable with derivatives $b', b'', b'''$ and $b''''$, respectively.

In the sections in main text, we considered several regularity conditions, as assumed in (\textbf{A1}), (\textbf{A1$\ast$}), and (\textbf{A2}). These conditions can be verified under the logistic regression model with a ``well-posed'' design $\bX$. To see this, we consider a simple random matrix setup where each entry of the design matrix $\bX$ is an i.i.d. standard normal random variable, i.e., $X_{ij} \overset{\iid}{\sim} \cN(0, 1)$. For simplicity, we take the covariance matrix to be the identity matrix $\bI_p$; this setting can be easily extended to a general covariance $\bSigma$ satisfying
\begin{align*}
    C^{-1} \leq \lambda_{\min} \big( \bSigma \big) \leq \lambda_{\max} \big( \bSigma \big) \leq C
\end{align*}
for some constant $C > 0$. With slight abuse of notation, hereafter, let $\bbP$ and $\bbE$ denote the joint probability measure and expectation corresponding to $(\bX, \bY)$, respectively.

Under the assumed random design setup, we can verify the conditions in (\textbf{A1}), (\textbf{A1$\ast$}), and (\textbf{A2}). First, one can easily check that $L_{t}(\theta)$ in \eqref{def:logit_likelihood} is \textit{stochastically linear} (with respect to the randomness in $\bY$) as follows:
\begin{align*}
    \zeta_{t}(\theta) = L_{t}(\theta) - \bbE_{t} L_{t}(\theta) = 
    \sum_{i \in I_{t}} \bigg[ \big( Y_i - \bbE_{t} (Y_i \mid \bX) \big) X_{i}^{\top} \theta \bigg], \quad \forall t \in [T],
\end{align*}
where $\bbE_{t} (Y_i \mid \bX) =  b'(X_{i}^{\top} \theta_0)$ for each $i \in I_{t}$.
To verify the remaining assumptions, we impose (\textbf{EX}).
Under the assumption (\textbf{EX}), we can prove the conditions in (\textbf{A1}), (\textbf{A1$\ast$}), and (\textbf{A2}) hold uniformly for all $t \in [T]$ with the following quantities:
\begin{align*}
    \bV_{t} = \bX_t^{\top} \bX_t/4, \quad
    \bV_{1:t} = \bX_{1:t}^{\top} \bX_{1:t}/4, \quad
    M_n = C_1, \quad 
    K_{\min} = C_2, \quad
    K_{\max} = C_3,
\end{align*}
where $C_1$ and $C_2$ are positive constants depending only on $K_1$, and $C_3$ is a universal constant.
Technical statements and proofs are deferred to Appendix \ref{sec:proof_logit_example}; see Propositions \ref{prop:eigenvalues_logit}, \ref{prop:logit_V_verification}, \ref{prop:operator_norm_logit} for precise statements.  

We now state that the online BvM theorem holds for the logistic regression model.
\begin{proposition} \label{prop:logit_fin_statement}
    Suppose that (\textbf{EX}) holds.
    Then, with $\bbP$-probability at least $1 - 5n^{-1} - 10e^{-n/72} -4(Np)^{-1}$, the following inequality holds uniformly for all $t \in [T]$:
    \begin{align*}
        d_{V} \Big( \Pi_{t}, \Pi(\cdot \mid \bD_{1:t}) \Big) 
        &\leq  
        C  \left( \dfrac{p_{\ast}^{3}}{n} \right)^{1/2}, \\
        d_{V} \bigg( \Pi_{t}, \cN\left( \fullthetaMLE, \FullFisher[t]{\theta_0}^{-1} \right) \bigg) 
        &\leq  
        C  \left( \dfrac{p_{\ast}^{3}}{n} \right)^{1/2},
    \end{align*}
    where $C = C(K_1)$.
\end{proposition} 

\subsection{Proof for Proposition \ref{prop:logit_fin_statement}} \label{sec:proof_logit_example}
Throughout this subsection, we follow the notations given in Appendix \ref{sec:logit_example_app} without explicitly referring to them. 
\begin{lemma} \label{lemma:random_matrix_eigenvalue}
Suppose that
\begin{align} \label{assume:least_eigenvalue_lemma}
    p \vee (4\log T) \leq n.
\end{align}
Then,
\begin{align} \label{eqn:least_eigenvalue_claim}
    \bbP \left\{ 
    \lambda_{\min} \left( \sum_{i \in I_{t}}  X_{i} X_{i}^{\top} \right) 
    \leq \dfrac{1}{9} n \quad
    \text{ for some } t \in [T]
    \right\}
    \leq 2e^{-n/4}
\end{align}
and 
\begin{align} \label{eqn:least_eigenvalue_claim2}
    \bbP \left\{ 
    \lambda_{\max} \left( \sum_{i \in I_{t}}  X_{i} X_{i}^{\top} \right) 
    \geq 9 n \quad
    \text{ for some } t \in [T]
    \right\}
    \leq 2e^{-n/4}.
\end{align}
\end{lemma}

\begin{proof}
By the equation (60) in \cite{wainwright2009sharp} and $p \leq n$, we have, for $t \in [T]$,
\begin{align*}
    \bbP \left\{ 
    \lambda_{\min} \left( \sum_{i \in I_{t}}  X_{i} X_{i}^{\top} \right) 
    \leq \dfrac{1}{9} n
    \right\}
    \leq 2e^{-n/2}.
\end{align*}
It follows that
\begin{align*}
    &\bbP \left\{ 
    \lambda_{\min} \left( \sum_{i \in I_{t}}  X_{i} X_{i}^{\top} \right) 
    \leq \dfrac{1}{9} n \
    \text{ for some } t \in [T]
    \right\} \\
    &\leq  
    T \cdot
    \max_{t \in [T]}  \bbP \left\{ 
    \lambda_{\min} \left( \sum_{i \in I_{t}}  X_{i} X_{i}^{\top} \right) 
    \leq \dfrac{1}{9} n
    \right\} 
    \leq 2e^{-n/2 + \log T}
    \overset{\eqref{assume:least_eigenvalue_lemma}}{\leq} 2e^{-n/4},
\end{align*}
completing the proof of \eqref{eqn:least_eigenvalue_claim}. 

The proof of \eqref{eqn:least_eigenvalue_claim2} is similar. 
By the equation (59) in \cite{wainwright2009sharp} and $p \leq n$, we have, for $t \in [T]$,
\begin{align*}
    \bbP \left\{ 
    \lambda_{\max} \left( \sum_{i \in I_{t}}  X_{i} X_{i}^{\top} \right) 
    \geq 9n
    \right\}
    \leq 2e^{-n/2}.
\end{align*}
It follows that
\begin{align*}
    &\bbP \left\{ 
    \lambda_{\max} \left( \sum_{i \in I_{t}}  X_{i} X_{i}^{\top} \right) 
    \geq 9 n \
    \text{ for some } t \in [T]
    \right\} \\
    &\leq  
    T \cdot
    \max_{t \in [T]}  \bbP \left\{ 
    \lambda_{\max} \left( \sum_{i \in I_{t}}  X_{i} X_{i}^{\top} \right) 
    \geq 9 n
    \right\}  
    \leq 2e^{-n/2 + \log T}
    \overset{\eqref{assume:least_eigenvalue_lemma}}{\leq} 2e^{-n/4},
\end{align*}
which completes the proof of \eqref{eqn:least_eigenvalue_claim2}.
\end{proof}

\begin{lemma} \label{lemma:design_row_norm}
We have
\begin{align} \label{eqn:design_row_norm_claim}
    \bbP \biggl\{ 
    \max_{i \in [N], j \in [p]} |X_{i,j}|
    >
    2 \sqrt{\log(Np)}
    \biggr\}
    \leq 2(Np)^{-1}
\end{align}
and
\begin{align} \label{eqn:design_row_norm_claim2}
    &\bbP \biggl\{ 
    \max_{i \in [N]} \left\| X_{i} \right\|_{2}^{2}
    > 
    4p\log(Np)
    \biggr\}
    \leq 2(Np)^{-1},
\end{align}
where $X_{i} = (X_{ij})_{j \in [p]} \in \bbR^{p}$.
\end{lemma}

\begin{proof}
Since $X_{ij} \overset{\iid}{\sim} \cN \left( 0, 1 \right)$, we have, for all $\omega \geq 0$, $i \in [N]$ and $j \in [p]$,
\begin{align*}
    \bbP \bigg( 
    |X_{ij}| > \omega
    \bigg) \leq 2 \exp \left( - \dfrac{\omega^2}{2} \right).
\end{align*}
It follows that 
\begin{align*}
    \bbP \bigg( 
    \max_{i \in [N], j \in [p]} |X_{ij}| > \omega
    \bigg) \leq  2Np \exp \left( - \dfrac{\omega^2}{2} \right).
\end{align*}
By taking $\omega = 2\sqrt{\log (Np)}$, we complete the proof of \eqref{eqn:design_row_norm_claim}.

Also, on the same event where the following inequality holds:
\begin{align*}
    \max_{i \in [N], j \in [p]} |X_{ij}| \leq 2\sqrt{\log (Np)},
\end{align*}
we have
\begin{align*}
    \max_{i \in [N]} \left\| X_{i} \right\|_{2}^{2} 
    \leq p \max_{i \in [N], j \in [p]} |X_{ij}|
    \leq p \left( 2\sqrt{\log (Np)} \right)^2.
\end{align*}
This completes the proof of \eqref{eqn:design_row_norm_claim2}.
\end{proof}

\begin{lemma} \label{lemma:GLM_b_ratio}
Let $b(\cdot) = \log(1 + \exp(\cdot))$.
Then, 
\begin{align*}
    \dfrac{b''\left( \eta_{1} \right)}{b''\left( \eta_{2} \right)}
    \leq e^{ 3\left| \eta_1 - \eta_2 \right|}, \quad \forall \eta_{1}, \eta_{2} \in \bbR.
\end{align*}
\end{lemma}
\begin{proof}
Let $\eta_1, \eta_2 \in \bbR$. 
Since $b''(\eta) = e^{\eta} / \left( 1 + e^{\eta} \right)^2$ for $\eta \in \bbR$, note that
\begin{align*}
\dfrac{b''\left( \eta_{1} \right)}{b''\left( \eta_{2} \right)}
= 
e^{\eta_1 - \eta_{2}}
\bigg(
\dfrac{
    1 + e^{\eta_2}
}{
    1 + e^{\eta_1}
}
\bigg)^2.
\end{align*}
Also,
\begin{align*}
\dfrac{ 1 + e^{\eta_2} }{ 1 + e^{\eta_1} }
= 1 + \dfrac{ e^{\eta_2} - e^{\eta_1} }{ 1 + e^{\eta_1} }
= 1 +  \dfrac{ e^{\eta_1} \left( e^{\eta_2 - \eta_1} - 1 \right) }{ 1 + e^{\eta_1} }
\leq  1 +  e^{|\eta_2 - \eta_1|} - 1 
= e^{|\eta_1 - \eta_2|}.
\end{align*}
It follows that 
\begin{align*}
\dfrac{b''\left( \eta_{1} \right)}{b''\left( \eta_{2} \right)} 
\leq  e^{\eta_1 - \eta_{2}} \times e^{2|\eta_1 - \eta_{2}|}  
\leq  e^{3|\eta_1 - \eta_{2}|},
\end{align*}
which completes the proof.
\end{proof}

\begin{lemma} \label{lemma:least_eigenvalue_logit}
For $\tau > 0$, suppose that
\begin{align} \label{assume:least_eigenvalue_logit}
    n \geq C \bigg( \log T \vee \big[ p \log \big( \tau^{2} p \log N \big) \big] \bigg)
\end{align}
for a large enough universal constant $C > 0$.
Then, 
\begin{align} \label{eqn:least_eigenvalue_logit}
    \dfrac{n}{1080 e^{2(\tau + 1)}} 
    \leq \min_{t \in T} \inf_{\theta \in \Theta (\bI_{p}, \tau)}  \lambda_{\min} \left( \bF_{t, \theta} \right) 
    \leq \max_{t \in T} \sup_{\theta \in \Theta} \lambda_{\max} \left( \bF_{t, \theta} \right) 
    \leq \dfrac{9}{4} n
\end{align}
with $\bbP$-probability at least $1 - 6e^{-n/72} - 2(Np)^{-1}$.
\end{lemma}

\begin{proof}
For $t \in [T]$ and $\theta \in \Theta$, note that
\begin{align*}
    \bF_{t, \theta} = \sum_{i \in I_{t}} \left[ b''\left( X_i^{\top} \theta \right) X_{i} X_{i}^{\top} \right].
\end{align*}
For $\tau > 0$ and $\epsilon \in (0, 1)$, let $\widehat{\Theta}_{\epsilon, \tau}$ be the $\epsilon$-cover of $\Theta (\bI_{p}, \tau)$. One can choose $\widehat{\Theta}_{\epsilon, \tau}$ so that $|\widehat{\Theta}_{\epsilon, \tau}| \leq (3\tau/\epsilon)^{p}$; see Proposition 1.3 of Section 15 in \cite{lorentz1996constructive}. Let $\theta \in \Theta (\bI_{p}, \tau)$. By the definition of $\widehat{\Theta}_{\epsilon, \tau}$, there exists $\widehat{\theta} (\theta) \in \widehat{\Theta}_{\epsilon, \tau}$ such that $\| \theta - \widehat{\theta} \|_{2} \leq \epsilon$.
For $\omega \geq 0$, let 
\begin{align*}
    \cI_{\omega}(\widehat{\theta}, t) = \cI_{\omega}(\widehat{\theta}, t, \tau) = \left\{ i \in I_{t} : \big| X_i^{\top} \widehat{\theta} \big| \leq \omega (\tau + 1) \right\}.
\end{align*}
Note that
\begin{align}
\begin{aligned} \label{eqn:least_eigenvalue_logit_eq1}
    \lambda_{\min} \left( \bF_{t, \theta} \right)
    &= 
    \lambda_{\min} \left( \sum_{i \in I_{t}} b'' ( X_i^{\top} \theta ) X_{i} X_{i}^{\top} \right)
    = 
    \lambda_{\min} \left( \sum_{i \in I_{t}} \dfrac{b''( X_i^{\top} \theta )}{b''( X_i^{\top} \widehat{\theta} )} b'' ( X_i^{\top} \widehat{\theta} ) X_{i} X_{i}^{\top} \right) \\    
    &\geq 
    \Bigg[ \min_{i \in [N]} \dfrac{b''( X_i^{\top} \theta )}{b''( X_i^{\top} \widehat{\theta} )} \Bigg]
    \lambda_{\min} \left( \sum_{i \in \cI_{\omega}(\widehat{\theta}, t)} b'' ( X_i^{\top} \widehat{\theta} ) X_{i} X_{i}^{\top} \right) \\
    \overset{\text{Lemma \ref{lemma:GLM_b_ratio}}}&{\geq}
    \exp\left( -3 \big\| \theta - \widehat{\theta} \big\|_{2} \max_{i \in [N]} \left\| X_{i} \right\|_{2}  \right)
    \lambda_{\min} \left( \sum_{i \in \cI_{\omega}(\widehat{\theta}, t)} b'' ( X_i^{\top} \widehat{\theta} ) X_{i} X_{i}^{\top} \right) \\    
    &\geq
    \exp\left( -3 \epsilon \cdot \max_{i \in [N]} \left\| X_{i} \right\|_{2}  \right)
    b''\big( \omega (\tau + 1) \big)  \lambda_{\min} \left( \sum_{i \in \cI_{\omega}(\widehat{\theta}, t)} X_{i} X_{i}^{\top} \right)
\end{aligned}
\end{align}
where the last inequality holds by the symmetry and monotonicity of $b''(\cdot)$ in the logistic regression model. 

First, for $\widehat{\theta} \in \widehat{\Theta}_{\epsilon, \tau}$ and $t \in [T]$,
we will prove that $| \cI_{2}(\widehat{\theta}, t) | \geq n/6$ with high probability.
Since $X_i^{\top} \widehat{\theta} \sim \cN(0, \| \widehat{\theta} \|_{2}^2)$ and 
\begin{align*}
    \| \widehat{\theta} \|_{2} 
    \leq \left\| \theta \right\|_{2} + \| \theta - \widehat{\theta} \|_{2}   
    \leq \tau + \epsilon \leq \tau + 1,
\end{align*}
we have, for $i \in I_{t}$,
\begin{align*}
    \bbP \bigg( \big| X_i^{\top} \widehat{\theta} \: \big| > \omega' (\tau + 1) \bigg)
    \leq 
    \bbP \bigg( \big| X_i^{\top} \widehat{\theta} \: \big| > \omega' \| \widehat{\theta} \|_{2} \bigg) \leq 2e^{-(\omega')^2/2}, \quad \forall \omega' \geq 0.
\end{align*}
By taking $\omega' = 2$, we have
\begin{align*}
    \bbP \bigg( \big| X_i^{\top} \widehat{\theta} \ \big| \leq 2(\tau + 1) \bigg) \geq 1 - 2e^{-2} \geq \dfrac{1}{3}.
\end{align*}
We will utilize the Chernoff-type left tail inequality (see Section 2.3 in \cite{vershynin2018high}). Let $S_n = \sum_{i=1}^{n} Z_i$,
where $Z_i \overset{\iid}{\sim} \operatorname{Bernoulli}(\eta)$ for some $\eta \in (0, 1)$. Then, for any $\delta \in (0, 1)$,
\begin{align*}
    \bbP \biggl\{ S_n \leq (1 - \delta) \eta n \biggr\} \leq \exp \left( - \dfrac{\delta^2}{3} \eta n \right).
\end{align*}
By taking $\delta = 1/2$ and $\eta = 1/3$ in the above display, we have, for $\widehat{\theta} \in \widehat{\Theta}_{\epsilon, \tau}$ and $t \in [T]$,
\begin{align*} 
    \bbP \bigg( |\cI_{2}(\widehat{\theta}, t)| \leq \dfrac{n}{6} \bigg) \leq e^{-n/36}.
\end{align*}
By taking $\epsilon = (4 \sqrt{p \log (Np)})^{-1}$, it follows that
\begin{align} 
\begin{aligned} \label{eqn:least_eigenvalue_logit_eq2}
    \bbP \left( \min_{\widehat{\theta} \in \widehat{\Theta}_{\epsilon, \tau}} \min_{t \in [T]} 
    |\cI_{2}(\widehat{\theta}, t)| \leq \dfrac{n}{6} \right) 
    &\leq (3\tau/\epsilon)^{p} \cdot T \cdot e^{-n/36} \\
    &= \exp\bigg( \dfrac{p}{2} \log(144\tau^{2}p \log(Np)) + \log T - \dfrac{n}{36} \bigg)
    \overset{\eqref{assume:least_eigenvalue_logit}}&{\leq} e^{-n/72}. 
\end{aligned}
\end{align}
Let
\begin{align*}
    \Omega_{n, 1} &= \biggl\{ |\cI_{2}(\widehat{\theta}, t)| \geq \dfrac{1}{6} n
        \ \text{ for all } t \in [T]  \text{ and } \widehat{\theta} \in \widehat{\Theta}_{\epsilon, \tau} \biggr\}, \\
    \Omega_{n, 2} &= \Biggl\{ 
        \lambda_{\min} \left( \sum_{i \in \cI_{2}(\widehat{\theta}, t)}  X_{i} X_{i}^{\top} \right) \geq \dfrac{1}{9} \left| \cI_{2} (\widehat{\theta}, t) \right| 
        \ \text{ for all } t \in [T] \text{ and } \widehat{\theta} \in \widehat{\Theta}_{\epsilon, \tau} \Biggr\}, \\
    \Omega_{n, 3} &= \Biggl\{ 
        \max_{i \in [N]} \left\| X_{i} \right\|_{2} \leq 2 \sqrt{p \log(Np)}   
    \Biggr\}.
\end{align*}
Since $p \vee 4\log T \leq n/6 \leq |\cI_{2}(\widehat{\theta}, t)|$ for all $t \in [T]$ on $\Omega_{n, 1}$, we can apply the results of Lemma \ref{lemma:random_matrix_eigenvalue} on $\Omega_{n, 1}$.
By the equation \eqref{eqn:least_eigenvalue_logit_eq2}, Lemmas \ref{lemma:random_matrix_eigenvalue} and \ref{lemma:design_row_norm},
\begin{align*}
\bbP \bigl\{ \Omega_{n, 1}^{\rm c} \bigr\} &\leq e^{-n/72}, \\
\bbP \bigl\{ \Omega_{n, 3}^{\rm c} \bigr\} &\leq 2(Np)^{-1} \\
\bbP \bigl\{ \Omega_{n, 2}^{\rm c} \mid \Omega_{n, 1}  \bigr\} 
&\leq (3\tau/\epsilon)^{p} \times 2e^{-(n/6)/4}
= 2\exp\left( -n/24 + \dfrac{p}{2} \log(144\tau^{2}p \log(Np)) \right) \\
\overset{\eqref{assume:least_eigenvalue_logit}}&{\leq} 2e^{-n/48}.
\end{align*}
By $1 - x \geq e^{-2x}$ and $e^{-y} \geq 1 - y$ for $x \in [0, 0.795]$ and $y \in \bbR$, note that
\begin{align*}
\bbP \bigl\{ \Omega_{n} \bigr\}
&\geq
1  
- \bbP \bigl\{ \Omega_{n, 1}^{\rm c} \bigr\}
- \bbP \bigl\{ \Omega_{n, 2}^{\rm c} \bigr\}
- \bbP \bigl\{ \Omega_{n, 3}^{\rm c} \bigr\} \\
&\geq 
1  
- 2\bbP \bigl\{ \Omega_{n, 1}^{\rm c} \bigr\}
- \bbP \bigl\{ \Omega_{n, 2}^{\rm c} \mid \Omega_{n, 1} \bigr\} 
- \bbP \bigl\{ \Omega_{n, 3}^{\rm c} \bigr\} \\
&\geq
1 - 2e^{-n/72} - 2e^{-n/48} - 2(Np)^{-1}  \\
&\geq 1 - 4e^{-n/72} - 2(Np)^{-1},
\end{align*}
where $\Omega_{n} = \Omega_{n, 1} \cap \Omega_{n, 2} \cap \Omega_{n, 3}$. 
On $\Omega_n$, therefore, we have
\begin{align*}
\min_{\theta \in \Theta_{n, \tau}} \min_{t \in [T]} \lambda_{\min} \left( \bF_{t, \theta} \right) 
\overset{\eqref{eqn:least_eigenvalue_logit_eq1}}&{\geq}
    \exp\left( -3 \epsilon \cdot \max_{i \in [N]} \left\| X_{i} \right\|_{2}  \right)
    b''\big( 2 (\tau + 1) \big)  
    \left( \dfrac{1}{9} \times \dfrac{n}{6} \right) \\
&\geq 
e^{-3/2} \times
\dfrac{ \exp\left( 2 (\tau + 1) \right) }{ \big[ 1 + \exp\left( 2 (\tau + 1) \right) \big]^2 } \times
\dfrac{n}{54} \\
&\geq 
\dfrac{n}{1080 e^{2(\tau + 1)}},
\end{align*}
where the third inequality holds by $e^{-3/2} \geq 1/5$ and $e^{x}/(1 + e^{x})^2 \geq 1/(4e^{x})$ for $x \geq 0$. 

The proof of the third inequality in \eqref{eqn:least_eigenvalue_logit} is simple. Since $b''(\cdot) \leq b''(0) = 1/4$, with $\bbP$-probability at least $1 - 2e^{-n/4}$, 
\begin{align*}
    \lambda_{\max} \left( \bF_{t, \theta} \right)
    = \lambda_{\max} \left( \sum_{i = 1}^{n} \left[ b''\left( X_i^{\top} \theta \right) X_{i} X_{i}^{\top} \right] \right) 
    \leq \dfrac{1}{4} \lambda_{\max} \left( \sum_{i = 1}^{n} X_{i} X_{i}^{\top} \right)
    \leq \dfrac{9}{4} n,
\end{align*}
where the second inequality holds by Lemma \ref{lemma:random_matrix_eigenvalue}. This completes the proof of \eqref{eqn:least_eigenvalue_logit}.
\end{proof}

The following proposition verifies \eqref{assume:A1_2}, \eqref{assume:A2_1} and \eqref{assume:A1ast_2}. 

\begin{proposition} \label{prop:eigenvalues_logit}
Suppose that (\textbf{EX}) holds.
Then, the following inequalities hold with $\bbP$-probability at least $1 - 4e^{-n/72} -2(Np)^{-1}$:
\begin{align} 
\label{eqn:A2_1_logit}
    \widetilde{K}_{\min} n
    \leq \min_{t \in T} \inf_{ \theta \in \Theta (\theta_0, \bI_{p}, 1/2)} \lambda_{\min} \left( \bF_{t, \theta} \right) 
    \leq \max_{t \in T} \sup_{ \theta \in \Theta} \lambda_{\max} \left( \bF_{t, \theta} \right) 
    \leq \widetilde{K}_{\max} n, \\
    \label{eqn:A1_2_logit} 
    \max_{t \in [T]} \sup_{\theta \in \Theta \left(\theta_0, \bI_{p}, 1/2 \right)} 
    \Bigg(
    \left\| \Fisher[t]{\theta}^{-1} \left( \dfrac{\bX_{t}^{\top} \bX_{t}}{4} \right) \right\|_{2}
    \vee
    \left\| \FullFisher[t]{\theta}^{-1} \left( \dfrac{\bX_{1:t}^{\top} \bX_{1:t}}{4} \right) \right\|_{2}
    \Bigg)
    \leq \dfrac{M^2}{9},
\end{align}    
where $\widetilde{K}_{\min}$ and $M$ are constants depending only on $K_1$, and $\widetilde{K}_{\max}$ is a universal constant.
\end{proposition}

\begin{proof}
    Since $\| \theta_0 \|_2 \leq K_1$, we have
    \begin{align*}
        \Theta (\theta_0, \bI_{p}, 1/2) \subseteq \Theta(\bI_p, K_1 + 1/2).
    \end{align*}
    By Lemma \ref{lemma:least_eigenvalue_logit} with $\tau = K_1 + 1/2$, \eqref{eqn:A2_1_logit} holds with the constants
    \begin{align*}
        \widetilde{K}_{\min} = \dfrac{1}{1080 e^{ 2K_1 + 3}}, \quad  \widetilde{K}_{\max} = \dfrac{9}{4},
    \end{align*}
    with $\bbP$-probability at least $1 - 8e^{-n/72} -4(Np)^{-1}$.    

    By the last display and Lemma \ref{lemma:random_matrix_eigenvalue}, on the same event on which \eqref{eqn:A2_1_logit} holds, we have
    \begin{align*}
        &\max_{t \in [T]} \sup_{\theta \in \Theta \left(\theta_0, \bI_{p}, 1/2 \right)} \left\| \Fisher[t]{\theta}^{-1} \left( \bX_{t}^{\top} \bX_{t}/4 \right) \right\|_{2}  \\
        &\leq 
        \bigg[ \min_{t \in [T]} \inf_{\theta \in \Theta \left(\theta_0, \bI_{p}, 1/2 \right)} \lambda_{\min} \left( \Fisher[t]{\theta} \right) \bigg]^{-1} 
        \bigg[ \max_{t \in [T]} \left( \bX_{t}^{\top} \bX_{t}/4 \right)  \bigg] \\
        &\leq 
        \left( 1080 e^{ 2K_1 + 3} \right) \left( \dfrac{9}{4} \right) = 2430 e^{ 2K_1 + 3}.
    \end{align*}
    Also,
    \begin{align*}
        &\max_{t \in [T]} \sup_{\theta \in \Theta \left(\theta_0, \bI_{p}, 1/2 \right)} \left\| \FullFisher[t]{\theta}^{-1} \left( \bX_{1:t}^{\top} \bX_{1:t}/4 \right) \right\|_{2}  \\
        &\leq 
        \bigg[ \min_{t \in [T]} \inf_{\theta \in \Theta \left(\theta_0, \bI_{p}, 1/2 \right)} \lambda_{\min} \left( \FullFisher[t]{\theta} \right) \bigg]^{-1} 
        \bigg[ \max_{t \in [T]} \left( \bX_{1:t}^{\top} \bX_{1:t}/4 \right)  \bigg] \\
        &\leq 
        \left( 1080 e^{ 2K_1 + 3} \right) \left( \dfrac{9}{4} \right) = 2430 e^{ 2K_1 + 3} 
        = \dfrac{ \big(  27\sqrt{30} e^{K_1 + 3/2} \big)^2  }{9}.
    \end{align*}
    The last two displays complete the proof of \eqref{eqn:A1_2_logit} by taking $M = 27\sqrt{30} e^{K_1 + 3/2}$.     
\end{proof}

\begin{lemma} \label{lemma:logit_V_verification}
    For any $\omega \geq 0$ and $t \in [T]$, we have
    \begin{align*}
        \bbP_{0, t} \left( \left\| \FisherBest^{-1/2} \nabla \zeta_{t} \right\|_{2} \geq 
        \sqrt{\operatorname{tr}\left( \FisherBest^{-1} \dfrac{\bX_{t}^{\top} \bX_{t}}{4} \right)} + \sqrt{2 \left\| \FisherBest^{-1} \dfrac{\bX_{t}^{\top} \bX_{t}}{4} \right\|_{2} \omega} \ \middle| \ \bX
        \right) \leq e^{-\omega},
    \end{align*}
\end{lemma}

\begin{proof}
Let $t \in [T]$ and $\cE_{t} = (\epsilon_{i})_{i \in I_{t}} \in \bbR^{n}$, where $\epsilon_{i} = Y_i - \bbE_{t} (Y_i \mid \bX)$.
Note that
\begin{align*}
    \nabla \zeta_{t} = \sum_{i \in I_{t}} \epsilon_{i} x_{i} = \bX_{t}^{\top} \cE_{t}, \quad \sup_{\eta \in \bbR} b''(\eta) \leq b''(0) = 1/4.
\end{align*}
By Lemma 6.1 in \cite{rigollet2012kullback}, we have, for any $u = (u_i)_{i \in [n]} \in \bbR^{n}$ with $\| u \|_{2} = 1$,
\begin{align*}
    u^{\top} \cE_{t} = \sum_{i = 1}^{n} u_{i} \epsilon_{n(t-1) + i} \sim \text{subG} \left( 1/4 \right), \quad \forall t \in [T],
\end{align*}
which is equivalent to $\cE_{t} \sim \text{subG} \left( 1/4 \right)$ for all $t \in [T]$. 
By the last display and $\bbE_{t} (\cE_{t} \mid \bX) = 0$ for all $t \in [T]$, we can apply Hanson-Wright inequality. 
By Lemma \ref{lemma:tech_subGaussian_Chaos}, we have, for any $\omega \geq 0$,
\begin{align*}
    \bbP_{0, t} \left( \left\| \FisherBest^{-1/2} \bX_{t}^{\top} \cE_{t} \right\|_{2}^{2} \geq 
    \dfrac{1}{4} \bigg[  \operatorname{tr}\left( \mathbf{B}_{t} \right) 
    + 2\sqrt{\operatorname{tr}\left( \mathbf{B}_{t}^{2} \right) \omega } 
    + 2\left\| \mathbf{B}_{t} \right\|_{2} \omega \bigg] \ \middle| \ \bX
    \right) 
    \leq e^{-\omega}, 
\end{align*}
where $\mathbf{B}_{t} = \FisherBest^{-1/2} \left( \bX_{t}^{\top} \bX_{t} \right) \FisherBest^{-1/2}$.
Since 
\begin{align*}
    \operatorname{tr}\left( \mathbf{B}_{t} \right) 
    + 2\sqrt{\operatorname{tr}\left( \mathbf{B}_{t}^{2} \right) \omega }
    + 2\left\| \mathbf{B}_{t} \right\|_{2} \omega
    &=
    \operatorname{tr}\left( \mathbf{B}_{t} \right) 
    + 2\left\| \mathbf{B}_{t} \right\|_{\rm F} \omega^{1/2} 
    + 2\left\| \mathbf{B}_{t} \right\|_{2} \omega \\
    &\leq 
    \bigg[ \sqrt{\operatorname{tr}\left( \mathbf{B}_{t} \right)} + \sqrt{2 \omega \left\| \mathbf{B}_{t} \right\|_{2}} \bigg]^{2}
\end{align*}
for any $\omega \geq 0$, we have
\begin{align*}
    \bbP_{0, t} \left( \left\| \FisherBest^{-1/2} \bX_{t}^{\top} \cE_{t} \right\|_{2} \geq 
    \dfrac{1}{2} \bigg[ \sqrt{\operatorname{tr}\big( \mathbf{B}_{t} \big)} + \sqrt{2 \omega \left\| \mathbf{B}_{t} \right\|_{2}} \bigg] \ \middle| \ \bX
    \right) 
    \leq e^{-\omega}.
\end{align*}
The above display is equivalent to
\begin{align*}
    \bbP_{0, t} 
    \left( 
    \left\| \FisherBest^{-1/2} \bX_{t}^{\top} \cE_{t} \right\|_{2} \geq 
    \sqrt{\operatorname{tr}\big( \widetilde{\mathbf{B}}_{t} \big)} + \sqrt{2 \omega \big\| \widetilde{\mathbf{B}}_{t} \big\|_{2}} \ \middle| \ \bX
    \right) 
    \leq e^{-\omega},    
\end{align*}
where $\widetilde{\mathbf{B}}_{t} = \FisherBest^{-1} \left( \bX_{t}^{\top} \bX_{t} / 4 \right)$.
This completes the proof.
\end{proof}

In the following Proposition \ref{prop:logit_V_verification}, we demonstrate that \eqref{assume:A1_1} and \eqref{assume:A1ast_1} are satisfied with the specified matrices $\bV_{t} = \bX_{t}^{\top} \bX_{t}/4$ and $\bV_{1:t} = \bX_{1:t}^{\top} \bX_{1:t}/4$, respectively. 

\begin{proposition} \label{prop:logit_V_verification}
    Suppose that (\textbf{EX}) holds. Then, with $\bbP$-probability at least $1 - 3n^{-1} - 6e^{-n/72} -2(Np)^{-1}$, the following inequalities hold uniformly for all $t \in [T]$:
    \begin{align*} 
        \left\| \FisherBest^{-1/2} \nabla \zeta_{t} \right\|_{2} 
        &\leq 
        r \big( \FisherBest, \ \bX_{t}^{\top} \bX_{t}/4, \ \log n + \log T \big),
        \\
        \left\| \FullFisherTilde[t]{\fullthetaBest}^{-1/2} \nabla \zeta_{1:t} \right\|_{2}  
        &\leq 
        r \big( \FullFisherTilde[t]{\fullthetaBest}, \ \bX_{1:t}^{\top} \bX_{1:t}/4, \ \log n + \log T \big), \\
        \left\| \FullFisher[t]{\theta_{0}}^{-1/2} \nabla \zeta_{1:t} \right\|_{2} 
        &\leq 
        r \big( \FullFisher[t]{\theta_{0}}, \ \bX_{1:t}^{\top} \bX_{1:t}/4, \ \log n + \log T \big),
    \end{align*}
    where
    \begin{align*}
        r(\bF, \bV, \omega)
        = 
        \sqrt{\operatorname{tr}\left( \bF^{-1} \bV \right)} + \sqrt{2 \omega \left\| \bF^{-1} \bV \right\|_{2} }, \quad \bF, \bV \in \symmPD, \ \omega \in \bbR_{+}. 
    \end{align*}
\end{proposition} 

\begin{proof}
The proof of the first assertion directly follows from Lemma \ref{lemma:logit_V_verification}. Note that
\begin{align*}
    &\bbP \left(
    \left\| \FisherBest^{-1/2} \nabla \zeta_{t} \right\|_{2} 
    \geq 
    \sqrt{\operatorname{tr}\left( \FisherBest^{-1} \dfrac{\bX_{t}^{\top} \bX_{t}}{4} \right)} + \sqrt{2 \left\| \FisherBest^{-1} \dfrac{\bX_{t}^{\top} \bX_{t}}{4} \right\|_{2} (\log n + \log T)} \
    \text{ for some } t \in [T]
    \ \middle| \ \bX
    \right) \\
    &\leq 
    T \cdot \max_{t \in [T]} \bbP_{t} \left(
    \left\| \FisherBest^{-1/2} \nabla \zeta_{t} \right\|_{2} 
    \geq 
    \sqrt{\operatorname{tr}\left( \FisherBest^{-1} \dfrac{\bX_{t}^{\top} \bX_{t}}{4} \right)} + \sqrt{2 \left\| \FisherBest^{-1} \dfrac{\bX_{t}^{\top} \bX_{t}}{4} \right\|_{2} (\log n + \log T)} 
    \ \middle| \ \bX
    \right) \\
    &\leq 
    T \cdot e^{-\log n - \log T} = n^{-1}.
\end{align*}
By integrating over the values of $\bX$, we have
\begin{align*}
    &\bbP \left(
    \left\| \FisherBest^{-1/2} \nabla \zeta_{t} \right\|_{2} 
    \geq 
    \sqrt{\operatorname{tr}\left( \FisherBest^{-1} \dfrac{\bX_{t}^{\top} \bX_{t}}{4} \right)} + \sqrt{2 \left\| \FisherBest^{-1} \dfrac{\bX_{t}^{\top} \bX_{t}}{4} \right\|_{2} (\log n + \log T)} \
    \text{ for some } t \in [T]
    \right) \\
    &\leq n^{-1},
\end{align*}
which completes the proof of the first assertion.

The proof of the second assertion is similar to that of Lemma  \ref{lemma:logit_V_verification}. Hence, we will provide the sketch of the proof.
Let $t \in [T]$ and $\cE_{1:t} = (\epsilon_{i})_{i \in I_{1:t}} \in \bbR^{nt}$.
Note that $\nabla \zeta_{1:t} = \bX_{1:t}^{\top} \cE_{1:t}$.
By Lemma \ref{lemma:tech_subGaussian_Chaos}, we have, for any $\omega \geq 0$,
\begin{align*}
    \bbP \left( \left\| \FullFisherTilde[t]{\fullthetaBest[t]}^{-1/2} \bX_{1:t}^{\top} \cE_{1:t} \right\|_{2}^{2} 
    \geq 
    \dfrac{1}{4} 
    \bigg[  
    \operatorname{tr}\big( \widetilde{\mathbf{B}}_{1:t} \big) + 2 \sqrt{\operatorname{tr}\big( \widetilde{\mathbf{B}}_{1:t}^{2} \big) \omega}
    + \big\| \widetilde{\mathbf{B}}_{1:t} \big\|_{2} \omega 
    \bigg] \ \middle| \ \bX
    \right) 
    \leq e^{-\omega}, 
\end{align*}
where $\widetilde{\mathbf{B}}_{1:t} = \FullFisherTilde[t]{\fullthetaBest[t]}^{-1/2} \left( \bX_{1:t}^{\top} \bX_{1:t} \right) \FullFisherTilde[t]{\fullthetaBest[t]}^{-1/2}$.
Since 
\begin{align*}
    \operatorname{tr}\big( \widetilde{\mathbf{B}}_{1:t} \big) 
    + 2 \sqrt{\operatorname{tr}\big( \widetilde{\mathbf{B}}_{1:t}^{2} \big) \omega}
    + \left\| \widetilde{\mathbf{B}}_{1:t} \right\|_{2} \omega
    \leq 
    \bigg[ \sqrt{\operatorname{tr}\big( \widetilde{\mathbf{B}}_{1:t} \big)} + 
    \sqrt{2 \omega \big\| \widetilde{\mathbf{B}}_{1:t} \big\|_{2}} \bigg]^{2},
\end{align*}
for any $\omega \geq 0$, we have
\begin{align*}
    \bbP \left( \left\| \FullFisherTilde[t]{\fullthetaBest[t]}^{-1/2} \bX_{1:t}^{\top} \cE_{1:t} \right\|_{2} \geq 
    \dfrac{1}{2} 
    \bigg[ \sqrt{\operatorname{tr}\big( \widetilde{\mathbf{B}}_{1:t} \big)} + 
    \sqrt{2 \omega \big\| \widetilde{\mathbf{B}}_{1:t} \big\|_{2}} \bigg]
    \ \middle| \ \bX
    \right) 
    \leq e^{-\omega}.
\end{align*}
It follows that
\begin{align*}
    &\bbP \left( \left\| \FullFisherTilde[t]{\fullthetaBest[t]}^{-1/2} \bX_{1:t}^{\top} \cE_{1:t} \right\|_{2} \geq 
    \dfrac{1}{2} 
    \bigg[ \sqrt{\operatorname{tr}\big( \widetilde{\mathbf{B}}_{1:t} \big)} + 
    \sqrt{2 \big\| \widetilde{\mathbf{B}}_{1:t} \big\|_{2} (\log n + \log T)}  \bigg]
    \text{ for some } t \in [T]
    \ \middle| \ \bX
    \right) \\
    &\leq 
    T \cdot e^{- \log n - \log T} = n^{-1}.
\end{align*}
By integrating over the values of $\bX$, we have
\begin{align*}
    &\bbP \left( \left\| \FullFisherTilde[t]{\fullthetaBest[t]}^{-1/2} \bX_{1:t}^{\top} \cE_{1:t} \right\|_{2} \geq 
    \dfrac{1}{2} 
    \bigg[ \sqrt{\operatorname{tr}\big( \widetilde{\mathbf{B}}_{1:t} \big)} + 
    \sqrt{2 \big\| \widetilde{\mathbf{B}}_{1:t} \big\|_{2} (\log n + \log T)}  \bigg]
    \text{ for some } t \in [T]
    \right) \\
    &\leq 
    n^{-1},
\end{align*}
which completes the proof of the second assertion.

Next, we will prove the third assertion, which is similar to the second assertion. 
Let $\widetilde{\scrE}$ be an event where $\bF_{1:t, \theta_0}$ is nonsingular.
By Lemma \ref{lemma:logit_V_verification}, there exists an event $\scrE$ such that on $\scrE$, $\bF_{1:t, \theta_0}$ is nonsingular and 
\begin{align*}
    \bbP \big( \scrE \big) \geq 1 - 6e^{-n/72} - 2(Np)^{-1}.
\end{align*}
Hence, we have
\begin{align*}
    \bbP \big( \widetilde{\scrE} \big) 
    \geq
    \bbP \big( \scrE \big) 
    \geq 1 - 6e^{-n/72} - 2(Np)^{-1}.
\end{align*}
In this proof, we denote
\begin{align*}
    R_{n, T}(\mathbf{B}_{1:t}) = 
    \dfrac{1}{2} 
    \bigg( \sqrt{\operatorname{tr}\left( \mathbf{B}_{1:t} \right)} + 
    \sqrt{2 \left\| \mathbf{B}_{1:t} \right\|_{2} (\log n + \log T)}  \bigg),
\end{align*}
where $\mathbf{B}_{1:t} = \FullFisher[t]{\theta_{0}}^{-1/2} \left( \bX_{1:t}^{\top} \bX_{1:t} \right) \FullFisher[t]{\theta_{0}}^{-1/2}$. 
It follows from Lemma \ref{lemma:tech_subGaussian_Chaos} that
\begin{align*}
    \bbP \left( \left\| \FullFisher[t]{\theta_{0}}^{-1/2} \bX_{1:t}^{\top} \cE_{1:t} \right\|_{2} \geq 
    R_{n, T}(\mathbf{B}_{1:t}) \
    \text{ for some } t \in [T]
    \ \middle| \ \bX, \ \widetilde{\scrE}
    \right) 
   \leq n^{-1},
\end{align*}
Consequently, we have
\begin{align*}
    &\bbP \left( \left\| \FullFisher[t]{\theta_{0}}^{-1/2} \bX_{1:t}^{\top} \cE_{1:t} \right\|_{2} \geq R_{n, T}(\mathbf{B}_{1:t})  \
    \text{ for some } t \in [T]
    \right) \\
    &\leq
    \bbE_{\bX} \bigg[ \bbP \left( \left\| \FullFisher[t]{\theta_{0}}^{-1/2} \bX_{1:t}^{\top} \cE_{1:t} \right\|_{2} \geq 
    R_{n, T}(\mathbf{B}_{1:t}) \
    \text{ for some } t \in [T]
    \ \middle| \ \bX, \ \widetilde{\scrE}
    \right) \mathds{1}_{\widetilde{\scrE}}
    \bigg] + \bbP \big( \widetilde{\scrE}^{\rm c} \big) \\   
    &\leq
    n^{-1} + \bbP \big( \scrE^{\rm c} \big) 
    \leq  
    n^{-1} + 6e^{-n/72} + 2(Np)^{-1}.
\end{align*}
This completes the proof of the third assertion.
Therefore, with $\bbP$-probability at least $1 - 3n^{-1} - 6e^{-n/72} - 2(Np)^{-1}$, the three assertions hold uniformly for all $t \in [T]$. 
\end{proof}

The following proposition verifies \eqref{assume:A2_2} and \eqref{assume:A2_3}.

\begin{proposition} \label{prop:operator_norm_logit}
Suppose that (\textbf{EX}) holds.
Then, 
\begin{align} \label{eqn:A2_23_logit}
\bbP \Bigg(
    \bigg[
    \max_{t \in [T]} \sup_{\theta \in \Theta} \left\| \nabla^{3} L_{t}(\theta) \right\|_{\rm op} \bigg] \vee
    \bigg[ \max_{t \in [T]} \sup_{\theta \in \Theta} \left\| \nabla^{4} L_{t}(\theta) \right\|_{\rm op} \bigg] \leq \widetilde{K}_{\max}' n
    \Bigg) \geq 1 - 2n^{-1},
\end{align} 
where $\widetilde{K}_{\max}' > 0$ is a universal constant.
\end{proposition}

\begin{proof}
    For $b(\cdot) = \log(1 + \exp(\cdot))$, note that
    \begin{align*}
        \sup_{\eta \in \bbR}  \bigg[ b'''\left( \eta \right) \vee b''''\left( \eta \right) \bigg] \leq 1.
    \end{align*}
    For $t \in [T]$, we have
    \begin{align*}
        &\sup_{\theta \in \Theta} \left\| \nabla^{3} L_{t}(\theta) \right\|_{\rm op}
        =
        \sup_{\theta \in \Theta} \sup_{u \in \bbR^{p} : \| u \|_{2} = 1} \sum_{i \in I_{t}} b'''\left( X_{i}^{\top} \theta \right) \left( X_i^{\top}u \right)^{3}
        \leq 
        \sup_{u \in \bbR^{p} : \| u \|_{2} = 1} \sum_{i \in I_{t}} \left| X_i^{\top}u \right|^{3}, \\
        &\sup_{\theta \in \Theta} \left\| \nabla^{4} L_{t}(\theta) \right\|_{\rm op}
        =
        \sup_{\theta \in \Theta} \sup_{u \in \bbR^{p} : \| u \|_{2} = 1} \sum_{i \in I_{t}} b''''\left( X_{i}^{\top} \theta \right) \left( X_i^{\top}u \right)^{4}
        \leq 
        \sup_{u \in \bbR^{p} : \| u \|_{2} = 1} \sum_{i \in I_{t}} \left| X_i^{\top}u \right|^{4}.
    \end{align*}
    By Lemma \ref{lemma:tech_logconcave_abs_moments}, for $t \in [T]$, $k \geq 2$ and $\omega, \tau \geq 1$, there exists some constants $D_1 = D_1(k) > 1$, $D_2 > 0$ and $D_3 = D_3(k) > 0$ such that
    \begin{align*}
       &\sup_{u \in \bbR^{p} : \| u \|_{2} = 1} \left| \dfrac{1}{n} \sum_{i \in I_{t}} \bigg( \left| \langle X_i, u \rangle \right|^{k} - \bbE \left| \langle X_i, u \rangle \right|^{k} \bigg) \right| \\
       &\leq 
       D_1 \tau \omega^{k-1} \log^{k-1} \left( \dfrac{2n}{p} \right) \sqrt{\dfrac{p}{n}} + 
       D_1\dfrac{ \omega^{k} p^{k/2} }{n} + 
       D_1 \left( \dfrac{p}{2n} \right)^{\omega}
    \end{align*}
    with a probability at least 
    \begin{align*}
        1 
        - \exp \left( - D_2 \omega \sqrt{p} \right) 
        - \exp \left( - D_3 \Bigg\{ \bigg[ \tau^{2} \omega^{2k-2} p \log^{(2k-2)}\left(\dfrac{2n}{p} \right) \bigg] \wedge \bigg[ \tau \omega^{-1} \sqrt{np} \log^{-1}\left( \dfrac{2n}{p} \right) \bigg] \Bigg\} \right).
    \end{align*}
    By (\textbf{EX}), we have
    \begin{align*}
        \left( \log T + \log n \right)^{k-1} \log^{(k-1)}\left(\dfrac{2n}{p} \right) \sqrt{\dfrac{p}{n}} 
        \leq \delta, \quad
        \left( \log T + \log n \right)^{k} p^{k/2} n^{-1}
        \leq \delta, \quad
        \dfrac{p}{2n} \leq \delta
    \end{align*}
    for a small enough constant $\delta > 0$ depending only $D_1$. Also,
    \begin{align*}
        (\log T + \log n)^{4} \log^{4} (2n/p) &\geq C' (\log n + \log T), \\
        \sqrt{np} \log^{-1}(2n/p)(\log T + \log n)^{-1} &\geq C' (\log n + \log T)
    \end{align*}
    for a large enough constant $C' > 0$ depending only $D_3$.    
    By taking 
    \begin{align*}
        \omega = \log T + \log n, \quad 
        \tau = 1,
    \end{align*}
    after some algebra, for any $k \in \{3, 4\}$, there exist some positive constants $c_1 = c_1(D_1)$ and $c_2 = c_{2}(D_3)$ such that
    \begin{align*}
        &\max_{t \in [T]} \sup_{u \in \bbR^{p} : \| u \|_{2} = 1} 
        \left| \dfrac{1}{n} \sum_{i \in I_{t}} \bigg( \left| \langle X_i, u \rangle \right|^{k} - \bbE \left| \langle X_i, u \rangle \right|^{k} \bigg) \right| \\
        &\leq 
        c_1 \bigg[ 
        \left( \log T + \log n \right)^{k-1} \log^{(k-1)}\left(\dfrac{2n}{p} \right) \sqrt{\dfrac{p}{n}}
        +
        \left( \log T + \log n \right)^{k} p^{k/2} n^{-1}
        + \dfrac{p}{2n}
        \bigg]
        \leq
        3
    \end{align*}
    with $\bbP$-probability at least 
    \begin{align*}
        &1 - e^{-\sqrt{p} (\log T + \log n) + \log T} \\
        &\quad - \exp \left( - c_2 \bigg[ \big\{ (\log T + \log n)^{4} \log^{4} (2n/p) \big\} \wedge \big\{ \sqrt{np} \log^{-1}(2n/p)(\log T + \log n)^{-1} \big\} \bigg] + \log T \right) \\
        &\geq
        1 - e^{-\log n -\log T + \log T} - e^{-\log n -\log T + \log T}
        =
        1 - 2n^{-1}.
    \end{align*}
    Also, for any $i \in I_{t}$ and $u \in \bbR^{p}$ with $\| u \|_{2} =  1$,
    \begin{align*}
        \bbE \left| \langle X_i, u \rangle \right|^{3} = 2\sqrt{\dfrac{2}{\pi}}, \quad 
        \bbE \left| \langle X_i, u \rangle \right|^{4} = 15.
    \end{align*}
    Therefore, we have
    \begin{align*}
        \max_{t \in [T]} \sup_{\theta \in \Theta} \left\| \nabla^{3} L_{t}(\theta) \right\|_{\rm op} \leq 
        \max_{t \in [T]} \sup_{u \in \bbR^{p} : \| u \|_{2} = 1} \sum_{i \in I_{t}} \left| X_i^{\top}u \right|^{3} 
        &\leq \left( 3 + 2\sqrt{\dfrac{2}{\pi}} \right) n, \\
        \max_{t \in [T]} \sup_{\theta \in \Theta} \left\| \nabla^{4} L_{t}(\theta) \right\|_{\rm op} \leq 
        \max_{t \in [T]} \sup_{u \in \bbR^{p} : \| u \|_{2} = 1} \sum_{i \in I_{t}} \left| X_i^{\top}u \right|^{4} 
        &\leq \left( 3 + 15 \right) n = 18n
    \end{align*}
    with $\bbP$-probability at least $1 - 2n^{-1}$. By taking $\widetilde{K}_{\max}' = (3 + 2\sqrt{2/\pi}) \vee 18 = 18$, we complete the proof.
\end{proof}

\begin{proof}[Proof of Proposition \ref{prop:logit_fin_statement}]
    Let 
    \begin{align*}
        \Omega_{1} &= \big\{ \text{ The three assertions in Proposition \ref{prop:logit_V_verification} hold uniformly for all } t \in [T] \big\}, \\
        \Omega_{2} &= \big\{ \eqref{eqn:A2_1_logit} \text{ and } \eqref{eqn:A1_2_logit} \text{ hold with the constants } \widetilde{K}_{\min}, \widetilde{K}_{\max} \text{ and } M \big\}, \\
        \Omega_{3} &= \big\{ \eqref{eqn:A2_23_logit} \text{ holds with the constant } \widetilde{K}_{\max}' \big\}.
    \end{align*}
    By Propositions \ref{prop:eigenvalues_logit}, \ref{prop:logit_V_verification} and \ref{prop:operator_norm_logit}, we have
    \begin{align*}
        \bbP (\Omega_{1}) &\geq 1 - 3n^{-1} - 6e^{-n/72} - 2(Np)^{-1}, \\
        \bbP (\Omega_{2}) &\geq 1 - 4e^{-n/72} - 2(Np)^{-1}, \\
        \bbP (\Omega_{3}) &\geq 1 - 2n^{-1}. 
    \end{align*}
    It follows that 
    \begin{align*}
        \bbP (\Omega_{1} \cap \Omega_{2} \cap \Omega_{3}) 
        \geq 1 - 5n^{-1} - 10e^{-n/72} - 4(Np)^{-1}.
    \end{align*}
    Let $\Omega = \Omega_{1} \cap \Omega_{2} \cap \Omega_{3}$. On $\Omega$, the assumptions in (\textbf{A1}), (\textbf{A1$\ast$}) and (\textbf{A2}) are satisfied when $K_{\min}, K_{\max}$ and $M_n$ are replaced by $\widetilde{K}_{\min}$, $(\widetilde{K}_{\max} \vee \widetilde{K}_{\max}' \vee K_2 \vee K_3)$ and $M$, respectively.
    Recall that 
    \begin{align*}
        \widetilde{K}_{\min} = (1080 e^{2K_1 + 3})^{-1}, \quad 
        \widetilde{K}_{\max} \vee \widetilde{K}_{\max}' = 18, \quad 
        M = 27\sqrt{30} e^{K_1 + 3/2}.
    \end{align*}
    By (\textbf{EX}), all conditions specified in Theorem \ref{thm:onlineBvM}, Propositions \ref{prop:eigenvalues_logit}, \ref{prop:logit_V_verification} and \ref{prop:operator_norm_logit}  hold. Therefore, the result of Proposition \ref{prop:logit_fin_statement} follows from Theorem \ref{thm:onlineBvM}.
\end{proof}

\section{Technical lemmas}

\subsection{General technical lemmas}
In this subsection, assume we are given a function $f : \Theta \rightarrow \bbR$ that is four times continuously differentiable. For $\theta \in \Theta$, let $\bF_{\theta} = - \nabla^2 f(\theta)$, and assume $\bF_{\theta}$ is nonsingular in this subsection.

\begin{lemma} \label{lemma:tech_Fisher_smooth}
    For a given $\theta \in \Theta$, suppose that $f$ satisfies the third order smoothness at $\theta$ with parameters $(\tau_3, \bF_{\theta}, r)$.
    Then,
    \begin{align*}
        \sup_{\theta' \in \Theta (\theta, \bF_{\theta}, r)}  \left\|  \bF_{\theta}^{-1/2} \bF_{\theta'} \bF_{\theta}^{-1/2} - \bI_{p} \right\|_{2} \leq \tau_{3} r.
    \end{align*}
    Consequently, for every $u \in \bbR^p$ and $\theta' \in \Theta (\theta, \bF_{\theta}, r)$, we have
    \bean
        (1 - \tau_{3} r) \left\| \bF_{\theta}^{1/2} u \right\|_{2}^{2}
        \leq \left\| \bF_{\theta'}^{1/2} u \right\|_{2}^{2} 
        \leq (1 + \tau_{3} r) \left\| \bF_{\theta}^{1/2} u \right\|_{2}^{2}.
    \eean
\end{lemma}
\begin{proof}
For $\theta' \in \Theta (\theta, \bF_{\theta}, r)$, we have
\begin{align*}
    &\left\|  \bF_{\theta}^{-1/2} \bF_{\theta'} \bF_{\theta}^{-1/2} - \bI_{p} \right\|_{2}
    = 
    \sup_{u \in \bbR^{p} : \| u \|_{2} = 1} \left| \langle \bF_{\theta}^{-1/2} \bF_{\theta'} \bF_{\theta}^{-1/2} - \bI_{p}, u^{\otimes 2} \rangle \right| \\ 
    &\leq
    \sup_{h \in \Theta (\bF_{\theta}, r)}
    \sup_{u \in \bbR^{p} : \| u \|_{2} = 1} \left| \langle \bF_{\theta}^{-1/2} \bF_{\theta + h} \bF_{\theta}^{-1/2} - \bI_{p}, u^{\otimes 2} \rangle \right| \\
    &=
    \sup_{h \in \Theta (\bF_{\theta}, r)}
    \sup_{u \in \bbR^{p} : \| u \|_{2} = 1} \left| \left \langle \bF_{\theta + h} - \bF_{\theta}, \left(\bF_{\theta}^{-1/2}u\right)^{\otimes 2}  \right \rangle \right| \\
    &\leq 
    \sup_{h, h' \in \Theta (\bF_{\theta}, r)}
    \sup_{u \in \bbR^{p} : \| u \|_{2} = 1} \left| \langle \nabla^{3} f(\theta + h') , \left(\bF_{\theta}^{-1/2} u \right)^{\otimes 2} \otimes h \rangle \right| \\
    \overset{\eqref{def:3_4_smooth_equiv}}&{\leq}
    \sup_{h \in \Theta (\bF_{\theta}, r)}
    \sup_{u \in \bbR^{p} : \| u \|_{2} = 1} \tau_{3} \left\| \bF_{\theta}^{1/2} \bF_{\theta}^{-1/2} u  \right\|_{2}^{2} \left\| \bF_{\theta}^{1/2} h \right\|_{2}
    \leq \tau_{3} r.
\end{align*}
This completes the proof.
\end{proof}

\begin{lemma} \label{lemma:tech_pre_bias_bound}
Let $\theta, \widetilde{\theta}, \beta \in \Theta$ and $r = \| \bF_{\theta}^{-1/2} \beta \|_{2}$.
Suppose that $f$ is concave and satisfies the third order smoothness at $\theta$ with parameters $(\tau_3, \bF_{\theta}, 4r)$.
Assume further that 
\begin{align*}
    \tau_{3} r \leq 1/16, \quad \nabla f(\theta) = 0, \quad \nabla f(\widetilde{\theta}) + \beta = 0.
\end{align*}
Then,
\begin{align*}
    \left\| \bF_{\theta}^{1/2} \big( \widetilde{\theta} - \theta \big) \right\|_{2} \leq 4r.
\end{align*}
\end{lemma}

\begin{proof}
Let $\Theta_{\theta, r} = \Theta (\theta, \bF_{\theta}, 4r)$, $\overline{\Theta}_{\theta, r} = \Theta (\bF_{\theta}, 4r)$ and 
\begin{align*}
    \partial \Theta_{\theta, r} = \left\{ \theta' \in \Theta : \left\| \bF_{\theta}^{1/2} \left( \theta' - \theta \right) \right\|_{2} = 4 r \right\}.
\end{align*}
Then, it suffices to prove that $\widetilde{\theta} \in \Theta_{\theta, r}$.

For $\theta' \in \Theta$, let $g(\theta') = f(\theta') + \langle \beta, \theta' \rangle$. Then, $\nabla g(\widetilde{\theta}) = 0$ and the map $\theta' \mapsto g(\theta')$ is concave.
By the concavity of $g(\cdot)$, for any $\theta' \in \Theta_{\theta, r}^{\rm c}$, we have
\begin{align} \label{eqn:pre_bias_bound_eq1}
    g(\overline{\theta}) \geq \omega g(\theta') + (1 - \omega)g(\theta),
\end{align}
where $\overline{\theta} = \omega \theta' + (1- \omega)\theta$ and $\omega = 4 r \| \bF_{x}^{1/2}(\theta' - \theta) \|_{2}^{-1} \in (0, 1)$.
One can easily check that $\overline{\theta} \in \partial \Theta_{\theta, r}$.
At the end of this proof, we will show that 
\begin{align} \label{eqn:pre_bias_bound_eq2}
    \sup_{\theta^{\circ} \in \partial \Theta_{\theta, r}} g(\theta^{\circ}) - g(\theta) \leq -2 r^{2} < 0.
\end{align}
It follows that, for any $\theta' \in \Theta_{\theta, r}^{\rm c}$,
\begin{align*}
    0 > -2 r^{2} 
    \geq  \sup_{\theta^{\circ} \in \partial \Theta_{x, r}} g(\theta^{\circ}) - g(x)
    \geq g(\overline{\theta}) - g(x)
    \overset{\eqref{eqn:pre_bias_bound_eq1}}{\geq} \omega \left[ g(\theta') - g(\theta) \right]
    \geq g(\theta') - g(\theta),
\end{align*}
which implies that $\widetilde{\theta} \in \Theta_{\theta, r}$. 

To complete the proof, we only need to prove \eqref{eqn:pre_bias_bound_eq2}. 
Let $\theta^{\circ} \in \partial \Theta_{\theta, r}$ and $u = \theta^{\circ} - \theta$.  
By Taylor's theorem, there exists some $\widetilde{u} \in \overline{\Theta}_{\theta, r}$ such that
\begin{align*}
    &g(\theta^{\circ}) - g(\theta) 
    = \nabla g(\theta)^{\top} u + \dfrac{1}{2} \langle \nabla^{2} g(\theta + \widetilde{u}), u^{\otimes 2} \rangle \\
    &= \big[ \nabla f(\theta) + \beta \big]^{\top} u + \dfrac{1}{2} \langle \nabla^{2} f(\theta + \widetilde{u}), u^{\otimes 2}  \rangle
    = \beta^{\top} u + \dfrac{1}{2} \langle \nabla^{2} f(\theta + \widetilde{u}), u^{\otimes 2}  \rangle \\
    &= \left[\bF_{\theta}^{-1/2} \beta \right]^{\top} \bF_{\theta}^{1/2} u - \dfrac{1}{2} \langle \bF_{\theta + \widetilde{u}}, u^{\otimes 2} \rangle \\
    \overset{\text{Lemma \ref{lemma:tech_Fisher_smooth}}}&{\leq}
    \bigg( \left\| \bF_{\theta}^{-1/2} \beta \right\|_{2}  - \dfrac{1}{2} \left( 1 -  4 \tau_{3} r \right) \left\| \bF_{\theta}^{1/2} u \right\|_{2}  \bigg) 
    \left\| \bF_{\theta}^{1/2} u \right\|_{2} \\
    &\leq  
    \bigg[ r  - 2\left( 1 -  4 \tau_{3} r \right) r \bigg] \times 4r, \quad \left( \because \left\| \bF_{\theta}^{1/2} u \right\|_{2} = 4 r \right)  \\
    &\leq - 2 r^{2},
\end{align*}
where the last inequality holds by $\tau_{3} r \leq 1/16$. This completes the proof.
\end{proof}

\begin{lemma} \label{lemma:tech_smooth_tau_bound}
    Let $\theta_{\rm c} \in \Theta$, $\bF \in \symmPD$ and $r \geq 0$ be given.
    Then, $f$ satisfies the third and fourth order smoothness at $\theta_{\rm c}$ with parameters $(\tau_3, \bF, r)$ and $(\tau_4, \bF, r)$, respectively,
    where
    \begin{align*}
        \tau_{3} &= \lambda_{\min}^{-3/2} ( \bF )  \sup_{\theta \in \Theta ( \theta_{\rm c}, \bF, r )} \left\| \nabla^{3} f(\theta) \right\|_{\rm op}, \\
        \tau_{4} &= \lambda_{\min}^{-2} ( \bF )  \sup_{\theta \in \Theta ( \theta_{\rm c}, \bF, r )} \left\| \nabla^{4} f(\theta) \right\|_{\rm op}.
    \end{align*}
\end{lemma}
\begin{proof}
Note that
\begin{align*}
&\sup_{u \in \Theta ( \bF, r)} \sup_{z \in \bbR^{p}} 
\dfrac{
\left| \langle \nabla^3 f(\theta_{\rm c} + u), z^{\otimes 3} \rangle \right|}{\left\| \bF^{1/2} z \right\|_{2}^{3}}
=
\sup_{u \in \Theta (\bF, r)} \sup_{z \in \bbR^{p}} 
\left| \left \langle \nabla^3 f(\theta_{\rm c} + u), \dfrac{z^{\otimes 3}}{\left\| \bF^{1/2} z \right\|_{2}^{3}} \right \rangle \right| \\
&\leq 
\lambda_{\min}^{-3/2} ( \bF )
\sup_{u \in \Theta (\bF, r)} \sup_{\substack{z \in \bbR^{p} : \| z \|_{2} = 1}} 
\left| \langle \nabla^3 f(\theta_{\rm c} + u), z^{\otimes 3} \rangle \right| \\
&\leq 
\lambda_{\min}^{-3/2} ( \bF) \sup_{\theta \in \Theta ( \theta_{\rm c}, \bF, r )} \left\| \nabla^{3} f(\theta) \right\|_{\rm op}.
\end{align*}
Also,
\begin{align*}
&\sup_{u \in \Theta ( \bF, r)} \sup_{z \in \bbR^{p}} 
\dfrac{
\left| \langle \nabla^4 f(\theta_{\rm c} + u), z^{\otimes 4} \rangle \right|}{\left\| \bF^{1/2} z \right\|_{2}^{4}}
=
\sup_{u \in \Theta (\bF, r)} \sup_{z \in \bbR^{p}} 
\left| \left \langle \nabla^4 f(\theta_{\rm c} + u), \dfrac{z^{\otimes 4}}{\left\| \bF^{1/2} z \right\|_{2}^{4}} \right\rangle \right| \\
&\leq 
\lambda_{\min}^{-2} ( \bF )
\sup_{u \in \Theta (\bF, r)} \sup_{\substack{z \in \bbR^{p} : \| z \|_{2} = 1}} 
\left| \langle \nabla^4 f(\theta_{\rm c} + u), z^{\otimes 4} \rangle \right| \\
&\leq 
\lambda_{\min}^{-2} ( \bF ) \sup_{\theta \in \Theta ( \theta_{\rm c}, \bF, r )} \left\| \nabla^{4} f(\theta) \right\|_{\rm op},
\end{align*}
which completes the proof.
\end{proof}

\subsection{Technical lemmas for TV distance}

\begin{lemma} \label{lemma:tech_Gaussian_comparison}
    For $\mu_{1}, \mu_{2} \in \bbR^{p}$ and $\bOmega_{1}, \bOmega_{2} \in \symmPD$, let $Q_1 = \cN(\mu_{1}, \bOmega_{1}^{-1}), Q_2 = \cN(\mu_{2}, \bOmega_{2}^{-1})$.
    Suppose that 
    \begin{align*}
        \left\| \bOmega_{2}^{-1/2} \bOmega_{1} \bOmega_{2}^{-1/2} - \bI_{p} \right\|_{2} \leq 0.684.
    \end{align*}
    Then,
    \begin{align*}
        d_{V} \left( Q_1, Q_2 \right) 
        \leq 
        \dfrac{1}{2} \Bigg( \left\| \bOmega_{1}^{1/2} \left( \mu_1 - \mu_2  \right)  \right\|_{2}^{2}  
        +  
        \left\| \bOmega_{2}^{-1/2} \bOmega_{1} \bOmega_{2}^{-1/2} - \bI_{p} \right\|_{\rm F}^{2}
        \Bigg)^{1/2}.
    \end{align*}
\end{lemma}
\begin{proof}
By Pinsker's inequality, we have
\begin{align*}
d_{V} \left( Q_1, Q_2 \right) \leq \bigg( \dfrac{1}{2} K \left( Q_1, Q_2 \right) \bigg)^{1/2}.
\end{align*}
By the definition of KL divergence, $K \left( Q_1, Q_2 \right)$ is equal to
\begin{align*}
\dfrac{1}{2}  \bigg[  \left\| \bOmega_{1}^{1/2} \left( \mu_1 - \mu_2  \right) \right\|_{2}^{2} + \operatorname{tr} \left( \bOmega_{2}^{-1/2} \bOmega_{1} \bOmega_{2}^{-1/2} - \bI_{p}  \right) - \operatorname{logdet} \left( \bOmega_{2}^{-1/2} \bOmega_{1} \bOmega_{2}^{-1/2} \right) \bigg].
\end{align*}
Let $(\lambda_{j})_{j \in [p]}$ be eigenvalues of $\mathbf{B} = \bOmega_{2}^{-1/2} \bOmega_{1} \bOmega_{2}^{-1/2} - \bI_{p}$. Then, the last display is represented by
\begin{align*}
&\dfrac{1}{2}  \bigg[  \left\| \bOmega_{1}^{1/2} \left( \mu_1 - \mu_2  \right) \right\|_{2}^{2} 
+ \sum_{j=1}^{p} \lambda_{j} - \sum_{j=1}^{p} \log (1 + \lambda_{j}) \bigg] \\   
&\leq  
\dfrac{1}{2}  \bigg[  \left\| \bOmega_{1}^{1/2} \left( \mu_1 - \mu_2  \right) \right\|_{2}^{2} 
+ \sum_{j=1}^{p} \lambda_{j} - \sum_{j=1}^{p} (\lambda_{j} - \lambda_{j}^{2}) \bigg] 
=
\dfrac{1}{2}  \bigg[  \left\| \bOmega_{1}^{1/2} \left( \mu_1 - \mu_2  \right) \right\|_{2}^{2} 
+ \sum_{j=1}^{p} \lambda_{j}^{2} \bigg] \\   
&=
\dfrac{1}{2}  \bigg[  \left\| \bOmega_{1}^{1/2} \left( \mu_1 - \mu_2  \right) \right\|_{2}^{2} 
+ \operatorname{tr}(\mathbf{B}^{2}) \bigg] 
=
\dfrac{1}{2}  \bigg[  \left\| \bOmega_{1}^{1/2} \left( \mu_1 - \mu_2  \right) \right\|_{2}^{2} 
+ \left\| \mathbf{B} \right\|_{\rm F}^{2} \bigg]. 
\end{align*}
where the first inequality holds by $\max_{j \in [p]}|\lambda_{j}| \leq 0.684$.
It follows that
\begin{align*}
d_{V} \left( Q_1, Q_2 \right) 
\leq \bigg( \dfrac{1}{2} K \left( Q_1, Q_2 \right) \bigg)^{1/2}    
\leq \dfrac{1}{2}  \bigg[  \left\| \bOmega_{1}^{1/2} \left( \mu_1 - \mu_2  \right) \right\|_{2}^{2} 
+ \left\| \mathbf{B} \right\|_{\rm F}^{2} \bigg]^{1/2}, 
\end{align*}
which completes the proof.
\end{proof}

\begin{lemma} \label{lemma:tech_Gaussian_comparison_lower}
    For $\mu_{1}, \mu_{2} \in \bbR^{p}$ and $\bOmega_{1}, \bOmega_{2} \in \symmPD$, let $Q_1 = \cN(\mu_{1}, \bOmega_{1}^{-1}), Q_2 = \cN(\mu_{2}, \bOmega_{2}^{-1})$.
    Suppose that 
    \begin{align*}
        d_{V} \left( Q_1, Q_2 \right) \leq  \dfrac{1}{600}.
    \end{align*}
    Let
    \begin{align*}
        \Delta = \left\| \bOmega_{1}^{1/2} \left( \mu_1 - \mu_2  \right)  \right\|_{2}  \vee \left\| \bOmega_{2}^{-1/2} \bOmega_{1} \bOmega_{2}^{-1/2} - \bI_{p} \right\|_{\rm F}.
    \end{align*}
    Then,
    \begin{align*}
        \dfrac{\Delta}{200} 
        \leq 
        d_{V} \left( Q_1, Q_2 \right) 
        \leq 
        \dfrac{\Delta}{\sqrt{2}}.
    \end{align*}
\end{lemma}
\begin{proof}
See Theorem 1.8 in \cite{arbas2023polynomial}.    
\end{proof}

\subsection{Technical lemmas for eigenvalue analysis}

\begin{lemma} \label{lemma:pre_theta_consistency}
    Suppose that (\textbf{A0})-(\textbf{A2}) hold. Also, assume that
    \begin{align} \label{assume:pre_theta_consistency_sample}
        n \geq C \left\| \bOmega_{t}^{1/2} \left( \theta_0 - \mu_{t} \right) \right\|_{2}^{2}
    \end{align}    
    on an event $\scrE$, where $C = C(K_{\min}, K_{\max})$ is a large enough constant.
    Then, on $\scrE$,
    \begin{align} \label{eqn:pre_theta_consistency_claim}
        \left\| \bF_{t+1, \theta_0}^{1/2} \left( \theta_0 - \thetaBest[t+1] \right) \right\|_{2} 
        \leq 
        \sqrt{2} \left\| \bOmega_{t}^{1/2} \left( \theta_0 - \mu_{t} \right) \right\|_{2}.
    \end{align}
\end{lemma}
\begin{proof}
 In this proof, we will work on the event $\scrE$ without explicitly mentioning it, and assume $C = C(K_{\min}, K_{\max})$ in \eqref{assume:pre_theta_consistency_sample} is large enough.
By the definition of $\thetaBest[t+1]$, we have
\begin{align*}
    \bbE_{t+1} L_{t+1}(\theta_0) - \dfrac{1}{2} \left\| \bOmega_{t}^{1/2} \left( \theta_0 - \mu_{t} \right) \right\|_{2}^{2} 
    &= \bbE_{t+1} \widetilde{L}_{t+1}(\theta_0)  
    \leq 
    \bbE_{t+1} \widetilde{L}_{t+1}(\thetaBest[t+1]) \\
    &= \bbE_{t+1} L_{t+1}(\thetaBest[t+1]) - \dfrac{1}{2} \left\| \bOmega_{t}^{1/2} \left( \thetaBest[t+1] - \mu_{t} \right) \right\|_{2}^{2} \\
    &\leq 
    \bbE_{t+1} L_{t+1}(\thetaBest[t+1]).
\end{align*}
It follows that 
\begin{align} \label{eqn:pre_theta_consistency_eq1}
    \bbE_{t+1} L_{t+1}(\thetaBest[t+1]) - \bbE_{t+1} L_{t+1}(\theta_0) \geq - \dfrac{1}{2} \left\| \bOmega_{t}^{1/2} \left( \theta_0 - \mu_{t} \right) \right\|_{2}^{2}.
\end{align}
In this proof, we denote $\Theta_{n, t+1} = \Theta (\theta_0, \Fisher[t+1]{\theta_0}, \sqrt{2} \| \bOmega_{t}^{1/2}(\theta_0 - \mu_{t})\|_{2} )$. 
For $\theta \in \Theta_{n, t+1}$, we have
\begin{align*}
    \left\| \theta - \theta_0 \right\|_{2} 
    &\leq \lambda_{\min}^{-1/2} (\Fisher[t+1]{\theta_0}) 
    \sqrt{2}\left\| \bOmega_{t}^{1/2} \left(\theta_0 - \mu_{t} \right) \right\|_{2} \\
    \overset{(\textbf{A2})}&{\leq} 
    (K_{\min} n)^{-1/2} \sqrt{2}\left\| \bOmega_{t}^{1/2} \left(\theta_0 - \mu_{t} \right) \right\|_{2}
    \overset{ \eqref{assume:pre_theta_consistency_sample} }{\leq} \dfrac{1}{2}.
\end{align*}
It follows that
\begin{align*}
    \Theta_{n, t+1} \subseteq \left\{ \theta \in \Theta : \left\| \theta - \theta_0 \right\|_{2} \leq 1/2 \right\}.
\end{align*}
By Lemma \ref{lemma:tech_smooth_tau_bound} and (\textbf{A2}), $\bbE_{t+1} L_{t+1}(\theta)$ satisfies the third order smoothness at $\theta_0$ with parameters
\begin{align*}
    \left( K_{\max} K_{\min}^{-3/2} n^{-1/2}, \: \Fisher[t+1]{\theta_0}, \: \sqrt{2} \left\| \bOmega_{t}^{1/2}(\theta_0 - \mu_{t}) \right\|_{2}  \right).
\end{align*}

We will prove \eqref{eqn:pre_theta_consistency_claim} by contradiction. Suppose $\thetaBest[t+1] \notin \Theta_{n, t+1}$.
Let 
\begin{align*}
    \partial \Theta_{n, t+1} = \left\{ \theta \in \Theta : \left\| \Fisher[t+1]{\theta_0}^{1/2} (\theta - \theta_0) \right\|_{2}  = \sqrt{2} \left\| \bOmega_{t}^{1/2}(\theta_0 - \mu_{t}) \right\|_{2} \right\}.
\end{align*}
For $\theta^{\circ} \in \partial \Theta_{n, t+1}$, we have
\begin{align*}
&\bbE_{t+1} L_{t+1}(\theta^{\circ}) - \bbE_{t+1} L_{t+1}(\theta_0)  \\
&\leq \bigg( \nabla \bbE_{t+1} L_{t+1}(\theta_0) \bigg)^{\top} \left( \theta^{\circ} - \theta_0 \right)  - \dfrac{1}{2} \inf_{\theta \in \Theta_{n, t+1}}  \left\| \Fisher[t+1]{\theta}^{1/2} (\theta^{\circ} - \theta_0) \right\|_{2}^{2} \\
&= - \dfrac{1}{2} \inf_{\theta \in \Theta_{n, t+1}}  \left\| \Fisher[t+1]{\theta}^{1/2}(\theta^{\circ} - \theta_0) \right\|_{2}^{2} \\
\overset{ \text{Lemma \ref{lemma:tech_Fisher_smooth}} }&{\leq}
-\dfrac{1}{2} \left( 1 -  K_{\max} K_{\min}^{-3/2} n^{-1/2} \sqrt{2} \left\| \bOmega_{t}^{1/2}(\theta_0 - \mu_{t}) \right\|_{2} \right) 
\left\| \Fisher[t+1]{\theta_0}^{1/2}(\theta^{\circ} - \theta_0) \right\|_{2}^{2} \\
&= 
-\dfrac{1}{2} \left( 1 -  K_{\max} K_{\min}^{-3/2} n^{-1/2} \sqrt{2} \left\| \bOmega_{t}^{1/2}(\theta_0 - \mu_{t}) \right\|_{2} \right) 2 \left\| \bOmega_{t}^{1/2}(\theta_0 - \mu_{t}) \right\|_{2}^{2} \\
\overset{ \eqref{assume:pre_theta_consistency_sample} }&{<}
-\dfrac{1}{2} \left\| \bOmega_{t}^{1/2}(\theta_0 - \mu_{t}) \right\|_{2}^{2}.
\end{align*}
where the second equality holds by the definition of $\partial \Theta_{n, t+1}$.
Consequently, we have
\begin{align*}
\bbE_{t+1} L_{t+1}(\thetaBest[t+1]) - \bbE_{t+1} L_{t+1}(\theta_0) 
< 
-\dfrac{1}{2} \left\| \bOmega_{t}^{1/2}(\theta_0 - \mu_{t}) \right\|_{2}^{2},
\end{align*}
by the concavity of the map $\theta \mapsto \bbE_{t+1}L_{t+1}(\theta)$, which contradicts to \eqref{eqn:pre_theta_consistency_eq1}. This completes the proof of \eqref{eqn:pre_theta_consistency_claim}.
\end{proof}

\begin{remark}
    The constant $C$ in \eqref{assume:pre_theta_consistency_sample} can be chosen as 
    \begin{align*}
        C = 8K_{\min}^{-1} \vee \big( 8 K_{\max}^{2} K_{\min}^{-3} + 1 \big).
    \end{align*}
\end{remark}

\begin{lemma} \label{lemma:ignorable_update}
    Suppose that (\textbf{A0})-(\textbf{A2}) hold. Also, on an event $\scrE$, assume that
    $\left\| \thetaBest - \theta_0 \right\|_{2} \leq 1/4$, and
    \begin{align} 
    \begin{aligned} \label{assume:ignorable_update_sample}
        n \geq C \Bigg(
        \left\| \Fisher[t+1]{\thetaBest}^{-1/2} \nabla \bbE_{t+1} L_{t+1}(\thetaBest) \right\|_{2}^{2} \vee \left\| \bOmega_{t}^{1/2} \left( \thetaBest[t] - \mu_{t} \right) \right\|_{2}^{2}
        \Bigg)
    \end{aligned}    
    \end{align}    
    for a large enough constant $C = C(K_{\min}, K_{\max})$.
    Then, on $\scrE$,
    \begin{align} \label{claim:ignorable_update}
    \begin{aligned}
        &\left\| \bF_{t+1, \thetaBest}^{1/2} \left( \thetaBest[t+1] - \thetaBest \right) \right\|_{2} \\
        &\leq 
        \bigg( 4 \left\| \Fisher[t+1]{\thetaBest}^{-1/2} \nabla \bbE_{t+1} L_{t+1}(\thetaBest) \right\|_{2} \bigg) \vee \bigg( 2 \left\| \bOmega_{t}^{1/2} \left( \thetaBest[t] - \mu_{t} \right) \right\|_{2} \bigg).
    \end{aligned}        
    \end{align}
\end{lemma}
\begin{proof}
In this proof, we work on the event $\scrE$ without explicitly referring it, and assume $C = C(K_{\min}, K_{\max})$ in \eqref{assume:ignorable_update_sample} is large enough.
By the definition of $\thetaBest[t+1]$, we have, for $t \in \{0, 1, ..., T -1\}$,
\begin{align*}
    &\bbE_{t+1} L_{t+1}(\thetaBest) - \dfrac{1}{2} \left\| \bOmega_{t}^{1/2} \left( \thetaBest - \mu_{t} \right) \right\|_{2}^{2} 
    = \bbE_{t+1} \widetilde{L}_{t+1}(\thetaBest)  
    \leq 
    \bbE_{t+1} \widetilde{L}_{t+1}(\thetaBest[t+1]) \\
    &= \bbE_{t+1} L_{t+1}(\thetaBest[t+1]) - \dfrac{1}{2} \left\| \bOmega_{t}^{1/2} \left( \thetaBest[t+1] - \mu_{t} \right) \right\|_{2}^{2}
    \leq 
    \bbE_{t+1} L_{t+1}(\thetaBest[t+1]).
\end{align*}
It follows that 
\begin{align} \label{eqn:ignorable_update_eq1}
    \bbE_{t+1} L_{t+1}(\thetaBest[t+1]) - \bbE_{t+1} L_{t+1}(\thetaBest) \geq - \dfrac{1}{2} \left\| \bOmega_{t}^{1/2} \left( \thetaBest - \mu_{t} \right) \right\|_{2}^{2}.
\end{align}
Let
\begin{align*}
    r_{t} = \bigg( 4 \left\| \Fisher[t+1]{\thetaBest}^{-1/2} \nabla \bbE_{t+1} L_{t+1}(\thetaBest) \right\|_{2} \bigg) \vee \bigg( 2 \left\| \bOmega_{t}^{1/2} \left( \thetaBest[t] - \mu_{t} \right) \right\|_{2} \bigg).
\end{align*}
In this proof, we denote $\Theta_{n, t+1} = \Theta (\thetaBest, \Fisher[t+1]{\thetaBest}, r_{t})$. 
For $\theta \in \Theta_{n, t+1}$, we have
\begin{align*}
    \left\| \theta - \thetaBest \right\|_{2} 
    \leq \lambda_{\min}^{-1/2} (\Fisher[t+1]{\thetaBest}) r_{t}
    \leq (K_{\min} n)^{-1/2} r_{t}
    \overset{\eqref{assume:ignorable_update_sample}}{\leq} \dfrac{1}{4},
\end{align*}
where the second inequality holds by $\| \thetaBest - \theta_0 \|_{2} \leq 1/4$ and (\textbf{A2}).
It follows that
\begin{align*}
    \Theta_{n, t+1} \subseteq \left\{ \theta \in \Theta : \left\| \theta - \theta_0 \right\|_{2} \leq 1/2 \right\}.
\end{align*}
By Lemma \ref{lemma:tech_smooth_tau_bound} and (\textbf{A2}), $\bbE_{t+1} L_{t+1}(\theta)$ satisfies the third order smoothness at $\thetaBest$ with parameters
\begin{align*}
    \left( K_{\max} K_{\min}^{-3/2} n^{-1/2}, \: \Fisher[t+1]{\thetaBest}, \: r_t \right).
\end{align*}

Next, we will prove \eqref{claim:ignorable_update} by contradiction. Suppose $\thetaBest[t+1] \notin \Theta_{n, t+1}$.
Let 
\begin{align} \label{eqn:ignorable_update_boundary_set}
    \partial \Theta_{n, t+1} = \left\{ \theta \in \Theta : \left\| \Fisher[t+1]{\thetaBest}^{1/2} (\theta - \thetaBest) \right\|_{2} = r_{t} \right\}.
\end{align}
For $\theta^{\circ} \in \partial \Theta_{n, t+1}$, Taylor's theorem gives
\begin{align*}
&\bbE_{t+1} L_{t+1}(\theta^{\circ}) - \bbE_{t+1} L_{t+1}(\thetaBest)  \\
&\leq \bigg( \nabla \bbE_{t+1} L_{t+1}(\thetaBest) \bigg)^{\top} \left( \theta^{\circ} - \thetaBest \right)  - \dfrac{1}{2} \inf_{\theta \in \Theta_{n, t+1}}  \left\| \Fisher[t+1]{\theta}^{1/2} (\theta^{\circ} - \thetaBest) \right\|_{2}^{2} \\
&= 
\bigg( \Fisher[t+1]{\thetaBest}^{-1/2} \nabla \bbE_{t+1} L_{t+1}(\thetaBest) \bigg)^{\top} \Fisher[t+1]{\thetaBest}^{1/2} \left( \theta^{\circ} - \thetaBest \right)  - \dfrac{1}{2} \inf_{\theta \in \Theta_{n, t+1}}  \left\| \Fisher[t+1]{\theta}^{1/2} (\theta^{\circ} - \thetaBest) \right\|_{2}^{2} \\
\overset{\text{Lemma \ref{lemma:tech_Fisher_smooth}}}&{\leq} 
\left\| \Fisher[t+1]{\thetaBest}^{-1/2} \nabla \bbE_{t+1} L_{t+1}(\thetaBest) \right\|_{2}
\left\| \Fisher[t+1]{\thetaBest}^{1/2} \left( \theta^{\circ} - \thetaBest \right) \right\|_{2} \\
&\qquad \quad  - \dfrac{1}{2} \left( 1 - K_{\max} K_{\min}^{-3/2} n^{-1/2} r_{t} \right) 
\left\| \Fisher[t+1]{\thetaBest}^{1/2} (\theta^{\circ} - \thetaBest) \right\|_{2}^{2} \\
\overset{\eqref{eqn:ignorable_update_boundary_set}}&{=} 
\bigg[ 
\left\| \Fisher[t+1]{\thetaBest}^{-1/2} \nabla \bbE_{t+1} L_{t+1}(\thetaBest) \right\|_{2}
- \dfrac{1}{2} \left( 1 - K_{\max} K_{\min}^{-3/2} n^{-1/2} r_{t} \right) r_{t}
\bigg] r_{t} \\
\overset{\eqref{assume:ignorable_update_sample}}&{<}
\bigg[ 
\dfrac{1}{4} r_{t}
- \dfrac{1}{2} \left( 1 - \dfrac{1}{4} \right) r_{t}
\bigg] r_{t} 
= -\dfrac{1}{8} r_{t}^{2}.
\end{align*}
Consequently, 
\begin{align*}
\bbE_{t+1} L_{t+1}(\thetaBest[t+1]) - \bbE_{t+1} L_{t+1}(\thetaBest) 
< -\dfrac{1}{8} r_{t}^{2} \leq -\dfrac{1}{2} \left\| \bOmega_{t}^{1/2} \left( \thetaBest[t] - \mu_{t} \right) \right\|_{2}^{2},
\end{align*}
by the concavity of the map $\theta \mapsto \bbE_{t+1} L_{t+1}(\theta)$, which contradicts to \eqref{eqn:ignorable_update_eq1}. This completes the proof of \eqref{claim:ignorable_update}.
\end{proof}

\begin{lemma} \label{lemma:ignorable_update2}
    Suppose that (\textbf{A0})-(\textbf{A2}) hold. Let $\alpha \in [1/2, 1]$.
    Also, assume that there exist some constants $D_1, D_2, D_3, D_{4} > 0$ such that
    \begin{align}
    \begin{aligned} \label{assume:ignorable_update_2}
        &\left\| \FisherBest^{1/2} \left( \thetaBest - \theta_0 \right)\right\|_{2} \leq D_1 M_n t^{\alpha} \sqrt{p_{\ast}}, \qquad 
        \left\| \FisherBest^{1/2} \left( \thetaMAP - \thetaBest \right)\right\|_{2} \leq D_2 M_n \sqrt{t^{-1} p_{\ast}}, \\
        &\left\| \bOmega_{t}^{1/2}(\thetaMAP - \mu_{t}) \right\|_{2} \leq  D_3 M_n \sqrt{t^{-1} p_{\ast}}, \qquad 
        \lambda_{\min} \left( \FisherBest \right) \geq D_4 nt
    \end{aligned}
    \end{align}    
    on an event $\scrE$.
    Assume further that, on $\scrE$,
    \begin{align} 
    \begin{aligned} \label{assume:ignorable_update_2_2}
        \left\| \FisherMAP^{-1/2} \bOmega_{t} \FisherMAP^{-1/2} - \bI_{p} \right\|_{2} \leq \dfrac{1}{4}, \quad 
        \left\| \FisherBest^{-1/2} \FisherMAP \FisherBest^{-1/2} - \bI_{p} \right\|_{2} \leq \dfrac{1}{8}, \quad 
        \left\| \thetaBest - \theta_0 \right\|_{2} \leq \dfrac{1}{4},
    \end{aligned}        
    \end{align}
    and 
    \begin{align} 
    \begin{aligned} \label{assume:ignorable_update_2_sample}
        n \geq C M_n^{2} t^{2\alpha-1} p_{\ast},
    \end{aligned}        
    \end{align}
    for a large enough constant $C = C(K_{\min}, K_{\max}, D_1, D_2, D_3, D_4)$.
    Then, on $\scrE$,
    \begin{align*}
        \left\| \bF_{t+1, \thetaBest}^{1/2} \left( \thetaBest[t+1] - \thetaBest \right) \right\|_{2} 
        \leq 
        K M_n t^{\alpha-1/2} \sqrt{p_{\ast}},
    \end{align*}
    where $K = K(D_1, D_2, D_3, D_4, K_{\min}, K_{\max})$.
\end{lemma}
\begin{proof}
In this proof, we work on the event $\scrE$ without explicitly referring it, and assume that $C = C(K_{\min}, K_{\max}, D_1, D_2, D_3, D_4)$ in \eqref{assume:ignorable_update_2_sample} is sufficiently large.
Note that
\begin{align*}
&\left\| \bOmega_{t}^{1/2}(\thetaBest - \mu_{t}) \right\|_{2} 
\leq
\left\| \bOmega_{t}^{1/2}( \thetaMAP - \thetaBest) \right\|_{2} 
+
\left\| \bOmega_{t}^{1/2}(\thetaMAP - \mu_{t}) \right\|_{2}  \\
&\leq 
\left\| \bOmega_{t}^{1/2} \FisherMAP^{-1/2} \right\|_{2} 
\left\| \FisherMAP^{1/2} \FisherBest^{-1/2} \right\|_{2}
\left\| \FisherBest^{1/2}(\thetaMAP - \thetaBest) \right\|_{2} 
+
\left\| \bOmega_{t}^{1/2}(\thetaMAP - \mu_{t}) \right\|_{2}  \\
&\leq 
\left( 1 + \left\| \FisherMAP^{-1/2} \bOmega_{t} \FisherMAP^{-1/2} - \bI_{p} \right\|_{2} \right)^{1/2}
\left( 1 + \left\| \FisherBest^{-1/2} \FisherMAP \FisherBest^{-1/2} - \bI_{p} \right\|_{2} \right)^{1/2} 
\left\| \FisherBest^{1/2}(\thetaMAP - \thetaBest) \right\|_{2} \\
&\qquad + \left\| \bOmega_{t}^{1/2}(\thetaMAP - \mu_{t}) \right\|_{2} \\
\overset{\eqref{assume:ignorable_update_2_2}}&{\leq}
\dfrac{3}{2} \left\| \FisherBest^{1/2}(\thetaMAP - \thetaBest) \right\|_{2} 
+ \left\| \bOmega_{t}^{1/2}(\thetaMAP - \mu_{t}) \right\|_{2} \\
\overset{\eqref{assume:ignorable_update_2}}&{\leq}
\dfrac{3}{2}\bigg[ D_2 M_n \sqrt{t^{-1} p_{\ast}} + \dfrac{2}{3} D_3 M_n \sqrt{t^{-1} p_{\ast}} \bigg] 
\leq 
3 \left( D_2 \vee D_3 \right) M_n \sqrt{t^{-1} p_{\ast}}.
\end{align*}
Also, by Taylor's theorem and $\nabla \bbE_{t+1} L_{t+1}(\theta_0) = 0$, we have
\begin{align*}
    \nabla \bbE_{t+1} L_{t+1}(\thetaBest) 
    &= 
    \nabla \bbE_{t+1} L_{t+1}(\thetaBest)  - \nabla \bbE_{t+1} L_{t+1}(\theta_0) 
    =
    - \overline \bF_{t+1} (\thetaBest, \theta_0) \left( \thetaBest - \theta_0 \right) \\
    &=
    - \overline \bF_{t+1} (\thetaBest, \theta_0) \FisherBest^{-1/2} \FisherBest^{1/2} \left( \thetaBest - \theta_0 \right)
\end{align*}
where
\begin{align*}
    \overline \bF_{t+1} (\thetaBest, \theta_0) 
    = -\int_{0}^{1} \nabla^{2} \bbE_{t+1} L_{t+1} \big( s\thetaBest + (1-s) \theta_0 \big) \rmd s
    = -\int_{0}^{1} \nabla^{2} L_{t+1} \big( s\thetaBest + (1-s) \theta_0 \big) \rmd s.
\end{align*}
Note that $s\thetaBest + (1-s) \theta_0 \in \Theta(\theta_0, \bI_{p}, 1/2)$ for all $s \in [0, 1]$ because $\left\| \thetaBest - \theta_0 \right\|_{2} \leq 1/4$.
It follows that
\begin{align*}
    &\left\| \Fisher[t+1]{\thetaBest}^{-1/2} \nabla \bbE_{t+1} L_{t+1}(\thetaBest) \right\|_{2}
    = 
    \left\| \Fisher[t+1]{\thetaBest}^{-1/2} \overline \bF_{t+1} (\thetaBest, \theta_0) \FisherBest^{-1/2} \FisherBest^{1/2} \left( \thetaBest - \theta_0 \right) \right\|_{2} \\
    &\leq 
    \left\| \Fisher[t+1]{\thetaBest}^{-1/2} \overline \bF_{t+1} (\thetaBest, \theta_0)^{1/2} \right\|_{2}
    \left\| \overline \bF_{t+1} (\thetaBest, \theta_0)^{1/2} \FisherBest^{-1/2} \right\|_{2}
    \left\| \FisherBest^{1/2} \left( \thetaBest - \theta_0 \right) \right\|_{2}  \\
    &\leq
    \lambda_{\min}^{-1/2} \big( \Fisher[t+1]{\thetaBest} \big)
    \lambda_{\max} \big( \overline \bF_{t+1} (\thetaBest, \theta_0) \big)
    \lambda_{\min}^{-1/2} \big( \FisherBest \big)    
    \left\| \FisherBest^{1/2} \left( \thetaBest - \theta_0 \right) \right\|_{2} \\
    \overset{\eqref{assume:ignorable_update_2}, (\textbf{A2})}&{\leq}
    \big( K_{\min} n \big)^{-1/2}
    \big( K_{\max} n \big)
    \big( D_4 n t \big)^{-1/2}
    \big( D_{1} M_n t^{\alpha} \sqrt{p_{\ast}} \big) \\
    &= 
    \big( K_{\min}^{-1/2} D_{4}^{-1/2} K_{\max} D_{1} \big)
    \big( M_n t^{\alpha-1/2} \sqrt{p_{\ast}} \big).
\end{align*}
Since
\begin{align*}
    \left\| \bOmega_{t}^{1/2}(\thetaBest - \mu_{t}) \right\|_{2} 
    \leq c_1 M_n \sqrt{t^{-1} p_{\ast}}, \quad 
    \left\| \Fisher[t+1]{\thetaBest}^{-1/2} \nabla \bbE_{t+1} L_{t+1}(\thetaBest) \right\|_{2} 
    \leq c_2 M_n t^{\alpha-1/2} \sqrt{p_{\ast}}
\end{align*}
for some positive constants $c_1 = c_1(D_2, D_3)$ and $c_2 = c_2(K_{\min}, K_{\max}, D_1, D_4)$, combining \eqref{assume:ignorable_update_2_sample} and the assumption $\| \thetaBest - \theta_0 \|_{2} \leq 1/4$, we can utilize Lemma \ref{lemma:ignorable_update}.
Therefore, Lemma \ref{lemma:ignorable_update} gives
\begin{align*}
    \left\| \bF_{t+1, \thetaBest}^{1/2} \left( \thetaBest[t+1] - \thetaBest \right) \right\|_{2} 
    \leq 
    \bigg( 4 \left\| \Fisher[t+1]{\thetaBest}^{-1/2} \nabla \bbE_{t+1} L_{t+1}(\thetaBest) \right\|_{2} \bigg)
    \vee 
    \bigg( 2 \left\| \bOmega_{t}^{1/2} \left( \thetaBest[t] - \mu_{t} \right) \right\|_{2} \bigg),
\end{align*}
which further upper bounded by
\begin{align*}
    \big( 4c_2 M_n t^{\alpha-1/2} \sqrt{p_{\ast}} \big)
    \vee
    \big(  2c_1 M_n \sqrt{t^{-1} p_{\ast}} \big)
    \leq 
    \big( 2c_1 + 4c_2 \big) M_n t^{\alpha-1/2} \sqrt{p_{\ast}},
\end{align*}
which completes the proof.
\end{proof}

\subsection{Deviation bounds for (sub-) Gaussian random vectors}

\begin{lemma} \label{lemma:tech_Gaussian_Chaos}
For $\bA \in \symmPD$ and $\mathbf{B} \in \symmPD$, let $Z \sim \cN(0, \bA^{-1})$ and $\bOmega = \bA^{-1/2} \mathbf{B} \bA^{-1/2}$.
Then, for every $\omega \geq 0$,
\begin{align*}
    \bbP \bigg( \left| \langle \mathbf{B} Z, Z \rangle - \operatorname{tr}\left( \bOmega \right) \right| >  2\left\| \bOmega \right\|_{\rm F}\sqrt{\omega} +  2\left\| \bOmega \right\|_{2}\omega \bigg)
    \leq 2e^{-\omega}.
\end{align*}
Furthermore, 
\begin{align*}
    \bbP \bigg( \left\| \mathbf{B}^{1/2} Z \right\|_{2}  >  \sqrt{\operatorname{tr}\left( \bOmega \right)} + \sqrt{2\left\| \bOmega \right\|_{2}\omega} \bigg)
    \leq e^{-\omega}.
\end{align*}
\end{lemma}
\begin{proof}
    See Theorem B.4 in \cite{spokoiny2024estimation}.
\end{proof}

\begin{lemma} \label{lemma:tech_subGaussian_Chaos}
Let $\bA \in \bbR^{p \times p} \succeq 0$, and $Z = (Z_i)_{i=1}^{n} \in \bbR^{n} \sim \operatorname{SubG}(\sigma^{2})$ be a random vector whose components are independent with $\bbE Z = 0$
Then, for every $\omega \geq 0$,
\begin{align*}
    \bbP \bigg( \langle \bA Z, Z \rangle  >  \sigma^{2} \bigg[ \operatorname{tr}\left( \bA \right) + 2\sqrt{\operatorname{tr}\left( \bA^{2} \right) \omega} + 2\left\| \bA \right\|_{2} \omega \bigg] \bigg)
    \leq e^{-\omega}.
\end{align*}
\end{lemma}
\begin{proof}
    See Theorem 1 in \cite{hsu2012tail}.
\end{proof}

\begin{lemma} \label{lemma:tech_Gaussian_Moments}
Let $Z \sim \cN(0, \bI_{p})$ and $\bA \in \symmPD$.
Then, 
\begin{align*}
    \bbE \langle \bA Z, Z \rangle^{4} \leq \big( \operatorname{tr}(\bA) + 3\left\| \bA \right\|_{2}  \big)^{4}.
\end{align*}
\end{lemma}
\begin{proof}
    See Theorem B.1 in \cite{spokoiny2024estimation}.
\end{proof}

\begin{lemma} \label{lemma:tech_logconcave_abs_moments}
Let $(Z_{i})_{i \in [n]}$ be \iid\ copies of an isotropic and log-concave probability measure on $\bbR^{p}$. 
Suppose that $n \geq p$.
Then, for any $k \geq 2$, there exists some constants $K_1 = K_1(k) > 1$, $K_2 > 0$ and $K_3 = K_3(k) > 0$ such that
\begin{align*}
   &\sup_{u \in \bbR^{p} : \| u \|_{2} = 1} \left| \dfrac{1}{n} \sum_{i=1}^{n} \bigg( \left| \langle Z_i, u \rangle \right|^{k} - \bbE \left| \langle Z_i, u \rangle \right|^{k} \bigg) \right| \\
   &\leq 
   K_1 t s^{k-1} \log^{k-1} \left( \dfrac{2n}{p} \right) \sqrt{\dfrac{p}{n}} + K_1\dfrac{ s^{k} p^{k/2} }{n} + K_1 \left( \dfrac{p}{2n} \right)^{s}, \quad \forall s, t \geq 1
\end{align*}
with probability at least 
\begin{align*}
    1 - \exp \left( - K_2 s \sqrt{p} \right) - \exp \left( - K_3 \Bigg\{ \bigg[ t^{2} s^{2k-2} p \log^{(2k-2)}\left(\dfrac{2n}{p} \right) \bigg] \wedge \bigg[ t s^{-1} \sqrt{np} \log^{-1}\left( \dfrac{2n}{p} \right) \bigg] \Bigg\} \right).
\end{align*}
\end{lemma}

\begin{proof}
    See Proposition 4.4 in \cite{adamczak2010quantitative}.
\end{proof}

\subsection{3-order Gaussian tensor analysis}
For $\theta, u \in \Theta$ and a three times differentiable function $f : \Theta \rightarrow \bbR$, let
\begin{align} 
\begin{aligned} \label{def:remainder_3_4_supp}
\cR_{3, f}(\theta, u) 
    &= 
    f(\theta + u) - f(\theta) 
    - \langle \nabla f(\theta), u \rangle 
    - \dfrac{1}{2} \langle \nabla^2 f(\theta), u^{\otimes 2} \rangle, \\
\cR_{4, f} (\theta, u) 
    &= 
    f(\theta + u) - f(\theta) 
    - \langle \nabla f(\theta), u \rangle 
    - \dfrac{1}{2} \langle \nabla^2 f(\theta), u^{\otimes 2} \rangle 
    - \dfrac{1}{6} \langle \nabla^3 f(\theta), u^{\otimes 3} \rangle. 
\end{aligned}
\end{align}
For a $3$-order symmetric tensor $\bT = (T_{ijk})_{i,j,k \in [p]} \in \bbR^{p \times p \times p}$, let
\begin{align*}
    \bT(u) = \langle \bT, u^{\otimes 3} \rangle, \quad \bT_{i} = (T_{ijk})_{j, k \in [p]} \bbR^{p \times p}, \quad \left\| \bT \right\|_{\rm F}^{2} = \sum_{i,j,k \in [p]} T_{ijk}^{2}.
\end{align*}

The following lemmas are from the Section B.7.1 in \cite{spokoiny2024estimation}. We reproduce them here for the sake of readability and completeness of proof.

\begin{lemma} \label{lemma:tech_Gaussian_Tensor_expectation_1}
For a symmetric $3$-order tensor $\bT \in \bbR^{p \times p \times p}$, let $Z = (Z_{i})_{i \in [p]} \sim \cN(0, \bI_{p})$ and $M = \left( M_i \right)_{i \in [p]}$, where $M_i = \operatorname{tr}(\bT_i)$.
Then, 
\begin{align*}
    \bbE \bigg( \bT(Z) - 3 \langle M, Z \rangle \bigg)^{2} = 6 \left\| \bT \right\|_{\rm F}^{2}, \quad 
    \bbE \bT^{2}(Z) = 6 \left\| \bT \right\|_{\rm F}^{2} + 9\left\| M \right\|_{2}^{2}.
\end{align*}
\end{lemma}
\begin{proof}
See Lemma B.32 in \cite{spokoiny2024estimation}.
\end{proof}

\begin{lemma} \label{lemma:tech_Gaussian_Tensor_expectation}
For a symmetric $3$-order tensor $\bT \in \bbR^{p \times p \times p}$, suppose that there exist some $\bF \in \symmPD$ and $\tau_3 \geq 0$ such that
\begin{align*}
    \bT(u) = \langle \bT , u^{\otimes 3} \rangle \leq  \tau_{3} \left\| \bF u \right\|_{2}^{3}, \quad \forall u \in \bbR^{p}.
\end{align*}
Let $\widetilde{Z} \sim \cN(0, \bD^{-1})$ for some $\bD \in \symmPD$ and $\bV = \bD^{-1/2} \bF \bD^{-1/2}$.
Then,
\begin{align*}
    \bbE [\{\bT(\widetilde{Z})\}^2] \leq 15 \tau_{3}^{2} \left\| \bV \right\|_{2} \operatorname{tr}^{2}(\bV)
\end{align*}
\end{lemma}
\begin{proof}
    See Lemma B.36 in \cite{spokoiny2024estimation}.
\end{proof}

\begin{lemma} \label{lemma:tech_mgf_approximation}
For a three times differentiable function $f : \Theta \rightarrow \bbR$ and $\theta \in \Theta$, suppose that $f$ satisfies the third order smoothness at $\theta$ with parameters $(\tau_3, \bF, r)$, where $\bF \in \symmPD$ and $\tau_3, r \geq 0$.
Let 
$\widetilde{Z} \sim \cN(0, \bD^{-1} )$ for some $\bD \in \symmPD$ and $\bV = \bD^{-1/2} \bF \bD^{-1/2}$. 
Then, for a random variable $G$ with $|G| \leq 1$,
\begin{align*}
\left| \bbE \left( e^{X} - 1 - X - \dfrac{X^2}{2} \right) G \right| \leq  \dfrac{5}{3} \epsilon^3 \exp(\epsilon^2),
\end{align*}
where 
\begin{align*}
    X = \cR_{3, f}(\theta, \widetilde{Z}) - \bbE \bigg[ \cR_{3, f}(\theta, \widetilde{Z}) \mathds{1}_{\Theta (\bF, r)}(\widetilde{Z}) \bigg], \quad 
    \epsilon = \tau_{3} \left\| \bV \right\|_{2} r^{2}/2
\end{align*}
\end{lemma}
\begin{proof}
    Combining with Lemma B.42 in \cite{spokoiny2024estimation}, this lemma is a special case of Lemma B.39 in \cite{spokoiny2024estimation}. The proof can be found therein.
\end{proof}

\end{document}